\documentclass[11pt,reqno]{amsart}

\numberwithin{equation}{section}
\usepackage{times}
\usepackage{amsmath,amsfonts,amstext,amssymb,amsbsy,amsopn,amsthm,eucal}
\usepackage{txfonts}
\usepackage{dsfont}
\usepackage{graphicx}   
\usepackage{hyperref}
\usepackage{accents}
\usepackage{enumerate}
\usepackage{xcolor}
\usepackage{verbatim}

\newcommand{\Lip}{\text{Lip}}

\newcommand{\Rm}{{\rm Rm}}
\newcommand{\Ric}{{\rm Ric}}

\newcommand{\Vol}{{\rm Vol}}
\newcommand{\diam}{{\rm diam}}

\newcommand{\inj}{{\rm inj}}
\newcommand{\rv}{{\rm v}}

\newcommand{\dN}{\mathds{N}}
\newcommand{\dP}{\mathds{P}}
\newcommand{\dQ}{\mathds{Q}}
\newcommand{\dR}{\mathbb{R}}

\newcommand{\dZ}{\mathds{Z}}

\newcommand{\cA}{\mathcal{A}}

\newcommand{\cC}{\mathcal{C}}

\newcommand{\cE}{\mathcal{E}}

\newcommand{\cH}{\mathcal{H}}

\newcommand{\cL}{\mathcal{L}}

\newcommand{\cN}{\mathcal{N}}

\newcommand{\cR}{\mathcal{R}}
\newcommand{\cS}{\mathcal{S}}

\newcommand{\cV}{\mathcal{V}}

\newtheorem{theorem}[equation]{Theorem}

\newtheorem{proposition}[equation]{Proposition}
\newtheorem{lemma}[equation]{Lemma}

\theoremstyle{definition}
\newtheorem{definition}[equation]{Definition}

\theoremstyle{remark}
\newtheorem{remark}[equation]{Remark}
\theoremstyle{remark}
\newtheorem{example}[equation]{Example}
\theoremstyle{remark}

\theoremstyle{remark}
\theoremstyle{remark}

\begin{document}

\author{Wenshuai Jiang and Aaron Naber}
\thanks{}
\thanks{}

\title[$L^2$ Curvature Bounds on Manifolds with Bounded Ricci Curvature]{$L^2$ Curvature Bounds on Manifolds with Bounded Ricci Curvature}

\address{W. Jiang, School of mathematical sciences, Zhejiang University, Zheda Road 38, hangzhou, zhejiang, China, 310027}
\address{W. Jiang,  School of Mathematics and Statistics, University of Sydney, Carslaw Building F07, Eastern Ave, Camperdown NSW 2006, Australia}
\email{wsjiang@zju.edu.cn}
\address{A. Naber, Department of Mathematics, 2033 Sheridan Rd., Evanston, IL 60208-2370}
\email{anaber@math.northwestern.edu}
\date{\today}
\maketitle

\begin{abstract}
Consider a Riemannian manifold with bounded Ricci curvature $|\Ric|\leq n-1$ and the noncollapsing lower volume bound $\Vol(B_1(p))>\rv>0$.  The first main result of this paper is to prove that we have the $L^2$ curvature bound $\fint_{B_1(p)}|\Rm|^2(x)\,dx < C(n,\rv)$, which proves the $L^2$ conjecture.  In order to prove this, we will need to first show the following structural result for limits.  Namely, if $(M^n_j,d_j,p_j) \longrightarrow (X,d,p)$ is a $GH$-limit of noncollapsed manifolds with bounded Ricci curvature, then the singular set $\cS(X)$ is $n-4$ rectifiable with the uniform Hausdorff measure estimates $H^{n-4}\big(\cS(X)\cap B_1\big)<C(n,\rv)$, which in particular proves the $n-4$-finiteness conjecture of Cheeger-Colding.  We will see as a consequence of the proof that for $n-4$ a.e. $x\in \cS(X)$ that the tangent cone of $X$ at $x$ is unique and isometric to $\dR^{n-4}\times C(S^3/\Gamma_x)$ for some $\Gamma_x\subseteq O(4)$ which acts freely away from the origin.
\end{abstract}

\tableofcontents

\section{Introduction}\label{s:intro}

The focus of this paper is to understand the regularity of Riemannian manifolds under the bounded Ricci and noncollapsing assumptions
\begin{align}\label{e:ricci_assumptions}
|\Ric_{M^n}|\leq n-1\, ,\;\;\;\;\Vol(B_1(p))>\rv>0\, .
\end{align}
A closely related problem, which will also play a central focus in this paper, is the study of Gromov-Hausdorff limit spaces
\begin{align}\label{e:ricci_assumptions_limit}
	(M^n_j,g_j,p_j)\stackrel{d_{GH}}{\longrightarrow} (X,d,p)\, ,
\end{align}
where the $M_j$ satisfy \eqref{e:ricci_assumptions}.\\

There is a good deal of history in studying the regularity of spaces satisfying \eqref{e:ricci_assumptions} or \eqref{e:ricci_assumptions_limit}.  Much of the early work focused on closed $4$-manifolds under the additional assumptions of bounded topology and bounded diameter.  The key use of these assumptions is that by the Chern-Gauss-Bonnet formula one can conclude in the four dimensional case that
\begin{align}
\chi(M^4)<A \implies \int_M |\Rm|^2(x)dx < C(A,\diam(M^4),\rv)\, .	
\end{align}

Though a series of paper \cite{A89},\cite{BKN89},\cite{T90} this was used to conclude that a limit space $X^4$ of four manifolds with bounded Ricci, topology, diameter, and uniform lower volume bounds would be a Riemannian orbifold with at most isolated singularities.  In particular, one gets from this that the singular set is of codimension four with bounded $n-4$ measure.  Though the singular set may not be orbifold in higher dimensions, it was conjectured during this time period that the codimension four and uniform finiteness would continue to hold in arbitrary dimension under only the bounded Ricci and noncollapsing assumptions of \eqref{e:ricci_assumptions}.  In particular, even in dimension four this was not understood because of the need for the topology and diameter assumptions.  There is a related conjecture was by Anderson \cite{Anderson_ICM94} which can be summarized as saying that the results in dimension four should not require the bounded topology and diameter assumptions.\\

A good deal of work toward these conjectures were done over the years.  In \cite{CCTi_eps_reg}, \cite{Cheeger}, \cite{CheegerTian05} a variety of new estimates were proved for spaces with bounded Ricci curvature with $L^p$ bounds on curvature.  In particular, the codimension four conjecture was proved in any dimension under the additional assumption of a $L^2$-bound on the curvature.  Though the statement was similar in nature to the four dimensional result, the techniques to exploit the $L^2$-bound in higher dimensions are substantially more involved.  This was extended in \cite{Cheeger} where under an assumed $L^2$ curvature bound it was proved that the {\it nonexceptional} part of the singular set was rectifiable.\\

In \cite{CheegerNaber_Codimensionfour} the codimension four conjecture was resolved in full generality, the proof of which required a variety of new estimates and techniques.  In addition, it was shown in \cite{CheegerNaber_Codimensionfour} that an improved result held in dimension four, where it could be proved under the assumptions of \eqref{e:ricci_assumptions} the $L^2$-bound on curvature was in fact automatic, and did not require any apriori assumptions on topology and diameter.  As a consequence, one could in fact show the topology was itself automatically bounded, which was a conjecture of Anderson \cite{Anderson_ICM94}.  In higher dimensions, only weaker curvature estimates were obtained in \cite{CheegerNaber_Codimensionfour}, where it was shown that the curvature $\Rm$ had apriori $L^p$-bounds for all $p<2$.  It was conjectured however in \cite{CheegerNaber_Codimensionfour},\cite{Na_14ICM} that the full $L^2$ curvature bound should hold in any dimension.\\

After the work of \cite{CheegerNaber_Codimensionfour}, there were still several open questions left over.  First, although the singular set was shown to be of codimension four, the $n-4$ finiteness conjecture of Cheeger-Colding \cite{ChC2} was still left open.  Additionally, no actual structure of the singular set was itself understood from the results of \cite{CheegerNaber_Codimensionfour}.  Fundamentally, this is because the new techniques of \cite{CheegerNaber_Codimensionfour} are built to deal with the codimension two part of the singular set under lower Ricci curvature assumptions, so that one can eventually rule out its existence entirely under the two sided Ricci bound.  Once the singular set is shown to not have a codimension two or three piece to it, it is automatically codimension four, but no information about the singular set is actually provided by such a construction.\\

The primary goal of this paper is to deal with these open questions.  We will  prove in Theorem \ref{t:main_L2_estimate} the $L^2$-conjecture of \cite{Na_14ICM},\cite{CheegerNaber_Codimensionfour}, by showing that under the assumptions of \eqref{e:ricci_assumptions} there exists the $L^2$ curvature bound
\begin{align}\label{e:intro:L2_est}
\fint_{B_1(p)} |\Rm|^2(x)\,dx < C(n,\rv).	
\end{align}

Additionally, in the case of a limit space $X$ of manifolds satisfying \eqref{e:ricci_assumptions} we will show that not only is the singular set $\cS(X)$ of codimension four, but we have the finiteness estimate
\begin{align}\label{e:intro:finiteness}
H^{n-4}(\cS(X)\cap B_1) < C(n,\rv)\, ,	
\end{align}
which proves the $n-4$ finiteness conjecture of Cheeger-Colding in \cite{ChC2}.  Structurally we will also show in Theorem \ref{t:main_limits} that the singular set $\cS(X)$ is $n-4$-rectifiable.  In fact, it follows from \cite{Cheeger}, \cite{Chen_Donaldson14} that \eqref{e:intro:L2_est} implies \eqref{e:intro:finiteness}, however our proof will necessarily go in the other direction, since we will be forced to tackle \eqref{e:intro:finiteness} first by very different means as a stepping stone toward \eqref{e:intro:L2_est}. \\

We will give a more complete description of the proof in Section \ref{ss:intro:outline}, however let us mention that there are several new points involved, many of which revolve around what we call a neck region.  In short, we will effectively decompose our manifold $M$ into two types of pieces: neck regions, which look like the singular space $\dR^{n-4}\times C(S^3/\Gamma)$ on many scales, and $\epsilon$-regularity regions which have scale invariant uniform curvature bounds.  The proof of the $L^2$ estimate will then rely on both our ability to give effective estimates for the number of pieces in this decomposition, and our ability to do more refined analysis on the neck regions themselves.  The challenge of doing analysis on the neck region is that there are an uncontrollable number of scales involved in a neck region, and to get global information one needs estimates which are {\it summable} and small over all these scales.  These estimates will depend heavily on a new type of superconvexity estimate which we will prove for the $L^1$ hessian of a harmonic function on these neck regions.  We refer the reader to Section \ref{ss:intro:outline}, where we give a much more detailed outline of the proof.\\

\subsection{Main Results on Manifolds}\label{ss:intro:main_manifolds}

Let us now discuss in precision our main results concerning pointed Riemannian manifolds $(M^n,g,p)$ with bounded Ricci curvature $|\Ric|\leq n-1$ and which satisfy the noncollapsing condition $\Vol(B_1(p))>\rv>0$.  Our main regularity result in this context is the following:

\begin{theorem}[$L^2$-Estimate]\label{t:main_L2_estimate}
Let $(M^n,g,p)$ be a pointed Riemannian manifold such that $|\Ric_{M^n}|\leq n-1$ and $\Vol(B_1(p))>\rv>0$, then there exists $C(n,\rv)>0$ such that
\begin{align}\label{e:main_Rm_est}
\fint_{B_1(p)} |\Rm|^2(x)\,dx\leq C\, .
\end{align}
\end{theorem}
\vspace{.5cm}

Let us remark that the above is certainly sharp in that one cannot expect $L^p$ estimates on the curvature for any $p>2$ or for the $L^2$ estimate to hold without the noncollapsing assumption, see Example \ref{sss:Example1_EH} and Example \ref{sss:Example2_collapsing_torus}.  An application of the above is to prove the following weak $L^4$ estimate on the injectivity radius, which follows immediately from \cite{CheegerNaber_Ricci} once the $L^2$ curvature bound has been established:\\

\begin{theorem}[Injectivity Radius Estimate]\label{t:main_inj_estimate}
Let $(M^n,g,p)$ be a pointed Riemannian manifold such that $|\Ric_{M^n}|\leq n-1$ and $\Vol(B_1(p))>\rv>0$, then there exists $C(n,\rv)>0$ such that we have the weak $L^4$ estimate on the injectivity radius given by
\begin{align}
\Vol\big(\{x\in B_1(p): \inj_x<r\}\big)<C(n,\rv) r^{4}.
\end{align}
\end{theorem}
\begin{remark}
One could replace the injectivity radius $\inj$ by the harmonic radius $r_h$ in the above theorem if one wishes.	 
\end{remark}
\vspace{.5cm}
Let us end by remarking on one more result, which tells us that if we additionally have a bound on the gradient of the Ricci curvature, in particular an Einstein manifold, then we also have a sharp apriori $L^p$ estimate on the gradient of the curvature.  Precisely:\\

\begin{theorem}[Gradient $L^{4/3}$-Estimate]\label{t:main_L4/3_estimate}
Let $(M^n,g,p)$ satisfy $|\Ric_{M^n}|\leq n-1$, $|\nabla \Ric|\leq A$ and the noncollapsing assumption $\Vol(B_1(p))>\rv>0$, then there exists $C(n,\rv,A)>0$ such that
\begin{align}\label{e:main_Rm_est}
\fint_{B_1(p)} |\nabla\Rm|^{4/3}(x)\,dx\leq C\, .
\end{align}
\end{theorem}
\vspace{.5cm}

\subsection{Main Results on Limit Spaces}\label{ss:intro:main_limits}

We now turn our attention to our main results on pointed Gromov-Hausdorff limits
\begin{align}
\label{e:limits}
(M^n_j,g_j,p_j)\stackrel{d_{GH}}{\longrightarrow} (X,d,p)
\end{align}
of sequences of manifolds satisfying the Ricci curvature bounds and noncollapsing of \eqref{e:ricci_assumptions}.  Let us first recall the definition of rectifiablity for our context(see more details about rectifiablity in \cite{Fed}).
\begin{definition}
A metric space $Z$ is $k$-rectifiable if there exists a countable collection of $H^k$-measurable subsets $Z_i \subset Z$, and bi-Lipschitz maps $\phi_i : Z_i \to \mathbb{R}^k$ such that $H^k(Z \setminus \cup_{i} Z_i) = 0.$
\end{definition}

Our main result in this direction is the following, which concerns the structure of the singular sets of such limits:\\

\begin{theorem}\label{t:main_limits}
Let $(M^n_j,g_j,p_j)\stackrel{d_{GH}}{\longrightarrow} (X,d,p)$ be a Gromov-Hausdorff limit of manifolds  with $|\Ric_{M^n_j}|\leq n-1$ and $\Vol(B_1(p_j))>\rv>0$.  Then the following hold
\begin{enumerate}
\item The singular set $\cS(X)$ is $n-4$ rectifiable.
\item In particular, we have the hausdorff measure estimate $H^{n-4}\big(\cS(X)\cap B_1\big) < C(n,\rv)$.
\item For $n-4$ a.e. $x\in \cS(X)$ the tangent cone of $X$ is unique and isometric to the conespace $\dR^{n-4}\times C(S^3/\Gamma_x)$, where $\Gamma_x\leq O(4)$ acts freely away from the origin.
\end{enumerate}
\end{theorem}
\begin{remark}
In fact, the proof gives stronger minkowski and packing estimates on the singular set.  Precisely, if $\{B_{r_i(x_i)}\}$ is any disjoint collection of balls with $x_i\in \cS(X)\cap B_1$ then $\sum r_i^{n-4}<C(n,\rv)$.  See \cite{NaVa_Rect_harmonicmap}. 	
\end{remark}

\vspace{.5cm}

In particular, the above proves the $n-4$ finiteness conjecture of Cheeger-Colding in \cite{ChC2}.\\

\subsection{Outline of the proof Theorem \ref{t:main_L2_estimate}}\label{ss:intro:outline}

 In this subsection we give an outline of the proofs of the main theorems of this paper.  Primarily, we will focus on the $L^2$-curvature estimate of Theorem \ref{t:main_L2_estimate}, however the same technical ingredients will go into the proofs of the other main results of the paper, including the structure results for limit spaces given by Theorem \ref{t:main_limits}.  This section is just an outline, and many of the computations are rough in nature, however the morals are the correct ones which will be applied throughout the paper.  For simplicity we will assume in the outline that $\Ric\equiv 0$, which will essentially save on having to discuss error terms that arise in the general case, which are of little consequence but often times quite involved (especially in the proof of the superconvexity estimate below).\\

 \subsubsection{$(\delta,\tau)$-Neck Regions}

 There are several new types of estimates as well as a new decomposition type theorem that the proofs rely on.  The new estimates and decompositions all center around the notion of what we call a $(\delta,\tau)$-neck region $\cN\subseteq B_2$.  The precise definition of a neck region is a bit technical, and we refer the reader to Section \ref{s:neck} for this, but roughly a $(\delta,\tau)$-neck is an open set
\begin{align}\label{e:intro:neck}
	\cN = B_2(p)\setminus \bigcup_{x\in \cC} \overline B_{r_x}(x) \equiv B_2\setminus \overline B_{r_x}(\cC)\, ,
\end{align}
where $\cC$ is a closed set of ball centers and $r_x<\delta$ is a radius function.  To qualify as a neck region, we will have for each $r_x<r<1$ that $B_{\delta^{-1}r}(x)$ is $\delta r $-Gromov-Hausdorff close to a ball in $\dR^{n-4}\times C(S^3/\Gamma)$, and that the ball centers look roughly like a covering of the singular set $\dR^{n-4}\times\{0\}$ at every scale.  Thus we can view $\cC$ as a discrete approximation to the singular set, and it is natural and convenient to associate to the neck region the $n-4$ packing measure given by
\begin{align}
\mu \equiv \sum r_x^{n-4}\,\delta_x\, .	
\end{align}

Before continuing let us discuss a simple example.  Analyzing this example will help solidify what it is we hope to hold for a general neck region:\\

\begin{example}\label{ex:neck_region}
Let $\cE^4$ be the standard Ricci flat four manifold given by the Eguchi-Hanson metric, see Example \ref{sss:Example1_EH}, and let $\cE^4_\eta\equiv \eta^{-1}\cE^4$ be the rescaled Ricci flat metric so that the central $2$-sphere has radius $\eta$.  Let us pick a point $y_c\in \cE^4$ which is an element of this central sphere.  Note then that $\big(\cE^4_\eta,y_c\big)\stackrel{\eta\to 0}{\longrightarrow} \big(\dR^4/\dZ_2, 0)$.  Now let us consider the space $M^n_\eta \equiv \dR^{n-4}\times \cE^4_\eta$, and let $r_x:\dR^{n-4}\to \dR^+$ be a positive function with $|\nabla r_x|<\delta$.  Then if we consider any discrete subset $\cC = \{x_i\}\subseteq B_2(0^{n-4})\times\{y_c\}$ with $\{B_{\tau^2 r_i}(x_i)\}$ a maximal disjoint subset, where $r_i=r_{x_i}$, then for all $\delta>0$ and $\tau<\tau(n)$ if $\eta\leq \eta(n,\delta)$ is sufficiently small then $\cN \equiv B_2(0,y_0)\setminus \bigcup \overline B_{r_x}(\cC)$ is a $(\delta,\tau)$-neck. $\square$
\end{example}
\vspace{.5cm}

There are two important pieces of information to take away from Example \ref{ex:neck_region}.  The first is that if we consider the packing measure
$\mu=\sum r_i^{n-4}\delta_{x_i}$ associated to the covering, then $\mu$ is uniformly Ahlfors regular.  More precisely, for all $r_i<r<2$ we have the estimate
\begin{align}
A(n,\tau)^{-1}r^{n-4}\leq \mu(B_r(x_i))\leq A(n,\tau) r^{n-4}\, .	
\end{align}
This holds because in the context of Example \ref{ex:neck_region} we have that $\{B_{r_i}(x_i)\}$ is a Vitali covering of $B_2(0^{n-4})$.  The second piece of information we get from Example \ref{ex:neck_region} is that we have curvature control on the neck region $\cN \equiv B_2\setminus \overline B_{r_x}(\cC)$.  In particular, regardless of the $(\delta,\tau)$-neck of the Example, there is a uniform $L^2$ bound on the curvature $\int_\cN |\Rm|^2 \leq \delta'$.  In fact, it is not hard to check that as $\delta\to 0$ we have that $\delta'\to 0$ in the example.  This is because for $\eta<<\delta$ a neck region is cutting out the central 2-spheres, which is where all the $L^2$ is concentrating.\\

Our main theorem on the structure of general neck regions is Theorem \ref{t:neck_region}, which tells us that the packing measure and $L^2$-curvature control which held in the previous easy example continue to hold on arbitrary $(\delta,\tau)$-neck regions.  The proofs of these points will take several new ingredients, which we will outline shortly, however let us first describe in detail what the results say.\\

\subsubsection{Structural Theorems on Neck Regions: Ahlfors Regularity}

Our first main structural result on neck regions given in Theorem \ref{t:neck_region} is that the packing measure of a neck region has uniform $n-4$-Ahlfors regularity bounds.  More precisely, with $\delta$ sufficiently small Theorem \ref{t:neck_region} tells us that:
\begin{align}\label{e:intro:ahlfor}
\text{ For each ball center $x\in \cC$ and $r_x<r$ with $B_{2r}(x)\subseteq B_2$ we have }A^{-1} r^{n-4} \leq \mu\Big(B_r(x)\Big)\leq A(n,\tau) r^{n-4}\, .
\end{align}
The proof of this uniform Ahlfors regularity is quite involved, and  we will see that the Ahlfors bound itself is tied into essentially every result of this paper.  The lower and upper bounds in the estimate are proved separately.  Let us mention just a few words about the lower bound now, and we will come back to the upper bound near the end of the outline.\\

The moral of the lower bound estimate is the following.  Roughly, the restriction of the ball centers $\cC\cap B_r(p)$ to any ball look {\it discretely} homeomorphic to a ball $B_r(0^{n-4})\subseteq \dR^{n-4}$.  This will be made precise by proving a discrete version of a Reifenberg theorem in Theorem \ref{t:neck_reifenberg}, which will find a collection $\cC'\subseteq B_r(0^{n-4})$ and a bih\"older mapping $\phi:\cC'\to \cC$ which satisfies a variety of properties.  Let us now also choose a $1$-lipschitz Gromov-Hausdorff map $u:B_r(p)\to B_r(0^{n-4})$.  Then the composition $u\circ\phi:\cC'\subseteq B_r(0^{n-4})\to B_r(0^{n-4})$ is a bilh\"older map which looks close to the identity.  If we ignore the discrete nature of the problem then we could pretend that $u\circ\phi$ is a degree 1 homeomorphism from $B_r$ to itself.  In particular, that would prove $A(n)^{-1}r^{n-4} =\Vol(u\circ\phi(B_r))\leq \Vol^{n-4}(\phi(B_r))\approx \Vol^{n-4}(\cC)\approx \mu(\cC\cap B_r)$, where we have used that $u$ is $1$-lipschitz and that $\mu$ is approximating the $n-4$ Hausdorff measure on $\cC$.  With a little work this argument will be made precise in Section \ref{ss:neck_region_smooth:lower_volume} .  We will come back to the outline of the upper Ahlfors bound, which is much more anlaytic in nature, at the end of the outline.\\

\subsubsection{Structural Theorems on Neck Regions: $L^2$-Estimate}

Let us now discuss our second structural result about $(\delta,\tau)$-neck regions from Theorem \ref{t:neck_region}.  Specifically, we have on a neck region $\cN\subseteq B_2(p)$ that we can prove the desired $L^2$-curvature bound
\begin{align}\label{e:intro:neck_L2}
\int_{\cN\cap B_1(p)} |\Rm|^2 (x)\,dx\leq \delta'\, .
\end{align}
The proof of the $L^2$ curvature estimate \eqref{e:intro:neck_L2} will, in fact, require that we have already proven the Ahlfors regularity \eqref{e:intro:ahlfor}.  In order to explain the main technical lemma involved in the proof, let us recall the definition of the $\cH$-volume given by
\begin{align}\label{e:intro:entropy}
\cH_t(x) \equiv \int_M (4\pi t)^{-\frac{n}{2}}e^{-\frac{d^2(x,y)}{4t}}\, dy\, .
\end{align}
We will see in Section \ref{s:L2_bound} that the $\cH$-volume is monotone nonincreasing.  Our main technical result toward the proof of the $L^2$-estimate \eqref{e:intro:neck_L2} on neck regions is Proposition \ref{p:local_L2_neck}, which will prove, under the assumption of the Ahlfors condition \eqref{e:intro:ahlfor}, that for each $x\in\cN$ with $2r=d(x,\cC)$ we have the estimate
\begin{align}\label{e:intro:neck_L2_entropy}
\int_{B_r(x)} |\Rm|^2(z)\,dz \leq C(n)\int_{B_{10r}(x)}|\cH_{10r^2}-\cH_{10^{-1}r^2}|(y)\,d\mu(y)\, .	
\end{align}
That is, if the measure $\mu$ satisfies the $n-4$ Ahlfors regularity condition, then we can measure the $L^2$ energy of a ball $B_r(x)$ in the neck region by the $\cH$-volume drop on that scale over $\mu$ in a slightly bigger ball.  Let us see how the $L^2$ estimate \eqref{e:intro:neck_L2} follows from the local estimate \eqref{e:intro:neck_L2_entropy}.  Indeed, it is not so difficult to build a Vitali covering
\begin{align}
\cN\cap B_1(p)\subseteq \bigcup_\alpha \bigcup_{i=1}^{N_\alpha} B_{s_{\alpha,i}}(x_{\alpha,i})\, ,	
\end{align}
where $d(x_{\alpha,i},\cC)=2s_{\alpha,i}$, $s_{\alpha,i}\in (2^{-\alpha-1}, 2^{-\alpha}]$, $s_\alpha=2^{-\alpha}$ and $\{B_{10^{-1}s_{\alpha,i}}(x_{\alpha,i})\}$ are disjoint.  Then we can roughly estimate
\begin{align}
	\int_{\cN\cap B_1(p)} |\Rm|^2(x)\,dx &\leq \sum_\alpha \sum_i \int_{B_{s_{\alpha,i}}(x_{\alpha,i})} |\Rm|^2(x)\,dx \leq C(n)\sum_\alpha \sum_i\int_{B_{10s_{\alpha,i}}(x_{\alpha,i})}|\cH_{10s_{\alpha,i}^2}-\cH_{10^{-1}{s_{\alpha,i}^2}}|(y)\,d\mu(y)\notag\\
	&\leq C(n)\sum_\alpha\int_{ B_{3/2}}\big|\cH_{40s^2_\alpha}-\cH_{40^{-1}s^2_\alpha}\big|(y)\,d\mu(y)\notag\\
	&\leq C(n)\int_{ B_{3/2}}\big|\sum_\alpha(\cH_{40s^2_\alpha}-\cH_{40^{-1}s^2_\alpha})\big|(y)\,d\mu(y)\notag\\
	&\leq C(n)\int_{B_{3/2}}\big|\cH_{40}-\cH_{40^{-1}r^2_y}\big|(y)\,d\mu(y)\, ,
\end{align}
where we have used the monotonicity of {$\cH$} to bring the sum inside the absolute value sign.  However, since $\cN$ is a neck region we have that $B_{\delta^{-1}}(y)$ and $B_{\delta^{-1}r_y}(y)$ are both Gromov-Hausdorff close to $\dR^{n-4}\times C(S^3/\Gamma)$, and therefore we have that $|\cH_{40}-\cH_{40^{-1}r_y^2}|(y)\to 0$ pointwise as $\delta\to 0$.  Using our Ahlfors condition again we see for $\delta$ sufficiently small that the $L^2$ estimate \eqref{e:intro:neck_L2} of Theorem \ref{t:neck_region} on the neck region is proved.

\subsubsection{Neck Decomposition Theorem}

Thus, we have now discussed several structural results from the paper which tell us that a general $(\delta,\tau)$-neck region analytically behaves much like we might hope from Example \ref{ex:neck_region}.  One is still in the position of understanding the relevance of these results, in particular in the context of proving the $L^2$ curvature estimate of Theorem \ref{t:main_L2_estimate}.\\

In more detail, in order to exploit the structural results of Theorem \ref{t:neck_region} about $(\delta,\tau)$-neck regions, we need to see that there are lots of such neck regions.  Otherwise, we are proving theorems about a set which may not really appear in practice.  Indeed, the next primary result of the paper which we wish to discuss is the neck decomposition of Theorem \ref{t:neck_decomposition}, which will show that {\it every} point either lies in a neck region or in an $\epsilon$-regularity ball.  More precisely, we can cover our space
\begin{align}\label{e:intro:neck_decomp}
	B_1(p)\subseteq \bigcup_a \cN_a \cup \bigcup_b B_{r_b}(x_b)\, ,
\end{align}
where $\cN_a\subseteq B_{2r_a}(x_a)$ are $(\delta,\tau)$-neck regions, and each ball $B_{r_b}(x_b)$ is a uniformly smooth ball in that the harmonic radius satisfies $r_h(x_b)>2r_b$, see Section \ref{ss:prelim:eps_reg} for a review of the harmonic radius.  The key aspect of the decomposition of Theorem \ref{t:neck_decomposition} is that we will prove the $n-4$ content bound
\begin{align}\label{e:intro:neck_decomp:content}
	\sum_a r_a^{n-4} + \sum_b r_b^{n-4} \leq C(n,\rv,\tau,\delta)\, .
\end{align}

We will mostly avoid discussing the proof of the above estimate in this outline, as it involves numerous technical covering arguments which first decomposes $B_1$ into five types of balls with the help of four parameters, and then proceeds to recover these balls until only the neck and $\epsilon$-regularity regions are left.  It is worth mentioning however that this decomposition, and in particular the content bound, relies heavily on the Ahlfors bound of \eqref{e:intro:ahlfor}.  Without \eqref{e:intro:ahlfor} the techniques of \cite{CheegerNaber_Ricci} could be used build this decomposition under the weaker content estimate $\sum_a r_a^{n-4-\delta} + \sum_b r_b^{n-4-\delta} \leq C(n,\rv,\delta)$, but as we will see shortly it is crucial for the applications in this paper, and in particular the $L^2$ curvature estimate of Theorem \ref{t:main_L2_estimate}, to have the sharp $n-4$ estimates.  \\

\subsubsection{Proving the $L^2$ Curvature Estimate}

Though we have not yet outlined the proof of the Ahlfors regularity result of Theorem \ref{t:neck_region}, let us now take a moment to see how the neck decomposition of Theorem \ref{t:neck_decomposition} combined with the structural results of Theorem \ref{t:neck_region} on neck regions we have discussed lead to a proof of the $L^2$-estimate of Theorem \ref{t:main_L2_estimate}.  Indeed, using the neck decomposition we can write
\begin{align}
	\int_{B_1(p)} |\Rm|^2(x)\,dx \leq \sum_a \int_{\cN_a} |\Rm|^2(x)\,dx + \sum_b \int_{B_{r_b}} |\Rm|^2(x)\,dx\, .
\end{align}

Let us now observe that on the regularity balls $B_{r_b}(x_b)$ that we can use the harmonic radius bound $r_h(x_b)>2r_b$ and standard elliptic estimates in order to prove the scale invariant curvature bound
\begin{align}\label{e:intro:reg_ball_L2}
	r_b^{4-n}\int_{B_{r_b}} |\Rm|^2(x)\,dx \leq C(n)\, .
\end{align}

On the other hand, applying the $L^2$ curvature estimate of Theorem \ref{t:neck_region} on $(\delta,\tau)$-neck regions explained in \eqref{e:intro:neck_L2} leads to the scale invariant estimate on neck regions given by
\begin{align}\label{e:intro:neck_L2:2}
	r_a^{4-n}\int_{\cN_a} |\Rm|^2(x)\,dx \leq C\, .
\end{align}

If we now combine the content estimate \eqref{e:intro:neck_decomp:content}, the regularity ball estimate \eqref{e:intro:reg_ball_L2}, and the $(\delta,\tau)$-neck estimate \eqref{e:intro:neck_L2:2}, then we see we can prove the desired $L^2$-curvature estimate by computing
\begin{align}
		\int_{B_1(p)} |\Rm|^2(x)\,dx &\leq \sum_a \int_{\cN_a} |\Rm|^2(x)\,dx + \sum_b \int_{B_{r_b}} |\Rm|^2(x)\,dx\, ,\notag\\
		&\leq C(n)\sum_a r_a^{n-4} + C(n)\sum_b r_b^{n-4} \leq C(n,\rv,\tau,\delta)\, ,
\end{align}
which would indeed finish the proof of Theorem \ref{t:main_L2_estimate}.\\

\subsubsection{Harmonic Splitting Functions on Neck Regions}

Therefore, what is left in our outline is to understand how to prove the upper bound of the Ahlfors regularity estimate \eqref{e:intro:ahlfor}, which is one of the main technical challenges of this paper.  This estimate itself is based heavily on a key new technical estimate for splitting functions.

More precisely, let $\cN\subseteq B_2(p)$ be a $(\delta,\tau)$-neck, so that in particular $B_{\delta^{-1}}(p)$ is Gromov-Hausdorff close to $\dR^{n-4}\times C(S^3/\Gamma)$.  Let $u:B_2(p)\to \dR^{n-4}$ be a harmonic $\delta$-splitting map associated to this geometry, so that we have the estimates
\begin{align}
	&\fint_{B_2(p)} |\nabla^2 u|^2(x)\,dx \, ,\,\,\, \fint_{B_2(p)} \big| \langle \nabla u_a,\nabla u_b \rangle - \delta_{ab}\big|(x)\,dx < \delta\, ,\notag\\
	&\sup_{B_2(p)}|\nabla u|\leq 1\, .
\end{align}

Our main new technical achievement is to show that if we restrict $u$ to the ball centers $\cC$, which recall act as a discretization of the singular set, then $u$ is a bilipschitz map over most of $\cC$.  To accomplish this we  need to address what is apriori a logical loop.  Namely, we will use this bilipschitz bound in order to prove the Ahlfors regularity, however we need the Ahlfors regularity in order to prove the bilipschitz bound in the first place.  We will address this issue in the next subsection of the outline, for now let us simply assume that for some $B>0$ that we have the (potentially weaker) Ahlfors regularity condition:
\begin{align}\label{e:intro:ahlfors_B}
\text{ For each ball center $x\in \cC$ with $B_{2r}(x)\subseteq B_2$ we have that }B^{-1} r^{n-4} \leq \mu\Big(B_r(x)\Big)\leq B r^{n-4}\, .
\end{align}

Mentally, one should view $B>>A$, where $A$ is the Ahlfors regularity constant we aim to prove.  Thus, our goal is effectively to show if one has a bad Ahlfors regularity constant $B$, then with $\delta$ small one in fact has a better bound of $A$ for free.  With a little technical footwork we will see how this can be used to deal with the apriori logical loop.

With this in mind, our main result for splitting functions on $(\delta,\tau)$-neck regions is Theorem \ref{t:splitting_neck_bilipschitz}, which tells us that for each $\epsilon>0$ if $\delta\leq \delta(n,\epsilon,\tau,B)$, then there exists a subset $\cC_\epsilon\subseteq \cC\cap B_1$ such that
\begin{align}\label{e:intro:bilipschitz}
\mu(\cC_\epsilon)>&\,(1-\epsilon)\mu(\cC\cap B_1)\, ,\notag\\
\forall x,y\in \cC_\epsilon \text{ we have }& \Big||u(x)-u(y)|-d(x,y)\Big|\leq \epsilon\,d(x,y)\, .
\end{align}

Let us spend a few moments on outlining the proof of the above.  To begin with, let us not try and hit every detail, and instead just focus on what is the main new technical estimate needed in the proof.  That is, in Theorem \ref{t:splitting_neck_summable_hessian} we prove the estimate
\begin{align}\label{e:intro:summable_hessian}
\int_{\cN\cap B_1}r_h^{-3}(x)|\nabla^2 u|(x)\,dx\approx \int_{B_1}\Big(\sum_{r_a=2^{-a}}\fint_{\cN\cap A_{r_{a+1,r_a}(x)}}r_a|\nabla^2 u|(z)\,dz\Big)d\mu[x]  < \epsilon^2\, ,	
\end{align}
where $\epsilon$ may be taken arbitrarily small so long as $\delta<\delta(n,\tau,\epsilon,B)$ is taken sufficiently small.  When combined with the Ahflors condition \eqref{e:intro:ahlfors_B} and a telescoping argument, one can conclude directly the bilipschitz estimate \eqref{e:intro:bilipschitz} from this, see Section \ref{s:neck_splitting} for details.

Therefore, we will focus on the estimate \eqref{e:intro:summable_hessian}.  Let us begin by considering the Green's function associated to the packing measure $\mu$ given by
\begin{align}
-\Delta G_\mu = \mu\, .	
\end{align}
Because $\mu$ has the Ahlfors regularity bounds \eqref{e:intro:ahlfors_B} and approximates the singular set of $\dR^{n-4}\times C(S^3/\Gamma)$, one can imagine $G_\mu$ being well approximated by $r_h^{-2}$, where $r_h$ is the harmonic radius, which is itself roughly the same as the distance to $\cC$.  Indeed, in Lemma \ref{l:green_function_distant_function} we will see that if we define the Green's distance function by $G_\mu = b^{-2}$, then we will have the estimates on the neck region $\cN$ given by
\begin{align}\label{e:intro:greens}
&C(n,B)^{-1} r_h \leq b \leq C(n,B) r_h\, ,\notag\\
&C(n,B)^{-1}  \leq |\nabla b| \leq C(n,B)\, ,\notag\\
&\Vol\big(\{b=r\}\cap B_{3/2}\big)\leq Cr^3\, .
\end{align}
Now if $\phi$ is a cutoff function for which $\phi\equiv 1$ on the neck region $\cN$, see Lemma \ref{l:neck:cutoff} for a precise construction, then we can define the quantities
\begin{align}
	&F(r) \equiv r^{-3}\int_{b=r} |\nabla^2 u|~|\nabla b|\phi(x)\, dx\, ,\notag\\
	&H(r)=rF(r)\, .
\end{align}
Let us observe that an estimate on $H(r)$ represents a scale invariant hessian estimate along the $b=r$ slice.  A key computation takes place in Proposition \ref{p:superconvexity}, where we will see that $H(r)$ satisfies the superconvexity type estimate
\begin{align}
r^2\ddot H + r\dot H - (1-\epsilon) H \geq -\epsilon\sum r_x^{n-4}\delta_{[C^{-1}r_x,Cr_x]} -\epsilon r \delta_{[0,C]}\, ,
\end{align}
where $C=C(n)$ and $\epsilon$ may be taken arbitrarily small so long as $\delta$ is sufficiently small (indeed, the above is actually simpler than Proposition \ref{p:superconvexity} due to the assumption ${\Ric}\equiv 0$ throughout the outline).  Combined with a maximum principle and the Ahlfors assumption \eqref{e:intro:ahlfors_B} one can easily conclude from this Theorem \ref{t:superconvexity_estimate}, which gives the Dini integral estimate
\begin{align}
	\int_0^\infty \frac{1}{r} H(r)\,dr\leq \epsilon\, .
\end{align}
Combining with the Green's function estimates \eqref{e:intro:greens} and a coarea formula one immediately concludes from this the estimate
\begin{align}
\int_{\cN} r_h^{-3}|\nabla^2 u|(x)\,dx &\leq C\int_M b^{-3}|\nabla^2 u|\phi(x)\,dx\, ,\notag\\
 &\leq C\int_0^\infty r^{-3}\int_{b=r} |\nabla^2 u| |\nabla b|\phi(x)\,dx dr\, ,\notag\\
 &= C\int_0^\infty \frac{1}{r}H(r)\,dr<\epsilon\, ,
\end{align}
which is the claimed estimate.

\subsubsection{Ahlfors Regularity on Neck Regions}

Now we end our outline by sketching how the bilipschitz estimate \eqref{e:intro:bilipschitz} implies the Ahlfors regularity estimate \eqref{e:intro:ahlfor}.  As mentioned previously, there is a seemingly logical loop in that we needed that Ahlfors condition in order to prove the bilipschitz estimate itself.\\

In order to circumnavigate this issue in Section \ref{s:neck_region_proof}, we will use an inductive procedure which is motivated from \cite{NaVa_Rect_harmonicmap}.  Precisely, let us consider the radii $r_\alpha\equiv 2^{-\alpha}$, and let our goal be to inductively prove the upper bound in \eqref{e:intro:ahlfor} for all $r\leq r_\alpha$, recalling that we have already outlined the lower bound independently.  First, since $M$ is a smooth manifold we have that $r_x>r_0>0$ must be uniformly bounded from below for all $x\in\cC$. In particular, if $r_\alpha$ is the largest radius for which $r_x\geq r_\alpha$ for all $x\in \cC$, then the result is trivial for this $\alpha$ by the simple definition of $\mu$.  This is the base step of our induction.\\

Therefore, our focus is to prove \eqref{e:intro:ahlfor} for some $r_\alpha$, assuming we have proved it for $r_{\alpha+1}$.  Let us begin with a weaker estimate, which will be useful to establish.  Namely, assume $B_{2r}(x)\subseteq B_2$ with $r\leq 10r_\alpha$.  Then by a standard covering argument we can cover $B_r(x)$ by at most $C(n)$ balls of radius $10^{-3}r\leq r_\alpha$, and therefore by applying our inductive assumption we have for some $B=C(n)A$ the strictly weaker estimate
\begin{align}
\mu\big( B_r(x)\big)\leq Br^{n-4}\, .	
\end{align}

While this estimate is not good enough for the inductive step, indeed if one were to iterate it in $\alpha$ the constant would blow up horribly, it is enough for us to apply Theorem \ref{t:splitting_neck_bilipschitz} and obtain the $\epsilon$-bilipschitz result of \eqref{e:intro:bilipschitz}.  More precisely, if $u:B_{2r}(x)\to \dR^{n-4}$ is a $\delta$-splitting function, which exists because we are in a $(\delta,\tau)$-neck, then there exists $\cC_\epsilon\subseteq \cC\cap B_r(x)$ such that
\begin{align}
\mu(\cC_\epsilon)>&\,(1-\epsilon)\mu(\cC\cap B_r)\, ,\notag\\
\forall x,y\in \cC_\epsilon \text{ we have }& \Big||u(x)-u(y)|-d(x,y)\Big|\leq \epsilon\,d(x,y)\, .
\end{align}

However, the set $\{B_{r_i}(x_i)\}_{\cC}$ is a Vitali covering, which tells us that for $\epsilon$ sufficiently small that $\{B_{r_i}(u(x_i))\}_{\cC\cap B_r}$ is itself a Vitali covering of $B_r(0^{n-4})$.  This immediately implies the improved Ahlfors upper bound
\begin{align}
	&\mu(\cC\cap B_r)\leq \frac{1}{1-\epsilon}\mu(\cC_\epsilon)=\sum_{\cC_\epsilon} r_i^{n-4}\leq C\,\Vol(B_{2r}(0^n))\leq A\, r^{n-4}\, .
\end{align}
This proves the inductive step, and therefore finishes the outline of the proof.
 \vspace{.5cm}

\section{Background and Preliminaries}\label{s:background}

In this section we review several standard constructions and techniques which will be used throughout the paper.

\subsection{Quantitative Stratification}\label{ss:prelim:quant_strat}

Let us briefly review the notion of symmetry and stratification, as these ideas are commonplace throughout this paper.  We begin by recalling the standard notion of symmetry:

\begin{definition}
Given a metric space $Y$ we define the following:
\begin{enumerate}
	\item We say $Y$ is {\it $k$-symmetric} at $y\in Y$ if there exists a pointed isometry $\iota:\dR^k\times C(Z)\to Y$ with $\iota(0^k,z_c)=y$, where $Z$ is compact and $z_c$ is a cone point.
	\item We say $Y$ is $k$-symmetric with respect to $L^k\subseteq Y$ if $L^k=\iota\big(\dR^k\times\{z_c\}\big)$.
\end{enumerate}
\end{definition}
\vspace{.3cm}

Associated to the notion of symmetry is that of stratification:

\begin{definition}
Given a metric space $Y$ we  define the {\it closed $k^{th}$-stratum} by
\begin{align}
\cS^k(X)=:\{x\in X: \text{ no tangent cone at $x$ is $(k+1)$-symmetric}\}
\end{align}
\end{definition}
\vspace{.3cm}

This notion of symmetry leads to a natural quantitative generalization, first introduced in \cite{CheegerNaber_Ricci}.  It is the notion of quantitative symmetry which will play the most important role for us in this paper, in particular in the discussion of neck regions.  Let us begin with a discussion of quantitative symmetry:\\

\begin{definition}\label{d:quant_symmetry}
Given a metric space $Y$ with $y\in Y$, $r>0$ and $\epsilon>0$, we say
\begin{enumerate}
\item $B_r(y)$ is $(k,\epsilon)$-symmetric if there exists a pointed $\epsilon r$-GH map $\iota:B_r(0^k,z_c)\subseteq \dR^k\times C(Z)\to B_r(y)\subseteq Y$ with $\iota(0^k,z_c)=y$.
\item $B_r(y)$ is $(k,\epsilon)$-symmetric with respect to $\cL^k_\epsilon\subseteq B_r(y)$ if $\cL^k_\epsilon \equiv \iota\big(\dR^k\times \{z_c\}\big)\cap B_r(y)$
\end{enumerate}
\end{definition}
\vspace{.5cm}

To state the definition in words, we say
that $B_r(y)\subseteq Y$ is $(k,\epsilon)$-symmetric if the ball $B_r(y)$ looks very close to having $k$-symmetries.  The quantitative stratification is then defined as follows:

\begin{definition}\label{d:quant_stratification}
For each $\epsilon>0$, $0<r<1$ and $k\in\dN$, define the closed quantitative $k$-stratum, $\cS^k_{\epsilon,r}(X)$, by
\begin{align}
\cS^k_{\epsilon,r}(X)\equiv \{x\in X:\text{ for no $r\leq s\leq 1$ is $B_s(x)$ a $(k+1,\epsilon)$-symmetric ball}\}\, .
\end{align}
\end{definition}
\vspace{.5cm}

Thus, the closed  stratum $\cS^k_{\epsilon,r}(X)$ is the collection of points such that no ball of size at least $r$ is almost $(k+1)$-symmetric.  The notion of the quantitative stratification plays an important role in our notion of a neck region, introduced in Section \ref{s:neck}.  The first main result of \cite{CheegerNaber_Ricci} is to show that for manifolds
 which are noncollapsed and have lower Ricci curvature bounds, the set $\cS^k_{\epsilon,r}(X)$ is small in a
 very strong sense.  To say this a little more carefully, if one pretends that the $k$-stratum
is a well behaved $k$-dimensional submanifold, then one would expect the volume of the $r$-tube
 around the set to behave like $Cr^{n-k}$. In \cite{CheegerNaber_Ricci} the following slightly weaker result was shown:

\begin{theorem}[Quantitative Stratification, \cite{CheegerNaber_Ricci}]\label{t:quant_strat}
Let $M^n$ satisfy $\Ric\geq -(n-1)$ with $\Vol(B_1(p))>{\rm v}>0$.  Then for every $\epsilon,\eta>0$
there exists $C=C(n,\rv,\epsilon,\eta)$ such that
\begin{align}
\Vol\left(B_r\left(\cS^k_{\epsilon,r}(M)\cap B_1(p)\right)\right)\leq C r^{n-k-\eta}\, .
\end{align}
\end{theorem}
\vspace{.5cm}

One of the consequences of the main Theorems of this paper is that for spaces with bounded Ricci curvature one can improve the above to $\Vol\left(B_r\left(\cS^{n-4}_{\epsilon,r}(M)\cap B_1(p)\right)\right)\leq C r^4$ for the top stratum of the singular set.\\

\subsection{Volume Monotonicity and Cone Structures}\label{ss:prelim:monotonicity}

When considering a Riemannian manifold with a lower Ricci bound $\Ric\geq -(n-1)\kappa g$, the key tool which separates the study of collapsed versus noncollapsed spaces is that of a monotone quantity.  Throughout this paper we will consider several, all of which are essentially equivalent, however it will be more convenient to work with one or another depending on the context.  Let us begin by recalling the classical volume ratio
\begin{align}\label{e:volume_ratio}
\cV_r(x) = \cV^\kappa_r(x) \equiv \frac{\Vol\big(B_r(x)\big)}{\Vol_{-\kappa}(B_r)}\, ,	
\end{align}
where $\Vol_{-\kappa}(B_r)$ is the volume of the ball of radius $r$ is the space form $M^n_{-\kappa}$ of constant curvature $-\kappa$.  It is a classic consequence of the Bishop-Gromov monotonicity that for each $x\in M$ that $\cV_r(x)$ is monotone nonincreasing in $r>0$ with $\cV_r(x)\to 1$ as $r\to 0$.  Note that the space being noncollapsed is now equivalent to there being a lower bound on $\cV_1$, so that we have a bounded monotone quantity.  The importance of this comes into play because there is a rigidity, which tells us that when this quantity is very pinched we must have symmetry.  Precisely:\\

\begin{theorem}[\cite{ChC1}]\label{t:metriccone_volumecone}
	Let $(M^n,g,p)$ with $\epsilon>0$, then there exists $\delta(n,\epsilon)>0$ such that if $\Ric\geq -\delta$ and $\cV_2(p)\geq (1-\delta)\cV_{1}(p)$, then $B_2(p)$ is $(0,\epsilon)$-symmetric.
\end{theorem}
\vspace{.5cm}

In Section \ref{s:L2_bound} we will discuss some generalizations of this point using some distinct monotone quantities.  The advantage of the approach of Section \ref{s:L2_bound} will be that we will be able to obtain some sharp estimates on the cone structures, which will be required in the proof of the $L^2$ estimates.\\

\subsection{Cone Splitting}

We saw in the previous subsection how to force $0$-symmetries by using a monotone quantity.  In this subsection we want to review one method of forcing higher orders of symmetries.  {The relevant concept for this paper is one introduced in \cite{CheegerNaber_Ricci} called cone splitting.  The main result in this direction from \cite{CheegerNaber_Ricci} is the following:}

\begin{theorem}[Cone-Splitting]\label{t:cone_splitting}
	For every $\epsilon,\tau>0$ there exists $\delta(n,\epsilon,\tau)>0$ such that if
	\begin{enumerate}
	\item $\Ric\geq -\delta$.
	\item $B_{2}(p)$ is $(k,\delta)$-symmetric with respect to $L^k_\delta\subseteq B_4(p)$.
	\item There exists $z\in B_1(p)\setminus B_\tau L^k_\delta$ such that $B_2(z)$ is $(0,\delta)$-symmetric.
	\end{enumerate}
Then $B_1(p)$ is $(k+1,\epsilon)$-symmetric.
\end{theorem}
\vspace{.25cm}

Therefore, the above is telling us that nearby $0$-symmetries interact to force higher order symmetries.\\

\subsection{Harmonic Radius and $\epsilon$-Regularity Theorems}\label{ss:prelim:eps_reg}

In this subsection we review two $\epsilon$-regularity theorems which will play a prominent role in this paper.  To make these results precise let us recall the notion of the harmonic radius:\\

\begin{definition}\label{d:harmonic_radius}
	For $x\in M^n$ we define its harmonic radius $r_h(x)>0$ to be the maximum over all $r>0$ such that there exists a mapping $\phi:B_r(x)\to \dR^n$ with the following properties:
	\begin{enumerate}
	    \item $\phi$ is a harmonic mapping.
		\item $\phi$ is a diffeomorphism onto its image with $B_r(0^n)\subseteq \phi(B_r(x))$, and hence defines a coordinate chart.
		\item The coordinate metric $g_{ij} = \langle \nabla\phi_i,\nabla \phi_j\rangle$ on $B_{r}(0^n)$ satisfies $||g_{ij}-\delta_{ij}||_{C^1(B_r)}<10^{-n}$.
	\end{enumerate}
\end{definition}
\begin{remark}
Note that by a standard implicit function type argument one has that $r_h(x)>0$ on any smooth manifold.	
\end{remark}
\begin{remark}
	The $C^1$-norm above is taken with respect to the scale invariant Euclidean norm.  That is, $||g_{ij}-\delta_{ij}||_{C^1(B_r)} \equiv \sup_{B_r}|g_{ij}-\delta_{ij}|+ r\,\sup_{B_r}|\partial_k g_{ij}|$.
\end{remark}
\begin{remark}\label{r:harmonic_rad_high_est}
Note that if we have the Ricci bound $|\Ric|\leq A$, then for every $\alpha<1$ and $p<\infty$ we have the apriori estimate $||g_{ij}-\delta_{ij}||_{C^{1,\alpha}(B_{r/2})}\, , ||g_{ij}-\delta_{ij}||_{W^{2,p}(B_{r/2})}<C(n,A,\alpha,p)$.  This follows from elliptic estimates which exploit the harmonic nature of the coordinate system.	
\end{remark}
\vspace{.25cm}

Now the $\epsilon$-regularity theorems of this subsection are meant to find weak geometric conditions under which we can be sure there exists a definite lower bound on the harmonic radius.  The first which we will discuss goes back to Anderson and tells us that if the volume of a ball is sufficiently close to that of a Euclidean ball, then we must have smooth estimates.  Precisely:\\

\begin{theorem}[\cite{Anderson_Einstein}]\label{t:eps_reg_volume}
There exists $\epsilon(n)>0$ such that if $(M^n,g,p)$ satisfies $|{\Ric}|\leq \epsilon$ and $\cV_2(p)>(1-\epsilon)$, then we have that $r_h(p)>1$.
\end{theorem}
\begin{remark}
Recall that $\cV_r(p) = \cV^\epsilon_r(p) \equiv \frac{\Vol(B_r(x)}{\Vol_{-\kappa}(B_r)}$, and thus $\cV_2>1-\epsilon$ gives us that the volume of $B_2(p)$ is close to the volume of $B_2(0^n)$.
\end{remark}
\vspace{.5cm}

We end with the following $\epsilon$-regularity result proved in \cite{CheegerNaber_Codimensionfour}.  This result can be viewed as a consequence of the proof of the codimension four conjecture, and tells us that if a ball has a sufficient amount of symmmetry, then the ball must be smooth.  Results like the following are why the notion of quantitative symmetry play such a crucial role in the analysis and estimates of this paper:\\

\begin{theorem}[\cite{CheegerNaber_Codimensionfour}]\label{t:eps_reg}
Let $(M^n,g,p)$ satisfy $|{\Ric}|\leq n-1$ and $\Vol(B_1(p))>\rv>0$.  Then there exists $\epsilon(n,\rv)>0$ such that if $B_2(p)$ is $(n-3,\epsilon)$-symmetric then $r_h(p)>1$.
\end{theorem}
\vspace{.5cm}

\subsection{Heat Kernel Estimates}\label{ss:heat_kernel}

In this subsection we record estimates on heat kernels on Riemannian manifolds with lower Ricci curvature bounds.  The estimates of this subsection are either classical or very minor modifications of classical estimates which are better suited to our purposes.  Let us summarize the basic estimates on heat kernels, which follow from the results in \cite{LiYau_heatkernel86}, \cite{SY_Redbook}, \cite{SoZha},\cite{Hamilton_gradient}, \cite{Kot_hamilton_gradient}:\\

\begin{theorem}[Heat Kernel Estimates]\label{t:heat_kernel}
	Let $(M^n,g,x)$ be a pointed Riemannian manifold with $\Vol(B_1(x))\ge \rv>0$ and ${\Ric}\ge -(n-1)$. Then for any $0<t\le 10$ and $\epsilon>0$, the heat kernel $\rho_t(x,y)=(4\pi t)^{-n/2}e^{-f_t}$ satisfies for all $y\in M$ that
	\begin{enumerate}
	\item $C(n,\epsilon) t^{-n/2} e^{-\frac{d(x,y)^2}{(4-\epsilon)t}}\leq \rho_t(x,y)$.
	\item $\rho_t(x,y)\leq C(n,\rv,\epsilon) t^{-n/2} e^{-\frac{d(x,y)^2}{(4+\epsilon)t}}$.
	\item $t|\nabla f_t|^2\leq C(n,\rv)\Big(1+\frac{d^2(x,y)}{t}\Big)$ . 
	\item $-C(n,\rv,\epsilon)+\frac{d^2(x,y)}{(4+\epsilon)t}\leq f_t \leq C(n,\rv,\epsilon)+\frac{d^2(x,y)}{(4-\epsilon)t }$
	\end{enumerate}
\end{theorem}
\begin{remark}
Noting that with the time $t$ bounded,  (1), (2)  and (4) follow directly from Li-Yau heat kernel estimates and the volume comparison.  (3) follows from (2) and \cite{SoZha},\cite{Hamilton_gradient}, \cite{Kot_hamilton_gradient}. 
\end{remark}

\vspace{.5cm}

\subsection{Harmonic Estimates on Spaces with Bounded Ricci}\label{ss:prelim:harmonic_est}

In this subsection we recall some estimates from \cite{CheegerNaber_Codimensionfour} on harmonic functions for noncollapsed spaces with bounded Ricci curvature.  We will also prove some generalizations of these estimates which will be used in the paper.  The main result is the following:

\begin{theorem}\label{t:prelim:harmonic_estimates}
	Let $(M^n,g,p)$ satisfy $|{\Ric}|\leq n-1$ with $\Vol(B_1(p))>\rv>0$, and let $u:B_2(p)\to \dR$ be a harmonic function with $|u|\leq 1$.  Then the following hold:
\begin{enumerate}
\item For each $0<q<4$ we have that $\int_{B_1(p)}|\nabla^2 u|^q(x)\,dx\leq C(n,\rv,q)$.
\item For each $0<q<2$ we have that $\int_{B_1(p)}|\nabla^3 u|^q(x)\,dx\leq C(n,\rv,q)$.
\end{enumerate}
\end{theorem}
\begin{remark}
Once the $L^2$ curvature bound of Theorem \ref{t:main_L2_estimate} has been established, we can use the techniques of \cite{Cheeger01} in order to produce full $L^4$ bounds on $|\nabla^2 u|$ and $L^2$ bounds on $|\nabla^3 u|$.  However, the weaker estimates of the above theorem will themselves be useful toward that goal.
\end{remark}

\begin{proof}
Note that by using standard Cheng-Yau estimates (or one of several others), we can conclude for every $r<2$ there exists $C_r(n)>0$ such that
	\begin{align}
		|\nabla u|\leq C_r\, .
	\end{align}
The estimate $(1)$ then follows directly from \cite{CheegerNaber_Codimensionfour}.  To prove $(2)$ will follow a similar proof structure.  First, let us consider the following:

{\bf Claim: } Let $B_{2r}(x)\subseteq B_2(p)$ with $r_{h}(x)\geq 2r$ and $q\leq 2$, then $r^{2q-n}\int_{B_r(x)}|\nabla^3 u|^q<C(n,q)$.\\

It is enough to prove the claim for $q=2$, the general case then follows from a H\"olders inequality.  For this let us compute the Bochner type formula
\begin{align}
	\Delta |\nabla^2 u|^2 = 2|\nabla^3 u|^2 + 2\big\langle\big(\nabla_jR_{ia}-\nabla_a R_{ij}+\nabla_i R_{ja}\big)\nabla^a u, \nabla^i\nabla^j u\big\rangle -4\text{Rm}(\nabla^2 u,\nabla^2 u) + 4{\Ric}(\nabla^a\nabla u,\nabla_a\nabla u)\, .\notag
\end{align}
Let $\phi:B_{2r}(x)\to \dR$ be such that $\phi\equiv 1$ on $B_r(x)$, $\phi\equiv 0$ outside of $B_{3r/2}(x)$ with $r|\nabla\phi|, r^2|\Delta\phi|\leq C(n)$.  Multiplying both sides of the above Bochner formula by $\phi$ and integrating and noting the pointwise upper bound $|\nabla^2u|\le C(n,\rv)r^{-1}$, we arrive, after a little manipulation, at
\begin{align}
2\int_{B_{2r}(x)} \phi(z) |\nabla^3 u|^2(z)\,dz \leq Cr^{-4} + 2\int_{B_{2r}(x)} \phi(z) \big\langle\big(\nabla_jR_{ia}-\nabla_a R_{ij}+\nabla_i R_{ja}\big)\nabla^a u, \nabla^i\nabla^j u\big\rangle(z)\,dz\, .
\end{align}
Integrating by parts the $\nabla {\Ric}$ term and using a H\"older's inequality we arrive at
\begin{align}
2\int_{B_{2r}(x)} \phi(z) |\nabla^3 u|^2(z)dz &\leq Cr^{-4} + C\int_{B_{2r}(x)} 
\Big(\phi |{\Ric}|\, |\nabla^2 u|^2 + |\nabla\phi|\, |{\Ric}| |\nabla^2 u| + \phi |{\Ric}| |\nabla^3 u| \Big)(z)dz\, ,\notag\\
&\leq Cr^{-4} + \frac{1}{10}\int_{B_{2r}(x)} \phi(z) |\nabla^3 u|^2(z)\, dz+ C\Vol(B_{2r}(x))\, .
\end{align}
Using the hessian estimate $(1)$ applied to $B_{2r}(x)$ we then get from this
\begin{align}
r^4\fint_{B_{2r}(x)} \phi(z) |\nabla^3 u|^2(z)\,dz\leq C+r^4\leq C\, ,
\end{align}
which proves the claim. $\square$\\

Now to finish the proof of the theorem we proceed as in \cite{CheegerNaber_Codimensionfour}.  Indeed, let $r_a\equiv 2^{-a}$, then for any $\epsilon>0$ we have by \cite{CheegerNaber_Codimensionfour} the estimate $\Vol\Big(B_{r_a}\Big(\{x\in B_1: r_{a+1}\leq r_{h}<r_{a+2}\}\Big)\Big)\leq C(n,\rv,\epsilon) r_a^{4-2\epsilon}$, where $r_h$ is the harmonic radius.  In particular, for each $r_a$ we can find a Vitali covering $\{B_{r_a}(x_{i,a})\}_1^{N_a}$ of $\{r_{a+1}\leq r_{h}<r_{a+2}\}$ with $N_a\leq C r_a^{4-n-2\epsilon}$.  Choosing $\epsilon<(2-q)/2$ and applying the Claim we arrive at the estimate
\begin{align}
\int_{B_1(p)}|\nabla^3 u|^q(z)dz \leq \sum_a \sum_a^{N_a} \int_{B_{r_a}(x_{i,a})}|\nabla^3 u|^q(z)dz \leq C\sum_a r_a^{n-2q} r_a^{4-n-2\epsilon} = C\sum_a r_a^{2(2-q-\epsilon)}\leq C(n,\rv,q)\, ,	
\end{align}
which proves the Theorem.
\end{proof}
\vspace{.5cm}

\subsection{$\delta$-Splitting Functions}\label{ss:splitting}

In this subsection we recall the basics about splitting functions, which act as a bridge between geometric notions and notions in analysis.  Let us begin with a definition, which is similar to the one introduced in \cite{CheegerNaber_Codimensionfour}:\\

\begin{definition}\label{d:splitting_function}
	We say $u=(u^1,u^2,\cdots,u^k):B_r(p)\to \dR^k$ is a $\delta$-splitting map if the following holds: 

\begin{enumerate}
	\item $\Delta u = 0$, which is to say that $\Delta u^i=0$ for all $i=1,\cdots, k$..
    \item $\sup_{a,b=1,\cdots,k}\fint_{B_r(p)}\big|\langle\nabla u^a,\nabla u^b\rangle - \delta^{ab}\big|(x)\,dx<\delta$.
    \item $r^2\fint_{B_r(p)}\big|\nabla^2 u\big|^2(x)\,dx:=\sum_{i=1}^kr^2\fint_{B_r(p)}\big|\nabla^2 u^i\big|^2(x)\,dx<\delta^2$.
	\item $\sup_{x\in B_{r}(p)} |\nabla u|(x) \leq 1$, where $|\nabla u|(x):=\sup_{v\in T_xM,|v|=1}|Du(v)|$. 
\end{enumerate}
\end{definition}
\vspace{.25cm}

The main result about splitting functions is that they are essentially equivalent to the ball $B_r$ splitting off an $\dR^k$-factor.  Precisely, we have the following:\\

\begin{theorem}\label{t:splitting_function}
	For every $\epsilon>0$ there exists $\delta(n,\epsilon)>0$ such that if $\Ric\geq -\delta$ then the following hold:
	\begin{enumerate}
		\item If $B_{\delta^{-1}}(p)$ is $\delta$-GH close to $B_{\delta^{-1}}(0^k,y)\subset \dR^k\times Y$, then there exists an $\epsilon$-splitting $u:B_2(p)\to \dR^k$.
		\item If there exists a $\delta$-splitting $u:B_{2}(p)\to \dR^k$, then $B_{1}(p)$ is $\epsilon$-GH close to $B_1(0^k,y)\subset\dR^k\times Y$.  Moreover, $(u,\phi): B_1(p)\to \mathbb{R}^k\times Y$ gives an $\epsilon$-GH map for some $\phi: B_1(p)\to Y$.
	\end{enumerate}
\end{theorem}
\begin{remark}
The above theorem is a mild extension of one of the main accomplishments of \cite{ChC1}, with the sharp gradient bound of condition $(4)$ having been proved in \cite{CheegerNaber_Codimensionfour}.	
\end{remark}
\vspace{.5cm}

\subsection{Examples}\label{ss:examples}

In this subsection, we  recall some standard examples which play an important role in guiding the results of this paper.  We discuss these in only minimal detail, and refer the reader to the appropriate references for more.

\begin{example}[The Eguchi-Hanson manifold]
\label{sss:Example1_EH}
The Eguchi-Hanson space $\cE = (T^*S^2,g)$ is a complete Ricci flat metric which is diffeomorphic to the cotangent bundle of $S^2$.  The asymptotic cone of $\cE$ converges rapidly to $\dR^4/\dZ_2$, where $\dZ_2$ acts on $\dR^4$ by $x\to-x$. That is, if one considers the spaces $\cE_\eta \equiv (T^*S^2,\eta g)$, then we have the Gromov Hausdorff limit $\cE_\eta\stackrel{\eta\to 0}{\longrightarrow} C(\dR\dP(3))= \dR^4/\dZ_2$.  This is the simplest example which shows that even under the assumption of Ricci  flatness and noncollapsing, Gromov-Hausdorff limit spaces can contain codimension 4 singularities.

 It is an interesting exercise, using the Chern-Gauss-Bonnet formula, that one has $\int_{\cE_\eta} |\Rm|^2(x)\,dx = 16\pi^2$ independent of $\eta$, however for $q>2$ it is also easy to check by a rescaling argument that $\int_{\cE_\eta} |\Rm|^q(x)\,dx\to \infty$ as $\eta\to 0$, which shows us that an apriori $L^2$ bound on the curvature is the most one can expect.
\end{example}
\vspace{.25cm}

\begin{example}[$L^2$ Blow up under collapsing]\label{sss:Example2_collapsing_torus}

We now briefly study an example of Anderson \cite{Anderson_Hausdorff}, which will tell us that the noncollapsing assumption is necessary in that there exists a sequence of manifolds $M^4_j$ such that
\begin{align}\label{e:example:blowup:1}
|\Ric_j|\to 0\, ,\;\;\;\;\;\diam(M^4_j)= 1, \;\;\;\; \Vol(M^4_j)\to 0,\;\;\;\; \int_{M_j} |\Rm_j|^2(x)\,dx\to \infty\, .
\end{align}

Indeed, the rough construction is as follows.  Anderson built a complete simply connected Ricci flat four manifold $(\cS^4,g)$ such that outside a compact set we have that $\cS^4$ is diffeomorphic, and in fact quickly becoming isometric, to the metric product $\dR^3\times S^1$.  The space $\cS^4$ has the property that $\int_{\cS^4}|\Rm|^2 = C>0$.

In order to build $M_j$ let us begin with the flat four torus $T^3\times S^1_{j^{-1}}$, where $T^3=(S^1)^3$ is the standard three torus and $S^1_{j^{-1}}$ is the circle of radius $j^{-1}$.  One may then glue copies of $\cS^4_{j^{-1}}\equiv j\cS^4$ into $T^3\times S^1_{j^{-1}}$.  As $j\to \infty$ one may glue an arbitrarily large number of copies while perturbing the Ricci flat condition arbitrarily small amount.  Each glued copy of $\cS^4_{j^{-1}}$ contributes roughly $C$ to the total $L^2$ norm of $|\Rm|$, and therefore it is easy to check that \eqref{e:example:blowup:1} is satisfied.

\end{example}
\vspace{.5cm}

\section{$(\delta,\tau)$-Neck Regions}\label{s:neck}

A central theme of this paper will be that of a $(\delta,\tau)$-neck region.  In this section we will define this, and then state our main results on such regions.  The proofs themselves will take place over the next several sections, as they will be a bit involved.  We begin with a formal definition:\\

\begin{definition}\label{d:neck}
We call $\cN\subseteq B_2(p)$ a $(\delta,\tau)$-neck region if there exists a closed subset $\cC = \cC_0\cup \cC_+=\cC_0\cup\{x_i\}\subset B_2(p)$  with $p\in \cC$ and a radius function $r:\cC\to \dR^+$ on $\cC_+$ and $r_x=0$ on $\cC_0$ such that $\cN \equiv B_2\setminus \overline B_{r_x}(\cC)$ satisfies
\begin{enumerate}
	\item[(n1)] $\{B_{\tau^2 r_x}(x)\}$ are pairwise disjoint.
	\item[(n2)] For each $r_x\leq r\leq 1$ there exists a $\delta r$-GH map $\iota_{x,r}:B_{\delta^{-1}r}(0^{n-4},y_c)\subseteq\dR^{n-4}\times C(S^3/\Gamma)\to B_{\delta^{-1}r}(x)$, where $\Gamma\subseteq O(4)$ is nontrivial  and $\iota_{x,r}(0^{n-4},y_c)=x$.\footnote{{In particular, the harmonic radius satisfies $r_h(x)\le r_x$.}} 
	\item[(n3)] For each $r_x\leq r$ with $B_{2r}(x)\subseteq B_2(p)$ we have that $\cL_{x,r}\equiv \iota_{x,r}\big(B_r(0^{n-4})\times\{y_c\}\big)\subseteq B_{\tau r}(\cC)$ and $\cC\cap B_r(x)\subset B_{\delta r}(\cL_{x,r})$.
	\item[(n4)] $|\Lip\, r_x|\leq \delta$.
\end{enumerate}	
For each $\tau\leq s$ we define the region $\cN_s\equiv B_2\setminus \overline B_{s\cdot r_x}(\cC)$.
\end{definition}

\begin{remark}
If $\cA\subseteq M$ is a closed subset with $a_x:\cA\to \dR^+$ a nonnegative continuous function, then the closed tube $\overline B_{a_x}(\cA)$ is by definition the set $\bigcup_{x\in\cA} \overline B_{a_x}(x)$.
\end{remark}
\begin{remark}
Recall if $f:\cA\subseteq M\to \dR$ is a function defined on some subset $\cA\subseteq M$, then the lipschitz constant of $f$ is defined $|\Lip\, f|\equiv \sup_{x,y\in\cA} \frac{|f(x)-f(y)|}{d(x,y)}$.	
\end{remark}
\begin{remark}
The notation $\cL_{x,r}$ is based on the comparible notation for the quantitative stratification given in Definition \ref{d:quant_symmetry}, since $B_r(x)$ is $(n-4,\delta)$-symmetric.	
\end{remark}

\begin{remark}\label{r:neck_region:1}
In the above we may take $B_2(p)$ to be either in a manifold $M^n$ or a limit space $M^n_j\to X$.  In the case when $B_2(p)$ is in a manifold then necessarily  $\cC_0=\emptyset$ and $ \#\cC<\infty$, since $\cC_0\subseteq \text{Sing}(X)$ and $\inf r_x>0$.	
\end{remark}
\begin{remark}
We also say $\cN\subseteq B_{2R}(p)$ is a neck region for some $R>0$ if on the rescaled ball $B_2(\tilde p) = R^{-1}B_{2R}(p)$ with  $r_x$ becomes $r_{x}/R$,  we have that $\cN$ satisfies the above.
\end{remark}
\begin{remark}
Note that if $\cN\subseteq B_2(p)$ is a $(\delta,\tau)$-neck region and $B_{2s}(q)\subseteq B_2(p)$, then $\cN\cap B_{2s}(q)$ defines a $(\delta,\tau)$-neck on $B_{2s}(q)$.	
\end{remark}
\vspace{.25cm}

An important concept associated to any neck region is the induced packing measure.  In the same manner that the set $\cC$ is a potentially discrete approximation of the singular set, the packing measure is a potentially discrete approximation of the $n-4$ hausdorff measure on this set.  More precisely, we have the following:

\begin{definition}
	Let $\cN \equiv B_2\setminus \overline B_{r_x}(\cC)$ be a neck region, then we define the associated packing measure
	\begin{align}
    \mu=\mu_\cN \equiv \sum_{x\in \cC_+} r_x^{n-4}\delta_{x} + \lambda^{n-4}|_{\cC_0}\, ,
\end{align}
where $\lambda^{n-4}|_{\cC_0}$ is the $n-4$-dimensional Hausdorff measure restricted to $\cC_0$.
\end{definition}
\vspace{.5cm}

Let us now state what is our main result on neck regions.  The proof of this result will take place over the course of the next several sections, as it will be one of the main technical accomplishments of the paper:

\begin{theorem}\label{t:neck_region}
Let $(M^n_j,g_j,p_j)\to (X,d,p)$ be a Gromov-Hausdorff limit with $\Vol(B_1(p_j))>\rv>0$, $\tau<\tau(n)$ and $\epsilon>0$.  Then for $\delta\leq \delta(n,\rv,\tau,\epsilon)$ if $|{\Ric}_j|<\delta$ and $\cN = B_2(p)\setminus \overline B_{r_x}(\cC)$ is a $(\delta,\tau)$-neck region, then the following hold:
\begin{enumerate}
	\item For each $x\in \cC$ and $r>r_x$ such that $B_{2r}(x)\subseteq B_2(p)$ the induced packing measure $\mu$ satisfies the Ahlfors regularity condition $ A(n,\tau)^{-1}r^{n-4}<\mu(B_r(x)) <A(n,\tau)r^{n-4} $.
	\item $\cC_0$ is $n-4$ rectifiable.
	\item $X$ is a manifold on $\cN$ and we have the $L^2$-curvature bound $\int_{\cN\cap B_1(p)} |\Rm|^2(x)\,dx < \epsilon$. \footnote{{Strictly speaking, $X$ is a smooth manifold with a $W^{2,p}$-metric on $\cN$ for all $p<\infty$, in particular the curvature condition is well defined.}}
\end{enumerate}
\end{theorem}
\vspace{.5cm}

Let us briefly outline this section.  We begin in Section \ref{ss:neck:basic_properties} by discussing some basic properties of neck regions.  This includes the definition of wedge regions and the construction of a canonical cutoff function, which will be used frequently in the analysis.  In Section \ref{ss:neck_approximation} we will show that every neck region on a singular limit space $M_j\to X$ may be approximated by neck regions on the smooth spaces $M^n_j$ themselves.  This will allow us to primarily work on manifolds and then pass our results off to the limit automatically.  Finally in Section \ref{ss:neck:reifenberg} we will prove a discrete Reifenberg type theorem for the effective singular set $\cC$.  This will end up playing an important role in the proof of the lower bound in our Ahlfors regularity.\\

\subsection{Basic Properties of Neck Regions}\label{ss:neck:basic_properties}

In this subsection we record some simple properties of neck regions which will play a role in the analysis.

\subsubsection{Regularity in the Neck Region}

Let us begin with the following, which tells us that the harmonic radius at a point in the neck region is roughly the distance of that point to the effective critical set.  The proof is immediate using $(n3)$ and $(n4)$ of a neck region, but it is worth mentioning the result explicitly:\\

\begin{lemma}\label{l:neck:distance_harm_rad}
	Let $(M^n,g,p)$ satisfy $\Vol(B_1(p))>\rv>0$ with $|{\Ric}|<\delta$ and $\cN\subseteq B_2(p)$ a $(\delta,\tau)$-neck region.  Then for each $\epsilon>0$ we have for $\delta<\delta(n,\rv,\epsilon)$ and $\tau<\tau(n,\epsilon)$ that for each $y\in \cN_{10^{-6}}$ it holds that
	\begin{align}
		(1-\epsilon)d(y,\cC)\leq r_h(y)\leq (1+\epsilon)d(y,\cC)\, .
	\end{align}
\end{lemma}
\vspace{.5cm}

\subsubsection{Wedge Regions in $(\delta,\tau)$-Necks}\label{sss:neck:wedge}

A region which will play a useful technical role for us is that of a wedge region.  Precisely:

\begin{definition}\label{d:wedge_region}
	Let $\cN = B_2(p)\setminus \overline B_{r_x}(\cC)$ be a $(\delta,\tau)$-neck region with $\tau,\delta<10^{-n}$.  Then for each center point $x\in \cC$ we define the wedge region $W_x$ associated to $x$ by
	\begin{align}
		W_x \equiv \Big\{y\in A_{10^{-2}r_x, 1}(x): d(y,x)< 2\cdot d(y,\cC)\Big\}\, .
	\end{align}
\end{definition}
\vspace{.5 cm}

Wedge regions are the correct scale invariant regions that many of our estimates will live on.  Let us point out that although the lipschitz condition $(n4)$ of the radius function is quite useful for many of the technical results, the only point where it plays a role of any consequential importance is in controlling the wedge regions.  Indeed, note that $W_x\subseteq \cN_{10^{-3}}$, however if a neck region consisted of an arbitrary covering which did not necessarily satisfy $(n4)$, then this may fail and the intuitive picture of a wedge region may not coincide with the actual picture.\\

\subsubsection{Cutoff Functions on Neck Regions}

In this subsection we build a natural cutoff function associated with a neck region.  We will record some of the basic properties and estimates associated to it, which will be useful throughout the paper.  Precisely:\\

\begin{lemma}\label{l:neck:cutoff}
	Let $(M^n,g,p)$ satisfy $\Vol(B_1(p))>\rv>0$ with $|{\Ric}|<\delta$ and $\cN\subseteq B_2(p)$ a $(\delta,\tau)$-neck region.  Then for $\tau<\tau(n)$ and $\delta<\delta(n,\rv,\tau)$ there exists a cutoff function $\phi=\phi_1\cdot \phi_2 \equiv \phi_\cN: B_2\to [0,1]$ such that
	\begin{enumerate}
	\item $\phi_1(y) = 1$ if $y\in B_{18/10}(p)$ with $\text{supp}(\phi_1) \subseteq B_{19/10}(p)$.
	\item $\phi_2(y)\equiv 1$ if $y\in \cN_{10^{-3}}$ with $\text{supp}(\phi_2)\cap B_2 \subseteq \cN_{10^{-4}}$.
	\item $|\nabla \phi_1|$, $|\Delta \phi_1|\leq C(n)$.
	\item $\text{supp}(|\nabla \phi_2|)\cap B_{19/10}\subseteq  A_{10^{-4}r_x,10^{-3}r_x}\big(\cC\big)$ with $r_x|\nabla \phi_2|$, $r^2_x|\nabla^2\phi_2|<C(n,\tau,\rv)$ in each $B_{r_x}(x)$.
	\end{enumerate}
\end{lemma}
\begin{proof}
Recall from \cite{ChC1} that for each annulus $A_{s,r}(x)$ we can build a cutoff function $\phi$ such that
\begin{align}\label{e:cutoff}
    &\phi\equiv 1 \text{ on }B_s(x)\, ,\;\;\phi\equiv 0 \text{ outside }B_r(x)\, ,\notag\\
	&|\nabla \phi|\leq \frac{C(n)}{|r-s|}\, ,\;\;|\Delta \phi|\leq \frac{C(n)}{|r-s|^2}\,\, .
\end{align}
It will be important to recall that $\phi$ is built as a composition $\phi=c\circ f$, where $c$ is a smooth cutoff on $\dR$ with $c=1$ on $[0,s^2]$ and $c=0$ outside $[0,r^2]$, and $f$ satisfies $\Delta f=2n$ on $B_{5r}(x)$ and is uniformly equivalent to the square distance $d^2_x$.

Thus we can let $\phi_1$ be the cutoff associated to the annulus $A_{\frac{18}{10},\frac{19}{10}}(p)$.  To build $\phi_2$ let us begin by defining for each $x\in \cC$ the cutoff $\phi_x=c_x\circ f_x$ as in \eqref{e:cutoff} associated to the annulus $A_{10^{-4}r_x,10^{-3}r_x}(x)$.  Using elliptic estimates we can get the pointwise estimate $|\nabla^2 f_x|<C(n) r_x^{-2}$ on $A_{10^{-4}r_x,10^{-3}r_x}(x)\cap \cN_{10^{-4}}$.  In particular we can get the estimate
\begin{align}\label{e:cutoff:1}
	r_x|\nabla \phi_x|,\, r_x^2|\nabla^2 \phi_x|<C(n) \text{ on }A_{10^{-4}r_x,10^{-3}r_x}(x)\cap \cN_{10^{-4}}\, .
\end{align}

Now with $\tau<10^{-5}$ let us define the cutoff function
\begin{align}
\phi_2 \equiv \prod \Big(1-\phi_x\Big)\, .	
\end{align}
We have that $\text{supp}\,\phi_2 \cap B_2 \subseteq \cN_{10^{-4}}$ and for each point $y\in B_2(p)$ we have by $(n4)$ that there are at most $C(n,\tau,\rv)$ of the cutoff functions $\phi_x$ which are nonvanishing at $y$.  In particular, the estimates \eqref{e:cutoff:1} then easily imply the required estimates on $\phi_2$.
\end{proof}

\vspace{.5 cm}

\subsection{Approximating Neck Regions on Limit Spaces}\label{ss:neck_approximation}

Most of our theorems on neck regions for limit spaces will be proved by first showing the corresponding results on smooth manifolds and then passing to a limit.  The reason for this is two fold.  First, it is simply more convenient to work on a manifold.  However, the more important reason is that there is a subtle point in the inductive proof of the Ahlfors regularity estimates which force one to first prove all results in the case $\cC_0=\emptyset$.  Essentially, this is because there is no base step for the induction argument otherwise.  The following approximation results will allow us to approximate the general case by such a scenario:\\

\begin{theorem}\label{t:neck_approximation}
	Let $(M^n_j,g_j,p_j)\stackrel{GH}{\to} (X,d,p)$ with $\Vol(B_1(p_j))>\rv>0$, $|\Ric_i|\le \delta$ and $\cN\equiv B_2(p)\setminus \overline B_{r_x}(\cC)$ a $(\delta,\tau)$-neck region.  Then there exists a sequence of $(\delta_i,\tau_i)$-neck regions $\cN_i\equiv B_2(p_i)\setminus \overline B_{r_{x,i}}(\cC_i)$ such that $\cN_i\to \cN$ in the following sense:
	\begin{enumerate}
	\item $\delta_i\to \delta$, $\tau_i\to \tau$.
	\item Let  $\psi_i:B_2(p_i)\to B_2(p)$ be the Gromov-Hausdorff maps then  $\psi_i(\cC_i)\to \cC$ in Hausdorff sense.
	\item $r_{x,i}\to r_x:\cC\to \dR^+$ uniformly \footnote{``uniform convergence'' means: $r_{x,r}\circ \phi_i\to r_x:\cC\to \mathbb{R}^+$ uniformly, where $\phi_i: \cC\to \cC_i$ is the Gromov-Hausdorff maps. } . 
	\item If $\mu_i$, $\mu$ are the packing measures of $\cN_i$ and $\cN$, respectively, then if we limit $\mu_i\to \mu_\infty$ we get the uniform comparison $\mu\leq C(n,\tau) \mu_\infty$.
	\item Conversely, we have the estimate $\mu_\infty \leq C(n,\tau)\,\mu$.	\end{enumerate}
\end{theorem}

\begin{remark}
	We will prove $(1)$-$(4)$ in full generality now.  The estimate $(5)$ we will prove under the assumption that $\cC_0$ is rectifiable, which will be a consequence of Theorem \ref{t:neck_region}.  We will only sketch $(5)$ as the main results in the paper will not rely on it.
\end{remark}
\begin{remark}
More generally, we will see without prior knowledge of rectifiability that $\mu_\infty$ is uniformly equivalent to $\mu_+=\mu\cap \cC_+$ plus the upper minkowski $n-4$-content of $\cC_0$.  The point then is that when $\cC_0$ is $n-4$-rectifiable we have that the $n-4$-content is equivalent to the $n-4$-Hausdorff measure.
\end{remark}

\vspace{.5cm}

The proof of Theorem \ref{t:neck_approximation} will be done in two steps.  The first is to approximate $\cN\subseteq B_2(p)$ by a sequence of neck regions $\tilde \cN_a\subseteq B_2(p)$ which also live in the singular space, but which are discrete.  After this is accomplished we will approximate much more directly the discrete neck regions $\tilde \cN_a$ by neck regions living in the smooth approximating spaces.  Let us begin by constructing the discrete neck regions $\tilde \cN_a$:\\

\begin{lemma}\label{l:neck_approximation}
	Let $(M^n_j,g_j,p_j)\stackrel{GH}{\to} (X,d,p)$ with $\Vol(B_1(p_j))>\rv>0$, $|\Ric_i|\le \delta$ and $\cN\equiv B_2(p)\setminus \overline B_{r_x}(\cC)$ a $(\delta,\tau)$-neck region.  Then there exists a sequence of $(\delta_a,\tau_a)$-neck regions $\tilde \cN_a\equiv B_2(p)\setminus \overline B_{r_{x}}(\tilde \cC_a)$ with $r_{x,a}>r_a>0$ uniformly bounded from below and $\tilde \cN_a\to \cN$ in the following sense:
	\begin{enumerate}
	\item $\delta_a\to \delta$, $\tau_a\to \tau$.
	\item $\tilde \cC_a\to \cC$ in the Hausdorff sense.
	\item $r_{x,a}\to r_x:\cC\to \dR^+$ uniformly.
	\item If $\tilde \mu_a$, $\mu$ are the packing measures of $\tilde \cN_a$ and $\cN$, respectively, then if we limit $\tilde \mu_a\to \tilde \mu_\infty$ we get the uniform comparison $\mu\leq C(n,\tau) \tilde \mu_\infty$.
	\item We have the estimate $\tilde \mu_\infty \leq C(n,\tau)\,\mu$.	\end{enumerate}
\end{lemma}
\begin{proof}
	For any $r>0$ let us define the function $\tilde r_x$ on $\cC$ given by $\tilde r_x \equiv \max\{r_x,r\}$.  Note that $|\Lip\, \tilde r_x|<\delta$, and that all the properties of a neck region are satisfied with $\cC$ and $\tilde r_x$ except potentially the Vitali condition.  Thus let us choose a maximal subset $\tilde \cC_r\equiv \{x^r_i\}\subseteq \cC$ such that the balls $\big\{B_{\tau^2 r_i}(x^r_i)\big\}$ are disjoint.  It is easy to check that $\tilde N_r \equiv B_2\setminus B_{\tilde r_x}(\tilde \cC_r)$ is then a $(\delta,\tau)$-neck region itself for each $r>0$.  We need to understand the convergence of these neck regions to $\cN$.  Conditions $(1)\to (3)$ are clear.  In order to see $(4)$ and $(5)$ we can focus on the set $\cC_0$, as on $\cC_+$ the limit of $\tilde \mu_r$ is exactly $\mu_+=\mu\cap \cC_+$.  Thus let us observe the following.  If $y\in \cC_0$ then for all $r<<s$ we have by the definition of $\tilde\mu_r$ the uniform minkowski estimates
\begin{align}
&r^{-4}\Vol\big(B_s(y)\cap B_r(\cC_0)\big)\leq C(n,\tau)\, \tilde\mu_r(B_{s+r}(y))\notag\\
&\tilde\mu_r(B_{s-r}(y)\cap B_r(\cC_0))\leq C(n,\tau)\,r^{-4}\Vol\big(B_{s}(y)\cap B_r(\cC_0)\big)\, .
\end{align}
Limiting $r\to 0$ proves that $\tilde\mu_\infty\cap \cC_0$ is equivalent to the minkowski $n-4$-content, which in particular proves $(4)$.  Given that $\cC_0$ is rectifiable, then a standard geometric measure theory argument tells us that the minkowski content is itself uniformly equivalent to the hausdorff measure, which thus proves $(5)$.
\end{proof}
\vspace{.5cm}

Let us now finish the proof of Theorem \ref{t:neck_approximation}:

\begin{proof}[Proof of Theorem \ref{t:neck_approximation}]
	
	Let us begin by remarking that because of Lemma \ref{l:neck_approximation} we need only consider the case when our neck region $\cN$ satisfies $\inf r_x>0$.  Indeed, in the general case we may approximate $\cN$ by a sequence $\tilde \cN_a$ of such neck regions with $\tilde \cN_a\to \cN$.  If we can therefore find for each of these neck regions smooth approximations $\cN_{ai}\subseteq B_2(p_i)$ with $\cN_{ai}\to \tilde \cN_a$, then by taking a diagonal subsequence we have approximated $\cN$ itself in the sense of Theorem \ref{t:neck_approximation}.\\
	
	Thus let us assume $\cN$ satisfies $\inf r_x>0$.  In particular notice that $\cC=\cC_+$ is a finite set in this case.  Let $\phi_i:B_{2}(p)\to B_{2}(p_i)$ be the $\epsilon_i$-Gromov Hausdorff maps.  For $i$ sufficiently large with $\epsilon_i<< \inf r_x$ note that we can consider the center points $\cC_i\equiv \{\phi_i(x)\}_{x\in \cC}$ with the radius function $r_{x,i} \equiv r_{\phi_i^{-1}(x)}$.  With $\cC$ finite it is then an easy exercise to check that $\cN_i\equiv B_2\setminus \overline B_{r_{x,i}}(\cC_i)$ are $(\delta_i,\tau_i)$-neck regions which satisfy the criteria of the Theorem, as claimed.
\end{proof}
\vspace{.5cm}

\subsection{Reifenberg and Lower Ahlfors Regularity}\label{ss:neck:reifenberg}

In this section we build a Reifenberg type map for the center points of our neck regions.  The constructions of this section generalize the usual Reifenberg construction of \cite{Reifenberg}, \cite{ChC2} in that the mapping is built with respect to a general covering.  The construction of this section is related to the constructions of \cite{NaVa_Rect_harmonicmap}.  Our main application of this will be in Section \ref{s:neck_region_proof} to prove the lower Ahlfors regularity bound of Theorem \ref{t:neck_region}.  Let us begin with the precise construction:

\begin{theorem}\label{t:neck_reifenberg}
	Let $(M^n_j,g_j,p_j)\stackrel{GH}{\to} (X,d,p)$ satisfy $\Vol(B_1(p_j))>\rv>0$ with $\Ric_j\geq -\delta$ and let $\cN\equiv B_2(p)\setminus \overline B_{r_x}(\cC)$ a $(\delta,\tau)$-neck region with $\max_{x\in \cC}r_x\le 10^{-3}$.  For each $\epsilon>0$ if $\tau<\tau(n)$ and $\delta\leq \delta(n,\rv,\epsilon)$, then there exists a map $\Phi: {\cC}\cap B_{19/10}(p)\to \dR^{n-4}$  with $\Phi(p)=0^{n-4}$ such that 
	\begin{enumerate}
		\item $(1-\epsilon)d(x,y)^{\epsilon}\leq \frac{|\Phi(x)-\Phi(y)|}{d(x,y)}\leq d(x,y)^{-\epsilon}$ 
		\item For each $x\in \cC\cap B_{9/5}(p)$ and $r_x\leq r\le 10^{-3}$ there exists $T_{x,r}=T\in\dR^{(n-4)\times(n-4)}$ such that the map $T\circ\Phi: B_{30r}(x)\cap \cC \to \dR^{n-4}$ is an $\epsilon r$-isometry and $T_{x,10^{-3}}=Id$.
		\item $|T_{x,r}\circ (T_{x,s})^{-1}|+|T_{x,s}\circ (T_{x,r})^{-1}|\le C(n,\epsilon)\left(\frac{r}{s}\right)^{\epsilon}$ for all $r_x\le s\le r\le 10^{-3}$ and $x\in\cC\cap B_{9/5}(p)$.
		\item For any $x,y\in \cC\cap B_{9/5}(p)$ and $10^{-3}\ge r\ge r_x$ with $ d(x,y)\le 10r$,  there exists $O_{xy,r}\in O(n-4)$ such that $|T_{x,r}\circ (O_{xy,r}T_{y,r})^{-1}-Id|\le \epsilon$. 
		\item If $T_x:= T_{x,r_x}$, then we have the covering  $B_{7/4}(0^{n-4})\subseteq \bigcup_{x\in \cC\cap B_{9/5}(p)} T_x^{-1}\bar B_{r_x}(\Phi(x))$, where $T_x^{-1}\bar B_{r_x}(\Phi(x)):=\{T_x^{-1}v+\Phi(x): v\in \bar B_{r_x}(0^{n-4})\}\subset \mathbb{R}^{n-4}$. 	\end{enumerate}
\end{theorem}

\begin{remark}
Actually, one can prove $T_x^{-1}\{B_{\tau^3 r_x}(\Phi(x))\}$ are disjoint in (5).  We omit the proof here since we don't need to use this. See also the claim in Subsection \ref{ss:neck_region_smooth:lower_volume} for a similar statement and proof.
\end{remark}

\begin{remark}
Condition $(1)$ tells us that the mapping $\Phi$ is uniformly bih\"older. 
\end{remark}
\begin{remark}
The only manner in which the neck region structure is used is to see that $\cC$ is a reifenberg set.  Indeed, it is not hard to check that the constructions of this theorem work for a general reifenberg type set.
\end{remark}
\begin{remark}
In the statement of Theorem \ref{t:neck_reifenberg} and its proof.  We call a map $f: B_r(x)\cap \cC\to \mathbb{R}^{n-4}$ is an $\epsilon r$-isometry if for any $y,z\in B_r(x)\cap \cC$ the following holds:
\begin{align}
\Big||f(z)-f(y)|-d(z,y)\Big|\le \epsilon r,
\end{align}
where the image $f(B_r(x)\cap \cC)$ may not be $\epsilon r$-dense in $B_r(f(x))\subset \mathbb{R}^{n-4}$.
\end{remark}
\begin{remark}
	 Let $(X,d,p)$ be a pointed compact metric space. If $f,g: B_r(p)\to \mathbb{R}^k$ are two $\epsilon r$-isometries with $f(p)=g(p)=0^k$ and both $f,g$ are $\tau r$-dense in $B_r(0^k)$ with $\tau<\tau(k)$,  then one can show that there exists $A\in O(k)$ such that $\sup_{x\in X}|f-Ag|\le C(k){\epsilon}$. 
	 We will use this fact several times in our proof.
	 Actually, in our proof the $\epsilon r $-isometry will be a restriction of some $\epsilon r$-GH map and the restricted map is $\tau r $-dense with $\tau$ small. 
	\end{remark}
\begin{proof}
Let us choose the constants $\eta\equiv 10^{-6n}\,\epsilon^2$ and $\delta\leq \eta^{10}$, and let us consider the scales $t_\beta \equiv \eta^{2\beta}$.  We will essentially prove the Theorem inductively on $\beta$.	More precisely, for each $x\in \cC$ let us define $r^\beta_x\equiv \max\{t_\beta,r_x\}$,  then our goal is to build maps $\Phi^\beta:\cC\cap B_{\frac{19}{10}+100t_{\beta-1}}(p)\to \dR^{n-4}$ such that the following hold: 
\begin{align}
     (\beta.0) &\;\Phi^0:\cC\to \dR^{n-4}\text{ is a }\eta^{10}\text{-isometry}\, ,\notag\\
	(\beta.1) &\forall \, x\in \cC\cap B_{{19}/{10}}(p)\; \exists\; \text{ linear } T^\beta_x \in \mathbb{R}^{(n-4)\times (n-4)}\text{ such that }\Phi_x^\beta\equiv T^\beta_x\circ \,\Phi^\beta:B_{50r_x^\beta}(x)\cap \cC\to \dR^{n-4} \text{ is a }\eta^6\, r_x^\beta\text{-isometry}\, ,\notag\\
	(\beta.2) &\sup_{y\in B_{40r_x^\beta}(x)\cap\cC}|T_x^\beta\circ (\Phi^\beta-\Phi^{\beta+1})|(y)\le \eta^5 r_x^\beta\,, \notag\\
		(\beta.3) &|T_x^\beta\circ (T_x^{\beta-1})^{-1}-Id|\le \eta^2 \text{  and } |T_x^{\beta-1}\circ (T_x^{\beta})^{-1}-Id|\le \eta^2\,. \notag\\
\end{align}

Before building the maps $\Phi^\beta$ let us check to see that we can finish the Theorem from their construction by choosing $\tau\le \tau(n)$.

By $(\beta.3)$ and noting that 
\begin{align}
T_x^\beta\circ (T_x^{\beta-2})^{-1}-Id=(T_x^\beta\circ (T_x^{\beta-1})^{-1}-Id)(T_x^{\beta-1}\circ (T_x^{\beta-2})^{-1}-Id)+(T_x^\beta\circ (T_x^{\beta-1})^{-1}-Id)+(T_x^{\beta-1}\circ (T_x^{\beta-2})^{-1}-Id)
\end{align}
we have $|T_x^\beta\circ (T_x^{\beta-2})^{-1}-Id|\le \eta^4+2\eta^2$. Inductively, we have for all $0\le \alpha<\beta$ that
\begin{align}\label{e:TxbetaTxalpha-Id}
|T_x^\beta\circ (T_x^{\alpha})^{-1}-Id|\le \sum_{k=1}^{\beta-\alpha-1}C_{\beta-\alpha}^k\eta^{2k}=(1+\eta^2)^{\beta-\alpha}-1,
\end{align}
where $C_{\beta-\alpha}^k=\frac{(\beta-\alpha)!}{k!(\beta-\alpha-k)!}$. This implies 
\begin{align}\label{e:TxbetaTxalpha1}
|T_x^\beta\circ (T_x^\alpha)^{-1}|\le (1+\eta^2)^{\beta-\alpha}.
\end{align}
Similarly, by using $|T_x^{\beta-1}\circ (T_x^{\beta})^{-1}-Id|\le \eta^2$ in $(\beta.3)$ we get 
\begin{align}\label{e:TxbetaTxalpha2}
|T_x^\alpha\circ (T_x^\beta)^{-1}|\le (1+\eta^2)^{\beta-\alpha}.
\end{align}
Noting that $T_x^0=Id$ we conclude 
\begin{align}\label{e:TxbetaTxbeta-1norm}
\max\{ |T_x^{\beta}|,|(T_x^{\beta})^{-1}|\}\le (1+\eta^2)^{\beta}.
\end{align}
Combining with $(\beta.2)$ we have 
\begin{align}
|\Phi^{\beta}(x)-\Phi^{\beta+1}(x)|\le \eta^5r_x^\beta(1+\eta^2)^{\beta}.
\end{align}
This implies for all $0\le \alpha\le \beta$ that
\begin{align}
|\Phi^{\beta+1}-\Phi^{\alpha}|(x)\le \sum_{k=\alpha}^{\beta} \eta^5r_x^k(1+\eta^2)^{k}\le \sum_{k=\alpha}^\beta \eta^5 \eta^{2k(1-\eta)}.
\end{align}
Thus $\{\Phi^\beta(x)\}$ is a Cauchy sequence. We can define $\Phi(x)=\lim_{\beta\to \infty}\Phi^\beta(x)$. In particular, we have 
\begin{align}\label{e:Phi-Phialpha}
|\Phi(x)-\Phi^\alpha(x)|\le \sum_{k=\alpha}^{\infty} \eta^5\eta^{2k(1-\eta)}\le \eta^4 \eta^{2\alpha(1-\eta)}.
\end{align}
Let us now check (1) in Theorem \ref{t:neck_reifenberg}. For any $x,y\in \cC\cap B_{{19}/{10}}(p)$, assume $\tau^2 r_x^{\alpha+1}< d(x,y)\le r_x^{\alpha} $. We have by $(\alpha.1)$ that $T_x^\alpha \Phi^\alpha: B_{50r_x^\alpha}(x)\cap \cC\to \mathbb{R}^{n-4}$ is a $\eta^6r_x^\alpha$\text{-isometry}. This implies 
\begin{align}
\Big||T_x^\alpha \Phi^\alpha(x)-T_x^\alpha\Phi^\alpha(y)|-d(x,y)\Big|\le \eta^6r_x^\alpha.
\end{align}
Using \eqref{e:TxbetaTxbeta-1norm} we conclude that 
\begin{align}\label{e:Phialpha_holder}
(1-\eta^3)(1+\eta^2)^{-\alpha}\le \frac{|\Phi^\alpha(x)-\Phi^\alpha(y)|}{d(x,y)}\le (1+\eta^3)(1+\eta^2)^\alpha.
\end{align}
The statement (1) follows now easily by the combination of \eqref{e:Phi-Phialpha} and \eqref{e:Phialpha_holder}. Thus we have proved (1). 
On the other hand, arguing similar as the proof of \eqref{e:Phi-Phialpha}, we can get 
\begin{align}\label{e:Phi-Phialpha2}
|T_x^{\alpha}\Phi(x)-T_x^\alpha\Phi^\alpha(x)|\le  \eta^4 r_x^\alpha.
\end{align}
By defining $T_{x,r}=T_{x}^\beta$ for $r_{x}^{\beta+1}<r\le r_{x}^\beta$ then (2), (3) follow directly from $(\beta.1)$, \eqref{e:TxbetaTxalpha1}, \eqref{e:TxbetaTxalpha2} and \eqref{e:Phi-Phialpha2}.

Let us now check (4). Actually, by $(\beta.1)$ and the choice of $T_{x,r}$ and $T_{y,r}$, we have that both $T_{x,r}\Phi: B_{50r}(x)\cap \cC\to \mathbb{R}^{n-4}$ and $T_{y,r}\Phi: B_{50r}(y)\cap \cC\cap B_{19/10}(p)\to \mathbb{R}^{n-4}$ are $\eta r$\text{-isometry}. Since $y\in B_{10r}(x)$, $x\in B_{9/5}(p)$ and $r\le 10^{-3}$,  we have $B_{10r}(x)\subset B_{50r}(x)\cap B_{50r}(y)\subset B_{19/10}(p)$. Thus  the maps $T_{x,r}\Phi, T_{y,r}\Phi: B_{10r}(x)\cap \cC\to \dR^{n-4}$ are $\eta r$\text{-isometry}. In particular, $T_{x,r}(\Phi-\Phi(x)), T_{y,r}(\Phi-\Phi(x)): B_{10r}(x)\cap \cC\to \dR^{n-4}$ are $\eta r$\text{-isometry}. On the other hand, by (n3) of neck region the maps $T_{x,r}(\Phi-\Phi(x)), T_{y,r}(\Phi-\Phi(x)): B_{10r}(x)\cap \cC\to \dR^{n-4}$ are $C(n)\tau r$\text{-dense}. Therefore, if $\tau\le \tau(n)$ then there exists $O_{xy,r}\in O(n-4)$ such that 
\begin{align}
C(n)\eta r&\ge \sup_{z\in B_{10r}(x)\cap \cC} |T_{x,r}(\Phi(z)-\Phi(x))-O_{xy,r}T_{y,r}(\Phi(z)-\Phi(x))|\\ \label{e:TxrTyrrotation}
&=\sup_{z\in B_{10r}(x)\cap \cC} |(T_{x,r}\circ (O_{xy,r}T_{y,r})^{-1}-Id)O_{xy,r}T_{y,r}(\Phi(z)-\Phi(x))|.
\end{align}
By (n3) of neck region and  noting that $O_{xy,r}T_{y,r}(\Phi-\Phi(x)): B_{r}(x)\cap \cC\to \dR^{n-4}$ is a $\eta r$\text{-isometry},  we have that the image $\{O_{xy,r}T_{y,r}(\Phi(z)-\Phi(x)): z\in \cC\cap B_{r}(x)\}$ is $C(n)\tau r$-dense in  $B_{r}(0^{n-4})\subset \dR^{n-4}$. Therefore, (4) of Theorem follows directly from \eqref{e:TxrTyrrotation} if $\tau\le \tau(n)$. Thus we have proved (4).

The main point now is to check that $\Phi$ satisfies the covering estimate of $(5)$.  As we shall see it is roughly a discrete version of the topological argument that a continuous degree one map must be onto.  Indeed, it follows from $(\beta.1)$ and $(n3)$ that for any $x\in \cC\cap B_{{19}/{10}}(p)$ and $10^{-3}\ge r\geq r_x$ with $B_{r}(x)\subset B_{19/10}(p)$ that 
\begin{align}\label{e:reifenberg:3}
\Big(A_{x,r}\circ\Phi\big(\cC\cap B_r(x)\big)\Big)\text{ is $C(n)(\tau+\eta) r$-dense in }B_{r}\big(\Phi(x)\big)\, ,	
\end{align}
where $A_{x,r}\circ\Phi(y)=T_{x,r}(\Phi(y)-\Phi(x))+\Phi(x)$. 
So now let us assume the covering statement of $(5)$ fails, and try to find a contradiction.  Thus let $\bar y\in B_{7/4}(0^{n-4})$ be a point such that $\bar y\not\in \bigcup_{x\in\cC\cap \bar B_{9/5}(p)} T_x^{-1}\bar B_{r_x}(\Phi(x))$, where we should notice that by \eqref{e:Phi-Phialpha} and $\Phi^0(p)=0^{n-4}$  we have $|\Phi(p)|\le \eta^4$.  Define $ s_{x}$ for each $x\in\cC\cap \bar B_{9/5}(p)$ to be the minimal $s$ such that
\begin{align}
\bar y\in T_{x,s}^{-1}\bar B_{s}(\Phi(x))
\end{align}
Thus we have that $s_x>r_x$ by our covering assumption.  Now let $y\in\cC\cap \bar B_{9/5}(p)$ be such that $s_y=\min_x s_x$, and hence $\bar y\in T_{y,2s_y}^{-1}\bar B_{2s_y}(\Phi(y))$.  To see this, by $(\beta.3)$ we have $|T_{y,s_y}T_{y,2s_y}^{-1}-Id|\le \eta$. Thus $\bar y\in T_{y,s_y}^{-1}\bar B_{s_y}(\Phi(y))\subset T_{y,2s_y}^{-1}\bar B_{2s_y}(\Phi(y))$.  Furthermore, if $\tau\le \tau(n)$ we must have $s_y\le 10^{-5}$ 
since by applying $\Phi:B_{9/5}(p)\cap \cC\to \dR^{n-4}$ to \eqref{e:reifenberg:3} we have that $\{\Phi(x):x\in B_{9/5}(p)\cap \cC\}$ is $C(n)\tau$-dense in $B_{7/4}(0^{n-4})$. 
Therefore, $B_{2s_y}(y)\subset B_{19/10}(p)$. Thus by \eqref{e:reifenberg:3} there exists $x\in \cC\cap B_{2s_y}(y)$ such that $|T_{y,2s_y}(\bar y-\Phi(x))|\le C(n)(\tau+\eta)s_y$. Noting that $|T_{y,2s_y}^{-1}|\le (2s_y)^{-\epsilon}$, we have $|\bar y-\Phi(x)|\le C(n)(\tau+\eta)s_y^{1-\epsilon}$. If $\tau\le \tau(n)$, we conclude by $(\beta.0)$ and \eqref{e:Phi-Phialpha} that $x\in \cC\cap B_{9/5}(p)$.
On the other hand, up to a rotation on $T_{y,2s_y}$ we notice that $|T_{y,2s_y}T_{x,2s_y}^{-1}-Id|<\epsilon$ by (4). Thus we get 
$|T_{x,2s_y}(\bar y-\Phi(x))|\le C(n)(\tau+\eta)s_y$. If $\tau<\tau(n)$, then by (3) this implies that $s_x<\frac{1}{2}s_y$, which is a contradiction. Hence we have finished the proof of (5). \\

Now we have shown that if the inductive condition $(\beta.1)-(\beta.3)$ holds for all $\beta\in \dN$ then our main Theorem follows.  Let us now focus on proving this inductive statement.  Before beginning let us begin by defining some basic useful structure which we will use:

\begin{enumerate}
\item[(s1)] Let $\cC^\beta =\{x^\beta_i\}\subseteq \cC\cap B_{\frac{19}{10}+t_{\beta-2}}(p)$ be nested maximal subsets such that $\{B_{\eta^{-2}r^\beta_x}(x)\}$ with $x\in \cC^\beta$ are disjoint, where $t_\beta=\eta^{2\beta}$.  In particular, we have that $\{B_{5\eta^{-2} r^\beta_x}(x)\}$ covers $\cC\cap B_{\frac{19}{10}+t_{\beta-2}}(p)$. 
\item[(s2)] Let $\psi^\beta_i:B_{10\eta^{-2} r^\beta_i}(x^\beta_i)\to \dR$ be an associated partition of unity.  Specifically, let $\psi:\dR\to [0,1]$ be a fixed smooth function with $\psi=1$ on $[0,5]$ and $\psi=0$ outside $[0,7]$, and let $\hat\psi^\beta_i(x)=\psi((r^\beta_i)^{-2}\eta^4 d(x,x^\beta_i)^2)$, then we define
\begin{align}\label{e:partition_unity_explicit}
	\psi^\beta_i(x) =\frac{\hat\psi^\beta_i(x)}{\sum_j \hat\psi^\beta_j(x)}.
\end{align}
Note by $(n4)$ that we have $r^\beta_i|\nabla \psi^\beta_i|\leq C(n)\eta^2$.
\item[(s3)] Let $\phi^\beta_i:B_{\eta^{-4}r^{\beta}_i}(x^\beta_i)\cap \cC\to \dR^{n-4}$ be a $\eta^{-4}\delta r_i^\beta$\text{-isometry} which maps $x_i^\beta$ to $ 0$. 
\end{enumerate}

Now we will build $\Phi^\beta$ inductively.  Let $\Phi^1=\Phi^0=\phi^{0}:\cC\to \dR^{n-4}$ be the chosen $\eta^{-4}\delta$\text{-isometry}.  We begin by writing a general ansatz for the construction of $\Phi^\beta:$ 
\begin{align}
\Phi^{\beta}(x) = \sum_{x^{\beta}_i}\psi^{\beta}_i(x)\Phi^{\beta}_i(x)\equiv \sum_{x^{\beta}_i}\psi^{\beta}_i(x)A_i^{\beta}\circ\phi^{\beta}_i(x) \equiv \sum_{x^{\beta}_i}\psi^{\beta}_i(x)(\alpha^{\beta}_i+J_i^{\beta}\circ\phi^{\beta}_i(x))\, ,	
\end{align}
where $\alpha^{\beta}_i\in \dR^{n-4}$ and $J^{\beta}_i\in \dR^{(n-4)\times(n-4)}$ are linear maps, so that $A^{\beta}_i = \alpha^{\beta}_i+J^{\beta}_i\circ$ is an affine map.  Thus we see that $\Phi^\beta$ is essentially a piecewise affine linear map, in some approximate sense. 

We will prove inductively the following claim, which will imply $(\beta.1)-(\beta.3)$. 

\noindent
\textbf{Claim:} If $\delta\le \delta(n,\eta,\rv)$ and $0<\tau\le \tau(n)$, there exists maps $\Phi^\beta:\cC\cap B_{\frac{19}{10}+100t_{\beta-1}}(p)\to \dR^{n-4}$ such that the following hold 
\begin{enumerate}
\item[$(\bar \beta.0)$] $\Phi^1=\Phi^0=\phi^0:\cC\to \dR^{n-4}$,
\item[($\bar\beta.1$)] For any $x\in \cC\cap B_{\frac{19}{10}+t_{\beta-1}}(p)$ there exists linear map $L_x^{\beta}\in \dR^{(n-4)\times(n-4)}$ such that $L_x^{\beta}\circ \Phi^{\beta}: B_{100r_x^{\beta}}(x)\cap \cC\to \mathbb{R}^{n-4}$ is an $\eta^{6}r_x^{\beta}$\text{-isometry},
\item[($\bar\beta.2$)] For any $x\in \cC\cap B_{\frac{19}{10}+t_{\beta-1}}(p)$, $\sup_{y\in B_{40r_x^{\beta-1}}(x)\cap\cC}|L_x^{\beta-1}\circ (\Phi^\beta-\Phi^{\beta-1})|(y)\le \eta^5 r_x^{\beta-1}$,
\item[($\bar\beta.3$)] For any $y,x\in \cC\cap B_{\frac{19}{10}+t_{\beta-1}}(p)$ with $d(x,y)\le 60r_x^{\beta}$, there exists $O_{xy}^\beta\in O(n-4)$ such that $|(L_x^\beta)^{-1}\circ L_y^{\beta}-O_{xy}^\beta|\le \eta^5$.  
\end{enumerate}
\vskip 3mm


\noindent
\textbf{Proof of Claim:} We will prove the result by induction on $\beta$.  Note that the claim holds trivially for $\beta=0,1$ by letting $L_x^{\beta}=Id$ for $\beta=0,1$.  Therefore let us assume for some $\beta$ that we have constructed $\Phi^\beta$ such that ($\bar\beta.1$), ($\bar\beta.2$) and ($\bar \beta.3$) hold, then we will construct $\Phi^{\beta+1}$ and prove that ($\overline{\beta+1}.1$), ($\overline{\beta+1}.2$) and  ($\overline{\beta+1}.3$) hold. 

  Note first that ($\overline{\beta+1}.3$) follows directly from ($\overline{\beta+1}.1)$, see also the above proof of (4) of Theorem, and hence it will suffice to prove ($\overline{\beta+1}.1$)  and  ($\overline{\beta+1}.2$).  In order to construct $\Phi^{\beta+1}$ it suffices to construct the affine maps $\Phi_i^{\beta+1}=A_i^{\beta+1}\circ \phi_i^{\beta+1}$. Since $\phi^{\beta+1}_i:B_{\eta^{-4}r^{\beta+1}_i}(x^{\beta+1}_i)\cap \cC\to \dR^{n-4}$ is a $\eta^{-4}\delta r_i^{\beta+1}$\text{-isometry}  and $L_i^\beta\circ \Phi^{\beta}: B_{100r^{\beta}_i}(x^{\beta+1}_i)\cap \cC\to \dR^{n-4}$ is a $\eta^6r_i^{\beta}$\text{-isometry} where $L_i^\beta=L_{x_i^{\beta+1}}^{\beta}$, there exist $\bar O_i^{\beta+1}\in O(n-4)$ and $\bar \alpha_i^{\beta+1}\in \mathbb{R}^{n-4}$ such that if $\delta\le \delta(n,\rv,\eta)$ we have
\begin{align}
\sup_{x\in B_{100r^{\beta}_i}(x^{\beta+1}_i)\cap \cC}\Big|L_i^{\beta}\circ \Phi^{\beta}(x)-(\bar \alpha_i^{\beta+1}+\bar O_i^{\beta+1} \phi^{\beta+1}_i(x))\Big|\le C(n)\eta^6 r_i^\beta,
\end{align}
where we also use the $C(n)\tau r_i^\beta$-dense of the imagine by (n3).
Denote $\Phi_i^{\beta+1}:=A_i^{\beta+1}\circ \phi_i^{\beta+1}:=(L_i^\beta)^{-1} (\bar \alpha_i^{\beta+1}+\bar O_i^{\beta+1} \phi^{\beta+1}_i)$.  In particular we have 
\begin{align}\label{e:LibetaPhibeta}
\sup_{x\in B_{100r^{\beta}_i}(x^{\beta+1}_i)\cap \cC}\Big|L_i^{\beta}\circ \Phi^{\beta}(x)-L_i^\beta \circ\Phi_i^{\beta+1}(x)\Big|\le C(n)\eta^6 r_i^\beta.
\end{align}
Let us define
\begin{align}\label{e:phibeta+1_define}
\Phi^{\beta+1}(x)=\sum_{x_i^{\beta+1}}\psi_i^{\beta+1}(x) \Phi_i^{\beta+1}(x).
\end{align}
By ($\bar\beta.3)$ and \eqref{e:LibetaPhibeta}, for any $x\in B_{60r_i^\beta}(x_i^{\beta+1})\cap \cC$, we have 
\begin{align}\label{e:Lxbeta_phibeta_phii}
\sup_{y\in B_{100r^{\beta}_i}(x^{\beta+1}_i)\cap \cC}\Big|L_x^{\beta}\circ \Phi^{\beta}(y)-L_x^\beta \circ\Phi_i^{\beta+1}(y)\Big|=\sup_{y\in B_{100r^{\beta}_i}(x^{\beta+1}_i)\cap \cC}\Big|L_x^{\beta}(L_i^{\beta})^{-1}L_i^\beta\circ (\Phi^{\beta}- \Phi_i^{\beta+1})(y)\Big|\le C(n)\eta^6 r_i^\beta.
\end{align}
Let us now check ($\overline{\beta+1}.2$). By the construction of $\Phi^{\beta+1}$ and \eqref{e:Lxbeta_phibeta_phii} and using the support of $\psi_i^{\beta+1}$, we have for any $x\in \cC\cap B_{\frac{19}{10}+10^4t_{\beta}}(p)$ and $y\in B_{40r_x^\beta}(x)\cap \cC$ that
\begin{align}\label{e:Lxbeta_Lxbeta+1}
|L_x^\beta\circ (\Phi^\beta-\Phi^{\beta+1})(y)|=|\sum_{x_i^{\beta+1}}\psi_i^{\beta+1}(y) (L_x^\beta \Phi_i^{\beta+1}(y)-L_x^\beta \Phi^\beta(y)|\le \sup_{\psi_i^{\beta+1}(y)\ne 0}|L_x^\beta (\Phi_i^{\beta+1}- \Phi^\beta)(y)|\le C(n)\eta^6r_x^\beta\le \eta^5r_x^\beta,
\end{align}
which proves $(\overline{\beta+1}.2)$.

Let us now check ($\overline{\beta+1}.1$), which is where the real work lives.  Let us first note by \eqref{e:Lxbeta_Lxbeta+1} and \eqref{e:Lxbeta_phibeta_phii} we have for each $x_i^{\beta+1}\in \cC^{\beta+1}\cap B_{\frac{19}{10}+10^4t_{\beta}}(p)$ and $x_j^{\beta+1}\in B_{60r_i^\beta}(x_i^{\beta+1})\cap \cC$ that
\begin{align}\label{e:Lxbetaphijphii}
\sup_{y\in B_{30r^{\beta}_i}(x_i^{\beta+1})\cap \cC}\Big|L_i^{\beta}\circ \Phi^{\beta+1}_j(y)-L_i^\beta \circ\Phi_i^{\beta+1}(y)\Big|\le C(n)\eta^6 r_i^\beta.\\ \label{e:Lxbetaphijphii1}
\sup_{y\in B_{40r_i^\beta}(x_i^{\beta+1})\cap\cC}|L_i^\beta (\Phi^\beta-\Phi^{\beta+1})(y)|\le C(n)\eta^6 r_i^\beta.
\end{align}
Combining with ($\bar\beta.1$) we get that $L_i^\beta \Phi^{\beta+1}: B_{40r_i^\beta}(x_i^{\beta+1})\to \mathbb{R}^{n-4}$ is an $C(n)\eta^6r_i^\beta$\text{-isometry}. Furthermore, let us remark that $L_x^{\beta}\Phi_i^{\beta+1}$ and $L_x^{\beta}\Phi^{\beta+1}$ is a $C(n)\eta^6r_i^{\beta}$\text{-isometric} equivalence whenever $x\in \cC$ is near $x_i^{\beta+1}$.  Clearly this is not enough to prove ($\overline{\beta+1}.1$) and we will need to find some form of improvement occuring at smaller scales.  We will argue by contradiction, however before doing this let us discuss the rough idea of where this improvement comes from.  Imagine a map $\bar\Phi:\mathbb{R}^{n-4}\to \mathbb{R}^{n-4}$ which is constructed by gluing affine maps $\bar\Phi_i: B_{R}(x_i)\subset \mathbb{R}^{n-4}\to \mathbb{R}^{n-4}$ which are all $\eta R$-close to isometries and each other. Thus $\bar \Phi$ is smooth and a direct computation shows that $|\nabla \bar\Phi-Id|\le C(n)\eta$ and $|\nabla^2\bar \Phi|\le C(n)\eta R^{-1}$, where we are using strongly that we are gluing together affine maps in the construction.  But now take some $x\in \dR^{n-4}$ and look at scales $r<\eta R$, we have by a Taylor series expansion that $\bar \Phi$ would be $\eta^2r$-close to the tangent mapping.  Said otherwise, if $L_x = d\bar\Phi_x^{-1}$ then $L_x\circ \bar\Phi_x$ is converging at a definite linear rate to an isometry, which is telling us our original $\eta R$\text{-isometry} condition is actually improving for small scales, modulo a linear transformation.  Our argument below is just making this a bit more precise. Let us begin the contradiction argument.

The statement ($\overline{\beta+1}.1$) is completely local and all the estimates and construction in the claim are scaling invariant. Let us prove this by contradiction.  Assume $0<\tau<\tau(n)$ there exists a sequence of $(\delta_K,\tau)$-neck regions $\cN_K\subset B_2(p_K)$, $\Phi^{\beta_K}_K:\cC_{K}\to \mathbb{R}^{n-4}$ satisfying $(\bar\beta_K.1)-(\bar\beta_K.3)$ for each $K$, and a sequence of  $\cC^{\beta_K+1}_K\subset \cC_K\cap B_{\frac{19}{10}+10t_{\beta_K}}(p_k)$, $\psi_{i,K}^{\beta_K+1},$ $\phi_{i,K}^{\beta_K+1}$ satisfying (s1)-(s3) with $\delta=\delta_K\to 0$,  but the constructing $\Phi^{\beta_K+1}_K$ defined in \eqref{e:phibeta+1_define} doesn't satisfy $(\overline{\beta_K+1}.1)$ for $x_{K}\in \cC_K\cap B_{\frac{19}{10}+5t_{\beta}}(p_K)$.  In order not to introduce too much notations, for a fixed $\beta$  where one can let $\beta=0$ if one wants, we can rescale the sequence so that  there exists a sequence of $(\delta_K,\tau)$-neck regions $\cN_K\subset B_2(p_K)$, $\Phi^{\beta}_K:\cC_{K}\to \mathbb{R}^{n-4}$ satisfying $(\bar\beta.1)-(\bar\beta.3)$ for each $K$, and a sequence of  $\cC^{\beta+1}_K\subset \cC_K\cap B_{\frac{19}{10}+10t_{\beta}}(p_k)$, $\psi_{i,K}^{\beta+1},$ $\phi_{i,K}^{\beta+1}$ satisfying (s1)-(s3) with $\delta=\delta_K\to 0$,  but the constructing $\Phi^{\beta+1}_K$ defined in \eqref{e:phibeta+1_define} doesn't satisfy $(\overline{\beta+1}.1)$ for $x_{K}\in \cC_K\cap B_{\frac{19}{10}+5t_{\beta}}(p_K)$. Without loss of generality, assume $\Phi_K^{\beta+1}(x_K)=0$. Since $\beta$ is fixed, taking the limit with $K\to \infty$, we get $\cN_K\to \cN_\infty=B_2(0^{n-4},y_c)\cap \mathbb{R}^{n-4}\times C(S^3/\Gamma)\setminus B_{r_x}(\cC_\infty)$ where $\cC_\infty$ is $\tau$-dense in $\mathbb{R}^{n-4}\times\{y_c\}\cap B_2(0^{n-4},y_c)$ by (n3) of neck region, $x_K\to x_\infty\in\cC_\infty\cap B_{\frac{19}{10}+5t_{\beta}}(0^{n-4},y_c)$ and $L_{x_K}^\beta\Phi^\beta_K\to \bar\Phi^\beta_{\infty},~~L_{x_K}^\beta\Phi^{\beta+1}_K\to \bar\Phi^{\beta+1}_{\infty}$, $\phi_{i,K}^{\beta+1}$ converges to isometric $\phi^{\beta+1}_{i,\infty}: B_{\eta^{-4}r_i^{\beta+1}}(x_{i,\infty}^{\beta+1})\cap \cC_\infty\to \dR^{n-4}$, and $L_{x_K}^\beta\Phi_{i,K}^{\beta+1}\to \bar \Phi_{i,\infty}^{\beta+1}$ which is affine and is a $C(n)\eta^6 r_i^\beta$\text{-isometry}.  Let us remark that here $r_x$ on the limit $\cC_\infty$ may be not zero since the rescaling sequence $t_\beta> r_x$ may be have comparable size with $r_x$. However, $r_x$ is equal to zero or not  does not affect our proof below. Furthermore, all the maps convergence are standard by noting that each map is uniformly Lipschitz. 
 
From \eqref{e:Lxbetaphijphii}, \eqref{e:Lxbetaphijphii1}  and \eqref{e:phibeta+1_define}, we have for all $x, x_{i,\infty}^{\beta+1}\in B_{1000r_{x_\infty}^\beta}(x_\infty)\cap \cC_\infty$  and $x_{j,\infty}^{\beta+1}\in B_{60r_i^\beta}(x_{i,\infty}^{\beta+1})\cap \cC_\infty$ that
\begin{align}
\bar\Phi^{\beta+1}_\infty(x)=\sum_{x_{i,\infty}^{\beta+1}}\psi_{i,\infty}^{\beta+1}(x) \bar \Phi_{i,\infty}^{\beta+1}(x).\\ \label{e:Lxbetaphijphiiinfty}
\sup_{y\in B_{30r^{\beta}_i}(x_{i,\infty}^{\beta+1})\cap \cC_\infty}\Big| \bar\Phi^{\beta+1}_{j,\infty}(y)-\bar\Phi_{i,\infty}^{\beta+1}(y)\Big|\le C(n)\eta^6 r_i^\beta.
\end{align}
We will show that there exists a matrix $A\in \mathbb{R}^{(n-4)\times (n-4)}$ such that $A\circ \bar \Phi^{\beta+1}_\infty: B_{200r_{x_\infty}^{\beta+1}}(x_\infty)\cap \cC_\infty \to \mathbb{R}^{n-4}$ is a $\eta^7 r_{x_\infty}^{\beta+1}$\text{-isometry}, which deduces a contradiction by letting $L_{x_K}^{\beta+1}=A\circ L_{x_K}^\beta$ in the contradiction sequence.

From the definition of $\psi_{j,K}^{\beta+1}$ and the convergence, we have $(r_j^{\beta+1})^2\eta^{-4}|\nabla^2\psi_{j,\infty}^{\beta+1}|+(r_j^{\beta+1})\eta^{-2}|\nabla \psi_{j,\infty}^{\beta+1}|\le C(n)$  on $\mathbb{R}^{n-4}\times\{y_c\}\cap B_2(0^{n-4},y_c)$. On the other hand, since $\bar \Phi_{i,\infty}^{\beta+1}$ is affine on $B_{\eta^{-4}r_i^{\beta+1}}(x_{i,\infty}^{\beta+1})\cap \cC_\infty$ and $\cC_\infty$ is $\tau \eta^{-4}r_i^{\beta+1}$-dense in $\mathbb{R}^{n-4}\times\{y_c\}\cap B_{\eta^{-4}r_i^{\beta+1}}(x_{i,\infty}^{\beta+1})$, the map $\bar \Phi_{i,\infty}^{\beta+1}$ naturally extends from $\cC_\infty$ to $\mathbb{R}^{n-4}\times\{y_c\}\cap B_{\eta^{-4}r_i^{\beta+1}}(x_{i,\infty}^{\beta+1})$. Thus $\bar\Phi_\infty^{\beta+1}$ is smooth on $B_{1000r_{x_\infty}^\beta}(x_\infty)\cap \mathbb{R}^{n-4}\times\{y_c\}\cap B_2(0^{n-4},y_c)$.

 Since $\bar \Phi^{\beta+1}_{i,\infty}: B_{100r_i^\beta}(x_{i,\infty}^{\beta+1})\cap \mathbb{R}^{n-4}\to \mathbb{R}^{n-4}$ is a $C(n)\eta^6r_x^\beta$\text{-isometry} and affine for $x_{i,\infty}^{\beta+1}\in B_{1000r_{x_\infty}^\beta}(x_\infty)$, without loss of generality, assume  $\sup_{B_{100r_{i_0}^\beta}(x_{i_0,\infty}^{\beta+1})\cap \mathbb{R}^{n-4} }|\bar\Phi^{\beta+1}_{i_0,\infty}-O|\le C(n)\eta^6r_{i_0}^\beta$ for a fixed $x_{i_0,\infty}^{\beta+1}\in B_{1000r_{x_\infty}^\beta}(x_\infty)$ and $O\in O(n-4)$. Since $\bar\Phi^{\beta+1}_{i_0,\infty}$ is affine, we thus have for all $y\in B_{1000r_{x_\infty}^\beta}(x_\infty)\cap \mathbb{R}^{n-4}$ that
\begin{align}
|\nabla_y \bar\Phi^{\beta+1}_{i_0,\infty}(y)-O|\le C(n)\eta^6,  \text{  and  } |\nabla^2\bar\Phi^{\beta+1}_{i_0,\infty}(y)|=0.
\end{align}
Similarly,  by \eqref{e:Lxbetaphijphiiinfty}, we have for all $x_{j,\infty}^{\beta+1}\in B_{1000r_{x_\infty}^\beta}(x_\infty)\cap \cC_\infty$  and $y\in B_{1000r_{x_\infty}^\beta}(x_\infty)\cap \cC_\infty$ that
\begin{align}\label{e:nablaybarphijinfty}
|\nabla_y \bar\Phi^{\beta+1}_{j,\infty}(y)-O|\le C(n)\eta^6,  \text{  and  } |\nabla^2\bar\Phi^{\beta+1}_{j,\infty}(y)|=0.
\end{align} 
Using the fact that $\bar\Phi^{\beta+1}_{j,\infty}$ is affine, $(r_j^{\beta+1})^2\eta^{-4}|\nabla^2\psi_{j,\infty}^{\beta+1}|+(r_j^{\beta+1})\eta^{-2}|\nabla \psi_{j,\infty}^{\beta+1}|\le C(n)$ and \eqref{e:Lxbetaphijphiiinfty}, \eqref{e:nablaybarphijinfty} we can compute that 
\begin{align}
\sup_{B_{100r_{x_\infty}^\beta}(x_\infty)\cap \mathbb{R}^{n-4}}\left(|\nabla \bar \Phi^{\beta+1}_\infty-O|+r_{x_\infty}^{\beta+1}\eta^{-2}|\nabla^2\bar\Phi^{\beta+1}_\infty|\right)\le C(n)\eta^6.
\end{align}

Now consider the affine transformation 
\begin{align}
A\circ \bar\Phi^{\beta+1}_\infty(y):=d\bar \Phi^{\beta+1}_\infty(x_\infty)^{-1}\circ\bar\Phi^{\beta+1}_\infty(y)
\end{align}
Note that this is chosen precisely so that $A \circ\bar\Phi^{\beta+1}_\infty(x_\infty)=0$ and $d(A\circ \bar\Phi^{\beta+1}_\infty)(x_\infty)= Id$.  In particular, it follows from a Taylor series, using the first and second derivative estimates on $\bar\Phi^{\beta+1}_\infty$ that for all $0<r\le r_{x_\infty}^{\beta}=\eta^{-2} r_{x_\infty}^{\beta+1}$
\begin{align}
\sup_{B_r(x_\infty)\cap \cC_\infty}|A\circ \bar\Phi^{\beta+1}_\infty-Id|\le C(n)\eta^8 (r_{x_\infty}^{\beta+1})^{-1}r^2\, . 	
\end{align}
In particular, let $r=200r_{x_\infty}^{\beta+1}$ we have 
\begin{align}
\sup_{B_{200r_{x_\infty}^{\beta+1}}(x_\infty)\cap \cC_\infty}|A\circ \bar\Phi^{\beta+1}_\infty-Id|\le C(n)\eta^8r_{x_\infty}^{\beta+1}\, , 
\end{align}
which implies $A\circ \bar\Phi^{\beta+1}_\infty: B_{200r_{x_\infty}^{\beta+1}}(x_\infty)\cap \cC_\infty\to \mathbb{R}^{n-4}$ is a $\eta^7r_{x_\infty}^{\beta+1}$\text{-isometry}. Hence we deduce a contradiction by letting $L_{x_K}^{\beta+1}=A\circ L_{x_K}^\beta$ in the contradiction sequence. Thus we have proved ($\overline{\beta+1}.1$) and the Claim. $\square$

Let us now check that the claim implies $(\beta.1)-(\beta.3)$. Let us define $T_{x}^\beta$ inductively for $x\in \cC\cap B_{\frac{19}{10}}(p)$.
For any $x\in \cC\cap B_{\frac{19}{10}}(p)$ define $T_{x}^0=Id$.  Assume $T_x^\beta$ is defined satisfying $(\beta.1), (\beta.1)$ and $(\beta.3) $. Let us now define $T_x^{\beta+1}$. 
Actually, we will define $T_x^{\beta+1}=O_x^{\beta+1}L_x^{\beta+1}$ for some $O_x^{\beta+1}\in O(n-4)$. In particular, $({\beta+1}.1)$ and $({\beta+1}.2)$ follow directly from $(\overline{\beta+1}.1)$ and $(\overline{\beta+2}.2)$. Let us now fix $O_x^{\beta+1}$ so that $(\beta+1.3)$ holds. We will only check $|T_x^{\beta+1}\circ (T_x^{\beta})^{-1}-Id|\le \eta^2$ since the argument for $|T_x^{\beta}\circ (T_x^{\beta+1})^{-1}-Id|\le \eta^2$ is the same.  By $( \beta.2)$, $( \beta.1)$ we have $T_x^\beta \Phi^{\beta+1}: B_{40r_x^\beta}(x)\cap \cC\to \mathbb{R}^{n-4}$ is a $C(n)\eta^5r_x^\beta$\text{-isometry} which in particular is a $C(n)\eta^3r_x^{\beta+1}$\text{-isometry}. By $(\overline{\beta+1}.1)$ we have that $L_x^{\beta+1}\Phi^{\beta+1}:B_{50r_x^{\beta+1}}(x)\cap \cC\to \mathbb{R}^{n-4}$ is a $\eta^6r_x^{\beta+1}$\text{-isometry}. Therefore, as the above proof of (4), there exists $O_x^{\beta+1}\in O(n-4)$ such that $T_x^{\beta+1}:=O_x^{\beta+1}L_x^{\beta+1}$ satisfies 
\begin{align}\notag
\sup_{y\in B_{50r_x^{\beta+1}}(x)\cap \cC}|(T_x^{\beta+1}\circ (T_x^\beta)^{-1}-Id)T_x^\beta (\Phi^{\beta+1}(y)-\Phi^{\beta+1}(x))|&=\sup_{y\in B_{50r_x^{\beta+1}}(x)\cap \cC}|(T_x^\beta -T_x^{\beta+1})(\Phi^{\beta+1}(y)-\Phi^{\beta+1}(x))|\\ 
&\le C(n)\eta^3r_x^{\beta+1}.
\end{align}
Now the same argument as the proof of (4) shows $(\beta+1.3)$ if $\tau\le \tau(n)$. Hence we finish the proof of the Theorem.
\end{proof}

\vspace{.5cm}

\section{Proving the $L^2$-curvature bound on Neck Regions}\label{s:L2_bound}

Recall that the proof of Theorem \ref{t:neck_region} will be done by a rather involved induction scheme, so that in the end each of the conclusions of Theorem \ref{t:neck_region} will be proved simultaneously.  This section is therefore dedicated to proving the $L^2$ curvature bound on neck regions under the assumption that we have already proved some Ahlfors regularity bound.  Precisely, our main result in this section is the following:\\

\begin{theorem}\label{t:L2_neck}
	Let $(M^n,g,p)$ satisfy $\Vol(B_1(p_j))>\rv>0$ with $\delta',B>0$ fixed.  Then if $\tau<\tau(n)$ and $\delta<\delta(n,\rv,\tau,\delta',B)$ are such that $\cN = B_2(p)\setminus \bigcup_{x\in \cC} \overline B_{r_x}(x)$ is a $(\delta,\tau)$-neck region for which
	\begin{enumerate}
	\item $|{\Ric}|<\delta$ ,
	\item For each $x\in \cC$ and $r_x<r$ with $B_{2r}(x)\subseteq B_2$ we have $B^{-1}r^{n-4} < \mu\big(B_r(x)\big)<Br^{n-4}$ ,
	\end{enumerate}
     then for each $B_{2r}(x)\subseteq B_2$ we have the curvature estimate $r^{4-n}\int_{\cN_{10^{-5}}\cap B_r(x)} |\Rm|^2(z)\,dz \leq \delta'$. In particular, we have that 	$\int_{\cN_{10^{-5}}\cap B_1(p)} |\Rm|^2(z)\,dz \leq \delta'$.
\end{theorem}

The above is the key estimate required for the proof of the $L^2$ curvature bound of Theorem \ref{t:neck_region}.3. The key result toward the proof of Theorem \ref{t:L2_neck} is the following local $L^2$ curvature estimate:

\begin{proposition}\label{p:local_L2_neck}
Let $(M^n,g,p)$ satisfy $\Vol(B_1(p))>\rv>0$ with $B>0$ fixed.  Then if $\tau<\tau(n)$ and $\delta<\delta(n,\rv,\tau,B)$ are such that $\cN = B_2(p)\setminus \bigcup_{x\in \cC} \overline B_{r_x}(x)$ is a $(\delta,\tau)$-neck region for which
	\begin{enumerate}
	\item $|{\Ric}|<\delta$ ,
	\item For each $x\in \cC$ and $r_x<r$ with $B_{2r}(x)\subseteq B_2(p)$ we have $B^{-1}r^{n-4} < \mu\big(B_r(x)\big)<Br^{n-4}$ ,
	\end{enumerate}
	then for some $K(n,B)>1$ and any $y\in \cN_{K(n,B)}$ with $d(y,\cC)=d(y,z)=r$ and $B_{r/10}(z)\subset B_2(p)$ for $z\in \cC$, we have
	\begin{align}
		\int_{B_{r/2}(y) } |\Rm|^2(z)\,dz \leq C(n,\rv,B,\tau)\delta r^{n-2}+ C(n,\rv,B,\tau)\int_{B_{r/20}(z)}\big|{\cH}_{10r^2}(x)-{\cH}_{r^2/10}(x)\big|d\mu(x),
	\end{align}
	where the precise definition of $\cH$-volume is \eqref{e:Hvolume_def}.
\end{proposition}
\vspace{.5cm}

 We will give a proof of Proposition \ref{p:local_L2_neck} and Theorem \ref{t:L2_neck} in the end of this section. The proof of Proposition \ref{p:local_L2_neck} relies on a few new and sharp estimates involving  about the ${\cH}$-volume and a pointwise curvature estimate in Lemma \ref{l:curvature_level_sets}.  

\subsection{Pointwise Riemann curvature estimates}
In this subsection, we get a pointwise bound on the curvature based on bounds on its Ricci curvature and control over some splitting functions. This estimate will provide us a manner in which to prove the local $L^2$ curvature estimate if we can find enough linear independently functions with good estimates.  Our main result of this subsection is the following:\\

\begin{lemma}[Curvature estimate of Level sets]\label{l:curvature_level_sets}
Let $(M^n,g,p)$ be a Riemannian manifold with  a map $\Phi=(h,f_1,\cdots, f_{n-4}): B_{10r}(p)\to \mathbb{R}_{+}\times \mathbb{R}^{n-4}$. Assume $f_0^2=h-\sum_{i=1}^{n-4}f_i^2\ge c_0r^2>0$ and $|f_i|\le c_0^{-1}r$ on $B_{2r}(y)\subset B_{10r}(p)$. Let $A=(a_{ij})$ be a $(n-3)\times (n-3)$ symmetric matrix with $a_{ij}(x)=\langle \nabla f_i(x),\nabla f_j(x)\rangle$.  Assume further $|\det A|(x)\ge c_0>0$ and $|\nabla f_i|\le c_1$ on $B_{2r}(y)$.  Then for any $x\in B_r(y)$, we have the following scale invariant estimate
$$r^4|\Rm|^2(x)\le C(n,c_0,c_1)\left(r^4|{\Ric}|^2+r^{2}|\nabla^3 h|^2+\sum_{i=1}^{n-4}r^4|\nabla^3 f_i|^2+\mathcal{F}+\mathcal{F}^2\right),$$
where $\mathcal{F}=|\nabla^2h-2g|^2+\sum_{i=1}^{n-4}r^2|\nabla^2f_i|^2$.
\end{lemma}
\begin{remark}
The picture here is that $B_{10r}(p)\approx\dR^{n-4}\times C(Z)$ with $B_{2r}(y)$ away from the cone point, where $f_1,...,f_{n-4}$ represent linear splitting functions and $f_0$ is the distance function coming from the cone factor.  Therefore $h$ represents the square distance coming from the cone point itself.
\end{remark}

\begin{proof}
Let $\Psi=(f_0,f_1,\cdots, f_{n-4}): B_{10r}(p)\to  \mathbb{R}^{n-3}$. For any $x\in B_{2r}(y)$, since $|\det A|\ge c_0>0$, by constant rank theorem, we have $N=\Psi^{-1}\Big(\Psi(x)\Big)\cap B_{2r}(y)$ is a smooth $3$ dimensional submanifold of $M$.
  Let $\{e_1,e_2,e_3\}$ be horizontal vector fields  on $M$ which form a (local) orthonormal basis of the level set $N$.  Then $\{\nabla f_0,\nabla f_1,\cdots, \nabla f_{n-4},e_1,e_2,e_3\}$ form a basis of $M$ at $x$. Note that to control the curvature $|\Rm|(x)$ we only need to control the curvature tensor with the $\{e_j\}$ factors in the basis.  Indeed, since 
  \begin{align}\label{e:Rmnabla3f}
 \Rm(X,Y,\nabla f,Z)=\nabla^3f(X,Y,Z)-\nabla^3f(Y,X,Z)\, ,
  \end{align}
  then if the curvature $\Rm$ involves one normal direction $\nabla f_i$ for some $1\le i\le n-4$ we have $|\Rm(\cdot,\cdot, \nabla f_i,\cdot)|\le 3|\nabla^3f_i|$. Similarly for $i=0$, we have that $|\Rm(\cdot,\cdot, \nabla f_0,\cdot)|=(2f_0)^{-1}|\Rm(\cdot,\cdot, \nabla f_0^2,\cdot)|\le 3f_0^{-1}|\nabla^3f_0^2|$.  Hence, to prove the lemma, it suffices to estimate $\Rm(e_i,e_j,e_k,e_l)$ with $i,j,k,l=1,2,3$.  We will use Gauss-Codazzi equation to estimate these terms.
  Let us consider the second fundamental form $\Pi$. Denote by $g^N$ the restricted metric on $N$. By definition, we know at $x$ that
\begin{align}
\Pi=v_0\nabla^2f_0+\sum_{i=1}^{n-4}v_i\nabla^2 f_i\, ,
\end{align}
where $v_i=-\sum _{j=0}^{n-4}a^{ij}\nabla f_j$ and $(a^{ij})=A^{-1}$. Since $\nabla^2f_0=(2f_0)^{-1}\big(\nabla^2h-2\sum_{i=0}^{n-4} \nabla f_i\otimes\nabla f_i-2\sum_{i=1}^{n-4} f_i\nabla^2f_i\big)$ and $\Pi$ only takes values in the horizontal vector fields, the term $2\sum_{i=0}^{n-4} \nabla f_i\otimes\nabla f_i$ gives no contribution to $\Pi$.  Therefore 
\begin{align}
\Pi=v_0f_0^{-1}g^N+\mathcal{E}
\end{align}
with $|\mathcal{E}|\le C(n,c_0,c_1)\Big(r^{-1}|\nabla^2h-2g|+ \sum_{i=1}^{n-4}|\nabla^2f_i|\Big)$.
By Gauss-Codazzi equation $\Rm(X,Y,Z,W)=\Rm^N(X,Y,Z,W)+\langle\Pi(Y,W),\Pi(X,Z)\rangle-\langle\Pi(X,W),\Pi(Y,Z)\rangle$, we have
\begin{align}
{\Ric}^N={\Ric}+2|v_0|^2f_0^{-2}g^N+\mathcal{E}_1\, ,
\end{align}
with $|\mathcal{E}_1|\le C(n,c_0,c_1)\left(r^{-1}|\nabla^3h|+\sum_{i=1}^{n-4}|\nabla^3f_i|+r^{-1}|\mathcal{E}|+|\mathcal{E}|^2\right)$, where we have used \eqref{e:Rmnabla3f}. 
Thus we have the scalar curvature estimate $R^N =6 |v_0|^2f_0^{-2}+\mathcal{E}_2$ with $|\mathcal{E}_2|\le C(n,c_0,c_1)\left(|{\Ric}|+|\mathcal{E}_1|\right)$.
On the other hand, since $\dim N=3$, we have
\begin{align}
\Rm^N=-\frac{R^N}{12}g^N\circ g^N-\left({\Ric}^N-\frac{R^N}{3}g^N\right)\circ g^N,
\end{align}
where $g^N\circ g^N$ is the Kulkarni-Nomizu product, see \cite{Petersen_RiemannianGeometry}. Then
\begin{align}
\Rm^N=-\frac{|v_0|^2}{2f_0^2}g^N\circ g^N+\mathcal{E}_3\, ,
\end{align}
with $|\mathcal{E}_3|\le C(n,c_0,c_1)|\mathcal{E}_2|$. Using the Gauss-Codazzi equation again, we finally arrive at
\begin{align}
|\Rm(e_i,e_j,e_k,e_l)|\le C(n,c_0,c_1)\left(|\mathcal{E}_3|+|\mathcal{E}|+|\mathcal{E}|^2\right).
\end{align}

Therefore, combining with the estimates on normal direction \eqref{e:Rmnabla3f}, we have
\begin{align}
|\Rm|(x)&\le C(n,c_0,c_1)\left(r^{-1}|\nabla^3h|+\sum_{i=1}^{n-4}|\nabla^3f_i|+|\mathcal{E}_3|+r^{-1}|\mathcal{E}|+|\mathcal{E}|^2\right)(x)\\
&\le C(n,c_0,c_1)\left(|{\Ric}|+r^{-1}|\nabla^3h|+\sum_{i=1}^{n-4}|\nabla^3f_i|+r^{-1}|\mathcal{E}|+|\mathcal{E}|^2\right)(x).
\end{align}
\end{proof}
By noting the curvature estimate above, to prove Proposition \ref{p:local_L2_neck}, we only need to find $n-4$ functions which satisfy the condition in Lemma \ref{l:curvature_level_sets}. The main purpose of the following subsections is to find and control such functions.\\

\subsection{$\cH$-volume on manifold with $\Ric\ge -\delta(n-1)$}
Let $(M,g,p)$ be pointed manifold with ${\Ric}\ge -(n-1)\delta$ and $\Vol(B_1(p))\ge \rv$. We define the $\cH$-volume as
\begin{align}\label{e:Hvolume_def}
\cH^\delta_t(p)=\int_M(4\pi t)^{-n/2}e^{-\frac{d^2(x,p)}{4t}}dx-\int_0^t\frac{1}{4s}\int_M( L_\delta(x)-2n)(4\pi s)^{-n/2}e^{-\frac{d^2(x,p)}{4s}}dx ds,
\end{align}
where $L_\delta(x)=2+2(n-1)d(p,x)\sqrt{\delta}\coth \big(\sqrt{\delta}d(p,x)\big)$.  Note that $L_0\equiv 2n$ for spaces with nonnegative Ricci curvature, so that this second term is purely a correction term.  On a first reading of this section we recommend the reader sets $\delta=0$, most of the formulas simplify quite a bit in this case.  By direct computation, we have
\begin{align}
\partial_t\mathcal{H}^\delta_t(p)
&=\int_M\big( \frac{d^2(p,x)}{4t^2}-\frac{L_\delta(x)}{4t}\big)(4\pi t)^{-n/2}e^{-\frac{d^2(x,p)}{4t}}dx\, .
\end{align}
Noting that
$
\Delta e^{-\frac{d^2(x,p)}{4t}}=\left(-\frac{1}{4t}\Delta d^2(p,x)+\frac{d^2(p,x)}{4t^2}\right)e^{-\frac{d^2(x,p)}{4t}}\, ,
$
we have
\begin{align}\label{e:derivative_cHt}
\partial_t\mathcal{H}^\delta_t(p)
&=\frac{1}{4t}\int_M\big( \Delta d^2(p,x)-{L_\delta(x)}\big)(4\pi t)^{-n/2}e^{-\frac{d^2(x,p)}{4t}}dx\le 0\, ,
\end{align}
where we use the Laplacian comparison in the last inequality.  
\begin{remark}\label{r:scaling_cHvolume}
By rescaling, one can check that $\cH^{\delta}_t(p)=\tilde{\cH}^{\delta t}_1(p)$ where $\tilde{\cH}$ is the $\cH$-volume of $(M,\tilde{g},p)=(M,t^{-1}g,p)$. 
\end{remark}

 For simplicity of notation, we will drop the $\delta$ of $\cH_t^\delta$ when there is no confusion, but one should keep in mind the dependence of $\cH_t$ on the lower Ricci curvature bound $-(n-1)\delta$.

Let us consider the following heat flow
\begin{align}\label{e:heat_flow}
&\partial_t f_t=\Delta f_t-2n\\
&f_0(x)=2nU(d(p,x))\, ,
\end{align}
where $U(r)=\int_0^r\sinh^{-(n-1)}(\sqrt{\delta} t)\int_0^t\sinh^{n-1}(\sqrt{\delta} s)ds dt.$

\begin{remark}\label{r:deltaf0_delta_d2}
By direct computations, we have 
\begin{align}
0\le 2n-\Delta f_0(x)=\frac{nU'(d(p,x))}{d(p,x)}\Big(L_\delta(x)-\Delta d^2(p,x) \Big)\le n\Big(L_\delta(x)-\Delta d^2(p,x)\Big).
\end{align}
\end{remark}

We begin by recording some basic points about $f_t$ which will be useful:

\begin{lemma}\label{l:heat_flow_estimate}
	 If $(M,g)$ satisfies ${\Ric}\ge -(n-1)\delta$ and $f_t$ is as in \eqref{e:heat_flow} then we have the following:
\begin{enumerate}
\item $\Delta f_t\le 2n$ for all $t\geq 0$.
\item $-2nt\leq f_t\leq f_0 = 2nU(d(p,x))$.
\item If $Q_t\equiv e^{-2(n-1)\delta t}(|\nabla f|^2-4f-8nt)-8(n-1)\delta (tf+2nt^2)$, then $Q_t\leq 0$ for all $t\geq 0$.
\item If $P_t(x)=e^{-2(n-1)\delta t}(|\nabla f|^2-4f+8t(\Delta f-2n))-8(n-1)\delta (tf+2nt^2)+8t^2\delta(n-1)(\Delta f-2n)$ then we have $P_t\leq 0$ for all $t\geq 0$.
\end{enumerate}
\end{lemma}
\begin{remark}
The last two expressions will be used to give sharp estimates on the gradient and hessian of our smooth approximation.
\end{remark}
\begin{proof}

Since $(\partial_t-\Delta)(f+2nt)=0$ we have 
\begin{align}
f_t(x)+2nt=\int_M f_0(y)\rho_t(x,dy).
\end{align}
This implies that 
\begin{align}
\Delta f_t=\partial_t f_t+2n=\int_Mf_0(y)\partial_t\rho_t(x,dy)=\int_M \Delta f_0(y)\rho_t(x,dy).
\end{align}
Therefore we arrive at 
\begin{align}\label{e:deltaft-2n}
(2n-\Delta f_t)(x)=\int_M (2n-\Delta f_0)(y)\rho_t(x,dy)\, .
\end{align}
By Laplacian comparison, we have $\Delta U(d(p,x))\le 1$. Thus $(2n-\Delta f_0)\ge 0$. Hence \eqref{e:deltaft-2n} implies (1).  From this we immediately get $\partial_tf\le 0$ which obtains for us the upper bound in $(2)$.  For the lower bound, we use $(\partial_t-\Delta)(f+2nt)=0$ to get $f_t(x)\ge -2nt$. \\

For the gradient estimate, we have by direct computation that
\begin{align}
(\partial_t -\Delta)(|\nabla f|^2-4f)&=-2|\nabla^2f|^2-2\Ric(\nabla f,\nabla f)+8n\le 2(n-1)\delta |\nabla f|^2+8n.
\end{align}
By noting our lower bound $f_t+2nt\geq 0$ we can conclude from this
\begin{align}
	(\partial_t-\Delta)Q_t(x)\le 0\, .
\end{align}
Combined with $Q_0\le 0$ we obtain $(3)$.  Finally, to prove $(4)$ we compute that
\begin{align}\label{e:evolution_Pt}
(\partial_t-\Delta)P_t\le -2e^{-2(n-1)\delta t}|\nabla^2 f-2g|^2\, .
\end{align}
Combined with $P_0\leq 0$ this proves the desired result.
\end{proof}
\vspace{.25cm}

Let us observe that a consequence of $(3)$ above is the gradient estimate
\begin{align}\label{e:gradient_f_t}
|\nabla f|^2&\le \big(4+8(n-1)\delta\,t\, e^{2(n-1)\delta t}\big)\big(f+2nt\big)\notag\\
&\le \big(4+8(n-1)\delta\,t\, e^{2(n-1)\delta t}\big)\big(f_0+2nt\big)\, .
\end{align}

Let us now prove a couple more refined estimates on $f_t$ that depend on the pinching of our $\cH$-volume.  Precisely, we have the following:

\begin{lemma}\label{l:cHfunctional}
Let $(M,g,p)$ satisfy $\Vol(B_1(p))\ge \rv>0$ and $\Ric\ge -(n-1)\delta$. Denote by $\eta=|\cH_{r^2/2}(p)-\cH_{2r^2}(p)|$ the $\cH$-volume pinching at scale $r$, then we have the estimates:
\begin{enumerate}
\item $\fint_{0}^{r^2} \fint_{B_{4r}(p)}|\nabla^2 f_t-2g|^2dydt\le C(n,\rv)(\delta r^2+\eta).$
\item For any $t\le r^2$, we have $\sup_{y\in B_{4r}(p)}|f_t-f_0|(y)\le C(n,\rv)\epsilon(\eta)\, r^2,$ where $\epsilon(\eta)\to 0$ if $\eta\to 0$.
\item For any $t\le r^2$, we have $\fint_{B_{4r}(p)}|\Delta f_t-2n|(z)\,dz+r^{-2}\fint_{B_{4r}(p)}|\nabla f_0-\nabla f_t|^2(z)\,dz\le C(n,\rv)\eta.$
\end{enumerate}
Moreover, if we assume further the Ricci curvature upper bound $|\Ric|\le (n-1)\delta$ and harmonic radius $r_h(x)\ge r$ for some $x\in B_{3r}(p)$, then for any $r^2/2\le t\le r^2$, we have
\begin{align}\label{e:hessian_parabolic_pointwise}
\sup_{B_{r/2}(x)}|\nabla^2 f_t-2g|^2+r^2\fint_{B_{r/2}(x)}|\nabla^3 f_t|^2(z)\,dz\le C(n,\rv)(\delta r^2+\eta).
\end{align}
\end{lemma}
\begin{proof}
By \eqref{e:deltaft-2n} we have
\begin{align}
(2n-\Delta f_t)(x)=\int_M (2n-\Delta f_0)(y)\rho_t(x,dy)\, .
\end{align}

Thus
\begin{align}
\int_M (2n-\Delta f_t)(x)\rho_{r^2/2}(p,dx)=\int_M\int_M(2n-\Delta f_0)(y)\rho_t(x,dy)\rho_{r^2/2}(p,dx)=\int_M (2n-\Delta f_0)(y)\rho_{t+r^2/2}(p,dy)\, .
\end{align}
Therefore,  by the heat kernel estimates of Theorem \ref{t:heat_kernel} and Remark \ref{r:deltaf0_delta_d2} and noting that $\Delta f_t\le 2n$, for any $t\le r^2$, one can get
\begin{align}
\fint_{B_{10r}(p)}|2n-\Delta f_t|(x)dx&\le C(n,\rv)\int_M (2n-\Delta f_t)(x)\rho_{r^2/2}(p,dx)\\
&\le C(n,\rv)\int_M (2n-\Delta f_0)(y)\rho_{t+r^2/2}(p,dy)\\
&\le C(n,\rv)\int_M \big( {L_\delta(y)}-\Delta d^2(p,y)\big)\rho_{t+r^2/2}(p,dy)\\
&\le C(n,\rv)\int_M\big( {L_\delta(y)}-\Delta d^2(p,y)\big) \Big(t+r^2/2\Big)^{-n/2}e^{-\frac{d(p,y)^2}{(4+2^{-1})(t+r^2/2)}}dy\\
&\le  C(n,\rv)\int_M\big( {L_\delta(y)}-\Delta d^2(p,y)\big) r^{-n}e^{-\frac{4d(p,y)^2}{27r^2}}dy\\ \label{e:meanvalue_cH7}
&\le C(n,\rv) |\cH_{7r^2/4}(p)-\cH_{2r^2}(p)|\\
&\le C(n,\rv)\eta\, ,
\end{align}
where the constants $C(n,\rv)$ change line by line, and we have used \eqref{e:derivative_cHt} and mean value equality to deduce \eqref{e:meanvalue_cH7}. 

Hence, we have
\begin{align}\label{e:L1f0_ft}
\fint_{B_{10r}(p)}|f_0-f_t|(x)dx\le \int_0^t\fint_{B_{10r}(p)}(2n-\Delta f_s)(x)dxds\le C(n,\rv)\eta\, t\, .
\end{align}
By the gradient estimate of $f_t$ in (\ref{e:gradient_f_t}), we therefore have $\sup_{B_{10r}(p)}|f_0-f_t|\le C(n,\rv)\epsilon(\eta)\, r^2$. The gradient $L^2$ in (3) follows from integrating by part and the $L^1$ estimate of $(2n-\Delta f)$. In fact, let $\phi$ be a cutoff function as in \cite{ChC1} with support in $B_{10r}(p)$ and $\phi\equiv 1$ on $B_{8r}(p))$ and $r|\nabla \phi|+r^2|\Delta \phi|\le C(n)$. By integrating by parts, we have
\begin{align}
\int_M\phi^2 |\nabla f_0-\nabla f_t|^2(x)\,dx &\le 2\int_M|\nabla\phi| \cdot |f_0-f_t|\cdot \phi|\nabla f_0-\nabla f_t|(x)\,dx+\int_M \phi^2 |\Delta f_0-\Delta f_t|\cdot |f_0-f_t|(x)\,dx\\ \nonumber
&\le \frac{1}{2}\int_M\phi^2 |\nabla f_0-\nabla f_t|^2(x)\,dx+C(n)r^{-2}\int_{B_{10r}(p)}|f_0-f_t|^2+ \sup_{B_{10r}(p)}|f_0-f_t|\int_{B_{10r}(p)}|\Delta f_0-\Delta f_t|(x)\,dx.
\end{align}
Combining with \eqref{e:L1f0_ft} and $\sup_{B_{10r}(p)}|f_0-f_t|\le C(n,\rv)\epsilon(\eta)\, r^2$ and the $L^1$ estimates of $|\Delta f_t-2n|$, we have $\int_M\phi^2 |\nabla f_0-\nabla f_t|^2(z)\,dz\le C(n,\rv)\eta r^{2+n}$. Hence, we prove (2) and (3).

To prove (1), recall that $P_t(x)=e^{-2(n-1)\delta t}(|\nabla f|^2-4f+8t(\Delta f-2n))-8(n-1)\delta (tf+2nt^2)+8t^2\delta(n-1)(\Delta f-2n)$ satisfies $P_t\leq 0$ from Lemma \ref{l:heat_flow_estimate}.  Let $\varphi$ be a cutoff function as in \cite{ChC1} with support in $B_{6r}(p)$ and $\varphi\equiv 1$ on $B_{5r}(p))$ and $r|\nabla \varphi|+r^2|\Delta \varphi|\le C(n)$. Multiplying such $\varphi$ to (\ref{e:evolution_Pt}) and integrating by parts, we have
\begin{align}
\int_0^{r^2}\fint_{B_{5r}(p)} |\nabla^2 f-2g|^2(z)\,dzdt\le C(n,\rv)\fint_{B_{6r}(p)} |P_{r^2}|(z)\,dz+C(n,\rv)\fint_0^{r^2}\fint_{B_{6r}(p)}|P_t|(z)\,dzdt.
\end{align}
To estimate $\fint_{B_{6r}(p)}|P_t|(z)\,dz $, one only needs to estimate $\fint_{B_{6r}(p)} \Big||\nabla f|^2-4f\Big|(z)\,dz$. This can be controlled by considering the evolution of $f_t^2-f_0^2$. In fact, we have
\begin{align}
(\partial_t-\Delta)(f_0^2-f^2)
&=(4n+8)(f-f_0)-2f_0(\Delta f_0-2n)-2(|\nabla f_0|^2-4f_0)+2(|\nabla f|^2-4f).
\end{align}
Let $\psi$ be a cutoff function as in \cite{ChC1} with support in $B_{10r}(p)$ and $\psi\equiv 1$ on $B_{6r}(p))$ and $r|\nabla \psi|+r^2|\Delta \psi|\le C(n)$. By noting $0\le (4f_0-|\nabla f_0|^2)(y)\le C(n)\delta d(p,y)^4$ for $d(p,y)\le 10$, we can show
\begin{align}
\int_0^{3r^2/2}\int_{M}e^{-2(n-1)\delta t}(|\nabla f|^2-4f)\psi (z)\,dzdt \ge -C(n,\rv)(\delta r^2+\eta )r^{4+n} .
\end{align}
Using the above, the $L^1$ estimate on the laplacian of $f_t$ and $|f_t|\le C(n)r^2$ on $B_{10r}(p)$ for $t\le 3r^2/2$, we have
\begin{align}\nonumber
\int_0^{3r^2/2}\int_{M}P_t\psi(z)\,dzdt\ge &-C(n,\rv)(\delta r^2+\eta )r^{4+n}\, .
\end{align}
Noting that $P_t\le 0$, we have
\begin{align}
\fint_0^{3r^2/2}\fint_{B_{6r}(p)}|P_t|(z)\,dzdt\le C(n,\rv)(\delta r^2+\eta )r^2\, ,
\end{align}
which finishes the proof of (1).

Now we wish to prove the estimates of \eqref{e:hessian_parabolic_pointwise}.  Indeed, under the assumption $r_h(x)\geq r$ we have that
\begin{align}\label{e:moser_iteration:1}
	r^{2q}\fint_{B_{3r/4}(x)}|\Rm|^q(z)\,dz<C(n,q)\, ,
\end{align}
for all $q<\infty$.  Using this, the estimates of \eqref{e:hessian_parabolic_pointwise} are fairly standard, and follow much the same path as the proof of Theorem \ref{t:prelim:harmonic_estimates}, so we will only sketch the argument. Denote $H_f=\nabla^2f-2g$,  then we can compute
\begin{align}\label{e:evolution_Hf}
(\partial_t-\Delta)|H_f|^2=-2|\nabla^3f|^2+Rm\ast H_f\ast H_f+\Ric\ast H_f+\nabla(\Ric(\nabla f,\cdot))\ast H_f\, ,
\end{align}
where $\ast$ means tensorial linear combinations and the exact expression can be computed as in Lemma \ref{l:harmonic_computations}.  Then we can apply a standard parabolic moser iteration using \eqref{e:moser_iteration:1} and (1) in order to conclude the pointwise estimate in (\ref{e:hessian_parabolic_pointwise}); see also \cite{ColdingMinicozzi_tangentcone} for an elliptic version estimate.  In order to conclude the $L^2$ estimate on $\nabla^3 f$, let us simply multiply \eqref{e:evolution_Hf} by a cutoff function and integrate using \eqref{e:moser_iteration:1} and the pointwise estimate on $|H_f|$, which finishes the sketch.
\end{proof}

\vspace{.5cm}

\subsection{$\cH$-volume and Local $L^2$ curvature estimates}
The main purpose of this subsection is to prove the local $L^2$ curvature estimate in Proposition \ref{p:local_L2_neck}. The key ingredient is the parabolic estimate of $\cH$-volume in Lemma \ref{l:cHfunctional} and pointwise curvature estimate in Lemma \ref{l:curvature_level_sets}.  First, let us introduce the concept of independent points.

\begin{definition}[$(k,\rho,r)$-independent points]\label{d:indep_point}
In a metric space $(X,d)$ a set of points $U=\{x_0,\cdots,x_k\}\subset B_{r}(x)$ is $(k,\rho,r)$-independent if for any subset $U'=\{x_0',\cdots,x_k'\}\subset \mathbb{R}^{k-1}$ we have 
\begin{align}
d_{GH}(U,U')\ge \rho \cdot r.
\end{align}
\end{definition}

\begin{remark}\label{r:remark_independentpoints}
Let $X\subset \mathbb{R}^n$, if there exists no $(k,\rho,r)$-independent set in $B_r(x)\cap X$, then $B_r(x)\cap X\subset B_{4\rho r}(\mathbb{R}^{k-1})$ for some $(k-1)$-plane $\mathbb{R}^{k-1}\subset \mathbb{R}^n$. To see this, if $B_r(x)\cap X$ is not a subset of $B_{3\rho r}(\mathbb{R}^{k-1})$ for any $(k-1)$-plane, then one can find $(k,\rho,r)$-independent set in $B_r(x)\cap X$ by induction on $k$.
\end{remark}


Now we are ready to prove Proposition \ref{p:local_L2_neck}.
\begin{proof}[Proof of Proposition \ref{p:local_L2_neck}]
The main idea for the proof is to use the pointwise curvature estimate in Lemma \ref{l:curvature_level_sets}. The key ingredient is to find $n-3$ functions $(h,u_1,\ldots,u_{n-4})$ which satisfy the conditions of this Lemma.  Intuitively, such $h$ should be a square distance function and $(u_1,\ldots,u_{n-4})$ splitting functions which form a cone map to $\mathbb{R}^{n-4}\times C(S^3/\Gamma)$. Based on this observation,  we will construct such functions in detail in the following paragraphs. First we will show the following claim:

\noindent \textbf{Claim 1:} For $\tau\le \tau(n)$ and $\delta\le \delta(n,B,\tau)$ there exist $\rho(n,B), L(n,B), K(n,B)>0$ such that for any $y\in \cN_{K}$ with $d(y,\cC)=d(y,z)=r$ satisfying $B_{r/10}(z)\subset B_2(p)$,  we have $(n-4,\rho,r/20)$-independent points $\{x_0,\cdots, x_{n-4}\}\subset \tilde{\cC}\subset B_{r/20}(z)$, where  $\tilde{\cC}=\{x\in \cC\cap B_{r/20}(z): \big|{\cH}_{10r^2}(x)-{\cH}_{r^2/10}(x)\big|\le L(n,B) \eta~\}$ and $\eta\equiv \fint_{B_{r/20}(z)}\big|{\cH}_{10r^2}(x)-{\cH}_{r^2/10}(x)\big|d\mu(x)$.


\textbf{Proof of Claim 1:} We will chose $K(n,B)=10^{10}\rho(n,B)^{-1}>10^{10}$ where $\rho(n,B)$ will be fixed later. 
By the maximal function argument, for any $\epsilon>0$, there exists $L= L(n,\epsilon)>0$ such that the set 
\begin{align}
\tilde{\cC}=\{x\in \cC\cap B_{r/20}(z): \big|{\cH}_{10r^2}(x)-{\cH}_{r^2/10}(x)\big|\le L(n,\epsilon) \eta~\}
\end{align}
satisfies 
\begin{align}\label{e:lowerbouncC}
\mu(\tilde{\cC})\ge (1-\epsilon)\mu(B_{r/20}(z))\ge (1-\epsilon) B^{-1}r^{n-4}20^{-(n-4)}\ge C(n)B^{-1}r^{n-4}.
\end{align}
From the definition of $(n-4,\rho,r)$-independent points and Remark \ref{r:remark_independentpoints} , it suffices to prove the following: 
\vskip 2mm
If $\rho=\rho(n,B):=\hat{C}(n)^{-1}B^{-2}$ then for all $k$ planes $P_k\subset \mathbb{R}^{n-4}\times \{y_c\}\subset \mathbb{R}^{n-4}\times C(S^3/\Gamma)$  with $k\le n-5$,  the following holds 
\begin{align}\label{e:PkcoveringtildecC}
B_{5\rho r}(\iota(P_k))\setminus \tilde{\cC}\ne\emptyset,
\end{align}
where $\iota$ is $\delta r$ GH map from $\mathbb{R}^{n-4}\times C(S^3/\Gamma)$.  
\vskip 2mm
To see \eqref{e:PkcoveringtildecC}, note that $\iota(P_k)\cap B_{r/20}(z)$ can be covered by $C(n)\rho^{-k}$ many $B_{\rho r}$-balls, which implies that 
\begin{align}\label{e:PkcoveringtildecC2}
\mu(B_{5\rho r}(\iota(P_k))\cap B_{r/20}(z))\le C(n)\rho^{-k} \cdot C(n)B (\rho r)^{n-4}\le C(n)B\rho^{n-4-k}r^{n-4}\le C(n)B\rho r^{n-4},
\end{align}
where we have used the fact that for each $B_{\rho r}(w)$ with $w\in B_{r/20}(z)$ and $\rho<1/40$ that 
\begin{align}\label{e:upperpackingrhow}
\mu(B_{6\rho r}(w))\le C(n)B (\rho r)^{n-4}.
\end{align}  
Actually, to see \eqref{e:upperpackingrhow} we first note that $r\ge  K(n,B)r_z=10^{10}\rho(n,B)^{-1}r_z$ by the chosen of $z$.  Since $|\text{Lip}  r_x|\le \delta$, for any $x\in \cC\cap B_{r}(z)$ we have $r_x\le 2r$. Thus for any $w\in B_{r/20}(z)$ if $B_{6\rho r}(w)\cap \cC=\emptyset$ then \eqref{e:upperpackingrhow} holds trivially. If $B_{6\rho r}(w)\cap \cC\ne \emptyset$ then there exists $x\in B_{6\rho r}(w)\cap \cC\subset B_r(z)\cap \cC$, hence $\mu(B_{6\rho r}(w))\le \mu(B_{12\rho r}(x))\le  B (12 \rho r)^{n-4}$ which proves \eqref{e:upperpackingrhow}.

Now the estimate \eqref{e:lowerbouncC} and \eqref{e:PkcoveringtildecC2} imply \eqref{e:PkcoveringtildecC} if $\rho(n,B)= \hat{C}(n)^{-1}B^{-2}$ for a large $\hat{C}(n)$. Therefore, we have proved Claim 1. $\square$

Now recall that $\eta=\fint_{B_{r/20}(z)}\big|{\cH}_{10r^2}(x)-{\cH}_{r^2/10}(x)\big|d\mu(x)$.  Applying Lemma \ref{l:cHfunctional} to each $x_i$, we have $n-3$ functions $f_{i,t}$ such that  $\fint_{0}^{r^2} \fint_{B_{4r}(x_i)}|\nabla^2 f_{i,t}-2g|^2dxdt\le C(n,\rv,B,\tau)(\delta r^2+\eta).$  For $1\le i\le n-4$, let $w_{i,t}= \left(f_{i,t}-f_{0,t}-d(x_0,x_i)^2\right)/2d(x_0,x_i)$.  

\noindent \textbf{Claim 2:} There exists an $(n-4,n-4)$-matrix $D$ with $|D|\le C(n,B,\tau)$ such that if $(v_{i,r^2})\equiv (w_{i,r^2})D$ then $\bar f_{0,r^2}\equiv f_{0,1}-\sum_{i=1}^{n-4} v_{i,r^2}^2$ satisfies $\bar f_{0,r^2}\ge C(n,B,\tau)r^2>0$ on $B_{5r/6}(y)$. Further, if we denote $v_{0,r^2}=\sqrt{\bar f_{0,r^2}}$ on $B_{5r/6}(y)$, then we have $\min_{x\in B_{r/2}(y)}|\det A|(x)\ge 1/2$, where $A(x)=\langle \nabla v_{i,r^2},\nabla v_{j,r^2}\rangle(x)$ for $i,j=0,\cdots, n-4$.

 \textbf{Proof of Claim 2:}  We prove this claim by contradiction, therefore let us assume this is not true. Then for $\delta_a\to 0$ there exists $(\delta_a,\tau)$-neck region $\cN_a\subset B_2(p_a)\subset M_a$ with $y_a\in \cN_{a,K(n,B)}$, $d(y_a,\cC_a)=r_a$ and $(n-4,\rho, r_a/20)$-independent points $\{x_{0,a},\cdots,x_{n-4,a}\}\subset {\cC}_a$, but there is no matrix $D_a$ with $|D_a|\le C(n,B,\tau)$ satisfying the claim for $\cN_a$, where $C(n,B,\tau)$ will be determined later.  Let us rescale each metric $g_a$ to $\tilde{g}_a$ such that $d(y_a,\cC_a)=1$. Taking limit we have $M_a\to \mathbb{R}^{n-4}\times C(S^3/\Gamma)$ and $\cC_a\to \cC_\infty\subset \mathbb{R}^{n-4}\times \{y_c\}$ with $y_a\to y_\infty$ and $x_{i,a}\to x_{i,\infty}$, where $\{x_{i,\infty}\}$ is a $(n-4,\rho,1/20)$-independent set.   By the $C^0$ estimate of Lemma \ref{l:cHfunctional} we have that $|\tilde{f}_{i,1,a}-d_{x_{i,a}}^2|\to 0$, and hence $\tilde{f}_{i,1,a}\to \tilde{f}_{i,1,\infty}\equiv d^2_{x_{i,\infty}}$.
 On the one hand, it is then clear that $\tilde{w}_{i,1,\infty}\equiv\left(\tilde{f}_{i,1,\infty}-\tilde{f}_{0,1,\infty}-d(x_{0,\infty},x_{i,\infty})^2\right)/2d(x_{0,\infty},x_{i,\infty})$ is a linear function. As $\{x_{i,\infty}\}$ is a $(n-4,\rho,1/20)$-independent set, one can choose a matrix $D_\infty$ with $|D_\infty|\le C(n,\rho)\equiv C(n,B,\tau)$ such that $(\tilde{v}_{i,1,\infty})\equiv(\tilde{w}_{i,1,\infty})D_\infty$ represents the standard coordinate functions of $\mathbb{R}^{n-4}\times \{y_c\}\subset \mathbb{R}^{n-4}\times C(S^3/\Gamma)$ with $(0^{n-4},y_c)=x_{0,\infty}$. Then $\bar f_{0,1,\infty}\equiv\tilde{f}_{0,1,\infty}-\sum_{i=1}^{n-4}\tilde{v}_{i,1,\infty}^2$ is the distance square function $d_{\mathbb{R}^{n-4}\times \{y_c\}}^2$. Thus the $(n-3,n-3)$ matrix $A_\infty$ defined in $B_1(y_\infty)\subset \mathbb{R}^{n-4}\times C(S^3/\Gamma)$ satisfies $A_\infty=(\delta_{ij})$. However, for the rescaled metric $\tilde{g}_a$ we have by the harmonic radius lower bound of $y_a$ in Lemma \ref{l:neck:distance_harm_rad} and the hessian estimate in Lemma \ref{l:cHfunctional} that
 $\tilde{f}_{i,1,a}\to \tilde{f}_{i,1,\infty}=d^2_{x_{i,\infty}}$ in $C^2$ sense on $B_{2/3}(y_\infty)$. By choosing $D_a=D_\infty$ for $a$ sufficiently large, this derives our contradiction and proves the result. $\square$

Now we plan to use such functions to prove our expected curvature estimates by using Lemma \ref{l:curvature_level_sets}. By Claim 2,
let us consider the map $\Phi=(h,u_1,\cdots, u_{n-4})\equiv (f_{0,r^2},v_{1,r^2}, \cdots, v_{n-4,r^2})$ on $B_{10r}(y)$. Using the harmonic radius lower bound $r_h(y)\ge 4r/5$ and the hessian estimate of $f_{i,r^2}$ in Lemma \ref{l:cHfunctional}, and noting that $u_i$ are linear combinations of the $f_{i,r^2}$ with uniformly bounded constants, we have the following scale invariant estimates
\begin{align}
\sup_{B_{3r/4}(y)}|\nabla^2h-2g|^2+r^2\fint_{B_{3r/4}(y)}|\nabla^3h|^2(z)\,dz+\sum_{i=1}^{n-4}\left(\sup_{B_{3r/4}(y)}r^2|\nabla^2u_i|^2+r^4\fint_{B_{3r/4}(y)}|\nabla^3u_i|^2(z)\,dz\right)\le C(n,\rv,B,\tau)(\eta+\delta r^2)\,.
\end{align}
Moreover, by the pointwise nondegeneration of $A(x)$ in Claim 2, we can now use Lemma \ref{l:curvature_level_sets} to deduce the curvature estimates.
In fact, for any $z\in B_{r/2}(y)$, we have scale invariant estimates
\begin{align}
r^4|\Rm|^2(z)\le C(n,\rv,B)\left(|{\Ric}|^2r^4+r^2|\nabla^3h|^2+\sum_{i=1}^{n-4}r^4|\nabla^3 u_i|^2+\mathcal{F}+\mathcal{F}^2\right)(z)\, ,
\end{align}
where $\mathcal{F}=|\nabla^2h-2g|^2+\sum_{i=1}^{n-4}r^2|\nabla^2u_i|^2$. By the pointwise hessian estimate for $u_i$, we have
\begin{align}
r^4|\Rm|^2(z)\le C(n,\rv,B,\tau)\left(|{\Ric}|^2r^4+r^2|\nabla^3h|^2+\sum_{i=1}^{n-4}r^4|\nabla^3 u_i|^2+|\nabla^2h-2g|^2+\sum_{i=1}^{n-4}r^2|\nabla^2u_i|^2\right)(z)\, .
\end{align}
Integrating over $B_{r/2}(y)$, we get
\begin{align}
r^4\fint_{B_{r/2}(y)}|\Rm|^2(z)\,dz\le C(n,\rv,B,\tau)(\eta +\delta r^2)= C(n,\rv,B,\tau)\Big(\delta r^2+\fint_{B_{r/20}(z)}\big|{\cH}_{10r^2}(x)-{\cH}_{r^2/10}(x)\big|d\mu(x)\Big)\, .
\end{align}
This completes the proof.
\end{proof}
\vspace{.5cm}

\subsection{Proof of the $L^2$ curvature estimate on neck region}
In this subsection, based on the local $L^2$ curvature estimate in Proposition \ref{p:local_L2_neck}, we prove Theorem \ref{t:L2_neck}.
\begin{proof}[Proof of Theorem \ref{t:L2_neck}]
Since the estimates are scale invariant, without loss of generality we will assume $r=1$. By the local $L^2$ curvature estimate of Proposition \ref{p:local_L2_neck}, for any $y\in\cN_{K(n,B)}$ with $s=d(y,\cC)=d(y,z)$ and $B_{s/10}(z)\subset B_2(p)$ which holds in particular for all $y\in B_1(p)\cap \cN_{K(n,B)}$,  we have the estimate
\begin{align}\label{e:neck_L2_entropy}
\int_{B_{s/2}(y)} |\Rm|^2(z)\,dz \leq C(n,\rv,B,\tau)\delta s^{n-2}+C(n,\rv,B,\tau)\int_{B_{s/20}(z)}|{\cH}_{10s^2}-{\cH}_{10^{-1}s^2}|(x)\,d\mu(x)\, .	
\end{align}

In order to use such an estimate, we first construct a Vitali covering. For any $x\in \cN_{K(n,B)}$ with $d(x,\cC)=s_x$, consider the covering $\{B_{s_x/5}(x),x\in \cN_{K(n,B)}\}$ of $\cN_{K(n,B)}$. We can choose a subcovering $\{B_{s_a/5}(x_a)\}$ such that $\{B_{s_a/40}(x_a)\}$ are disjoint and
\begin{align}
\cN_{K(n,B)}\cap B_1(p)\subseteq \bigcup_a  B_{s_{a}/2}(x_{a})\, ,
\end{align}
where $s_a=s_{x_a}$.
Rearranging $x_a$ such that $d(x_{\alpha,i},\cC)=s_{\alpha,i}=d(x_{\alpha,i},z_{\alpha,i})$ with $s_{\alpha,i}\in (2^{-\alpha-1}, 2^{-\alpha}]$ and $z_{\alpha,i}\in\cC$, then we have
\begin{align}
\cN_{K(n,B)}\cap B_1(p)\subseteq \bigcup_\alpha \bigcup_{i=1}^{N_\alpha} B_{s_{\alpha,i}/2}(x_{\alpha,i})\, ,	
\end{align}
 and $\{B_{40^{-1}s_{\alpha,i}}(x_{\alpha,i})\}$ are disjoint.  Moreover, by Ahlfors assumption, for any fixed $\alpha$ we have that $\sharp_i\{B_{s_{\alpha,i}/2}(x_{\alpha,i})\}\le C(n,B,\rv)s_\alpha^{4-n}$ with $s_\alpha=2^{-\alpha}$.
Then we have
\begin{align}
	\int_{\cN_{K(n,B)}\cap B_1(p)} |\Rm|^2(z)\,dz &\leq \sum_\alpha \sum_i \int_{B_{s_{\alpha,i}/2}(x_{\alpha,i})} |\Rm|^2(z)\,dz \\
&\leq C(n,\rv,B,\tau)\sum_\alpha \sum_i\left(\delta s_{\alpha,i}^{n-2} +\int_{B_{s_{\alpha,i}/20}(z_{\alpha,i})}
|{\cH}_{10s_{\alpha,i}^2}-{\cH}_{10^{-1}s_{\alpha,i}^2}|(x)\,d\mu(x)\right)\notag\\
	&\leq C(n,\rv,B,\tau)\sum_\alpha \delta s_\alpha^{2}+C(n,\rv,B,\tau)\sum_\alpha\int_{ B_{3/2}(p)}\big|{\cH}_{40s^2_\alpha}-{\cH}_{40^{-1}s^2_\alpha}\big|(x)\,d\mu(x)\notag
\end{align}
By the monotonicity of $\cH$-volume, we have
\begin{align}
	\int_{\cN_{K(n,B)}\cap B_1(p)} |\Rm|^2(z)\,dz&\leq C(n,\rv,B,\tau)\delta+C(n,\rv,B,\tau)\int_{ B_{3/2}(p)}\big|\sum_\alpha({\cH}_{40s^2_\alpha}-{\cH}_{40^{-1}s^2_\alpha})\big|(x)\,d\mu(x)\notag\\
	&\leq C(n,\rv,B,\tau)\delta+C(n,\rv,B,\tau)\int_{ B_{3/2}(p)}\big|{\cH}_{40}-{\cH}_{40^{-5}r^2_x}\big|(x)\,d\mu(x)\, .
\end{align}
On the other hand, noting that $B_2(p)$ is $\delta$-close to $\dR^{n-4}\times C(S^3/\Gamma)$, by $\epsilon$-regularity Lemma \ref{l:neck:distance_harm_rad} on neck region, for any $\delta'>0$ if $\delta\le \delta(n,\delta')$ we have
\begin{align}
\int_{(\cN_{10^{-5}}\setminus \cN_{K(n,B)})\cap B_1(p)} |\Rm|^2(z)\,dz\le \sum_{x\in B_1(p)\cap \cC}\int_{B_{K(n,B)r_x}(x)\setminus \cup_{y\in \cC}B_{10^{-5}r_y}(y)}|\Rm|^2(z)\,dz\le \sum_{x\in B_1(p)\cap \cC} C(n,B) \delta' r_x^{n-4}\le C(n,B)\delta'.
\end{align}
Thus we arrive at 
\begin{align}
\int_{\cN_{10^{-5}}\cap B_1(p)} |\Rm|^2(z)\,dz\le C(n,B)\delta'+C(n,\rv,B,\tau)\delta+C(n,\rv,B,\tau)\int_{ B_{3/2}(p)}\big|{\cH}_{40}-{\cH}_{40^{-5}r^2_x}\big|(x)\,d\mu(x)\, .
\end{align}
 However, since $\cN$ is a neck region we have that both $B_{\delta^{-1}}(x)$ and $B_{\delta^{-1}r_x}(x)$ are Gromov-Hausdorff close to $\dR^{n-4}\times C(S^3/\Gamma)$, and therefore by volume convergence we have that $|\cH_{40}-\cH_{40^{-5}r_x^2}|(x)\to 0$ as $\delta\to 0$.  To see this, it suffices to prove the following claim.
 
\noindent \textbf{Claim:} For any $\epsilon>0$ and $t\le 10$, if $\delta\le \delta(n,\rv,\epsilon)$ and 
\begin{align}
d_{GH}( B_{\delta^{-1}\sqrt{t}}(x), B_{\delta^{-1}\sqrt{t}}(0^{n-4},y_c))\le \sqrt{t} \delta, ~~\text{ where $(0^{n-4},y_c)$ is a cone vertex of $\dR^{n-4}\times C(S^3/\Gamma)$}
\end{align}
then $|\cH_t(x)-\bar \cH_{\Gamma}|\le \epsilon$, where the constant $\bar \cH_{\Gamma}$ depending on on $|\Gamma|$ and $n$ is the $\cH$-volume of cone $\dR^{n-4}\times C(S^3/\Gamma)$ .

\textbf{Proof of Claim:} By the scaling property of $\cH$-volume in Remark \ref{r:scaling_cHvolume}, it suffices to prove the case $t=1$. For simplicity, let us denote $\cH_1(p)=\int_M L(x,\delta)dx$ and $\bar \cH_{\Gamma}=\int_X L(x)dx$ with $X=\mathbb{R}^{n-4}\times C(S^3/\Gamma)$, which satisfies $L(x,\delta)\to L(x)$ uniformly in compact set if $\delta\to 0$. Due to the exponential decay of $L(x,\delta)$, $L(x)$ and growth control of volume by the volume comparison, the integral of $L(x,\delta)$ over $M\setminus B_{R}(x)$ and the integral of $L(x)$ over $\dR^{n-4}\times C(S^3/\Gamma)\setminus B_{R}(0^{n-4},y_c))$ both must be small than $\epsilon/10$ if $R\ge R(n,\rv,\epsilon)$. For the integral over $B_{R}$, by volume convergence we have if $\delta\le \delta(n,\rv,\epsilon)$ that
\begin{align}
|\int_{B_R(p)}L(x,\delta)dx-\int_{B_R(0^{n-4},y_c))}L(x)dx|\le \epsilon/10.
\end{align}
Combining the estimates on $B_R$ and outside $B_R$ we have proved the Claim. $\square$

  Hence, for any $\delta\leq \delta(n,\rv,\delta',\tau)$, we have proved Theorem \ref{t:L2_neck}.
\end{proof}
\vspace{.5cm}

\section{Splitting Functions on Neck Regions}\label{s:neck_splitting}

The main goal of this section is to do some analysis on neck regions, and more specifically to study the behavior of splitting functions on neck regions.  We will show that splitting functions are in fact better behaved than the standard basic estimates would lead one to believe.  This will be the key technical ingredient in the proof of Theorem \ref{t:neck_region}.\\

More precisely, with $\delta,\delta',\tau,B>0$ fixed background parameters, we will be interested throughout this section in the following assumptions:
\begin{align}\label{e:assumptions_splitting}
\text{(S1)}& \text{ $|{\Ric}|<\delta$}\, .\notag\\
\text{(S2)}& \text{ $\cN = B_2(p)\setminus \overline B_{r_x}(\cC)$ is a $(\delta,\tau)$-neck region}\, .\notag\\
\text{(S3)}& \text{ For each $x\in \cC$ and $r_x<r$ with $B_{2r}(x)\subseteq B_2$ we have $ B^{-1}r^{n-4}<\mu(B_r(x)) <B r^{n-4}$}\, .\notag\\
\text{(S4)}& \text{ $u:B_4(p)\to \dR^{n-4}$ is a $\delta'$-splitting function}\, .
\end{align}

Note that it will eventually be a consequence of Theorem \ref{t:neck_region} that we can take $B=A(n,\tau)$, and therefore $(S3)$ will be a redundant assumption.  However, at this stage it is important to not make this assumption, as the results of this section will factor heavily in the proof of this point.  The main goal of this section is to prove the following: \\

\begin{theorem}\label{t:splitting_neck_bilipschitz}
Let $(M^n_j,g_j,p_j)\to (X,d,p)$ be a limit space with $\Vol(B_1(p_j))>\rv>0$.  Then for every $\epsilon,B>0$ and $\tau<\tau(n)$ if $\delta,\delta'\leq \delta(n,\epsilon,\tau,B,\rv)$ is such that assumptions (S1)-(S4) of \eqref{e:assumptions_splitting} hold, then there exists a subset $\cC_\epsilon\subseteq \cC\cap B_1$ such that
\begin{enumerate}
\item[(b1)] $\mu(\big(\cC\cap B_1\big)\setminus \cC_\epsilon)<\epsilon$.
\item[(b2)] For each $x,y\in \cC_\epsilon$ we have that $1-\epsilon<\frac{|u(x)-u(y)|}{d(x,y)}<1+\epsilon$.
\item[(b3)] For each $x\in \cC_\epsilon$ and $r_x\leq r\leq 1$ we have that $u:B_r(x)\to \dR^{n-4}$ is an $\epsilon$-splitting.
\end{enumerate}
\end{theorem}
\begin{remark}
Recall that we say a limit $M^n_j\to X$ satisfies $|\Ric|<\delta$ if $|\Ric_j|<\delta_j$ where $\delta_j\to \delta$.
\end{remark}
\begin{remark}
Essentially the entire section will focus on the case when $X$ is a manifold, since by applying the neck approximation of Theorem \ref{t:neck_approximation} we will immediately conclude the general statement.	
\end{remark}

\vspace{.5cm}

Before continuing let us make some observations about the above result.  The difficulty in proving the result is due to the smallness of the set $\cC$.  If a $\delta$-splitting function were to have pointwise estimates on the hessian, as it does for instance in the bounded curvature case, then the result would be trivial.  However, with only $L^2$ bounds on the hessian apriori, one can only use the hessian information to prove $(b3)$ away from a set of codimension $2+\epsilon$, which is far away from the requirements of $(b1)$.  In fact, even if we assume the the main theorem of this paper holds, which gives us an $L^2$ bound on the curvature, then the most one can prove is $L^4$ estimates on the hessian, see \cite{Cheeger}, which are still not strong enough to prove $(b1)-(b3)$.  Therefore, we see there is a crucial gap in the known estimates for splitting functions and what is needed for the above Theorem.  Closing this gap will require some new estimates which will be presented throughout this section, and which will depend fundamentally on the two sided Ricci bound.

\vspace{.5cm}

\subsection{Estimates on Standard Green's Functions on Neck Regions}\label{ss:splitting:Greens_standard}

In this section we discuss standard Green's functions on neck regions.  We will use these estimates in the next section to discuss Green's functions with respect to the packing measure $\mu$.

\begin{definition}\label{d:Green_neck}
	Given $(M^n,g,p)$ with $\cN\subseteq B_2(p)$ a $(\delta,\tau)$-neck region with packing measure $\mu_\cN$, let us consider $\bar \mu=\mu_{\cN}|_{B_{39/20}(p)}$ and we define the following:
	\begin{enumerate}
		\item We denote by $G_x(y)$ a Green's function at $x$.  That is, $-\Delta_y G_x(y) = \delta_x$, where $\delta_x$ is the dirac delta at $x$.
		\item We denote by $G_{\bar \mu}(y)$ the function $G_{\bar \mu}(y) = \int_{B_{39/20}(p)} G_x(y)\, d\mu(x)$ a Green's function which solves $-\Delta G_{\bar\mu} = \bar \mu$.
		\item We denote by $b(y)\equiv G^{-1/2}_{\bar \mu}$ the $\bar \mu$-Green's distance function to the center points $\cC$.
	\end{enumerate}
\end{definition}
\vspace{.5cm}

The intuition is that $b$ should behave in a manner which is comparable to the Green's function from the singular set in $\dR^{n-4}\times C(S^3/\Gamma)$, which itself is a multiple of $d(\cS^{n-4},\cdot)^{-2}\approx r_h^{-2}$. 
The main result of this subsection will be to prove just that.  Precisely:\\

\begin{lemma}\label{l:green_function_distant_function}
	Let $(M^n,g,p)$ satisfy $\Vol(B_1(p))>\rv>0$.  Then for given $B>0$ there exists  $\delta(n,B,\rv)$, $\tau(n)$, $C(n,B,\rv)>0$ and a Green function $G_{\bar \mu}$ such that if the assumptions (S1)-(S3) of \eqref{e:assumptions_splitting} hold, then
	\begin{enumerate}
	\item For $x\in \cN_{10^{-6}}\cap B_{39/20}(p)$, we have $C^{-1}\,d(x,\cC)\leq b(x)\leq C\,d(x,\cC)$.
	\item For $x\in \cN_{10^{-6}}\cap B_{39/20}(p)$, we have $C^{-1}\leq |\nabla b|\leq C$.
	\end{enumerate}
\end{lemma}

Before giving the proof of Lemma \ref{l:green_function_distant_function}, let us point out that the advantage of $\mu$ after restricting to $B_{39/20}(p)$. We have the following uniform Ahlfors regularity.
\begin{lemma}
Let $(M^n,g,p)$ satisfy $\Vol(B_1(p))>\rv>0$ and the assumptions (S1)-(S3) of \eqref{e:assumptions_splitting}, then for all $x\in \cC\cap B_{39/20}(p)$, $r_x\le r\le 1$ we have 
\begin{align}
C(n)^{-1}B^{-1}r^{n-4}\le \bar\mu(B_r(x))\le C(n)Br^{n-4}.
\end{align}
\end{lemma}
\begin{proof}
For the upper bound, if $r\le 1/40$  it follows directly by (S3) since $B_{2r}(x)\subset B_2(p)$. If $r\ge 1/40$, it follows from $\bar \mu(B_r(x))\le \mu(B_{39/20}(p))\le C(n)B$ where the last inequality is based on (S3) and a Vitali covering $\{B_{1/40}(x_\alpha),x_\alpha\in \cC\cap B_{39/20}(p)\}$ of $B_{39/20}(p)$. 
On the other hand,  the lower bound follows directly by (S3) and the fact that there exists $B_{r/3}(\tilde{x})\subset B_r(x)\cap B_2(p)$ with $\tilde{x}\in \cC$.
\end{proof}

This subsection is dedicated to proving the above Lemma \ref{l:green_function_distant_function}, though we will need to work through several preliminaries first.  Let us begin by collecting together a list of useful computations:\\

\begin{lemma}\label{l:green_formula}
	Given $(M^n,g,p)$ with $\cN\subseteq B_2(p)$ a $(\delta,\tau)$-neck region and $b(y)$ an associated Green's distance function, then for a smooth compactly supported function $f:B_4(p)\rightarrow \dR$ the following hold:
\begin{enumerate}
 \item ${\bar \mu}[f] = -\int_{b\leq r} G_{{\bar \mu}}\Delta f(z)\,dz +2 r^{-3}\int_{b=r}f|\nabla b|(x)\, dx + r^{-2}\int_{b\leq r}\Delta f(z)\,dz$.
 \item ${\bar \mu}[f] = -\int_{b\leq r} G_{{\bar \mu}}\Delta f(z)\,dz +2 r^{-3}\int_{b=r}f|\nabla b|(x)\, dx + r\frac{d}{dr}\left(r^{-3}\int_{b=r}f|\nabla b|(x)\,dx\right)$,
\end{enumerate}
\end{lemma}
\begin{proof}
Let us observe that $\big(b^{-2}-r^{-2}\big)$ vanishes on $b=r$, is smooth in a neighborhood of this set, and $\text{supp}\Big\{\Delta\big(b^{-2}-r^{-2}\big)\Big\}\subseteq \{b<r\}$.  Thus we can use standard properties of the distributional laplacian to compute
\begin{align}
\int_{b\leq r} \Delta f\left(b^{-2}-r^{-2}\right)(z)\,dz &= -{\bar \mu}\big[f\big]-\int_{b=r}f\langle \frac{\nabla b}{|\nabla b|},\nabla b^{-2}\rangle(z)\,dz \notag\\
	&= -{\bar \mu}\big[f\big] + 2r^{-3}\int_{b=r} f|\nabla b|(z)\,dz\, ,
\end{align}
which proves the first formula.  To compute the second formula let us first compute
\begin{align}
\Delta b = 3b^{-1}|\nabla b|^2\, ,	
\end{align}
and recall that the mean curvature of the $b=r$ level set is given by $\text{div}\big(\frac{\nabla b}{|\nabla b|}\big)$.  Therefore we can compute

\begin{align}\label{e:greens_formula:1}
\frac{d}{dr}\left(r^{-3}\int_{b=r}f|\nabla b|(z)\,dz\right) = &-3r^{-4}\int_{b=r}f|\nabla b| (z)\,dz+ r^{-3}\int_{b=r}\langle \nabla f, \frac{\nabla b}{|\nabla b|}\rangle(z)\,dz \notag\\
~~\text{ }&+ r^{-3}\int_{b=r} f \langle \nabla|\nabla b|, \frac{\nabla b}{|\nabla b|^2}\rangle(z)\,dz + r^{-3}\int_{b=r} f \text{div}\big(\frac{\nabla b}{|\nabla b|}\big)(z)\,dz\notag\\
=& r^{-3}\int_{b\le r}\Delta f(z)\,dz -3r^{-4}\int_{b=r}f|\nabla b|(z)\,dz+ r^{-3}\int_{b=r} f \langle \nabla|\nabla b|, \frac{\nabla b}{|\nabla b|^2}\rangle(z)\,dz \notag\\
~~\text{ }&+ r^{-3}\int_{b=r} f \Big(\frac{\Delta b}{|\nabla b|}-\langle \nabla b, \frac{\nabla|\nabla b|}{|\nabla b|^2}\rangle\Big)(z)\,dz\notag\\
=&r^{-3}\int_{b\le r}\Delta f(z)\,dz\, .
\end{align}

\end{proof}

With the help of the above we can now compute the following, which will be the key use of the Green's distance function:\\

\begin{lemma}\label{l:ode_Green_function}
	Given $(M^n,g,p)$ with $\cN\subseteq B_2(p)$ a $(\delta,\tau)$-neck and $b(y)=b_\cN(y)$ the associated Green's distance function, then the following hold:
	\begin{enumerate}
	\item For every compactly supported $f:B_4(p)\to \dR$ we have
	\begin{align}\nonumber
		\frac{d}{dr}\Big(r^{-3}\int_{b=r} f\, |\nabla b|(z)\,dz\Big) = \, r^{-3}\int_{b\leq r}\Delta f(z)\,dz\,.
	\end{align}
	\item \begin{align}\nonumber
		r\frac{d^2}{dr^2}\Big(r^{-3}\int_{b=r} f\, |\nabla b|(z)\,dz\Big)+3\frac{d}{dr}\Big(r^{-3}\int_{b=r} f\, |\nabla b|(z)\,dz\Big) = \, r^{-2}\int_{b=r} \Delta f \,|\nabla b|^{-1}(z)\,dz.
	\end{align}
	\end{enumerate}
\end{lemma}
\begin{proof}
The proof of (1) is just \eqref{e:greens_formula:1} in the previous Lemma.  Taking derivative of (2) in Lemma \ref{l:green_formula} with respect to $r$, we have

\begin{align}
r\frac{d^2}{dr^2}\Big(r^{-3}\int_{b=r} f\, |\nabla b|(z)\,dz\Big)+3\frac{d}{dr}\Big(r^{-3}\int_{b=r} f\, |\nabla b|(z)\,dz\Big) &= \,r^{-2}\frac{d}{dr} \left(\int_{b\le r} \Delta f (z)\,dz\right)\\ \nonumber
&=\,r^{-2}\frac{d}{dr} \left(\int_0^rds\int_{b=s}\Delta f\,|\nabla b|^{-1} (z)\,dz\right)\\ \nonumber
&=r^{-2}\int_{b=r}\Delta f\,|\nabla b|^{-1}(z)\,dz.
\end{align}

\end{proof}
\vspace{.5cm}

\subsubsection{Green function estimates on manifolds with Ricci curvature lower bound}
In order to prove Lemma \ref{l:green_function_distant_function}, we first consider the Green function $G_x$ for any center point $x\in\cC$ in a neck region.  Besides the basic expected estimates, we need to see that at every point $y\in\cN$ in the neck region itself there is a fixed direction $v_y\in T_yM$ for which $G_x$ has positive gradient in the direction of $v_y$.  This will be used heavily when we integrate to construct $G_{\cN}$ in order to see that the gradient of $G_{{\bar \mu}}$ has a definite lower bound in the neck region.  Precisely, we have the following:\\

\begin{lemma}\label{l:green_function_cone_point}
Let $(M^n,g,p)$ satisfy $\Vol(B_1(p))>\rv>0$ and $|{\Ric}|<\delta$ with $\cN=B_2(p)\setminus \bigcup_{x\in \cC} \overline B_{r_x}(x)$ a $(\delta,\tau)$-neck region.  For each $R>0$ if $\delta\leq \delta(n,R,\rv)$ and $0<\tau\le \tau(n)$ then there exists $C(n,\rv)>0$ such that $\forall$ $x\in \cC$ there exists a Green's function $G_x$ such that:
\begin{enumerate}
\item For all $y\in B_{\delta^{-1/2}}(p)$ we have $C^{-1}d_x^{2-n}(y)\le G_x(y)\le Cd_x^{2-n}(y)$.
\item For all $y\in B_{\delta^{-1/2}}(p)$ we have $|\nabla_yG_x|(y)\le Cd_x^{1-n}(y)$
\item For all $y\in \cN_{10^{-5}}$ if $r=d(y,\cC)$ and $x\in B_{Rr}(y)\cap \cC$ then there exists a unit vector $v_y\in T_yM$ independent of $x$  such that
\begin{align}\label{e:green_nabla_lowerbound_I}
\langle\nabla_yG_x(y),v_y\rangle\,>\, c(R,n) r^{1-n}>0,
\end{align}
\item In particular, if $y\in \cN_{10^{-5}}$ with $r=d(y,\cC)$ and $x\in B_{10r}(y)\cap \cC$ then
\begin{align}\label{e:green_nabla_lowerbound_II}
\langle\nabla_yG_x(y),v_y\rangle\,>\, c(n) r^{1-n}.
\end{align}
\end{enumerate}
\end{lemma}
\begin{proof}
For each $x\in \cC$, since the estimates in (1) and (2) are scale invariant, in order to estimate the Green's function on the ball $B_{\delta^{-1/2}}(p)$ with $|{\Ric}|<\delta$, it suffices to construct and estimate the Green function on the ball $B_1(p)$ with $|{\Ric}|<1$.  In this case we construct the Green function $G_x$ in the following manner.
Let $${G}_{x,0}(y)=\int_0^1\rho_t(x,y)dt\, .$$ Then by heat kernel estimate of Theorem \ref{t:heat_kernel}, we can compute
$C^{-1}d_x^{2-n}(y)\le G_{x,0}(y)\le Cd_x^{2-n}(y)$. Additionally, we can compute that $\Delta {G}_{x,0}(y)=\int_0^1\partial_t \rho_t(x,y)dt=\rho_1(x,y)-\delta_x(y).$ Therefore let us solve $G_{x,1}$ by
\begin{align}\label{e:green_rho1}
\Delta G_{x,1}(y)=\rho_1(x,y),
\end{align}
on $B_{2}(x)$ with $G_{x,1}(y)= G_{x,0}(y)$ on $\partial B_{2}(x)$. We define our Green's function
$$G_x=G_{x,0}-G_{x,1}\, .$$
We claim that $G_x$ satisfies the Lemma. Indeed, by noting the uniform bound on the heat kernel $\rho_1(x,y)$ we may use a standard maximal principle and Cheng-Yau gradient estimates on (\ref{e:green_rho1}) as in \cite{Cheeger01},   in order to show that $|G_{x,1}|+|\nabla G_{x,1}|<C(n)$ in $B_{4/3}(p)$ is uniformly bounded. Coupling with the estimates of $G_{x,0}$, we proven the estimates on $G_x$. Thus we have proved (1) and (2).

Now we only give a proof of (\ref{e:green_nabla_lowerbound_I}), the argument is the same for (\ref{e:green_nabla_lowerbound_II}). We argue by contradiction. Therefore assume for some $R>0$ that there exists a $(\delta_i,\tau)$-neck regions $\cN_i$ with $\delta_i\to 0$, $y_i\in \cN_{i,10^{-5}}$ and $r_i=d(y_i,\cC_i)$ such that
\begin{align}
\sup_{v\in T_{y_i}M,~|v|=1}\inf _{x_{i}\in B_{Rr_i}(y_i)\cap \cC_i}\langle\nabla_yG_{x_i}(y_i),v\rangle\,<\, i^{-1}r_i^{1-n}.
\end{align}
 Scaling $B_{Rr_i}(y_i)$ to ball $\tilde{B}_{R}(y_i)$ and denoting the corresponding Green function to be $\tilde{G}_{x_i}$, then
\begin{align}
\sup_{v\in T_{y_i}M, ~|v|=1}\inf _{x_{i}\in \tilde{B}_{R}(y_i)\cap \cC_i}\langle\nabla_y\tilde{G}_{x_i}(y_i),v\rangle\,<\, i^{-1}.
\end{align}
To deduce a contradiction, we will show that the Green function $\tilde{G}_{x_i}$  converges to a function $D\,d^{2-n}_{x}$ on  $\mathbb{R}^{n-4}\times C(S^3/\Gamma)$ with constant $C^{-1}(n,\rv)<D<C(n,\rv)$.  Since the convergence is $C^1$ on the neck region due to the harmonic radius control, we can take $v_y$ to be any vector which approximates the radial direction on the $C(S^3/\Gamma)$ factor in order to conclude the result.

Thus, we only need to show the Green function $\tilde{G}_{x_i}\to D\,d^{2-n}_{x}$. On one hand, we notice that $\tilde{B}_{\delta_i^{-1/2}}(x_i)$ converges to the same limit $\mathbb{R}^{n-4}\times C(S^3/\Gamma)$ for any sequence $x_{i}\in \tilde{B}_{R}(y_i)\cap \cC_i$.  On the other hand, by $(1)$ and $(2)$ of the Lemma we have on $\tilde{B}_{\delta_i^{-1/2}}(x_i)$, we have
$C^{-1}d_{x_i}^{2-n}\le \tilde{G}_{x_i}\le Cd_{x_i}^{2-n}$ and $|\nabla \tilde{G}_{x_i}|\le Cd_{x_i}^{1-n}$.  By standard Arzel\'a-Ascoli we have that $\tilde{G}_{x_i}$ converges to a function $G_x$ on the limit space $\mathbb{R}^{n-4}\times C(S^3/\Gamma)$ which satisfies the estimates
\begin{align}\label{e:green_cone_upper_lower}
C^{-1}d_{x}^{2-n}\le {G}_{x}\le Cd_{x}^{2-n},~~~\, \mbox{ and }\,~~~\, |\nabla G_x|\le Cd_x^{1-n}\, .
\end{align}
In fact, if we use some RCD theory it is not hard to see that $G_x$ is itself the Green's function at $x$, however we will prove something slightly weaker in order to avoid such techniques. Indeed, since $\tilde G_{x_i}$ converges smoothly on the regular part of $\mathbb{R}^{n-4}\times C(S^3/\Gamma)$ we at least have that $G_x$ is harmonic away from the singular set.  If we lift $G_x$ to a function $G_{\tilde x}$ on $\dR^n$, we get away from $\tilde x$ that $G_{\tilde x}$ is locally lipschitz and harmonic away from a set of zero capacity.  Hence, $G_x$ is harmonic away from $\tilde x$ with the bounds \eqref{e:green_cone_upper_lower}.  Now the only harmonic functions on $\dR^n\setminus \tilde x$ with estimates \eqref{e:green_cone_upper_lower} are multiples of the Green's function.  Hence we have $G_x=D\, d_x^{2-n}$ for some constant $C^{-1}(n,\rv)<D<C(n,\rv)$ as claimed, which finishes the proof.
\end{proof}
\vspace{.5cm}

\subsubsection{Proof of Lemma \ref{l:green_function_distant_function}}
Noting that $G_{\bar \mu}(y) = \int_{B_{39/20}(p)} G_x(y)\, d{\bar \mu}(x)$ we will use the pointwise Green function estimates in Lemma \ref{l:green_function_cone_point} and the Ahlfors regularity assumption in order to conclude the proof of Lemma \ref{l:green_function_distant_function}.  Indeed, for any $y\in \cN_{10^{-5}}\cap B_{39/20}(p)$ let $r\equiv d(y,\cC)$, and let us estimate the upper bound of $G_{\bar \mu}(y)$ as follows:
\begin{align}
G_{\bar \mu}(y)&\le C\int_{B_{39/20}(p)} d_x^{2-n}(y)d{\bar \mu}(x)=C\int_{B_{10r}(y)\cap B_{39/20}(p)}d_x^{2-n}(y)d{\bar \mu}(x)+C\sum_{i=1}^\infty\int_{A_{10^ir,10^{i+1}r}(y)\cap B_{39/20}(p)}d_x^{2-n}(y)d{\bar \mu}(x) \\
&\le Cr^{2-n}\int_{B_{10r}(y)\cap B_{39/20}(p)}d{\bar \mu}(x)+Cr^{2-n}\sum_{i=1}^\infty 10^{-i(n-2)}\int_{A_{10^ir,10^{i+1}r}(y)\cap B_{39/20}(p)}d{\bar \mu}(x)\, ,\notag\\
&\leq C\cdot B r^{-2}\Big(1+\sum 10^{-2i}\Big)\equiv C(n,\rv,B) r^{-2}\, ,
\end{align}
as claimed.  To prove the lower bound of $G_{\bar \mu}$ we can similarly compute
\begin{align}\label{e:lowerboundGmu}
G_{\bar \mu}(y)&\ge C^{-1}\int_{B_{39/20}(p)} d_x^{2-n}(y)d{\bar \mu}(x)\ge C^{-1}\int_{B_{10r}(y)\cap B_{39/20}(p)}d_x^{2-n}(y)d{\bar \mu}(x)\notag\\
&\ge C^{-1} 10^{2-n}r^{2-n}{\bar \mu}\Big(B_{8r}(x_0)\cap B_{39/20}(p)\Big)\ge C^{-1}(n,\rv,B)r^{-2},
\end{align}
where $d(x_0,y)\le 2r=2d(y,\cC)$ and $x_0\in\cC$. 
Since $b^{-2}(y)=G_{\bar \mu}(y)$, we have $C^{-1}r\le b(y)\le Cr$, or that $C^{-1}\,d(y,\cC)\leq b(y)\leq C\,d(y,\cC)$. Hence we have proven (1) of Lemma \ref{l:green_function_distant_function}. For the gradient estimate, using the same computational strategy as above we have
\begin{align}
|\nabla G_{\bar \mu}(y)|=\int_{B_{39/20}(p)} |\nabla G_x(y)|d{\bar \mu}(x)\le C\int_{B_{39/20}(p)} d_x^{1-n}(y)d{\bar \mu}(x)\le Cr^{-3}.
\end{align}
By noting that  $b^{-2}=G_{\bar \mu}$, then $2 b^{-3}|\nabla b|=|\nabla b^{-2}|=|\nabla G_{\bar \mu}|\le Cr^{-3}$. By the upper bound estimate of $b$, we have $|\nabla b|\le C$.  For the gradient lower bound,  for any fixed unit vector $v\in T_yM$, we have
\begin{align}
\langle\nabla G_{\bar \mu}(y),v\rangle&=\int_{B_{39/20}(p)} \langle\nabla G_x(y),v\rangle d{\bar \mu}(x)=\int_{B_{rR}(y)}\langle\nabla G_x(y),v\rangle d{\bar \mu}(x)+\int_{M\setminus B_{rR}(y)}\langle \nabla G_x(y),v\rangle d{\bar \mu}(x).
\end{align}
By the gradient upper bound estimate $|\nabla G_x(y)|\le Cd_x^{1-n}(y)$, we have
\begin{align}
|\int_{M\setminus B_{rR}(y)}\langle\nabla G_x(y),v\rangle d{\bar \mu}(x)|\le \int_{M\setminus B_{rR}(y)}|\nabla G_x(y)|d{\bar \mu}(x)\le CR^{-3}r^{-3} .
\end{align}
On the other hand, for fixed $R$ we have by the Green function estimates of Lemma \ref{l:green_function_cone_point} that if $\delta\leq \delta(R,n,\rv)$, then there is a unit vector $v_y\in T_yM$ such that
\begin{align}
\int_{B_{rR}(y)}\langle \nabla G_x(y),v_y\rangle d{\bar \mu}(x)\ge \int_{B_{10r}(y)}r^{1-n}c(n)d{\bar \mu}(x)\ge C(n,B)r^{-3}.
\end{align}
Therefore, combining the estimates above, we have
\begin{align}
\langle\nabla G_{\bar \mu}(y),v_y\rangle\,\ge \int_{B_{rR}(y)}\langle\nabla G_x(y),v_y\rangle d{\bar \mu}(x)-\Big|\int_{M\setminus B_{rR}(y)}\langle \nabla G_x(y),v_y\rangle d{\bar \mu}(x)\Big|\,\ge \, C_0(B,n)r^{-3}-C_1(B,n)R^{-3}r^{-3}.
\end{align}
Choosing $R=R(\rv,n,B)$ large enough we conclude $
\langle\nabla G_{\bar \mu}(y),v_y\rangle\,\ge\, \frac{C(B,n)}{2}r^{-3}$.  In particular, this gives us the estimate $|\nabla G_{\bar \mu}|(y)\ge C r^{-3}$ and hence the desired estimate $|\nabla b|(y)\ge C$ for $y\in \cN_{10^{-5}}\cap B_{39/20}(p)$.  This finishes the proof of the Lemma.\qed

\subsection{Harmonic Function Estimates}

In this short subsection we record several estimates about harmonic functions.  We begin with some basic computations, mainly for the convenience of the reader.  Our list is the following:\\

\begin{lemma}\label{l:harmonic_computations}
	Let $\Delta u=0$ be a harmonic function on an open set.  Then the following hold:
	\begin{enumerate}
		\item $\Delta \nabla_i u = R_{ia}\nabla^a u$.
		\item $\Delta |\nabla u|^2 = 2|\nabla^2 u|^2 +2{\Ric}(\nabla u,\nabla u)$.
		\item $\Delta \nabla_i\nabla_j u = \big(\nabla_jR_{ia}-\nabla_a R_{ij}+\nabla_i R_{ja}\big)\nabla^a u-2R^{\,a\;\;b}_{\;\;i\;\;j}\nabla_a\nabla_b u + R_{ia}\nabla^a\nabla_j u + R_{aj}\nabla_i\nabla^a u$.
		\item $\Delta |\nabla^2 u|^2 = 2|\nabla^3 u|^2 + 2\big\langle\big(\nabla_jR_{ia}-\nabla_a R_{ij}+\nabla_i R_{ja}\big)\nabla^a u, \nabla^i\nabla^j u\big\rangle -4\text{Rm}(\nabla^2 u,\nabla^2 u) + 4{\Ric}(\nabla^a\nabla u,\nabla_a\nabla u)$.
	\end{enumerate}
\end{lemma}
\vspace{.5cm}

Let us now record the following hessian estimates which will play a role in subsequent subsections:\\

\begin{lemma}\label{l:harmonic_estimates}
    Let $(M^n,g,p)$ satisfy $\Vol(B_1(p))>\rv>0$ with $\Ric\geq -\delta$.  Then if $u:B_4(p)\to \dR^{n-4}$ is a $\delta'$-splitting function, then the following hessian estimate holds for any $x\in B_2(p)$:
    \begin{align}
	\int_{B_2(x)} d_x^{2-n}|\nabla^2 u|^2(z)\,dz \leq C(n,\rv)\, ,
    \end{align}
    where $d_x(y)=d(x,y)$ is the distance function.
\end{lemma}
\begin{remark}
Let us make the important observation that the constant in the above estimate cannot be taken to be small, simply bounded.	
\end{remark}
\begin{proof}
	Using Lemma \ref{l:harmonic_computations} we have the formula
	\begin{align}\label{e:harmonic_est:1}
		\frac{1}{2}\Delta|\nabla u|^2 + \delta|\nabla u|^2 \geq |\nabla^2 u|^2\, .
	\end{align}
	Now for $x\in B_1$ let $G_x(y)$ be the Green's function.  Recall from Section \ref{ss:splitting:Greens_standard} that for $y\in B_4$ we have the estimate
	\begin{align}
		C(n,\rv)^{-1}d_x^{2-n}(y)\leq G_x(y)\leq C(n,\rv)d_x^{2-n}\, .
	\end{align}
	Now recall as in \cite{ChC1} that we may build a cutoff function $\phi$ such that $\phi\equiv 1$ on $B_2(x)$, $\phi\equiv 0$ outside $B_{5/2}(x)$ and such that $|\nabla\phi|,|\Delta\phi|\leq C(n)$.  Therefore let us multiply both sides of \eqref{e:harmonic_est:1} by $\phi G_x$ and integrate in order to compute
\begin{align}
	\int_{B_2(x)} d_x^{2-n}|\nabla^2 u|^2(z)\,dz&\leq C(n,\rv)\int_{B_3(x)} \phi G_x|\nabla^2 u|^2(z)\,dz\notag\\
	&\leq C(n,\rv)\int_{B_3(x)} \phi G_x \Delta(|\nabla u|^2-1)(z)\,dz + C(n,\rv)\delta\int_{B_3(x)} \phi G_x |\nabla u|^2(z)\,dz\, ,\notag\\
	&\leq C(n,\rv)\left( 1-|\nabla u|^2(x) + \delta+\int_{B_3(x)} \Big(\big(\Delta \phi G_x+2\langle\nabla\phi,\nabla G_x\rangle\big)\big(|\nabla u|^2-1\big)\Big)\right)(z)\,dz\, ,\notag\\
	&\leq C(n,\rv)\, ,
\end{align}
as claimed.
\end{proof}
\vspace{.5cm}

By using the Green's function estimates of Lemma \ref{l:green_function_distant_function} we can argue as above to prove the following estimate:

\begin{lemma}\label{l:hessian_L2_green}
Let $(M^n,g,p)$ satisfy $\Vol(B_1(p))>\rv>0$. If  the assumptions (S1)-(S4) of \eqref{e:assumptions_splitting} hold with $\delta,\delta'<\delta(n,B,\rv)$ and $0<\tau\le \tau(n)$, then
\begin{align}
\int_{ B_{19/10}(p)} b^{-2}|\nabla^2u|^2(z)\,dz\le C(n,B,\rv).
\end{align}

\end{lemma}

\begin{proof}
 By Bochner formula, we have 
\begin{align}
\frac{1}{2}\Delta|\nabla u|^2 + \delta|\nabla u|^2 \geq |\nabla^2 u|^2\, .
\end{align}
Choose a cut-off function $\phi$ such that $\phi\equiv 1$ on $B_{19/10}(p)$ and $\phi\equiv 0$ outside $B_{39/20}(p)$ such that $|\nabla \phi|+|\Delta \phi|\le C(n)$. Then we can compute 
\begin{align}
\phi^2 |\nabla^2u|^2&\le \frac{1}{2}\phi^2\Delta|\nabla u|^2 + \delta \phi^2|\nabla u|^2\\
&=\frac{1}{2}\Delta (\phi^2|\nabla u|^2)-2\langle \nabla \phi^2,\nabla |\nabla u|^2\rangle -\Delta \phi^2 |\nabla u|^2+\delta \phi^2|\nabla u|^2\\
&\le \frac{1}{2}\Delta (\phi^2|\nabla u|^2)+\frac{1}{2}\phi^2 |\nabla^2u|^2+32 |\nabla\phi|^2|\nabla u|^2-\Delta \phi^2 |\nabla u|^2+\delta \phi^2|\nabla u|^2\\
&\le \frac{1}{2}\Delta (\phi^2|\nabla u|^2)+\frac{1}{2}\phi^2 |\nabla^2u|^2+C(n)(|\nabla \phi|^2+|\Delta\phi|+\phi^2)|\nabla u|^2\,.
\end{align}
This implies that 
\begin{align}
\phi^2 |\nabla^2u|^2&\le\Delta (\phi^2|\nabla u|^2)+C(n)(|\nabla \phi|^2+|\Delta\phi|+\phi^2)|\nabla u|^2\,.
\end{align}

Noting that $b^{-2}=G_{\bar \mu}$ and $\sup |\nabla u|\le 1$ we have 
\begin{align}\label{e:b2nabla21}
\int_{ B_{19/10}(p)} b^{-2}|\nabla^2u|^2(z)\,dz&\le \int_{M} G_{\bar \mu} |\nabla ^2u|^2\phi^2(z)\,dz\\
&\le \int_{M} G_{\bar \mu} \Big(\Delta (\phi^2|\nabla u|^2)+C(n)(|\nabla \phi|^2+|\Delta\phi|+\phi^2)|\nabla u|^2\Big)(z)\,dz\\
&\le \int_{M} G_{\bar \mu} \Delta (\phi^2|\nabla u|^2)(z)\,dz+ C(n) \int_{B_{39/20}(p)} G_{\bar \mu}(y)\,dy .
\end{align}
By Lemma \ref{l:green_function_distant_function} we have $B_{39/20}(p)\subset \{b\le C(n,B,\rv)\}$. 
Therefore by Lemma \ref{l:green_formula} we get

\begin{align}\label{e:b2nabla22}
\int_{M} G_{\bar \mu} \Delta (\phi^2|\nabla u|^2)(z)\,dz&\le {\bar \mu} [(\phi^2|\nabla u|^2]+C(n,B,\rv)\int_{B_{2}(p)}|\Delta (\phi^2|\nabla u|^2)|(z)\,dz\\
&\le  C(n,B,\rv)+{\bar \mu} [(\phi^2|\nabla u|^2]\\
&\le C(n,B,\rv)+C(n){\bar \mu}[1]\\
&\le C(n,B,\rv)+C(n){\bar \mu}(B_{39/20}(p))\le C(n,B,\rv).
\end{align}
On the other hand, noting that $G_{\bar \mu}(y)=\int_{B_{39/20}(p)} G_x(y)d{\bar \mu}(x)$ and the estimate of $G_x(y)\le C(n,\rv)d(x,y)^{2-n}$ in Lemma \ref{l:green_function_cone_point}, we get 
\begin{align}\label{e:b2nabla23}
\int_{B_{39/20}(p)} G_{\bar \mu}(y)dy&=\int_{B_{39/20}(p)}\int_{B_{39/20}(p)} G_x(y)d{\bar \mu}(x)dy\le C(n,\rv)\int_{B_{39/20}(p)}\int_{B_{39/20}(p)} d(x,y)^{2-n}dyd{\bar \mu}(x)\\
&\le C(n,\rv) \int_{B_{39/20}(p)}d{\bar \mu}(x)\le C(n,\rv,B).
\end{align}
Combining \eqref{e:b2nabla21}, \eqref{e:b2nabla22} and \eqref{e:b2nabla23} we get the desired estimates. 
\end{proof}

\vspace{.5cm}

\subsection{Scale Invariant Hessian Estimates}

In this subsection we prove that a splitting function on a neck region continues to have scale invariantly small hessian on all scales within the neck region.  This is not generally true for a harmonic function which is not living on a neck region, and the estimate will play an important role in our analysis in subsequent sections (particularly on making certain errors small, instead of just bounded, which will be crucial).  In the next subsection we will prove a much stronger estimate which basically says this error is not just small but summably small, at least when averaged over $\cC$, which is what is required to prove the gradient estimate in \eqref{e:splitting_annular}.

The main result of this subsection is the following:

\begin{theorem}[Pointwise Scale Invariant Hessian Estimate]\label{t:splitting_neck_pointwise_hessian}
Let $(M^n,g,p)$ satisfy $\Vol(B_1(p))>\rv>0$.  Then for every $\epsilon>0$ and $\tau<\tau(n)$ if $\delta,\delta'\leq \delta(n,\epsilon,\tau,\rv)$ is such that assumptions (S1)-(S4) of \eqref{e:assumptions_splitting} hold, then we have for every $x\in\cC$ and $r_x<r<1$ that
\begin{align}
r^2\fint_{B_r(x)} |\nabla^2 u|^2(z)\,dz < \epsilon.
\end{align}
\end{theorem}
\begin{remark}
Let us observe a corollary:  If $y\in \cN_{10^{-5}}$ with $2r=d(y,\cC)\approx r_h\approx b$, then for $\delta,\delta'\leq \delta(n,\epsilon,\tau,\rv)$ we can use elliptic estimates to obtain $r|\nabla^2 u|<\epsilon$ on $B_r(y)$.	We can rephrase this as the pointwise hessian estimate $|\nabla^2 u|(y)<\epsilon r_h^{-1}(y)$ on $\cN_{10^{-5}}$.
\end{remark}
\vspace{.5cm}

The strategy for the proof of the above will be by contradiction.  Assuming the result is false, we will take a sequence of $(\delta,\tau)$-neck regions with $\delta'$-splitting functions such that $\delta,\delta'\to 0$, but for which the conclusions of the theorem presumably fail.  We will prove in subsection \ref{sss:blowups_proof_pointwise_hessian} that the limit harmonic function is actually linear, and therefore has vanishing hessian.  One must be able to conclude from this the almost vanishing of the hessian for the original sequence of harmonic functions.  Thus a key technical result will be shown in subsection \ref{sss:weak_sobolev_convergence}, which will prove, among other things, a form of $H^{2}$-convergence for the sequence of harmonic functions.\\

\subsubsection{Weak Sobolev Convergence of Limiting Harmonic Functions}\label{sss:weak_sobolev_convergence}

The main goal of this subsection is to prove the following convergence result, which will play an important role in subsequent sections.

\begin{theorem}\label{t:weak_sobolev_convergence}
Let $(M^n_j,g_j,p_j)\to (X,d,p)$ satisfy $|\Ric_j|\leq n-1$ and $\Vol(B_1(p_j))>\rv>0$.  Assume for some $R,C>0$ that $u_j:B_R(p_j)\to \dR$ is a sequence of harmonic functions with $|u_j|\leq C$.  Then there exists harmonic $u:B_R(p)\to \dR$ such that after possibly passing to a subsequence
\begin{enumerate}
	\item $u_j\to u$ uniformly on compact subsets of $B_R$.
	\item For each $x_j\to x\in X$ with $B_{2r}(x_j)\subseteq B_{R}(p_j)$ we have $|\nabla u|(x) \leq \liminf_j||\nabla u_j||_{L^\infty(B_{r}(x_j))}$.
	\item For each $0<p<4$ we have that $|\nabla^2 u_j|^p\to |\nabla^2 u|^p$ in measure.  In fact, for each $x_j\to x\in X$ with $B_{2r}(x_j)\subseteq B_{R}(p_j)$ we have that $\int_{B_r(x_j)}|\nabla^2 u_j|^p(z)\,dz\to \int_{B_r(x)}|\nabla^2 u|^p(z)\,dz$.
\end{enumerate}
\end{theorem}
\begin{remark}
Recall that a harmonic function on $X$ is in the RCD sense, which is to say $\int\langle \nabla u,\nabla v\rangle(z)\,dz = 0$ for every compactly supported lipschitz function $v$.	
\end{remark}

\begin{remark}
For clarity, we mean in the above that $\int_{B_r(x)}|\nabla^2 u|^p(z)\,dz\equiv \int_{\cR(X)\cap B_r(x)}|\nabla^2 u|^p(z)\,dz$, where $\cR(X)$ is the regular set of $X$.  Though we will not explicitly say it, we will see as a consequence of the proof that for $p<4$ we have that $|\nabla^2 u|^p$ cannot have a distributional term on $\cS(X)$, so that this is reasonable.
\end{remark}

\begin{proof}
    Points $(1)$ and $(2)$ are standard, see for instance \cite{ChC2}, and only require a lower bound on the Ricci curvature.  In a little detail, if the $u_j$ are harmonic with the uniform bounds $|u_j|\leq C$, then by using standard Cheng-Yau estimates (or one of several others), we can conclude that for every $r<R$ there exists $C_r>0$ such that
	\begin{align}
		|\nabla u_j|\leq C_r\, ,
	\end{align}
	uniformly.  Using a standard Ascoli argument one can conclude the existence of $u:B_R\to \dR$ such that $u_j\to u$ uniformly such that $(2)$ holds.  \\
	
	To see that $u$ is harmonic one in fact only needs the lower Ricci bound, however to vastly simplify the argument and avoid bringing optimal transport techniques into the paper we will give a simplified argument which exploits the two sided Ricci bound.  Indeed, if $\cR(X)\subseteq X$ is the regular set, then we know that $u_i\to u$ in $C^{2,\alpha}$ on $\cR(X)$.  In particular, $u$ is harmonic on $\cR(X)$.  By \cite{CheegerNaber_Codimensionfour} we also know that $\cS(X) = X\setminus \cR(X)$ is of codimension four, which also gives us that $\cS(X)$ is a capacity zero set.  Since $u$ is lipschitz, and in particular $H^1$, we then know by standard Dirichlet form theory \cite{FOT_Dirichlet} that $u$ is therefore harmonic on all of $X$. \\
	
	The proof of $(3)$ requires heavily the two sided bound on the Ricci curvature and the codimension four nature of the singular set as in \cite{CheegerNaber_Codimensionfour}.  Using Theorem 7.20 of \cite{CheegerNaber_Codimensionfour}, see also Theorem \ref{t:prelim:harmonic_estimates}, one can conclude for every $p<4$ and $r<R$ that there exists $C^p_r$ such that
	\begin{align}\label{e:weak_sobolev:1}
		\int_{B_r(x_j)}|\nabla^2 u_j|^p(z)\,dz < C^p_r\, .
	\end{align}

	Let us define the effective regular and singular sets
	\begin{align}
	&\cR_r\equiv \{x: r_h(x)>r\}\, ,\notag\\
	&\cS_r\equiv \{x: r_h(x)\leq r\}\, .	
	\end{align}
	Recall that the singular set $\cS(X)\equiv \cS_0(X)$ is a set of codimension four, and in fact we have the much stronger volume estimate
    \begin{align}
    \Vol\big(\cS_r\cap B_R\big) \equiv \Vol\big(B_r\big\{r_h<r\big\}\cap B_R\big)\leq C_\epsilon(n,\rv,R)\,r^{4-\epsilon}\, .	
    \end{align}

Now on the regular set $\cR_r(X)$ we have, due to our lower bound on the harmonic radius, for every $\alpha<1$ that $u_j\to u$ in $C^{2,\alpha}$ uniformly on compact subsets.  Now let $x_j\in M_j\to x\in X$ such that $B_{2r}(x_j)\subseteq B_R$, and let us fix $p<4$.  For every $s>0$ it holds from the $C^{2,\alpha}$ convergence we just commented on that
\begin{align}
	\int_{\cR_s\cap B_r(x_j)} |\nabla^2 u_j|^p(z)\,dz \to \int_{\cR_s(X)\cap B_r(x)} |\nabla^2 u|^p(z)\,dz\, .
\end{align}
To finish the proof, we need to see that $\int_{\cS_s\cap B_r(x_j)} |\nabla^2 u_j|^p(z)\,dz\to 0$ as $s\to 0$, uniformly in the $u_j$.  Therefore let us choose $p<p'<4$ so that $\int_{B_r(x_j)}|\nabla^2 u_j|^{p'}(z)\,dz<C'$ uniformly as in \eqref{e:weak_sobolev:1}.  Then a Cauchy inequality gives us that
\begin{align}
\int_{\cS_s\cap B_r(x_j)} |\nabla^2 u_j|^p(z)\,dz &\leq \Vol(\cS_s\cap B_r)^{p'-p}\big(\int_{\cS_s\cap B_r(x_j)} |\nabla^2 u_j|^{p'}(z)\,dz\big)^{\frac{p}{p'}} 	\notag\\
&\leq C(n,R,p,p')s^{(4-\epsilon)(p'-p)}\, ,
\end{align}
which shows in particular that $\int_{\cS_s\cap B_r(x_j)} |\nabla^2 u_j|^p(z)\,dz\to 0$ as $s\to 0$, as claimed.
\end{proof}
\vspace{.5cm}

\subsubsection{Blow ups and the Proof of Theorem \ref{t:splitting_neck_pointwise_hessian}}\label{sss:blowups_proof_pointwise_hessian}

Let us now use the tools of the previous subsection in order to finish the proof of Theorem \ref{t:splitting_neck_pointwise_hessian}.  Thus, let us assume the result is false, so that for some $\epsilon>0$ there exists $\delta_j,\delta'_j\to 0$ together with a sequence $M^n_j$ with $\Vol(B_1(p_j))>\rv>0$ and $\cN_j\subseteq B_2(p_j)$ $(\delta_j,\tau)$-neck regions with $u_j:B_4\to \dR^{n-4}$ $\delta'_j$-splitting maps such that for some $x_j\in \cC_j$ and $r_{x_j}<r_j<1$ we have that
\begin{align}
r_j^2\fint_{B_{r_j}(x_j)}|\nabla^2 u_j|^2(z)\,dz\geq \epsilon\, .	
\end{align}

Let us begin by observing that $r_j\to 0$, since for any $r>0$ fixed we always have the trivial estimate
\begin{align}
r^2\fint_{B_{r}(x_j)}|\nabla^2 u_j|^2(z)\,dz\leq r^{2-n}\int_{B_{1}(x_j)}|\nabla^2 u_j|^2(z)\,dz\leq r^{2-n}\delta'_j\to 0\, .	 
\end{align}

Now let us consider the rescaled spaces $\tilde M_j = r_j^{-2}M_j$ with $\tilde u_j=r_j^{-1}\big(u_j-u_j(x_j)\big): B_{r_j^{-1}}(x_j)\to \dR^{n-4}$ the renormalized maps.  Notice that we have the equality
\begin{align}
	r_j^2\fint_{B_{r_j}(x_j)}|\nabla^2 u_j|^2(z)\,dz = \fint_{\tilde B_1(\tilde x_j)} |\nabla^2 \tilde u_j|^2(\tilde{z})\,d\tilde{z}\, .
\end{align}

Now by the definition of neck regions we have that
\begin{align}
\tilde M_j \to \dR^{n-4}\times C(S^3/\Gamma)\, ,	
\end{align}
and by Theorem \ref{t:weak_sobolev_convergence} we have that
\begin{align}
\tilde u_j\to \tilde u:	\dR^{n-4}\times C(S^3/\Gamma)\to \dR^{n-4}\, ,
\end{align}
where $\tilde u$ is a harmonic map for which $|\nabla \tilde u|\leq \liminf |\nabla \tilde u_j| = 1$.  However, we can lift $\tilde u$ to a $\Gamma$-invariant harmonic map $\tilde u:\dR^{n}\to \dR^{n-4}$, and since we have a global gradient bound we can apply Liouville's theorem to see we have that $\tilde u$ is a linear map.  In particular,
\begin{align}
|\nabla^2 \tilde u| \equiv 0\, .	
\end{align}
By again applying Theorem \ref{t:weak_sobolev_convergence} we have that
\begin{align}
r_j^2\fint_{B_{r_j}(x_j)}|\nabla^2 u_j|^2(z)\,dz = \fint_{\tilde B_1(\tilde x_j)} |\nabla^2 \tilde u_j|^2(\tilde{z})\,d\tilde{z}\to \fint_{\tilde B_1} |\nabla^2 \tilde u|^2(\tilde{z})\,d\tilde{z} = 0\, ,
\end{align}
which is our desired contradiction. $\square$ \\

\subsection{Summable Hessian Estimates}

In this section we prove a vastly refined version of Theorem \ref{t:splitting_neck_pointwise_hessian} which tells us that the $L^1$-norm of the hessian is scale invariantly summable.  As we will see, this is the key estimate in the proof of Theorem \ref{t:neck_region}.\\

\begin{theorem}[$L^1$ Summable Hessian Estimate]\label{t:splitting_neck_summable_hessian}
Let $(M^n,g,p)$ satisfy $\Vol(B_1(p))>\rv>0$.  Then for every $\epsilon,B>0$ and $\tau<\tau(n)$ if $\delta,\delta'<\delta(n,\epsilon,B,\tau,\rv)$ is such that assumptions (S1)-(S4) of \eqref{e:assumptions_splitting} hold, then we have the estimate
\begin{align}
\int_{\cN_{10^{-3}}\cap B_{18/10}(p)} r_h^{-3} |\nabla^2 u|(z)\,dz < \epsilon\, ,
\end{align}
where $\cN_{10^{-3}}$ is the enlarged neck region as in Definition \ref{d:neck}.
\end{theorem}
\vspace{.5cm}

To prove the above we will first discuss a new superconvexity estimate in Section \ref{sss:superconvexity}.  Analysis of the derived ode together with the ${\bar \mu}$-Green's estimates of Section \ref{ss:splitting:Greens_standard}  will then be applied in order to conclude Theorem \ref{t:splitting_neck_summable_hessian} itself.

\subsubsection{The Superconvexity Equation}\label{sss:superconvexity}

In this subsection we prove a new superconvexity estimate on the $L^1$ Hessian with respect to the ${\bar \mu}$-Green's distance.  The estimates of this section are the cornerstone of Theorem \ref{t:splitting_neck_summable_hessian} and hence Theorem \ref{t:splitting_neck_bilipschitz}.  Let us begin by stating our main result for this subsection:\\

\begin{proposition}\label{p:superconvexity}
Let $(M^n,g,p)$ satisfy $\Vol(B_1(p))>\rv>0$.  Then for every $\epsilon,B>0$ and $\tau<\tau(n)$ if assumptions (S1)-(S4) of \eqref{e:assumptions_splitting} hold with $\delta,\delta'\leq \delta(n,\epsilon,B,\tau,\rv)$, then for $F(r)=r^{-3}\int_{b=r}\sqrt{|\nabla^2u|^2+{\epsilon}^2}\phi_\cN|\nabla b|(z)\,dz$ and $H(r)=rF(r)$, we have the following evolution equation
\begin{align}\nonumber
r^2H''(r)+rH'(r)-H(r)\ge e(r)=\int_{b=r}\left(-\sqrt{|\nabla^2 u|^2+{\epsilon}^2}~|\Delta \phi_\cN|-c(n)|Rm|\cdot|\nabla^2u|\phi_\cN+U_1+U_2\right)|\nabla b|^{-1}(z)\,dz,
\end{align}
where $U_1=2\langle\nabla\phi_\cN,\nabla \sqrt{|\nabla^2 u|^2+{\epsilon}^2}\rangle$ and $U_2=\Big(\nabla_j(R_{ia}\nabla^au)-\nabla_a (R_{ij}\nabla^au)+\nabla_i (R_{ja}\nabla^au)\Big)\ast \nabla^i\nabla^ju \frac{\phi_\cN}{\sqrt{|\nabla^2 u|^2+{\epsilon}^2}}$.
\end{proposition}
\begin{remark}
The function $\phi_\cN$ is the neck cutoff function from Lemma \ref{l:neck:cutoff}.
\end{remark}

\begin{remark}
Let us remark that $H(r)$ represents the scale invariant $L^1$ norm of the hessian over the surface $b=r$.  Our main goal will then be to get a Dini integral estimate on $H(r)$, see Theorem \ref{t:superconvexity_estimate}.
\end{remark}

\begin{proof}
Using Lemma \ref{l:harmonic_computations} we can compute
\begin{align}
\frac{1}{2}\Delta (|\nabla^2u|^2+{\epsilon}^2)=|\nabla^3u|^2-\langle \Rm\ast \nabla^2u,\nabla^2u\rangle+\nabla({\Ric}(\nabla u,\cdot))\ast \nabla^2u,
\end{align}
where $\langle\Rm\ast \nabla^2u,\nabla^2u\rangle=2R^{\,a\;\;b}_{\;\;i\;\;j}\nabla_a\nabla_b u \cdot \nabla^i\nabla^ju$
and
$\nabla({\Ric}(\nabla u,\cdot))\ast \nabla^2u=\Big(\nabla_j(R_{ia}\nabla^au)-\nabla_a (R_{ij}\nabla^au)+\nabla_i (R_{ja}\nabla^au)\Big)\nabla^i\nabla^ju$.  Let $f=\phi_\cN \sqrt{|\nabla^2 u|^2+{\epsilon}^2}$. Then
\begin{align}
\Delta f&=\Delta \phi_\cN ~\sqrt{|\nabla^2 u|^2+{\epsilon}^2}+\Delta \sqrt{|\nabla^2 u|^2+{\epsilon}^2}\, \phi_\cN+2\langle\nabla\phi_\cN,\nabla \sqrt{|\nabla^2 u|^2+{\epsilon}^2}\rangle\\ \nonumber
&\ge -\sqrt{|\nabla^2 u|^2+{\epsilon}^2}~|\Delta \phi_\cN|-c(n)|Rm|\cdot|\nabla^2u|\phi_\cN+U_1+U_2,
\end{align}
where $U_1$ and $U_2$ are defined as in the lemma.
Thus if $H(r)=rF(r)=r^{-2}\int_{b=r}f|\nabla b|(z)\,dz$ with $b^{-2}=G_{\bar \mu}$, then by the Green formula and the computation of Lemma \ref{l:ode_Green_function}, we have
\begin{align}
H''(r)+\frac{1}{r}H'-\frac{1}{r^2}H&=r^{-2}\int_{b=r}\Delta f~|\nabla b|^{-1}(z)\,dz\\ \nonumber
&\ge r^{-2}\int_{b=r}\left(-\sqrt{|\nabla^2 u|^2+{\epsilon}^2}~|\Delta \phi_\cN|-c(n)|Rm|\cdot|\nabla^2u|\phi_\cN+U_1+U_2\right)|\nabla b|^{-1}(z)\,dz.
\end{align}
Thus we finish the proof.
\end{proof}

\vspace{.5cm}

Of course, in order to exploit the above superconvexity we will want to apply a maximum principle in order to obtain bounds for $H$.  To accomplish this we first need to find special solutions of the above ode to compare to:\\

\begin{lemma}[Special solution]\label{l:special_solution3}
Consider on the interval $(0,R)$ the differential equation
$$h''(r)+\frac{1}{r}h'(r)-\frac{1}{r^2}h(r)={g(r)},$$
for some bounded $g(r)$ which may change signs. Then we have a solution $h(r)$ such that
\begin{align}
\int_0^R\frac{h(r)}{r}dr=-\int_0^Rsg(s)ds+\frac{1}{2R}\int_0^Rs^2g(s)ds,
\end{align}
with $h(R)=-\frac{1}{2R}\int_0^Rs^2g(s)ds$ and $h(0)=0$.
\end{lemma}
\begin{proof}Let us define $h(r)=\frac{1}{2}\left(-r\int_{r}^R{g}ds-r^{-1}\int_{0}^rs^2{g}ds\right)$, which one can easily check is a solution.  We will see that $h(r)$ satisfies the desired estimates. Indeed,
\begin{align}
-2\int_0^R r^{-1}h(r)dr&=\int_0^R\int_{r}^R{g(s)}dsdr+\int_0^Rr^{-2}\int_{0}^rs^2{g}dsdr\\ \nonumber
&=\int_0^Rg(s)\int_0^sdr ds+\int_0^Rs^2g(s)\int_s^Rr^{-2}drds\\ \nonumber
&=2\int_0^Rsg(s)ds-\frac{1}{R}\int_0^Rs^2g(s)ds
\end{align}
Hence we finish the proof.
\end{proof}

\vspace{.5cm}

We will now use the above to provide estimates on the $L^1$ norm of the $\delta'$-splitting functions on a neck region:\\

\begin{theorem}\label{t:superconvexity_estimate}
Let $(M^n,g,p)$ satisfy $\Vol(B_1(p))>\rv>0$.  Then for every $\epsilon,B>0$ and $\tau<\tau(n)$ if $\delta,\delta'\leq \delta(n,\epsilon,B,\tau,\rv)>0$ is such that assumptions (S1)-(S4) of \eqref{e:assumptions_splitting} hold and if $F(r)=r^{-3}\int_{b=r}\sqrt{|\nabla^2u|^2+{\epsilon}^2}\phi_\cN|\nabla b|(z)\,dz$ with $H(r)\equiv r F(r)$, then we have the Dini estimate:
\begin{align}
\int_0^\infty \frac{1}{r}H(r)dr\leq C(n,\rv,B,\tau)\epsilon\, .
\end{align}
\end{theorem}
\begin{proof}
By the definition of $\phi_\cN$ and Lemma \ref{l:green_function_distant_function}, there exists $R=R(n,\rv,B)$ such that $\phi_\cN\equiv 0$ on $\{b\ge R\}$. Therefore we have
$$\int_0^\infty \frac{1}{r}H(r)\,dr=\int_0^R\frac{1}{r}H(r)\,dr\, .$$
Now let us begin by using Lemma \ref{l:special_solution3} to choose a special solution $h_1(r)$ and $h_2(r)$ of our ode such that
\begin{align}
h_1''+\frac{1}{r}h_1'-\frac{1}{r^2}h_1= r^{-2}\int_{b=r}U_1 |\nabla b|^{-1}(z)\,dz=r^{-2}\int_{b=r}2\langle\nabla\phi_\cN,\nabla \sqrt{|\nabla^2 u|^2+{\epsilon}^2}\rangle|\nabla b|^{-1}(z)\,dz,\notag\\
\end{align}
 with $h_1(R)=-\frac{1}{R} \int_{b\le R}\langle\nabla\phi_\cN,\nabla \sqrt{|\nabla^2 u|^2+{\epsilon}^2}\rangle(z)\,dz $, $h_1(0)=0$ and
\begin{align}\label{e:H_1_integral}
\int_0^R\frac{h_1(r)}{r}dr=-2\int_{b\le R}b^{-1}\langle\nabla\phi_\cN,\nabla \sqrt{|\nabla^2 u|^2+{\epsilon}^2}\rangle(z)\,dz+\frac{1}{R}\int_{b\le R}\langle\nabla\phi_\cN,\nabla \sqrt{|\nabla^2 u|^2+{\epsilon}^2}\rangle(z)\,dz.
\end{align}
On the other hand, we have special solution $h_2(r)$ of
\begin{align}
h_2''+\frac{1}{r}h_2'-\frac{1}{r^2}h_2= r^{-2}\int_{b=r}U_2 |\nabla b|^{-1}(z)\,dz,
\end{align}
such that $h_2(R)=-\frac{1}{2R} \int_{b\le R}U_2(z)\,dz $, $h_2(0)=0$ and
\begin{align}\label{e:H_2_integral}
\int_0^R\frac{h_2(r)}{r}dr=-\int_{b\le R}b^{-1}U_2(z)\,dz+\frac{1}{2R}\int_{b\le R}U_2(z)\,dz,
\end{align}
where $U_2=\Big(\nabla_j(R_{ia}\nabla^au)-\nabla_a (R_{ij}\nabla^au)+\nabla_i (R_{ja}\nabla^au)\Big)\nabla^i\nabla^ju\Big) \frac{\phi_\cN}{\sqrt{|\nabla^2 u|^2+{\epsilon}^2}}$ as in Proposition \ref{p:superconvexity}.

Now let us use Lemma \ref{l:special_solution3} one last time to produce a special solution $h_3(r)$ of
\begin{align}
h_3''+\frac{1}{r}h_3'-\frac{1}{r^2}h_3= r^{-2}\int_{b=r}\left(-\sqrt{|\nabla^2 u|^2+{\epsilon}^2}~|\Delta \phi_\cN|-c(n)|Rm|\cdot|\nabla^2u|\phi_\cN\right) |\nabla b|^{-1}(z)\,dz,
\end{align}
such that $h_3(R)=-\frac{1}{2R} \int_{b\le 2}\left(-\sqrt{|\nabla^2 u|^2+{\epsilon}^2}~|\Delta \phi_\cN|-c(n)|Rm|\cdot |\nabla^2u|\phi_\cN\right)(z)\,dz $, $h_3(0)=0$ and
\begin{align}\label{e:H_3_integral}
\int_0^R\frac{h_3(r)}{r}dr=\int_{b\le R}\left(-b^{-1}+\frac{1}{2R}\right)\left(-\sqrt{|\nabla^2 u|^2+{\epsilon}^2}~|\Delta \phi_\cN|-c(n)|Rm||\nabla^2u|\phi_\cN\right)(z)\,dz\, .
\end{align}
Finally let us choose $h_4(r)$ such that
\begin{align}
h_4''+\frac{1}{r}h_4'-\frac{1}{r^2}h_4=0
\end{align}
with $h_4(0)=0$ and $h_4(R)=|h_1(R)|+|h_2(R)|+|h_3(R)|+|H(R)|=|h_1(R)|+|h_2(R)|+|h_3(R)|\equiv A$.  Note that we can explicitly solve $h_4(r)=Ar/R$. On the other hand, we now have that
\begin{align}\label{e:H1H2H3H4}
\left(H-\sum_{i=1}^4 h_i\right)''(r)+\frac{1}{r}\left(H-\sum_{i=1}^4 h_i\right)'(r)-\frac{1}{r^2}\left(H-\sum_{i=1}^4 h_i\right)(r)\ge 0,
\end{align}
with $\left(H-\sum_{i=1}^4 h_i\right)(0)=0$ and $\left(H-\sum_{i=1}^4 h_i\right)(R)\le 0$.  Therefore, 
we have $H- \sum_{i=1}^4 h_i\le 0$. Actually, assume $(H- \sum_{i=1}^4 h_i)(r_0)=\max_{0\le r\le R} (H- \sum_{i=1}^4 h_i)(r)>0$, we can compute
\begin{align}
\left(H-\sum_{i=1}^4 h_i\right)''(r_0)+\frac{1}{r_0}\left(H-\sum_{i=1}^4 h_i\right)'(r_0)-\frac{1}{r^2_0}\left(H-\sum_{i=1}^4 h_i\right)(r_0)\le -\frac{1}{r^2_0}\left(H-\sum_{i=1}^4 h_i\right)(r_0)<0,
\end{align}
which contradicts to \eqref{e:H1H2H3H4}. Thus 
we have $H- \sum_{i=1}^4 h_i\le 0$.  Observe also that $H\ge 0$ while $h_i$ may change signs.  We may now estimate
\begin{align}\label{e:H_bound}
\int_0^R \frac{H(r)}{r}dr\le \sum_{i=1}^4\int_0^R \frac{h_i(r)}{r}dr.
\end{align}
Therefore, to estimate the Dini integral for $H(r)$, we only need to control the Dini integral of each $h_i$.  Beginning with $h_1$, we have by (\ref{e:H_1_integral}) and integrating by parts that
\begin{align}
\int_0^R\frac{h_1(r)}{r}dr&=-2\int_{b\le R}b^{-1}\langle\nabla\phi_\cN,\nabla \sqrt{|\nabla^2 u|^2+{\epsilon}^2}\rangle(z)\,dz +\frac{1}{R}\int_{b\le R}\langle\nabla\phi_\cN,\nabla \sqrt{|\nabla^2 u|^2+{\epsilon}^2}\rangle(z)\,dz\\ \nonumber
&\le C(n,\rv,B)\left(\int_{b\le R}\Big(b^{-2}|\nabla \phi_\cN|+b^{-1}|\Delta\phi_\cN|\Big) \sqrt{|\nabla^2 u|^2+{\epsilon}^2}\right)(z)\,dz\, ,
\end{align}
where we have used that $\phi_\cN\equiv 0$ on $\{b=R\}$ and $|\nabla b|\le C(n,\rv,B)$ on $\text{ supp}\, \phi_\cN$.  Now using Theorem \ref{t:splitting_neck_pointwise_hessian} and Lemma \ref{l:green_function_distant_function} with $\delta,\delta'\leq \delta(\epsilon,n,\rv)\le \epsilon^5$ we can obtain the apriori estimate $|\nabla^2u|\le \epsilon b^{-1}$ on $\text{supp }\phi_\cN$.  Plugging this in gives us
\begin{align}
\int_0^R\frac{h_1(r)}{r}dr
&\le C(n)\epsilon\left(\int_{b\le R}\big(b^{-3}|\nabla \phi_\cN|+b^{-2}|\Delta\phi_\cN|\big)\right)(z)\,dz\, .
\end{align}
To estimate $h_2$ let us see by (\ref{e:H_2_integral}), integration by parts, and our Ricci curvature bound that we have
\begin{align}
\int_0^R\frac{h_2(r)}{r}dr&=-\int_{b\le R}b^{-1}U_2(z)\,dz+\frac{1}{2R}\int_{b\le R}U_2(z)\,dz\\ \nonumber
&\le C(n)\delta\left( \int_{b\le R}b^{-1} |\nabla \frac{\nabla^2u}{\sqrt{|\nabla^2 u|^2+{\epsilon}^2}}|\phi_\cN(z)\,dz+\int_{b\le R} b^{-2}|\nabla b| \phi_\cN(z)\,dz+\int_{b\le R} b^{-1}|\nabla \phi_\cN|(z)\,dz\right)\\ \nonumber
&\le  C(n)\left( \frac{\delta}{\epsilon}\int_{b\le R}b^{-1} |\nabla^3u|\phi_\cN(z)\,dz+\delta\int_{b\le R} b^{-2}|\nabla b| \phi_\cN(z)\,dz+\delta\int_{b\le R} b^{-1}|\nabla \phi_\cN|(z)\,dz\right)\, ,
\end{align}
where $\delta$ comes from the Ricci bound $|\Ric|\le \delta$.
Finally for $h_3$, by (\ref{e:H_3_integral}) and H\"older inequality we have
\begin{align}
\int_0^R\frac{h_3(r)}{r}dr&=\int_{b\le R}\left(-b^{-1}+\frac{1}{4}\right)\left(-\sqrt{|\nabla^2 u|^2+{\epsilon}^2}~|\Delta \phi_\cN|-c(n)|\Rm|\cdot |\nabla^2u|\phi_\cN\right)(z)\,dz\\ \nonumber
&\le C(n)\epsilon\int_{b\le R} b^{-2}|\Delta\phi_\cN|(z)\,dz+C\left(\int_{b\le R}|\Rm|^2\phi_\cN(z)\,dz\right)^{1/2}\left(\int_{b\le R}b^{-2}|\nabla^2u|^2\phi_\cN(z)\,dz\right)^{1/2}.
\end{align}
By applying the $L^2$ curvature estimate on neck region in Theorem \ref{t:L2_neck} and the hessian $L^2$ estimate in Lemma \ref{l:hessian_L2_green} with $\epsilon$ we therefore get
\begin{align}
\int_0^R\frac{h_3(r)}{r}dr
&\le C(n)\epsilon\int_{b\le R} b^{-2}|\Delta\phi_\cN|(z)\,dz+C(n,B)\epsilon\, .
\end{align}
To estimate $h_4$, it suffices to estimate $A=|h_1(R)|+|h_2(R)|+|h_3(R)|$. Actually, from the formulas of $h_1(R), h_2(R), h_3(R)$ and the formulas \eqref{e:H_1_integral}, \eqref{e:H_2_integral}, \eqref{e:H_3_integral} and the above estimates, we have already proved 
\begin{align}
|A|\le  C(n,B,\rv)\epsilon\left(1+\int_{b\le R}b^{-2}|\Delta \phi_\cN|(z)\,dz+\int_{b\le R}b^{-3}|\nabla\phi_\cN|(z)\,dz+ \int_{b\le R}b^{-1}|\nabla^3u|\phi_\cN(z)\,dz\right)\, .
\end{align}

Combining all these with \eqref{e:H_bound} we get
\begin{align}\label{e:L1_summable}
\int_0^R \frac{H(r)}{r}dr\le C(n,B,\rv)\epsilon\left(1+\int_{b\le R}b^{-2}|\Delta \phi_\cN|(z)\,dz+\int_{b\le R}b^{-3}|\nabla\phi_\cN|(z)\,dz+ \int_{b\le R}b^{-1}|\nabla^3u|\phi_\cN(z)\,dz\right)\, .
\end{align}

Hence, to get the Dini estimate it suffices to estimate each term of (\ref{e:L1_summable}).  In fact, from the definition of $\phi_\cN$ in Lemma \ref{l:neck:cutoff}, the definition of neck region, and by the comparison of $b$ and $d_\cC$ in Lemma \ref{l:green_function_distant_function}, one can easily get
\begin{align}\nonumber
\int_{b\le R}b^{-2}|\Delta \phi_\cN|(z)\,dz&+\int_{b\le R}b^{-3}|\nabla\phi_\cN|(z)\,dz\\ \nonumber
&\le C(n,B,\rv,\tau)\left(\int_{d_\cC\le 3}d_\cC^{-2}|\Delta \phi_\cN|(z)\,dz+\int_{d_\cC\le 3}d_\cC^{-3}|\nabla\phi_\cN|(z)\,dz\right)
\\ \nonumber
&\le C(n,B,\rv,\tau)\left(\sum_{x\in\cC\cap B_{19/10}(p)}\int_{B_{r_x}(x)}\left(d_\cC^{-2}|\Delta \phi_\cN|+d_\cC^{-3}|\nabla \phi_\cN|\right)(z)\,dz+\int_{\text{ supp }\phi_\cN}d_\cC^{-3}(z)\,dz\right)\\ \nonumber
&\le  C(n,B,\rv,\tau)\left(\sum_{x\in \cap B_{19/10}(p)}r_x^{n-4}+\int_{\{d_\cC\le 3\}\cap \text{ supp }\phi_\cN}d_\cC^{-3}(z)\,dz\right)\\ \nonumber
&\le C(n,B,\rv,\tau)\mu_\cN(B_{19/10}(p))+C(n,B,\rv)\int_{\{d_\cC\le 3\}\cap \text{ supp }\phi_\cN}d_\cC^{-3}(z)\,dz\\ \label{e:dcC-30}
&\le C(n,B,\rv,\tau)+C(n,B,\rv,\tau)\int_{\{d_\cC\le 3\}\cap \text{ supp }\phi_\cN}d_\cC^{-3}(z)\,dz\, ,
\end{align}
where the term $\int_{\text{ supp }\phi_\cN}d_\cC^{-3}(z)\,dz$ comes from the derivative of $\phi_1$ in Lemma \ref{l:neck:cutoff}.  We can compute that 
\begin{align}\label{e:dcC-31}
\int_{\text{ supp }\phi_\cN}d_\cC^{-3}(z)\,dz&\le \sum_{i=0}^\infty\int_{\text{ supp }\phi_\cN \cap \{2^{-i}\le d(x,\cC)\le 2^{-i+1}\}}d_\cC^{-3}(z)\,dz\\
&\le \sum_{i=0}^\infty 2^{3i}\Vol(\text{ supp }\phi_\cN \cap \{2^{-i}\le d(x,\cC)\le 2^{-i+1}\}).
\end{align}
Choose a Vitali covering $\{B_{2^{-i+3}}(x_\alpha),x_\alpha\in \cC,\alpha=1,\cdots, N\}$ of $\text{ supp }\phi_\cN \cap \{2^{-i}\le d(x,\cC)\le 2^{-i+1}\}$ such that $\{B_{2^{-i}}(x_\alpha),x_\alpha\in \cC,\alpha=1,\cdots,N\}$ are pairwise disjoint. We get 
\begin{align}\label{e:dcC-321}
 C(n,B)\ge \mu(B_{19/10}(p))\ge \mu(\text{ supp }\phi_\cN)\ge\sum_{\alpha=1}^N\mu(B_{2^{-i}}(x_\alpha))\ge N B^{-1}2^{-(n-4)i}6^{-(n-4)}.
\end{align}
We should remark that the last inequality in \eqref{e:dcC-321} always holds even $2^{-i}\le r_{x_\alpha}$ or $2^{-i}\ge d\big(x_\alpha,\partial B_2(p)\big)$. Actually, for the first case ($2^{-i}\le r_{x_\alpha}$) the inequality is obvious from the definition of $\mu$, for the second case ($2^{-i}\ge d\big(x_\alpha,\partial B_2(p)\big)$  ) we can find $B_{2^{-i}/3}(\tilde{x}_\alpha)\subset B_2(p)\cap B_{2^{-i}}(x_\alpha)$ with $\tilde{x}_\alpha\in \cC$.
Therefore, we have $N\le C(n,B)2^{(n-4)i}$. Thus 
\begin{align}\label{e:dcC-32}
\Vol(\text{ supp }\phi_\cN \cap \{2^{-i}\le d(x,\cC)\le 2^{-i+1}\})\le \sum_{\alpha=1}^N\Vol(B_{2^{-i+3}}(x_\alpha)) \le C(n,B)2^{(n-4)i} C(n)2^{n(-i+3)}\le C(n,B)2^{-4i}.
\end{align}
Combining \eqref{e:dcC-31} and \eqref{e:dcC-32} we get 
\begin{align}
\int_{\text{ supp }\phi_\cN}d_\cC^{-3}(z)\,dz\le C(n,B).
\end{align}
Putting in \eqref{e:dcC-30} we arrive at 
\begin{align}
\int_{b\le R}b^{-2}|\Delta \phi_\cN|(z)\,dz+\int_{b\le R}b^{-3}|\nabla\phi_\cN|(z)\,dz\le C(n,\rv,B,\tau).
\end{align}

On the other hand, by using the $L^p$ estimates arising from the codimension four theorem of \cite{CheegerNaber_Codimensionfour}, we know from Theorem \ref{t:prelim:harmonic_estimates} that $\int_{B_{10}(p)}|\nabla^3u|^{3/2}\le C(n,\rv)$. Therefore, by the H\"older inequality we have
\begin{align}
\int_{b\le R}b^{-1}|\nabla^3u|\phi_\cN(z)\,dz\le \left(\int_{b\le R}b^{-3}\phi_\cN(z)\,dz\right)^{1/3}\left(\int_{b\le R}|\nabla^3u|^{3/2}\phi_\cN(z)\,dz\right)^{2/3}\le C(n,B,\rv,\tau)\, .
\end{align}

Thus we have proved the Lemma.
\end{proof}
\vspace{.5cm}

\subsubsection{Proof of Theorem \ref{t:splitting_neck_summable_hessian}}

Let us first observe that
\begin{align}
	\int_{\cN_{10^{-3}}\cap B_{18/10}(p)} r_h^{-3} |\nabla^2 u|(z)\,dz<\int_{M} r_h^{-3} |\nabla^2 u|\phi_\cN(z)\,dz\, ,
\end{align}
where $\phi_\cN$ is the cutoff function from Lemma \ref{l:neck:cutoff}.  In order to finish the proof let us apply Theorem \ref{t:superconvexity_estimate} with $\epsilon'$ in order to conclude
\begin{align}
	\int_0^\infty r^{-3}\int_{b=r} |\nabla^2 u|\phi_\cN(z)\,dz dr<\int_0^\infty r^{-3}\int_{b=r} \sqrt{|\nabla^2 u|^2+{\epsilon'}^2}\phi_\cN(z)\,dzdr  < C(n,B,\rv,\tau)\epsilon'\, .
\end{align}
By using the coarea formula we can rewrite this as
\begin{align}
\int_{M} b^{-3}\,|\nabla b|\, |\nabla^2 u| \phi_\cN(z)\,dz < \epsilon'\, .	
\end{align}
Finally, using the Green's function estimates $|\nabla b|> C^{-1}(n,B)$ and $b\leq C(n,B)\,r_h$ from Lemma \ref{l:green_function_distant_function} this gives us the estimate
\begin{align}
	\int_M r_h^{-3} |\nabla^2 u|\phi_\cN(z)\,dz < C(n,B,\rv,\tau)\epsilon'<\epsilon\, ,
\end{align}
where in the last line we have chosen $\epsilon'<C(n,B,\rv,\tau)^{-1}\epsilon$ in order to finish the proof. $\square$
\vspace{.5cm}

\subsection{Gradient Estimates}

In this subsection we provide the main gradient estimate required in the proof of Theorem \ref{t:splitting_neck_bilipschitz}.  Precisely, we have the following:\\

\begin{theorem}\label{t:neck_gradient}
Let $(M^n,g,p)$ satisfy $\Vol(B_1(p))>\rv>0$.  For every $\epsilon,B>0$ and $\tau<\tau(n)$, if (S1)-(S4) of \eqref{e:assumptions_splitting} hold with $\delta,\delta'<\delta(n,\epsilon,B,\tau,\rv)$, then for every $x\in\cC\cap B_1(p)$ and $r_x<r<2/3$ we have that
	\begin{align}
		\fint_{B_r(x)}\big|\langle \nabla u_a,\nabla u_b\rangle-\delta_{ab}\big|(z)\,dz < \epsilon +C(n,\rv)\int_{W_x} r_h^{1-n} |\nabla^2 u|(z)\,dz\, .
	\end{align}
\end{theorem}
\vspace{.5cm}

The proof of the above will rely on first proving a telescoping type estimate in Lemma \ref{l:telescope_estimate}.  This estimate will be applied iteratively along a wedge region to then conclude the above.\\

\subsubsection{Telescoping Estimate}

Our main result of this subsection is a telecoping type estimate which allows us to compare the average of a function along different nearby balls.  Precisely, we have the following:\\

\begin{lemma}\label{l:telescope_estimate}
	Let $(M^n,g,p)$ satisfy $\Ric\geq -(n-1)$, and let $f:B_4(p)\to \dR$ be a $H^1$-function.  Then if $B_{r_1}(x_1),B_{r_2}(x_2),B_{r_3}(x_3)\subseteq B_2$ satisfy $r_1,r_2>10^{-n}r_3$ and are such that $B_{2r_1}(x_1),B_{2r_2}(x_2)\subseteq B_{r_3}(x_3)$, then we have the estimate
	\begin{align}
		\Big|\fint_{B_{r_1}(x_1)}f(z)\,dz - \fint_{B_{r_2}(x_2)} f(z)\,dz\Big| < C(n)r_3\fint_{B_{r_3}(x_3)}|\nabla f|(z)\,dz
	\end{align}
\end{lemma}
\begin{remark}
In practice, we will be applying this estimate to $f=|\langle\nabla u^a,\nabla u^b\rangle-\delta^{ab}|$, where $u:B_4\to \dR^{n-4}$ is one of our splitting functions.
\end{remark}
\begin{proof}
	The proof is an application of the Poincare inequality for spaces with lower Ricci curvature bounds.  To see this, let us choose $r'\equiv r_3-\max\{r_1,r_2\}>0$, so that $B_{r_1}(x_1),B_{r_2}(x_2)\subseteq B_{r'}(x_3)\subseteq B_{r_3}(x_3)$.  Now we can estimate
\begin{align}
	\Big|\fint_{B_{r_1}(x_1)}f(z)\,dz - \fint_{B_{r'}(x_3)} f(z)\,dz\Big| &\leq \Big|\fint_{B_{r_1}(x_1)}\big|f(x) - \fint_{B_{r'}(x_3)} f(z)\,dz\big|\Big|\,dx\, ,\notag\\
	&\leq \Vol(B_{r_1}(x_1))^{-1}\Big|\int_{B_{r'}(x_3)}\big|f(x) - \fint_{B_{r'}(x_3)} f(z)\,dz\big|\Big|\,dx\, ,\notag\\
	&\leq C(n)\Big|\fint_{B_{r'}(x_3)}\big|f(x) - \fint_{B_{r'}(x_3)} f(z)\,dz\big|\Big|dx\leq C(n)r_3\fint_{B_{r_3}(x_3)}|\nabla f|(z)\,dz\, ,
\end{align}
where in the last line we have used volume monotonicity and the Poincare inequality \cite{SY_Redbook} using that $r_3-r'>10^{-n}r_3$.  Applying the same argument to $B_{r_2}(x_2)$ gives the estimate
\begin{align}
	\Big|\fint_{B_{r_2}(x_2)}f(z)\,dz - \fint_{B_{r'}(x_3)} f(z)\,dz\Big| \leq C(n)r_3\fint_{B_{r_3}(x_3)}|\nabla f|(z)\,dz\, ,
\end{align}
and by combining these we obtain the desired result.
\end{proof}
\vspace{.5cm}

\subsubsection{Proof of Gradient Estimate of Theorem \ref{t:neck_gradient}}

The proof will come in two steps.  To begin with, let us fix $\gamma=1-\frac{1}{10}$ and consider the sequence of scales $s_\alpha\equiv \gamma^{\alpha}$, and for $x\in \cC$ let us consider a sequence of points $x_1,x_2,\ldots$ such that $x_\alpha\in \partial B_{s_\alpha}(x)$ is the maximizer of $d(\cdot,\cC)|_{\partial B_{s_\alpha}(x)}$ 
Note that by our Gromov-Hausdorff condition we have that $1+\delta>d(x_\alpha,\cC) s_\alpha^{-1}>1-\delta$, at least for $s_\alpha>r_x/10$.  Let us observe the ball inclusions given by
\begin{align}
	B_{10^{-1}s_\alpha}(x_\alpha), B_{10^{-1}s_{\alpha+1}}(x_{\alpha+1})\subseteq B_{4^{-1}s_\alpha}(x_\alpha)\subseteq W_x\, .
\end{align}

In particular, if we consider the function $f=|\langle\nabla u^a,\nabla u^b\rangle-\delta^{ab}|$ then we can apply Lemma \ref{l:telescope_estimate} in order to conclude
\begin{align}
\Big|\fint_{B_{20^{-1}s_\alpha}(x_\alpha)}f(z)\,dz - \fint_{B_{20^{-1}s_{\alpha+1}}(x_{\alpha+1})} f (z)\,dz\Big|&\leq C(n)s_\alpha\fint_{B_{4^{-1}s_\alpha}(x_\alpha)}|\nabla f|(z)\,dz\, ,\notag\\
&\leq C(n)s_\alpha\fint_{B_{4^{-1}s_\alpha}(x_\alpha)}|\nabla^2 u|(z)\,dz\, ,\notag\\
&\leq C(n)s_\alpha^{1-n}\int_{W_x\cap A_{s_\alpha/10,10s_\alpha}(x)} |\nabla^2 u|(z)\,dz\, .
\end{align}
In particular, for every $s_\alpha>r_x/10$ we can sum in order to obtain the estimate
\begin{align}
	\Big|\fint_{B_{20^{-1}s_\alpha}(x_\alpha)}f(z)\,dz - \fint_{B_{20^{-1}}(x_{0})} f(z)\,dz \Big|&\leq C(n)\sum_{s_\alpha>r_x/10} s_\alpha^{1-n}\int_{W_x\cap A_{s_\alpha/10,10s_\alpha}(x)} |\nabla^2 u|(z)\,dz\notag\\
	&\leq C(n)\int_{W_x} r_h^{1-n}|\nabla^2 u|(z)\,dz\, .
\end{align}
Since $u:B_2(p)\to \dR^{n-4}$ is a $\delta$-splitting we have for $\delta<\delta(\epsilon)$ that
\begin{align}
	\Big|\fint_{B_{20^{-1}}(x_{0})} f(z)\,dz \Big|&\leq 10^{-2}\epsilon\, ,
\end{align}
which by combining with our previous estimates therefore gives us
\begin{align}
	\fint_{B_{20^{-1}s_\alpha}(x_\alpha)}\big|\langle\nabla u^a,\nabla u^b\rangle-\delta^{ab}\big|(z)\,dz&\leq 10^{-2}\epsilon+ C(n)\int_{W_x} r_h^{1-n}|\nabla^2 u|(z)\,dz\, ,
\end{align}
for all $s_\alpha>r_x/10$.\\

Now in order to finish the proof we will again apply Lemma \ref{l:telescope_estimate}.  Indeed, let us fix $r>r_x$ and pick $s_\alpha$ such that $|s_\alpha-r|$ is minimized.  Then by applying Lemma \ref{l:telescope_estimate} we have the estimate
\begin{align}
\Big|\fint_{B_{20^{-1}s_\alpha}(x_\alpha)}&\big|\langle\nabla u^a,\nabla u^b\rangle-\delta^{ab}\big|(z)\,dz - \fint_{B_{r}(x)}\big|\langle\nabla u^a,\nabla u^b\rangle-\delta^{ab}\big|(z)\,dz\, \Big|\, ,\notag\\
&\leq C(n) r\fint_{B_{2r}(x)}|\nabla f|(z)\,dz\leq C(n) r\fint_{B_{2r}(x)}|\nabla^2 u|(z)\,dz\, .
\end{align}
But if we apply the scale invariant estimate of Theorem \ref{t:splitting_neck_pointwise_hessian} with $\epsilon'=C(n)^{-2}10^{-2}\epsilon^2$ then we obtain the estimate
\begin{align}
\Big|\fint_{B_{20^{-1}s_\alpha}(x_\alpha)}&\big|\langle\nabla u^a,\nabla u^b\rangle-\delta^{ab}\big|(z)\,dz - \fint_{B_{r}(x)}\big|\langle\nabla u^a,\nabla u^b\rangle-\delta^{ab}\big|(z)\,dz\, \Big|\leq 10^{-1}\epsilon\, .
\end{align}
Combining this with our previous estimates gives us our desired estimate
\begin{align}
	\fint_{B_{r}(x)}\big|\langle\nabla u^a,\nabla u^b\rangle-\delta^{ab}\big|(z)\,dz&\leq \epsilon+ C(n)\int_{W_x} r_h^{1-n}|\nabla^2 u|(z)\,dz\, ,
\end{align}
as claimed. $\square$
\vspace{.5cm}

\subsection{$\mu$-Splitting Estimates and Proof of Theorem \ref{t:splitting_neck_bilipschitz}}

In this subsection we now build the ingredients required to finish the proof of Theorem \ref{t:splitting_neck_bilipschitz}.

\subsubsection{$\mu$-Maximal Function Estimates}

We begin in this subsection by studying a maximal function defined with respect to the $\mu$-measure.  This will be used in the next subsection to define for us our bilipschitz set of center points.  Our main result for this subsection is the following:\\

\begin{proposition}\label{p:splitting_neck_region}
Let $(M^n,g,p)$ satisfy $\Vol(B_1(p))>\rv>0$.  Then for every $\epsilon,B>0$ and $\tau<\tau(n)$ if $\delta,\delta'\leq \delta(n,\epsilon,B,\tau,\rv)$ is such that assumptions (S1)-(S4) of \eqref{e:assumptions_splitting} hold, then if we define the $\mu$-supported maximal function $m_u:\cC\to \dR^+$ by 
$$m(x) \equiv \sup_{r_x\leq r\leq 2/3}\Bigg(\fint_{B_{r}(x)}\big|\langle\nabla u_a,\nabla u_b\rangle - \delta_{ab}\big|(z)\,dz + r^2\fint_{B_{r}(x)}\big|\nabla^2 u\big|^2(z)\,dz\Bigg)\, ,$$
then we have the estimate
\begin{align}\label{e:splitting_annular}
\int_{B_1(p)} m(x)\,d\mu(x) < \epsilon \, .
\end{align}
\end{proposition}
\begin{remark}
The hessian part of the estimate is a consequence of Theorem \ref{t:splitting_neck_pointwise_hessian}.  Therefore it is the gradient aspect of the above estimate which is the challenging part of the Theorem.	
\end{remark}

\vspace{.5cm}

In order to prove the Proposition we need a lemma which relates integrals over wedge regions to integrals over neck regions.  Precisely:\\

\begin{lemma}\label{l:wedge_integral}
	Let $(M^n,g,p)$ satisfy $\Vol(B_1(p))>\rv>0$ with $\tau<\tau(n)$ and $\delta,\delta'\leq \delta(n,B,\tau,\rv)$ such that assumptions (S1)-(S4) of \eqref{e:assumptions_splitting} hold.  Then for $f:B_2\to \dR^+$ a nonnegative function then we can estimate
	
\begin{align}
\int_{B_1(p)}\left(\int_{W_x\cap B_{18/10}(p)} r_h^{4-n}(y) f(y)\,dy\right)\, d\mu(x) \leq C(n,B)\int_{\cN_{10^{-3}}\cap B_{18/10}(p)} f(x)\,dx\, .	
\end{align}
\end{lemma}
\begin{remark}
This is essentially an effective version of a Fubini type theorem.	
\end{remark}
\begin{proof}
We have by Lemma \ref{l:neck:distance_harm_rad} that $r_h(y)\ge C(n) d_\cC(y)$. Hence, we only need to estimate $\int_{B_1}\Big(\int_{W_x} d_\cC^{4-n} f\Big)\, d\mu$.
Moreover, since $W_x\subset \cN_{10^{-2}}$ we can assume without loss that $\text{ supp }f\subset \cN_{10^{-3}}\cap B_{18/10}(p)$.
Let $s_i=10^{-i}$. We rewrite the integral in the following manner.
\begin{align}
\int_{B_1(p)}\left(\int_{W_x} d_\cC^{4-n}(y) f(y) \,dy\right)d\mu(x) &=\int_{B_1(p)}\left(\sum_{i=0}^\infty \int_{W_{x}\cap \{s_{i+1}\le d_\cC(y)\le s_i\}} d_\cC^{4-n}(y) f(y)\,dy\right) d\mu(x) \\ \nonumber
&\le \sum_{i=0}^\infty s_{i+1}^{4-n}\int_{B_1(p)}\left(\int_{W_x\cap \{s_{i+1}\le d_\cC(y)\le s_i\}} f\, dy\right)d\mu(x).
\end{align}
Let $\chi_i(x,y)$ be a function on $M\times M$ such that $\chi_i(x,y)=1$ if $y\in W_{x}\cap \{s_{i+1}\le d_\cC(y)\le s_i\}$, otherwise $\chi_i(x,y)=0$. By Fubini theorem and Ahlfors assumption on the measure $\mu$, we have 
\begin{align}
 \int_{B_1(p)}\int_{W_{x}\cap \{s_{i+1}\le d_\cC\le s_i\}} f(y)dy\, d\mu(x)&=\int_{B_1(p)}\int_{\{s_{i+1}\le d_\cC(y)\le s_i\}} \chi_i(x,y)f(y)\,dy\, d\mu(x)\\ \nonumber
& =\int_{\{s_{i+1}\le d_\cC\le s_i\}}\left( \int_{B_1(p)}\chi_i(x,y)\, d\mu(x)\right)f(y)\,dy\\ \nonumber
 &\le \int_{\{s_{i+1}\le d_\cC\le s_i\}}f(y)\mu(B_{3s_i}(y)\cap B_1(p))\, dy\le C(n,B)s_i^{n-4} \int_{\{s_{i+1}\le d_\cC\le s_i\}}f(y)\,dy\, . 
\end{align}

Combining with the previous computation we obtain
\begin{align}
\int_{B_1(p)}\int_{W_x} d_\cC^{4-n} f(z)\,dz d\mu 
&\le \sum_{i=0}^\infty C(n,B)\int_{\{s_{i+1}\le d_\cC\le s_i\}}f(y)\,dy\le C\int_{\cN_{10^{-3}}\cap B_{18/10}(p)}f(z)\,dz.
\end{align}
\end{proof}
\vspace{.5cm}

We are now in a position to prove Proposition \ref{p:splitting_neck_region}:\\

\begin{proof}[Proof of Proposition \ref{p:splitting_neck_region}]
Let us consider the two maximal functions
\begin{align}
	&m_1(x) \equiv \sup_{r_x\leq r\leq 2/3}\fint_{B_{r}(x)}\big|\langle\nabla u_a,\nabla u_b\rangle - \delta_{ab}\big|(z)\,dz\, ,\notag\\
	&m_2(x) \equiv \sup_{r_x\leq r\leq 2/3}r^2\fint_{B_{r}(x)}\big|\nabla^2 u\big|^2(z)\,dz\, .
\end{align}
We clearly have the estimate $m(x)\leq m_1(x)+m_2(x)$, so it is enough to estimate each of these individually.\\

Let us begin by estimating $m_2(x)$.  By applying the scale invariant estimates of Theorem \ref{t:splitting_neck_pointwise_hessian} with $\frac{1}{2}B^{-1}\epsilon$ we have for every $x\in \cC$ and $r_x\leq r\leq 2/3$ the pointwise estimate
\begin{align}
	r^2\fint_{B_{r}(x)}\big|\nabla^2 u\big|^2(z)\,dz < B^{-1}\frac{\epsilon}{2}\, ,
\end{align}
and hence we clearly have the much weaker estimate
\begin{align}
	\int_{B_1(p)}m_2(x)\,d\mu(x) = \int_{B_1(p)} \left(\sup_{r_x\leq r\leq 2/3}r^2\fint_{B_{r}(x)}\big|\nabla^2 u\big|^2(z)\,dz\right)d\mu(x) < \frac{\epsilon}{2}\, .
\end{align}

The more challenging part of the Proposition is therefore the gradient estimate of $m_1(x)$.  To deal with this we apply Theorem \ref{t:neck_gradient} with $\frac{1}{4}B^{-1}\epsilon$ in order to see that for every $x\in \cC$ and $r_x\leq r\leq 2/3$ we have the pointwise estimate
\begin{align}
		\fint_{B_r(x)}\big|\langle \nabla u_a,\nabla u_b\rangle-\delta_{ab}\big| (z)\,dz< \frac{1}{4}B^{-1}\epsilon +C(n,\rv)\int_{W_x\cap B_{18/10}(p)} r_h^{1-n} |\nabla^2 u|(z)\,dz\, ,
\end{align}
which of course gives the maximal function estimate
\begin{align}
		m_1(x)< \frac{1}{4}B^{-1}\epsilon +C(n,\rv)\int_{W_x\cap B_{18/10}(p)} r_h^{1-n} |\nabla^2 u|(z)\,dz\, .
\end{align}
Now if we apply Lemma \ref{l:wedge_integral} we can then estimate
\begin{align}
	\int_{B_1} m_1(x)\,d\mu &< \frac{1}{4}\epsilon + C(n,\rv)\int_{B_1}\Big(\int_{W_x\cap B_{18/10}(p)} r_h^{1-n} |\nabla^2 u|(z)\,dz\Big)\,d\mu(x)\, \notag\\
	&\leq \frac{1}{4}\epsilon + C(n,\rv,B)\int_{\cN_{10^{-3}}\cap B_{18/10}(p)} r_h^{-3} |\nabla^2 u|(z)\,dz\, .
\end{align}
To finish the proof let us now apply Theorem \ref{t:splitting_neck_summable_hessian} with constant $\epsilon'=C(n,\rv,B)^{-1}\frac{1}{4}\epsilon$ in order to conclude
\begin{align}
	\int_{B_1} m_1(x)\,d\mu &\leq \frac{1}{4}\epsilon + C(n,\rv,B)\int_{\cN_{10^{-3}}\cap B_{18/10}(p)} r_h^{-3} |\nabla^2 u|\, \notag\\
	&\leq \frac{1}{4}\epsilon+\frac{1}{4}\epsilon = \frac{1}{2}\epsilon\, .
\end{align}
Combining this with our estimates on $m_2(x)$ finishes the proof.
\end{proof}

\vspace{.5cm}

\subsubsection{Proof of Theorem \ref{t:splitting_neck_bilipschitz} for Smooth Manifolds.}

Now equipped with Proposition \ref{p:splitting_neck_region} we are now in a position to finish the proof of Theorem \ref{t:splitting_neck_bilipschitz} in the case when $X=M$ is a smooth manifold.  Indeed, let us pick $\delta,\delta'<\delta(n,\rv,\tau,B,\epsilon)$ such that Proposition \ref{p:splitting_neck_region} holds with $(\epsilon')^2$, where $\epsilon'=\epsilon'(n,B,\epsilon)$ will be chosen later.  That is, we have the estimate
\begin{align}
\int_{B_1} m(x)\,d\mu(x) < (\epsilon')^2 \, ,
\end{align}
where
$$m(x) \equiv \sup_{r_x\leq r\leq 2/3}\Bigg(\fint_{B_{r}(x)}\big|\langle\nabla u_a,\nabla u_b\rangle - \delta_{ab}\big|(z)\,dz + r^2\fint_{B_{r}(x)}\big|\nabla^2 u\big|^2(z)\,dz\Bigg)\, .$$

Let us now define $\cC_\epsilon\subseteq \cC\cap B_1$ to be the subset such that
\begin{align}
\cC_\epsilon \equiv \{x\in \cC\cap B_1: m(x)<\epsilon'\}\, .	
\end{align}
Notice that because of our integral estimate on $m(x)$ a weak $L^1$ argument gives us for $\epsilon'<\epsilon'(n,B,\epsilon)$ that
\begin{align}
\mu\big(\big(\cC\cap B_1\big)\setminus \cC_\epsilon\big)<\epsilon\, .	
\end{align}

Let us now see that $u:\cC_\epsilon\to \dR^{n-4}$ is a $1+\epsilon$-bilipschitz map.  For any $\epsilon''>0$ let us first show that if $\epsilon'\le \epsilon'(n,\rv,\tau,B,\epsilon'')$ and $\delta,\delta'\le \delta(n,\rv,\tau,B,\epsilon',\epsilon'')$ then for any $x\in \cC_\epsilon$ and $r_x\le r\le 1$, 
\begin{align}\label{e:epsilon''GHmap}
u:B_{\epsilon'' r}\big(\cL_{x,2r}\big)\to \dR^{n-4}\, ,
\end{align}
is itself a $10\epsilon'' r$-GH map. To see this, note that since $B_{r\delta^{-1}}(x)$ is $\delta r$ close to $\mathbb{R}^{n-4}\times C(S^3/\Gamma)$, by Theorem  \ref{t:splitting_function} if $\delta\le \delta'(n,\rv,\epsilon')$ then there exists a $\epsilon'$-splitting map $\tilde{u}: B_{2r}(x)\to \mathbb{R}^{n-4}\subset \mathbb{R}^{n-4}\times C(S^3/\Gamma)$ such that 
\begin{align}
\tilde{u}\circ \iota: \mathbb{R}^{n-4}\times \{y_c\} \cap B_{2r}(0^{n-4},y_c)\to \mathbb{R}^{n-4}
\end{align}
is a $\epsilon'r$-GH map, where $\iota$ is $\delta r$-GH map from $\mathbb{R}^{n-4}\times C(S^3/\Gamma)$ in Definition \ref{d:neck}. This implies
\begin{align}
\tilde{u}: \cL_{x,2r}\to \mathbb{R}^{n-4},
\end{align}
is an $2\epsilon' r$-GH map. Let us assume without loss of generality that $u(x)=\tilde{u}(x)$. We claim that for any $\epsilon''>0$ there exists $(n-4)\times (n-4)$ matrix $O\in O(n-4)$ such that if $\delta,\delta'\le \delta(n,\rv,\tau,B,\epsilon'')$ and $\epsilon'\le \epsilon'(n,\rv,B,\epsilon'')$ we have 
\begin{align}
\sup_{y\in B_{2r}(x)}|Ou-\tilde{u}|(y)\le \epsilon'' r/100.
\end{align}
Actually this follows directly from a contradiction argument by noting that if we choose a contradiction sequence then the both $u$ and $\tilde{u}$ would converge uniformly in compact subset and $W^{1,2}$-sense (see Theorem \ref{t:weak_sobolev_convergence} ) to the linear coordinates of $\mathbb{R}^{n-4}\subset \mathbb{R}^{n-4}\times C(S^3/\Gamma)$. Therefore, we get 
$Ou: \cL_{x,2r}\to \mathbb{R}^{n-4},$ is an $\epsilon'' r/10$-GH map which implies $u: \cL_{x,2r}\to \mathbb{R}^{n-4}$ is an $\epsilon'' r/10$-GH map. This implies 
\begin{align}
u: B_{\epsilon'' r}(\cL_{x,2r})\to \mathbb{R}^{n-4}
\end{align}
is an $10\epsilon'' r$-GH map which proves \eqref{e:epsilon''GHmap}.  \\
Now we can use \eqref{e:epsilon''GHmap} to finish our proof. For any $x\in \cC_\epsilon$ and $y\in \cC$ we have $d(x,y)\ge \tau^{2}r_x$.  Let $r=\tau^{-2}d(x,y)\ge r_x$.
By cone-splitting theorem \ref{t:cone_splitting} and $\epsilon$-regularity Theorem \ref{t:eps_reg}, if $\delta\le \delta(n,\rv,\epsilon'',B,\rv)$ we have $\cC\cap B_{2r}(x)\subset B_{\epsilon'' r}(\cL_{x,2r})$ which is also stated in Definition \ref{d:neck}.  Therefore, we have 
\begin{align}
\Big||u(x)-u(y)|-d(x,y)\Big|\le 10\epsilon'' r.
\end{align}
Since $r=\tau^{-2}d(x,y)$, if $\epsilon''\le \epsilon''(\epsilon,\tau)=\tau^2 \epsilon/10$ we conclude that 
\begin{align}
\Big||u(x)-u(y)|-d(x,y)\Big|\le \epsilon d(x,y).
\end{align}
Hence we have proved $u: \cC_\epsilon\to \mathbb{R}^{n-4}$ is a $1+\epsilon$-bilipschitz map if $\delta',\delta\le \delta(n,\rv,\tau,B,\epsilon)$, which finishes the proof of Theorem \ref{t:splitting_neck_bilipschitz}. $\square$

\vspace{.5cm}

\subsubsection{Proof of Theorem \ref{t:splitting_neck_bilipschitz} for General Limit Spaces}

We wish to now use a relatively straight forward limiting argument based on the neck approximation of Theorem \ref{t:neck_approximation} in order to conclude the proof of Theorem \ref{t:splitting_neck_bilipschitz} for general limit spaces $M^n_j\to X$.

Thus let $\epsilon'<<\epsilon$, which will later be fixed so that $\epsilon'=\epsilon'(n,\tau,\epsilon)$, and let $2\delta,2\delta'$ be the corresponding constants so that Theorem \ref{t:splitting_neck_bilipschitz} holds on smooth manifolds with $\epsilon'$.  If $\cN\subseteq B_2(p)$ is a $(\delta,\tau)$-neck region with $u:B_4(p)\to \dR^{n-4}$ a $\delta'$-splitting function, then by applying Theorem \ref{t:neck_approximation} let us consider a sequence of neck regions $\cN_j\to \cN$ which converge in the sense of $(1)- (4)$ of Theorem \ref{t:neck_approximation}, and let $u_j:B_4(p_j)\to \dR^{n-4}$ a sequence of $\delta'$-splitting functions with $u_j\to u$ uniformly.  Let $\cC_{\epsilon,j}\subseteq \cC_j$ be a sequence which satisfy the desired $\epsilon'$-bilipschitz and measure estimates given by Theorem \ref{t:splitting_neck_bilipschitz}, and let us consider the Hausdorff limit set $\cC_{\epsilon,j}\to \cC_\epsilon\subseteq \cC$.

Note that $\cC_\epsilon$ satisfies the desired $\epsilon$-bilipschitz control of Theorem \ref{t:splitting_neck_bilipschitz} since $u_j\to u$ uniformly.  To finish the proof we need to understand measure estimates on $\mu(B_1\setminus \cC_\epsilon)$.  However recall condition $(4)$ of our neck approximation Theorem, which tells us that if $\mu_j\to \mu_\infty$ then in particular $\mu\leq C(n,\tau)\mu_\infty$.  Since $B_1(p)\setminus \cC_\epsilon$ is an open set, we can conclude
\begin{align}
\mu(B_1\setminus \cC_\epsilon)\leq \mu_\infty\big(B_1\setminus \cC_\epsilon\big)\leq \liminf \mu_j\big(B_1\setminus \cC_{\epsilon,j}\big) < C(n,\tau)\epsilon' < \epsilon\, ,	
\end{align}
where in the last line we have taken $\epsilon' <C(n,\tau)^{-1}\epsilon$, which finishes the proof. $\square$

\vspace{1cm}

\section{Completing the Proof of Theorem \ref{t:neck_region}}\label{s:neck_region_proof}

In this section we use the tools of Sections \ref{s:L2_bound} and  \ref{s:neck_splitting} in order to prove Theorem \ref{t:neck_region}.  The outline of this section is as follows.  In Section \ref{ss:neck_region_smooth:lower_volume} we will first prove the lower volume bound of Theorem \ref{t:neck_region}.1.  The proof of this lower bound will depend on the discrete Reifenberg maps constructed in Section \ref{ss:neck:reifenberg}.  In Section \ref{ss:neck_region_proof:smooth} we will then complete the proof the remaining statements of Theorem \ref{t:neck_region} in the case when we are on a manifold by an inductive construction.  Finally in Section \ref{ss:neck_region_proof:limit} we will complete the remainder of the proof by considering a general limit space.\\

\subsection{Proof of Lower Ahlfors Regularity bound of Theorem \ref{t:neck_region}.1 }\label{ss:neck_region_smooth:lower_volume}
In this subsection we will prove the lower Ahlfors regularity bound of Theorem \ref{t:neck_region}.1.  Without any loss of generality we can focus on proving the lower bound
\begin{align}\label{e:lower_ahlfors:1}
\mu(B_1(p))>A(n)^{-1}>0\, ,	
\end{align}
as it will be clear the verbatim construction will tackle a general ball since $B_{2r}(x)\cap \cN\subset B_{2r}(x)$ is a $(\delta,\tau)$-neck for $x\in \cC$, $r\ge r_x$ and $B_{2r}(x)\subset B_2(p)$.  Furthermore, it will suffices to consider $\max_{x\in\cC\cap B_1(p)}r_x\le 10^{-5n}$ since otherwise the bound \eqref{e:lower_ahlfors:1} follows trivially by the definition of measure $\mu$. 

 Therefore with $\delta'<<1$ small let us consider a harmonic $\delta'$-splitting function $u:B_4(p)\to \dR^{n-4}$.  The proof of \eqref{e:lower_ahlfors:1} is based on a discrete degree type argument.  More precisely, let us begin by defining the discrete ${\epsilon}$-Reifenberg map
\begin{align}
\Phi: \cC\cap B_{19/10}(p)\to B_2(0^{n-4})\, ,	
\end{align}
given by Theorem \ref{t:neck_reifenberg}.  Note since both $\Phi$ and $u$ are ${\epsilon}$\text{-isometries}, that after possibly composing $u$ with an orthogonal transformation we can assume $|\Phi-u|<C(n)\epsilon$.  Let us denote $\cC'=\cC'_0\cup\cC'_+\equiv \Phi(\cC\cap B_{19/10}(p))\subseteq B_2(0^{n-4})$, and since $\Phi$ is bih\"older by Theorem \ref{t:neck_reifenberg}.1 we can consider the inverse mapping $\Phi^{-1}:\cC'\to \cC\subseteq B_2(p)$.  Let us now compose this with the splitting function $u$ in order to construct the map
\begin{align}
F\equiv u\circ \Phi^{-1}:\cC'\subseteq \dR^{n-4}\to \dR^{n-4}\, .	
\end{align}
Let us remark on some properties of the mapping $F$:
\begin{align}
|F-Id|<C(n)\epsilon ,\\ \label{e:Ahlforslowerbound_covering}
F\big(T_{x'}^{-1}B_{10r_{x'}}(x')\big)\subseteq B_{20r_{x'}}(F(x')), \text{ for all $x'\in\cC'_+$,}
\end{align}
where $r_{\Phi(x)}=r_x$ and $T_{\Phi(x)}=T_x$ with $\Phi(x)=x'$ is as in Theorem \ref{t:neck_reifenberg}. Recall that $T_{x'}^{-1}B_{10r_{x'}}(x'):=\{x'+T_{x'}^{-1}y: y\in B_{10r_{x'}}(0^{n-4})\}$. The first property follows from the combined Gromov-Hausdorff behaviors of $\Phi$ and $u$, while the second property follows from Theorem \ref{t:neck_reifenberg}.2 together with the lipschitz bound $|\nabla u|\leq 1$.\\

Now using Theorem \ref{t:neck_reifenberg} let us consider the covering of $\dR^{n-4}\setminus \cC'_0$ given by
\begin{align}
\dR^{n-4}\setminus \cC'_0\subseteq A_{\frac{3}{2},\infty}(0^{n-4})\cup \bigcup_{x\in\cC_+} T_x^{-1}\bar B_{r_x}\big(\Phi(x)\big)= A_{\frac{3}{2},\infty}(0^{n-4})\cup \bigcup_{x'\in\cC'_+} T_{x'}^{-1}\bar B_{r_{x'}}\big(x'\big)\, .
\end{align}
 Let us consider a partition of unity $\phi\cup \{\phi_{x'}\}$ with respect to this covering such that $\text{supp} \phi_{x'}\subset T_{x'}^{-1}B_{3r_{x'}}(x')$ and $\text{supp }\phi\subset A_{5/4,\infty}(0^{n-4})$. The following claim compares the supports of $\phi_{x'}$ and $\phi_{y'}$, which will be used in the calculations \eqref{e:Ahlforslowerbound_GzIdz} and \eqref{e:Ahlforslowerbound_GyGy'} below. 
 
\vskip 2mm

\noindent
\textbf{Claim:} If $T_{x'}^{-1}B_{3r_{x'}}(x')\cap T_{y'}^{-1}B_{3r_{y'}}(y')\ne \emptyset$, then  $T_{y'}^{-1}B_{3r_{y'}}(y')\subset T_{x'}^{-1}B_{10r_{x'}}(x')$\footnote{Recall again that $T_{x'}^{-1}B_{3r_{x'}}(x'):=\{x'+T_{x'}^{-1}z: z\in B_{3r_{x'}}(0^{n-4})\}$ } .

\vskip 2mm
 
 Proof of Claim:  From the definition of $T_{x'}^{-1}B_{3r_{x'}}(x')$ and $T_{y'}^{-1}B_{3r_{y'}}(y')$, the estimate $T_{x'}^{-1}B_{3r_{x'}}(x')\cap T_{y'}^{-1}B_{3r_{y'}}(y')\ne \emptyset$ is equivalent to $B_{3r_{x'}}(T_{x'}x')\cap T_{x'}T_{y'}^{-1}B_{3r_{y'}}(T_{x'}y')\ne \emptyset$.
We will  show that 
 \begin{align}\label{e:distance_xy_Ahlforslowerboundproof}
 d(x,y)\le 4(r_x+r_y).
 \end{align}
Let us first assume \eqref{e:distance_xy_Ahlforslowerboundproof} and finish the proof of the claim. Since $|\text{Lip }r_x|\le \delta$ by (n4), \eqref{e:distance_xy_Ahlforslowerboundproof} implies that $d(x,y)\le 9r_x$ and $r_y\le 11r_x/10$. Moreover, by Theorem \ref{t:neck_reifenberg} there exists $O\in O(n-4)$ such that $|T_{x'}(OT_{y'})^{-1}-Id|\le \epsilon$. Since $T_{x'}(OT_{y'})^{-1}B_{3r_{y'}}(T_{x'}y')=T_{x'}T_{y'}^{-1}O^{-1}B_{3r_{y'}}(T_{x'}y')=T_{x'}T_{y'}^{-1}B_{3r_{y'}}(T_{x'}y')$, without loss of generality, let us now assume $|T_{x'}T_{y'}^{-1}-Id|\le \epsilon$. Therefore, if $\epsilon\le \epsilon(n)$, by $B_{3r_{x'}}(T_{x'}x')\cap T_{x'}T_{y'}^{-1}B_{3r_{y'}}(T_{x'}y')\ne \emptyset$, we get $|T_{x'}x'-T_{x'}y'|\le 3r_{x'}+34r_{y'}/11\le 64r_{x'}/10$. Thus the statement $T_{y'}^{-1}B_{3r_{y'}}(y')\subset T_{x'}^{-1}B_{10r_{x'}}(x')$, which is equivalent to $T_{x'}T_{y'}^{-1}B_{3r_{y'}}(T_{x'}y')\subset B_{10r_{x'}}(T_{x'}x')$, follows directly if $\epsilon\le \epsilon(n)$. Hence we have proved the claim by assuming \eqref{e:distance_xy_Ahlforslowerboundproof}.
 
Let us now prove \eqref{e:distance_xy_Ahlforslowerboundproof}.  Without loss of generality let us assume $r_y\le r_x$. 
and assume $d(x,y)=r$. Noting $T_{x,2r}\Phi$ is an $\epsilon r$\text{-isometry}, we conclude that 
\begin{align}
|T_{x,2r}(\Phi(x)-\Phi(y)|\ge (1-\epsilon)r.
\end{align}
On the other hand, by Theorem \ref{t:neck_reifenberg} we have $|T_xT_{x,2r}^{-1}|\le (2r/r_x)^\epsilon$, $|T_{y,2r}T_y^{-1}|\le (2r/r_y)^\epsilon$  and $|T_{x,2r}\circ (O_{xy,2r}T_{y,2r})^{-1}-Id|\le \epsilon$ for some $O_{xy,2r}\in O(n-4)$. Thus we have 
\begin{align}
|T_{x}T_{y}^{-1}|=|\Big(T_xT_{x,2r}^{-1}\Big)\Big(T_{x,2r}\circ (O_{xy,2r}T_{y,2r})^{-1}\Big)O_{xy,2r}\Big(T_{y,2r}T_y^{-1}\Big)|\le (1+\epsilon)(2r/r_x)^\epsilon (2r/r_y)^\epsilon.
\end{align}
Since $B_{3r_{x'}}(T_{x'}x')\cap T_{x'}T_{y'}^{-1}B_{3r_{y'}}(T_{x'}y')\ne \emptyset$, we thus have 
\begin{align}\label{e:Txphixphiyupper}
|T_{x}(\Phi(x)-\Phi(y))|\le 3r_x+3r_y|T_{x}T_{y}^{-1}|\le 3r_x+(1+\epsilon)3r_y(2r/r_x)^\epsilon (2r/r_y)^\epsilon.
\end{align}
Using the estimate $|T_{x,2r}T_{x}^{-1}|\le (2r/r_x)^\epsilon$ we have
\begin{align}\label{e:Txphixphiylower}
|T_{x}(\Phi(x)-\Phi(y))|\ge |T_{x,2r}T_x^{-1}|^{-1}|T_{x,2r}(\Phi(x)-\Phi(y)|\ge (1-\epsilon)r({2r}/{r_x})^{-\epsilon}.
\end{align}

Now \eqref{e:distance_xy_Ahlforslowerboundproof} follows directly from \eqref{e:Txphixphiylower} and \eqref{e:Txphixphiyupper} if $\epsilon\le \epsilon(n)$.  Hence we have proved the claim. $\square$

Let us define the  mapping $G:\dR^{n-4}\to \dR^{n-4}$ given by
\begin{align}
	G(y') = \begin{cases}
 	\phi(y')\cdot Id(y') + \sum_{x'\in \cC'} \phi_{x'}(y') F(x')\,&\text{ if }y'\in\dR^{n-4}\setminus \cC'_0\, ,\notag\\
 	\phi(y')y'+(1-\phi(y'))F(y') &\text{ if }y'\in \cC'_0\, .
 \end{cases}
\end{align}
Since $F$ itself is continuous on $\cC'$, the map $G$ is continuous. 
 Using the properties of $F$ and Theorem \ref{t:neck_reifenberg} we can conclude the following properties for $G$:
\begin{enumerate}
	\item $|G-Id|<C(n)\epsilon+10^{-n}$,
	\item $G(y')=Id(y')=y'$ if $y'\not\in B_3(0^{n-4})$,
	\item $G\big(T_{y'}^{-1}B_{r_{y'}}(y')\big)\subseteq B_{40r_{y'}}(G(y'))$ for all $y'\in\cC'_+\cap B_1(0^{n-4})$.
\end{enumerate}
Actually, (2) follow directly from the definition of $G$. To see (1), for any $z\in \mathbb{R}^{n-4}$ if $z\not\in \cup_{y'\in \cC'_{+}}\text{supp} \phi_{y'}$ we have $G(z)=Id(z)$ or $G(z)=F(z)$ for $z\in \cC_0'$ and thus $|G(z)-Id(z)|\le C(n)\epsilon$. On the other hand, if $z\in T_{y'}^{-1}B_{3r_{y'}}(y')$ for some $y'\in \cC_+'$ we have by the above Claim that if $\phi_{x'}(z)\ne 0$ then $x'\in T_{y'}^{-1}B_{10r_{y'}}(y')$. Then we can compute for $z\in T_{y'}^{-1}B_{3r_{y'}}(y')$ that
\begin{align}\label{e:Ahlforslowerbound_GzIdz}
|G(z)-Id(z)|&=|\sum_{x'\in \cC'} \phi_{x'}(z) (F(x')-Id(z))|\\
&\le \sum_{x'\in \cC'} \phi_{x'}(z) |F(x')-F(y')|+\sum_{x'\in \cC'} \phi_{x'}(z) |F(y')-Id(y')|+\sum_{x'\in \cC'} \phi_{x'}(z)|y'-z|\\
&\le \sup_{x'\in T_{y'}^{-1}B_{10r_{y'}}(y')}|F(x')-F(y')|+C(n)\epsilon+3r_{y'}|T_{y'}^{-1}|.\\ \label{e:Ahlforslowerbound_Fx'_Fy'}
&\le C(n)\epsilon+20r_{y'}+3r_{y'}^{1-\epsilon}\\
&\le C(n)\epsilon+10^{-n}.
\end{align}
where in \eqref{e:Ahlforslowerbound_Fx'_Fy'} we have used \eqref{e:Ahlforslowerbound_covering}
 and $|T_{y'}^{-1}|\le r_{y'}^{-\epsilon}$ from Theorem \ref{t:neck_reifenberg}, and the last inequality follows from the fact that $\max_{x\in\cC\cap B_1(p)}r_x\leq 10^{-5n}$ and $|{\text{Lip}}r_x|\le \delta^2$.

 Let us check the third property. The proof is similar to the above computations. For any $y'\in \cC'_+\cap B_1(0^{n-4})$  and $z\in T_{y'}^{-1}B_{r_{y'}}(y')$, and noting that $\text{supp }\phi\subset A_{5/4,\infty}(0^{n-4})$, we have by the definition of $G$ that 
\begin{align}\notag 
|G(z)-G(y')|&=|\sum_{x'\in \cC'} \phi_{x'}(y') F(x')-\sum_{x'\in \cC'} \phi_{x'}(z) F(x')|\le |\sum_{x'\in \cC'} \phi_{x'}(y') F(x')-F(y')|+|\sum_{x'\in \cC'} \phi_{x'}(z) F(x')-F(y')|\\ \notag
&\le \sum_{x'\in \cC'} \phi_{x'}(y') |F(x')-F(y')|+\sum_{x'\in \cC'} \phi_{x'}(z) |F(x')-F(y')|\le 2\sup_{x'\in T_{y'}^{-1} B_{10r_{y'}}(y')}|F(x')-F(y')|\\  \label{e:Ahlforslowerbound_GyGy'}
&\le 40r_{y'},
\end{align}
 where the last inequality follows from \eqref{e:Ahlforslowerbound_covering}. Therefore, we have proved (3).

The second condition above tells us that $G$ is a degree one map from the ball $B_3$ to itself, and in particular we must have that $G$ is onto.  Combining with the first condition we in particular get that
\begin{align}
G\big(B_{3/4}(0^{n-4}\big)\supseteq B_{1/2}(	0^{n-4})\, .
\end{align}
Further combining this with the third conditions this tells us that
\begin{align}
B_{1/2}(	0^{n-4})\subseteq G\big(B_{3/4}(0^{n-4})\big)\subseteq  G\big(\bigcup_{x'\in \cC'\cap B_{7/8}(0^{n-4})} T_{x'}^{-1}\overline B_{r_{x'}}(x')\big)\subseteq \bigcup_{x'\in \cC'_+\cap B_{7/8}(0^{n-4})} B_{40r_{x'}}(G(x'))\cup \big(G(\cC'_0)\cap B_{7/8}(0^{n-4}) \big)\, .
\end{align}
But this immediately leads to the estimate
\begin{align}
\mu(B_1(p)) &= \sum_{x\in \cC_+\cap B_1(p)} r_{x}^{n-4} + \lambda^{n-4}\big(\cC_0\cap B_1(p)\big)\geq \sum_{x'\in \cC'\cap B_{7/8}(0^{n-4})} r_{x'}^{n-4} +\lambda^{n-4}\Big(G(\cC'_0)\cap B_{7/8}(0^{n-4})\Big)\notag\\
&= C(n)\sum_{x'\in \cC'\cap B_{7/8}(0^{n-4})} (40r_{x'})^{n-4}+\lambda^{n-4}\Big(G(\cC'_0)\cap B_{7/8}(0^{n-4})\Big)\notag\\
&\geq C(n)\Vol\Big(\bigcup_{x'\in \cC'\cap B_{7/8}} B_{40r_{x'}}(G(x'))\cup (G(\cC'_0)\cap B_{7/8}(0^{n-4}))\Big)\geq C(n)\Vol\Big(B_{1/2}(0^{n-4})\Big) \equiv A(n)^{-1}\, ,	
\end{align}
where we have used in the first line that $|\nabla u|\leq 1$ in order to conclude $\lambda^{n-4}\big(\cC_0\cap B_1(p)\big)\geq \lambda^{n-4}\big(u\big(\cC_0\cap B_1(p)\big)\big)=\lambda^{n-4}\big(G\big(\cC'_0\cap B_1(0^{n-4})\big)\big)\geq \lambda^{n-4}\Big(G(\cC'_0)\cap B_{7/8}(0^{n-4})\Big)$.  This completes the proof of the lower volume bound. $\square$

\vspace{.5cm}

\subsection{Induction Scheme and Proof of Theorem \ref{t:neck_region} on a Manifold}\label{ss:neck_region_proof:smooth}

Let us now focus our attention on finishing the proof of Theorem \ref{t:neck_region} in the case when we are on a smooth manifold $M$.  Let us begin by making several observations.  First, if $M$ is a smooth manifold then we have that $\inf r_x>0$, see Remark \ref{r:neck_region:1}.  Therefore, as we have already shown the lower Ahlfors regularity bound we are left with just proving the upper Ahlfors regularity of Theorem \ref{t:neck_region}.1 and the $L^2$ curvature estimate of Theorem \ref{t:neck_region}.3.  Additionally, by Theorem \ref{t:L2_neck} we have that the $L^2$ estimate of Theorem \ref{t:neck_region}.3 follows from the Ahlfors regularity of Theorem \ref{t:neck_region}.1, and thus we can focus in this subsection on only proving the upper Ahlfors regularity estimate.\\

Thus let us now outline our induction strategy.  For $\alpha\in \dN$ let us consider the following statement:
\begin{align}
(\alpha)\;\;\; &\text{For $x\in \cC$ with $r_x<r\leq 2^{-\alpha}$ and $B_{2r}(x)\subseteq B_2$ we have $\mu(B_r(x)) < A \, r^{n-4}$}\, .\notag
\end{align}

Our strategy is to find $A\leq A(n,\tau)$ such that we can prove statement $(\alpha)$ inductively.  If we can prove the inductive statement for $\alpha=0$ then we will have completed the theorem.\\

We need to begin with the base step of the induction process, and for that let us make the following observation.  We have already commented that our neck region $\cN = B_2\setminus \overline B_{r_x}(\cC_+)$ contains only balls of positive radius since we are on a manifold, and indeed since $\inf r_x>0$ we even have the (ineffective) lower bound
\begin{align}
  \min r_x > 2^{-\alpha_0}>0\, ,	
\end{align}
for some $\alpha_0\in \dN$ sufficiently large.  Therefore, we have that the inductive statement $(\alpha_0)$ trivially holds, since $r_x>2^{-\alpha_0}$ for all $x$ and thus there is no content to the statement.\\

Now let us assume we have proved the inductive statement $(\alpha+1)$, and from that and the constructions of Sections \ref{s:L2_bound} and \ref{s:neck_splitting} we will prove the inductive statement $(\alpha)$.  Thus for the remainder of this subsection let $x\in \cC$ with $r_x<r\leq 2^{-\alpha}$ and $B_{2r}(x)\subseteq B_2$.

Let us first deal with the easy case and observe that if $r_x>10^{-3}r$ is relatively large compared to our ball size, then we have both $\cN\cap B_{10^3 r}(x)\subseteq B_{10^3 r}(x)\setminus \overline B_{10^{-6}r}(\cC)$ and that $B_{10^3 r}(x)\setminus \overline B_{10^{-6}r}(\cC)$ is $\delta r$-GH close to $B_{10^3 r}\setminus \overline B_{10^{-6}r}(\dR^{n-4}\times\{0\})\subseteq \dR^{n-4}\times C(S^3/\Gamma)$.  Further, by the lipschitz condition $(n4)$ we have for every ball center $x_i\in \cC\cap B_{10^3 r}(x)$ that $r_{x_i}>r_x-10^3\delta\,r>10^{-6}r$.  Thus using the definition of $\mu$ we see that $(\alpha)$ trivially holds when $A(n,\tau)$ is taken to be the maximal number of disjoint balls of radius $10^{-6}\tau^2$ that one can fit in $B_{2}(0^{n-4})$.\\

Therefore we may assume that $r_x\leq 10^{-3}r$ without any loss, and our aim is to prove $(\alpha)$ for our ball $B_r(x)$.  The beginning point is to first prove a weaker estimate.  Thus let $y\in\cC\cap B_{2r}(x)$, and note by the lipschitz condition $(n4)$ of our neck region that $r_y< r_x + 2\delta r<10^{-2}r$.  Now let us consider a radius $r_y<s\leq r$ such that $B_{2s}(y)\subseteq B_{2r}(x)$.  If $s\leq \frac{r}{2}$ then $B_s(y)$ satisfies the hypothesis of $(\alpha)$ and therefore we have the upper bound $\mu(B_s(y))\leq A s^{n-4}$.  On the other hand, if $s> \frac{r}{2}$, then by a standard covering argument we may cover $B_s(y)\cap \cC$ by at most $C(n)$ balls $\{B_{r/4}(x_i)\}$.  Since the hypothesis holds for each of the balls $B_{r/2}(x_i)$ we can estimate to get the strictly weaker estimate
\begin{align}
(*)\;\;\; \mu(B_s(y))	\leq \sum \mu(B_{r/4}(x_i)) \leq C(n) A \big(\frac{r}{4}\big)^{n-4} \leq C(n) A s^{n-4}\equiv B(n,\tau)\, s^{n-4}\, .
\end{align}

In particular, we see that for every $y\in \cC\cap B_{2r}(x)$ and $r_y<s$ such that $B_{2s}(y)\subseteq B_{2r}(x)$ we have the weaker estimate $\mu(B_s(y))\leq B\, s^{n-4}$.\\

In order to exploit this weaker estimate let us now observe that for $\delta'>0$ we can choose $\delta\leq \delta(n,\tau,\delta')$ such that by Theorem \ref{t:splitting_function} there exists a $(n-4,\delta')$-splitting
\begin{align}
u:B_{8r}(x)\to \dR^{n-4}\, .	
\end{align}
If we consider Theorem \ref{t:splitting_neck_bilipschitz} then we see if $\delta,\delta'\leq \delta(n,\tau,\rv,B,\epsilon)=\delta(n,\tau,\rv,\epsilon)$ then there exists a subset $\cC_\epsilon\subseteq \cC\cap B_r(x)$ such that
\begin{align}\label{e:neck_region_proof:1}
&\mu\Big(\big(\cC\cap B_r(x)\big)\setminus \cC_\epsilon\Big) < \epsilon r^{n-4}\, ,\notag\\
&\text{For all $y,z\in \cC_\epsilon$ we have $1-\epsilon\leq \frac{|u(y)-u(z)|}{d(y,z)}\leq 1+\epsilon$}\, .
\end{align}

But let us now consider the consequences of this estimate.  Restricting to the set $\cC_\epsilon$ we have that $\{B_{\tau^2 r_x}(x)\}_{x\in\cC_\epsilon}$ are disjoint balls.  However, using \eqref{e:neck_region_proof:1} this tells us that the image balls
\begin{align}
	\big\{ B_{10^{-1}\tau^2 r_x}(u(x))\big\}_{x\in\cC_\epsilon}\subseteq B_{2r}(0^{n-4})\, ,
\end{align}
are also disjoint.  But then we automatically get the packing estimate
\begin{align}
	\sum_{x\in \cC_\epsilon} r_x^{n-4}&=C(n,\tau)\sum_{x\in \cC_\epsilon} \Big(10^{-1}\tau^2 r_x\Big)^{n-4}\leq C(n,\tau) \Vol\Big(\sum_{x\in \cC_\epsilon} B_{10^{-1}\tau^2 r_x}(u(x))\Big)\notag\\
	&\leq C(n,\tau) \Vol\Big(B_{2r}(0^{n-4})\Big)\leq C(n,\tau) r^{n-4}\equiv  \frac{1}{2}A(n,\tau) r^{n-4}\, .
\end{align}
But using the first estimate of \eqref{e:neck_region_proof:1} this then gives us
\begin{align}
\mu\big(B_r(x)\big)	\leq \mu\big(\cC_\epsilon\big)+\epsilon r^{n-4} = \sum_{x\in \cC_\epsilon} r_x^{n-4}+\epsilon r^{n-4} \leq A(n,\tau)\, r^{n-4}\, .
\end{align}
This proves the estimate $(\alpha)$ for $B_r(x)$, and since $B_r(x)$ was arbitrary this completes the inductive step of the proof, and hence the proof of Theorem \ref{t:neck_region} itself. $\square$
\vspace{.5cm}

\subsection{Proof of Theorem \ref{t:neck_region}}\label{ss:neck_region_proof:limit}

Let us now consider the case of a general limit space $M^n_j\to X$.  First observe that by the neck approximation of Theorem \ref{t:neck_approximation} we can consider a sequence of neck regions $\cN_j\subseteq B_2(p_j)$ such that $\cN_j\to \cN$ in the sense of $(1)\to(4)$ of Theorem \ref{t:neck_approximation}.  In particular, if we consider the packing measures $\mu_j$ of $\cN_j$ then we may limit $\mu_j\to \mu_\infty$ so that
\begin{align}
\mu \leq C(n,\tau)\, \mu_\infty\, .
\end{align}
In particular, by applying the results of the previous two subsections we can immediately conclude that the measure $\mu$ satisfies the required Ahlfors regularity of Theorem \ref{t:neck_region}.1.

Now if we apply Theorem \ref{t:neck_region}.3 to each of the neck regions $\cN_j$, then by observing that $\cN_j\to \cN$ in $C^{1,\alpha}\cap W^{2,p}$ on all compact subsets, we immediately obtain the $L^2$ estimate
\begin{align}
\int_{\cN\cap B_1} |\Rm|^2(z)\,dz < \epsilon\, ,	
\end{align}
which finishes the proof of the $L^2$ estimate of Theorem \ref{t:neck_region}.3.\\

To finish the proof we need to prove Theorem \ref{t:neck_region}.2, which is to say we need to see that $\cC_0$ is rectifiable.  The main claim needed for this result is the following:\\

{\bf Claim:}  For each $\epsilon>0$ if $\delta<\delta(n,\tau,\rv,\epsilon)$, then for $x\in \cC$ with $B_{2r}(x)\subseteq B_2$ there exists a closed subset $\cR_\epsilon\big(B_r(x)\big)\subseteq \cC_0\cap B_r(x)$ such that $\cR_\epsilon$ is bilipschitz to an open subset of $\dR^{n-4}$ and $\mu\big(B_r(x)\cap( \cC_0\setminus \cR_\epsilon)\big)<\epsilon r^{n-4}$.\\

Indeed, for such a ball $B_{2r}(x)\subseteq B_2$ let us choose a $\delta'$-splitting function $u:B_{4r}(x)\to \dR^{n-4}$, which for $\delta<\delta(n,\delta')$ exists by theorem \ref{t:splitting_function}.  Using Theorem \ref{t:splitting_neck_bilipschitz} we see that if $\delta',\delta<\delta(n,\tau,\rv,\epsilon)$ then there exists a subset $\cC_\epsilon\subset \cC$ such that the restriction $u:\cC_\epsilon\to \dR^{n-4}$ is uniformly bilipschitz and $\mu\big(B_r(x)\setminus \cC_\epsilon\big)<\epsilon r^{n-4}$.  We now define $\cR_\epsilon\equiv \cC_0\cap \cC_\epsilon$, and see this is our desired set. $\square$ \\

To finish the proof is now a measure theoretic argument.  Indeed, let $\{y_j\}\subseteq \cC_0$ be a countable dense subset and let us consider the set
\begin{align}
	\cR\equiv \bigcup_{B_{2r}(y_j)\subseteq B_2:\, r\in \dQ} \cR_\epsilon(B_r(y_j))\, .
\end{align}

Clearly $\cR\subseteq \cC_0$ is rectifiable, as it can be identified as a countable union of rectifiable sets.  Notice that because the sets $\cR_\epsilon$ are closed we also have the identification
\begin{align}
	\cR= \bigcup_{B_{2r}(y)\subseteq B_2} \cR_\epsilon(B_r(y))\, ,
\end{align}
so that it does not matter if we union over all points or scales or just some countable dense collection.\\

Now in order to conclude that $\cC_0$ is rectifiable we need to see that $\mu(\cC_0\setminus \cR) = \lambda^{n-4}(\cC_0\setminus \cR)=0$.  Thus let us assume that $\cC_0\setminus \cR$ has positive $n-4$ Hausdorff measure.  In particular, by standard arguments we have that $n-4$ a.e. point is a density point, which is to say that if $\lambda^{n-4}(\cC_0\setminus \cR)>0$ then for some dimensional constant $\epsilon_n>0$ and $x\in \cC_0\setminus \cR$ there are radii $r_a\to 0$ such that
\begin{align}
\lim \frac{\lambda^{n-4}\big((\cC_0\setminus \cR)\cap B_{r_a}(x)\big)}{r_a^{n-4}}	 > \epsilon_n>0\, .
\end{align}
But by taking $\epsilon<\epsilon_n$ and using the definition of $\cR$ we immediately arrive at a contradiction, and therefore $\cR\subseteq \cC_0$ is a full measure subset and $\cC_0$ is rectifiable.  This completes the proof of Theorem \ref{t:neck_region} $\square$

\vspace{1cm}
\section{Neck Decomposition Theorem}\label{s:neck_decomposition}

In this section we focus on decomposing our manifolds $M^n$ into two basic type of pieces, namely neck regions and $\epsilon$-regularity regions with uniformly bounded curvature.  The key aspect of this decomposition will be our ability to control the number of pieces in this decomposition, a result which will depend heavily on the results of Section \ref{s:neck} and is sharp.  We begin with the decomposition in the case of smooth spaces:\\

\begin{theorem}[Neck Decomposition]\label{t:neck_decomposition}
		Let $(M^n_i,g_i,p_i)\to (X,d,p)$ satisfy $\Vol(B_1(p_i))>\rv>0$ and the Ricci bound $|\Ric_i|\leq n-1$.  Then for each $0<\delta<\delta(n,\rv)$ and $0<\tau\le \tau(n)$ we can write
	\begin{align}
		&B_1(p)\cap \cR(X)\subseteq \bigcup_a \big(\cN_a\cap B_{r_a}\big) \cup \bigcup_b B_{r_b}(x_b)\, , \notag\\
		&B_1(p)\cap \cS(X) \subseteq \bigcup_a \big(\cC_{0,a}\cap B_{r_a}\big)\cup \tilde S(X)
	\end{align}
	such that
	\begin{enumerate}
		\item $\cN_a\subseteq B_{2r_a}(x_a)$ are $(\delta,\tau)$-neck regions.
		\item $B_{r_b}(x_b)$ satisfy $r_h(x_b)>2r_b$, where $r_h$ is the harmonic radius.
		\item $\cC_{0,a}\subseteq B_{2r_a}(x_a)\setminus \cN_a$ is the singular set associated to $\cN_a$.
		\item $\tilde\cS(X)$ is a singular set of $n-4$ measure zero.
		\item $\sum_a r_a^{n-4} + \sum_b r_b^{n-4} + H^{n-4}\big(\cS(X)\big) \leq C(n,\rv,\delta,\tau)$.
	\end{enumerate}
\end{theorem}
\vspace{.5cm}

The proof will require a series of covering arguments, where we will first decompose our ball $B_1(p)$ into five types of pieces.  Over the course of several subsections we will then break down these pieces until we are left with only the two which appear in the theorem itself.  \\

\subsection{Notation and Ball Decomposition Types}\label{ss:neck_decomp:notation}

The proof of Theorem \ref{t:neck_decomposition} will require multiple covering steps and lemmas, where we will first decompose $B_1(p)$ into a much larger collection of balls, and then proceed to recover those balls which are not either neck regions or smooth regions.  In order to avoid as much confusion as is reasonable, we will introduce in this subsection a variety of notation which will hold throughout this section.  In particular, we will introduce and discuss the various ball types which will appear.  \\

Let us begin by introducing some notation which will play a role.  Throughout the proof we have the underlying constants $\delta$ and $\tau$, however for rigor sake we will need several other constants floating around that we can exploit.  Throughout the proof these constants will eventually be fixed, however let us mention that they will roughly behave as
\begin{align}
0<\eta(n,\delta',\tau)<<\delta'(n,\delta,\tau)<<\delta<<\epsilon(n,\tau)<<\tau<<1\, .	
\end{align}

Let us now move to a distinct notational point.  When studying a ball we will often be interested in how the volume pinching behaves at any given point in comparison to a background reference value.  This reference value will be denoted by $\overline V$, and in practice will typically be given by
\begin{align}\label{e:neck_decom:bar_V}
\overline V \equiv \inf_{y\in B_4(p)} \cV_{\eta^{-1}}(y)\, .
\end{align}
Notice that it will be convenient to move up $\eta^{-1}$-scales in our volume measurements.  In addition to this, we will be interested in those points which remain volume pinched after we drop many scales.  Thus, let us also define the sets
\begin{align}\label{e:neck_decom:E_eta}
	\cE_\eta(x,r)\equiv \{y\in B_{4r}(x): \cV_{\eta r}(y)<\overline V + \eta\}\, .
\end{align}
Therefore $\cE_\eta(x,r)$ represents those points of $B_{4r}(x)$ whose value ratio has not increased to some definite amount $\eta$ more than our reference value.  Using the cone splitting of Theorem \ref{t:cone_splitting} this means that these are the very $0$-symmetric points of $B_{4r}(x)$.\\

Equipped with this terminology, which will be used consistently throughout this section, let us now consider the various ball types which will appear in the proof.  The conventions below will allow us to keep oriented throughout the proof of the various conditions we will be forced to consider.  After defining the sets carefully we will discuss in words their meaning:\\

\begin{enumerate}
\item[(a)] A ball $B_{r_a}(x_a)$ with subscript $a$ will be such that there exists a $(\delta,\tau)$-neck region $\cN_a\subseteq B_{2r_a}(x_a)$.
\item[(b)] A ball $B_{r_b}(x_b)$ with subscript $b$ will be such that	 $r_b(x_b)\geq 2r_b$.
\item[(c)] A ball $B_{r_c}(x_c)$ with subscript $c$ will be such that $\Vol\big(B_{\delta r_d}\cE_\eta(x_d,r_d)\big)> \epsilon\,\delta^4\,r_d^n$ 
\item[(d)] A ball $B_{r_d}(x_d)$ with subscript $d$ will be such that $\Vol\big(B_{\delta r_d}\cE_\eta(x_d,r_d)\big)\leq \epsilon\,\delta^4\,r_d^n$ and $\cE_\eta\ne \emptyset$.
\item[($\tilde{d}$)] A ball $B_{{r}_{\tilde{d}}}({x}_{\tilde{d}})$ with subscript $\tilde{d}$ will be such that $\Vol\big(B_{\delta r_{\tilde{d}}}\cE_\eta(x_{\tilde{d}},r_{\tilde{d}})\big)\leq \epsilon\,\delta^4\,r_{\tilde{d}}^n$.
\item[(e)] A ball $B_{r_e}(x_e)$ with subscript $e$ will be such that $\cE_\eta(x_e,r_e)=\emptyset$, that is we have the volume jump $\inf_{y\in B_{4r_e}(x_e)} \cV_{\eta r_e}(y)\geq \overline V + \eta$.
\item[(f)] A ball $B_{r_f}(x_f)$ with subscript $f$ will be a typical ball for which we have no apriori knowledge of additional structure.
\end{enumerate}

Of the various ball types above, only the first two take part in the final result of Theorem \ref{t:neck_decomposition}.  Notice also that the $(c)$, $(d)$, $(\tilde{d})$ and $(e)$ balls require a background parameter $\overline V$ in their definition.  Let us maybe take a moment to explain in words the import of these ball types which will appear in the proof.  A $c$-ball is one for which there is a lot of points which are volume pinched on many scales, from a $n-4$ content point of view.  In particular, we will see by Theorem \ref{t:n-4_content_conesplitting} that a $c$-ball is $\delta'$-GH close to $\dR^{n-4}\times C(S^3/\Gamma)$, and therefore is a ball for which we expect a neck region to exist on, though maybe it has not yet been built.  A $d$-ball is one for which, from a $n-4$ content point of view, there are very few volume pinched points on that scale.  On the one hand this is bad in that we do not understand the geometry of $d$-balls, while on the other this makes them easy to recover until we can get down to a scale we do understand.  A $\tilde{d}$-ball is a $d$-ball or $e$-ball. We introduce this ball just because sometimes we need to handle $d$-ball and $e$-ball simultaneously.  To explain the $e$-ball, let us now note that the final proof of the decomposition theorem will be by an induction process, where we will induct on the lower volume of the ball.  Therefore, an $e$-ball is a ball for which the volume has increased by some definite amount after dropping some scales, and thus we will be able to apply our inductive hypothesis to deal with it.  Finally, an $f$-ball is one for which we may know nothing.  Typically, the $f$-balls will not appear in the statements of results, but only in the proofs until we have categorized their behavior.\\

The next several subsections will consider a sequence of covering lemmas which will allow us to build this final decomposition.\\

\subsection{Cone Splitting and Codimension Four}

In this subsection we prove a handful of preliminary results, which will be used to eventually help fix the constants $\delta'$ and $\eta$, in terms of $\tau$ and $\delta$, throughout the proof.  Our first result in this direction is essentially a new viewpoint on the notion of cone splitting in \cite{CheegerNaber_Ricci}.  Instead of trying to find independent points, we will instead measure the content of volume pinched points.  In general these two notions need not coincide, however in the case of cone-splitting it is not hard to see that they do.  The precise result is phrased in a manner most convenient for the applications of this section:\\

\begin{theorem}[Content Cone-Splitting]\label{t:content_cone_splitting}
	Let $(M^n,g,p)$ satisfy $Vol(B_1(p))>v>0$ with $0<\delta,\zeta\le \delta(n,\rv)$ and $\epsilon>0$.  Then there exists $\eta(n,\rv,\delta,\epsilon,\zeta)>0$ such that if
	\begin{enumerate}
		\item $\Ric\geq -\eta$,
		\item If $\overline V\equiv \inf_{B_{4}(p)}\cV_{\eta^{-1}}(y)$ and $\cE_\eta(p,1)\equiv \{y\in B_{4}(p): \cV_{\eta}(y)<\overline V + \eta\}$ satisfies
	\begin{align}
		\Vol\big(B_{\delta}\cE_\eta\big)> \epsilon\, \delta^{k}\, ,
	\end{align}
	\end{enumerate}
	then there exists $q\in B_4(p)$ such that $B_{\zeta^{-1}}(q)$ is $(n-k,\zeta^2)$-symmetric.
\end{theorem}
\begin{proof}
We will prove the result inductively on $k$. If $k=n$, then this just follows from Cheeger-Colding's metric cone Theorem \ref{t:metriccone_volumecone}.  Assume for $k\le n$ that the theorem holds. We will prove the theorem for $k-1$. Assume we have
$	\Vol\big(B_{\delta}\cE_\eta\big)> \epsilon\, \delta^{k-1}\ge \epsilon\, \delta^{k}.$
Let $\zeta'<<\zeta$ be fixed later. By induction, if $\eta\le \eta(n,\rv,\zeta',\delta,\epsilon)$, there exists $q\in B_1(p)$ such that $B_{{\zeta'}^{-1}}(q)$ is $(n-k,{\zeta'}^2)$-symmetric with respect to $L_{\zeta'}^{k}$. By GH-approximation and a covering argument as in the proof of Proposition \ref{p:local_L2_neck}, it is easy to show $\Vol(B_{2\delta} L_{\zeta'}^{k}\cap B_1(p))\le C(n,\rv)\delta^{k}.$
Thus for $\delta\le \delta(n,\rv)$ small, we have $\Vol\big(B_{\delta}\cE_\eta\setminus B_{2\delta} L_{\zeta'}^{k}\big)\ge \frac{\epsilon}{2}\delta^{k-1}$.  In particular, there exists $z\in \cE_\eta\setminus B_{\delta} L_{\zeta'}^{k}$. For $\eta$ small and metric cone Theorem \ref{t:metriccone_volumecone}, we have $B_{\zeta'^{-1}}(z)$ is $(0,{\zeta'}^2)$-symmetric. By the cone-splitting of Theorem \ref{t:cone_splitting} and choosing $\zeta'=\zeta'(\zeta,\delta,n,\rv)$ small, we have that $B_{\zeta^{-1}}(q)$ is $(n-k+1,\zeta^2)$-symmetric. Thus we finish the inductive step of the proof.
\end{proof}

\vspace{.5 cm}

The main application of the above in this paper will be to the $k=4$ case, under the assumption of a two-sided bound on the Ricci curvature.  In this case one can combine the above with the $\epsilon$-regularity of theorem \ref{t:eps_reg} in order to conclude that a ball which is sufficiently $n-4$-symmetric must actually be Gromov Hausdorff close to $\dR^{n-4}\times C(S^3/\Gamma)$, which is what will allow us to build our neck regions in subsequent sections.  Our precise result is the following:\\

\begin{theorem}\label{t:n-4_content_conesplitting}
	Let $(M^n,g,p)$ satisfy $Vol(B_1(p))>v>0$ with $\delta,\epsilon>0$.  Then for $\eta\leq \eta(n,\rv,\delta,\epsilon)$ with $\overline V\equiv \inf_{B_{4}(p)}\cV_{\eta^{-1}}(y)$ and $\cE_\eta(p,1)\equiv \{y\in B_{4}(p): \cV_{\eta}(y)<\overline V + \eta\}$, if we have
	\begin{align}
	    &|{\Ric}|<\eta\, ,\notag\\
		&\Vol\big(B_{\delta}\cE_\eta\big)> \epsilon\, \delta^{4}\, ,
	\end{align}
	then either:
	\begin{enumerate}
		\item We have that $r_h(p)>\delta^{-1}$, or
		\item There exists $q\in B_4(p)$ such that $B_{\delta^{-1}}(q)$ is pointed $\delta$-GH close to $B_{\delta^{-1}}(0^{n-4},y_c)\subseteq \dR^{n-4}\times C(S^3/\Gamma)$, where $\Gamma\leq O(4)$ is nontrivial and acts freely away from the origin with $y_c$ the cone point.
	\end{enumerate}
\end{theorem}
\begin{proof}
By Theorem \ref{t:content_cone_splitting} we have for $\eta\leq \eta(n,\rv,\delta,\tau,\epsilon)$ that there exists $q\in B_1(p)$ such that $B_{4\delta^{-1}}(q)$ is $\delta$-Gromov Hausdorff close to $\dR^{n-4}\times C(Z)$, where $C(Z)$ is a cone metric space.  Let us begin with the following:\\

{\bf Claim 1:}  For $\eta\leq \eta(n,\rv,\delta)$ sufficiently small we have that $Z=S^3/\Gamma$ is a smooth three manifold.\\

Indeed, assume this is not the case, then by applying Theorem \ref{t:content_cone_splitting} we have a sequence $(M^n_j,g_j,q_j)$ such that $B_{\delta_j^{-1}}(q_j)$ is $\delta_j$-GH close to $\dR^{n-4}\times C(Z_j)$, but we do not have that $B_{4\delta^{-1}}(q_j)$ is $\delta$-GH close to $\dR^{n-4}\times C(S^3/\Gamma)$.  By passing to a limit we have that $M_j\to X\equiv \dR^{n-4}\times C(Z_\infty)$.  Let us first see that $Z_\infty$ is a smooth manifold.  Indeed, if $Z_\infty$ were not smooth, then we see that the singular set of $X$ has codimension at most three.  However, by \cite{CheegerNaber_Codimensionfour} the singular set of $X$ is at most of codimension four, which tells us that $Z_\infty$ is indeed smooth.  Additionally, since $\eta_j\to 0$ we have that $X$ is Ricci flat on its smooth component.  It is not difficult to check for a smooth manifold $Z_\infty$ that $C(Z_\infty)$ is Ricci flat iff $Z_\infty$ has constant sectional curvature $1$, which in particular implies that $Z_\infty=S^3/\Gamma$.  Therefore for far enough in the sequence this means we have that $B_{4\delta^{-1}}(q_j)$ is $\delta$-GH close to $\dR^{n-4}\times C(S^3/\Gamma)$, which is a contradiction and so proves the claim.$\square$\\

Finally, to finish the proof we now distinguish between the cases when $\Gamma$ is trivial or nontrivial.  If $\Gamma$ is trivial then $B_{4\delta^{-1}}(q)$ is $\delta$-GH close to $\dR^n$, which by theorem \ref{t:eps_reg} implies $r_h(q)>2\delta^{-1}$, and so $r_h(p)>\delta^{-1}$ as in case $(1)$.  On the other hand if $\Gamma$ is nontrivial then this is exactly case $(2)$, which finishes the proof.
\end{proof}
\vspace{.5cm}

\subsection{Weak $c$-Ball Covering}

Now recall that the essence of a $(\delta,\tau)$-neck region $\cN$ is that $\cN$ is a smooth region which on many scales looks like $\dR^{n-4}\times C(S^3/\Gamma)$.  In order to prove the decomposition of Theorem \ref{t:neck_decomposition} with the sharp $n-4$ content estimates, it will be important to build neck regions which are in some sense maximal.  That is, if in the construction we dealt with neck regions which could be strictly enlarged, then it might be possible to build the covering of Theorem \ref{t:neck_decomposition} badly, so that the content estimates fail.\\

With this in mind, in the next two subsections we decompose a ball which begins with many points that are volume pinched on many scales, which is to say we decompose a $c$-ball in the notation of Section \ref{ss:neck_decomp:notation}.  The resulting decomposition will produce one large $(\delta,\tau)$-neck region, together with an $n-4$ content of smaller balls.  The $(\delta,\tau)$-neck region will satisfy a form of maximal condition preventing it from being extended further down, and the remaining balls will be such that for most points there is a large drop in the volume ratio.\\

Now this decomposition will take place over the next two subsections.  The reason being is that in this section we will first  decompose a $c$-ball to produce one large {\it weak} neck region, as well as an $n-4$ content of smaller balls.  Then in the next subsection we will see how to refine the weak neck region into a standard neck region.  Roughly, a weak neck region is a neck region except we do not insist that the lipschitz condition $(n4)$ of Definition \ref{d:neck} holds.  That is, we do not have control on how the ball sizes vary.  The advantage of first building a weak neck region is that the construction, if somewhat technical as always, is straightforward.  If one has enough symmetries you cover the approximate $\dR^{n-4}$-singular set coming from a Gromov-Hausdorff map with some balls of one only scale smaller, then on each new ball you ask if one of the other conditions of the neck region has suddenly failed.  If so you stop, if not you show that again there must be a sufficient number of symmetries to recover the approximate $\dR^{n-4}$-singular set of the new ball, repeat...  The disadvantage of such a construction is that the ball sizes will have no apriori control and can vary fairly wildly, an issue which we will deal with in the next subsection with slightly more finesse.  To make all of this precise, let us begin by defining what we mean by a {\it weak} neck region:\\

\begin{definition}\label{d:weak_neck}
We call $\tilde \cN\subseteq B_2(p)$ a weak $(\delta,\tau)$-neck region if there exists a closed subset $\tilde\cC = \tilde\cC_0\cup \tilde\cC_+=\tilde\cC_0\cup\{x_i\}\subset B_2(p)$ with $p\in\tilde\cC$ and a radius function $r:\tilde\cC\to \dR^+$ with $r_x>0$ on $\tilde\cC_+$ and $r_x=0$ on $\tilde\cC_0$ such that $\tilde\cN \equiv B_2\setminus \overline B_{r_x}(\cC)$ satisfies
\begin{enumerate}
	\item[($\tilde n1$)] $\{B_{\tau^{2} r_x}(x)\}$ are pairwise disjoint. 
	\item[($\tilde n2$)] For each $r_x\leq r\leq 1$ there exists a $\delta r$-GH map $\iota_{x,r}:B_{\delta^{-1}r}(0^{n-4},y_c)\subseteq\dR^{n-4}\times C(S^3/\Gamma)\to B_{\delta^{-1}r}(x)$, where $\Gamma\subseteq O(4)$ is nontrivial.
	\item[($\tilde n3$)] For each $r_x\leq r$ with $B_{2r}(x)\subseteq B_2(p)$ we have that $\cL_{x,r}\equiv \iota_{x,r}\big(B_r(0^{n-4})\times\{y_c\}\big)\subseteq B_{\tau(r+r_y)}(\cC)$ and $\cC\cap B_r(x)\subseteq B_{\delta r}(\cL_{x,r})$, where $B_{\tau (r+r_y)}(\cC)=\cup_{y\in \cC}B_{\tau (r+r_y)}(y)$. 
\end{enumerate}	

\end{definition}
\begin{remark}
For precision sake, a weak neck region not only eliminates 	the lipschitz condition of $(n4)$, but condition $(n3)$ has changed slightly in order to accurately take this into account.
\end{remark}
\begin{remark}
It is worth keeping an example in mind, as this clarifies the definition.  Thus if one considered Example \ref{ex:neck_region} in the outline, then an example of a weak neck region is given by letting $r_x$ be any positive function which is not necessarily lipschitz.
\end{remark}

That is, a weak neck region is one for which all the usual conditions of a neck region hold, except that possibly the ball radii are not varying in a manner which has small lipschitz constant.  The effect of this is that it is possible there are nearby balls of uncontrollably different sizes.  In practice, this will be quite inconvenient for the analysis, and therefore in the next subsection we will see how to further refine the weak $(\delta,\tau)$-neck constructed in this subsection into an actual $(\delta,\tau)$-neck.  The main result of this subsection is the following:\\

\begin{proposition}[Weak $c$-Ball Covering]\label{p:neck_decomp:c_ball}
Let $(M^n_i,g_i,p_i)\to (X,d,p)$ satisfy $\Vol(B_1(p_i))>\rv>0$ with $\delta>0$ and $0<\epsilon,\tau<\tau(n)$, then for $\eta\leq \eta(n,\rv,\delta,\epsilon,\tau)$ let us assume the following holds:
\begin{enumerate}
\item $|\Ric|\leq \eta$,
\item If $\overline V\equiv \inf_{B_{4}(p)}\cV_{\eta^{-1}}(y)$ and $\cE_\eta \equiv \{x\in B_4(p): \cV_\eta(x)<\overline V+\eta\}$ then $\Vol\big(B_{\delta}\cE_\eta\big)> \epsilon\cdot \delta^{4}$.
\item $r_h(p)<2$.
\end{enumerate}
	then we can write $B_1(p) \subseteq \cC_0\cup \tilde\cN \cup \bigcup_d B_{r_d}(x_d)$, where
\begin{enumerate}
	\item[(a)] $\tilde\cN = B_{11}(q)\setminus\Big(\cC_0\cup \bigcup_{\tilde{d}} B_{r_{\tilde{d}}}(x_{\tilde{d}})\Big)$ is a weak $(\delta,\tau)$-neck region, where $q\in B_4(p)$.
	\item[(d)] For each ${\tilde{d}}$-ball $B_{r_{\tilde{d}}}(x_{\tilde{d}})$ if $\cE_{\tilde{d}}\equiv \{x\in B_{4r_{\tilde{d}}}(x_{\tilde{d}}): \cV_{\eta r_{\tilde{d}}}(x)<\overline V+\eta\}$, then we have $\Vol\big(B_{\delta r_{\tilde{d}}}\cE_{\tilde{d}}\big)\leq \epsilon\cdot \delta^4\,r_{\tilde{d}}^n$.
\end{enumerate}
\end{proposition}
\begin{remark}
Condition (2) above says that $B_1(p)$ is a $c$-ball, while condition $(3)$ above says that $B_1$ is not a $b$-ball.	 
\end{remark}

\begin{proof}
	Let us consider $\delta'\le \delta'(n,\rv,\tau,\delta)<<\delta$ to be fixed later. Denote $s=10^{-10n}$. To prove the result we will build on $B_2(p)$ a sequence of weak $(\delta,\tau)$-necks $\tilde \cN_i$, where each is a refinement of the last, with the eventual goal of arriving at a weak $(\delta,\tau)$-neck which cannot be extended any further.  Before discussing the inductive procedure for this construction, we need to build the base neck region $\cN^1$.  Precisely, by conditions (2) and (3) we can, for $\eta\leq\eta(n,\delta',\epsilon,\tau)$, apply Theorem \ref{t:n-4_content_conesplitting} to get that $B_{\delta'^{-1}}(p')$ is  $\delta'$-GH close to $\dR^{n-4}\times C(S^3/\Gamma)$ for some $p'\in B_4(p)$.  In particular, $B_{12}(p')$ is $(n-4,\delta')$-symmetric with respect to some $\cL^1\equiv \iota^1\big(\dR^{n-4}\times \{y_0\}\big)\cap B_{12}(p')$, where $\iota^1:B_{\delta'^{-1}}(0^{n-4},y_0)\subseteq \dR^{n-4}\times C(S^3/\Gamma)\to B_{\delta'^{-1}}(p')$ is the Gromov-Hausdorff map.  Notice that $\cL^1$ acts as an approximate singular set on the ball.  Let us now consider a covering
	\begin{align}
		\cL^1\subseteq \bigcup B_{\tau^{3/2}\cdot s}(x^1_f)\subseteq  \bigcup B_{s}(x^1_f)\, ,
	\end{align}
	where $x^1_f\in \cL^1$ and $\{B_{\tau^{5/3}\cdot s}(x^1_f)\}$ is a maximal disjoint collection.  We then define
	\begin{align}
		\tilde\cN_1 \equiv B_{12}(p')\setminus \bigcup B_{s}(x^1_f)\, ,
	\end{align}
	and it is easy to check that for $\eta\leq\eta(n,\epsilon,\tau,\delta')$ this is indeed a weak $(\delta,\tau)$-neck.  Before moving on to the inductive construction let us further separate the singular balls into a couple of better groups.  To accomplish this let us recall that for a ball $B_r(x)$ we can define the $\eta$-pinched subset $\cE_\eta$ by
	\begin{align}
		\cE_\eta(x,r)\equiv \{y\in B_{4r}(x): \cV_{\eta r}(y)<\overline V + \eta\}\, .
	\end{align}
	Thus for each ball $B_{s}(x^1_f)$ in our covering if we consider the pinched points $\cE_\eta(x^1_f,s)$ then we can consider one of two cases:
	\begin{enumerate}
		\item[(c)] $\Vol\big(B_{\delta\cdot s}\cE_\eta(x^1_f,s)\big)> \epsilon\, \delta^4\,s^{n}$,
		\item[(${\tilde{d}}$)] $\Vol\big(B_{\delta\cdot s}\cE_\eta(x^1_f, s)\big)\leq \epsilon\, \delta^4\, s^{n}$.
	\end{enumerate}
	Of course, in case $(c)$ above we have that $B_{s}(x^1_f)$ is a $c$-ball, while in case $({\tilde{d}})$-above we have that $B_{s}(x^1_f)$ is a ${\tilde{d}}$-ball.  Therefore, let us break our balls into these two subsets and write
	\begin{align}
		\tilde\cN_1 \equiv B_{12}(p')\setminus\Big(  \bigcup_c B_{s}(x^1_c)\cup \bigcup_{\tilde{d}}  B_{s}(x^1_{\tilde{d}})\Big)\, .
	\end{align}
	
	Before discusing the inductive procedure, let us remark an obvious fact which we use several times in the following construction: if a ball is close to a  nontrivial cone $\mathbb{R}^{n-4}\times C(S^3/\Gamma)$ at a given scale then the ball is also close to a nontrivial cone at next scale and the approximated cone vertices are close in distance by $\epsilon$-regularity Theorem \ref{t:eps_reg}. 
	Now let us move on to discuss the inductive procedure for the construction of the weak $(\delta,\tau)$-necks $\tilde\cN^i$.  Thus, let us assume we have constructed the weak $(\delta,\tau)$-neck $\tilde\cN^i$ given by
	\begin{align}
		\tilde\cN^i \equiv  B_{12}(p')\setminus \Big( \bigcup_c B_{s^i}(x^i_c)\cup \bigcup_{j\leq i} \bigcup_{\tilde{d}} B_{s^j}(x^j_{\tilde{d}})\Big)\, ,
	\end{align}
	such that the following additional conditions hold:
	\begin{enumerate}
	\item[i.1] Each	$c$-ball $B_{s^i}(x^i_c)$ satisfies $\Vol\big(B_{\delta\cdot s^i}\cE_\eta(x^i_c,s^i)\big)> \epsilon\,\delta^4\, (s^{i})^{n}$ and there exists an approximate singular set $\cL_c^i=\iota_{i,c}(\mathbb{R}^{n-4}\times\{y_{i,c}\}\cap B_2(0^{n-4},y_{i,c}))$ where the map $\iota_{i,c}: B_{{\delta'}^{-1}s^i}(0^{n-4},y_{i,c})\cap \mathbb{R}^{n-4}\times C(S^3/\Gamma_{i,c})\to B_{{\delta'}^{-1}s^i}(\hat{x}_c^i)$  is a $\delta' s^i$-GH map for some $\hat{x}^i_c\in B_{\delta^2s^i}(x_c^i)$. 	
	\item[i.2] Each ${\tilde{d}}$-ball $B_{s^j}(x^j_{\tilde{d}})$ satisfies $\Vol\big(B_{\delta\cdot s^j}\cE_\eta(x^j_{\tilde{d}},s^j)\big)\leq \epsilon\,\delta^4\,(s^{j})^{n}$.
	\item[i.3] Denote $\tilde{\cC}_i:=\{x_c^i\}\bigcup\{x_{\tilde{d}}^j,j\le i\}$ and the associated $r_x={s}^i$ for $x=x_c^i, x_{\tilde{d}}^i$ and $r_x={s}^j$ for $x=x_{\tilde{d}}^j$. Then the balls $\{B_{\tau^{5/3}\cdot r_x}(x),  x\in \tilde{\cC}_i\}$ are pairwise disjoint. 
		\item[i.4] Every $c$-ball has radius ${s}^i$.
		\item[i.5] $\tilde{\cC}_{i-1}\subset \tilde{\cC}_i$, where the center of a $c$-ball in $\tilde{\cC}_{i-1}$ is a center of a $\tilde{d}$-ball or $c$-ball in $\tilde{\cC}_{i}$ with radius ${s}^i$,  and the center of $\tilde{d}$-ball in $\tilde{\cC}_{i-1}$ is still a $\tilde{d}$-ball in $\tilde{\cC}_{i}$ with the same radius.
		\item[i.6] Denote $\cL^{i-1}\equiv\bigcup_c \Big(\cL^{i-1}_c\cap B_{2{s}^{i-1}}(x_c^{i-1})\Big)\cap B_{12}(p')$, where $\cL^{i-1}_c$ is the approximate singular set corresponding to each $c$-ball $B_{{s}^{i-1}}(x_c^{i-1})$ as i.1. 
		Then the centers $\tilde{\cC}_i\setminus \tilde{\cC}_{i-1}$ with radius ${s}^i$ are chosen such that  $\tilde{\cC}_i\setminus \tilde{\cC}_{i-1}\subset \cL^{i-1}\setminus\Big(\cup_{x_c^{i-1}}B_{2\tau^{5/3}\cdot {s}^i}(x_c^{i-1}) \bigcup_{j\leq i-1} B_{2\tau^{5/3}\cdot{s}^j}(x^j_{\tilde{d}})\Big)$ and $\cL^{i-1}\setminus\Big(\cup_{x_c^{i-1}}B_{2\tau^{5/3}\cdot {s}^i}(x_c^{i-1}) \bigcup_{j\leq i-1} B_{2\tau^{5/3}\cdot{s}^j}(x^j_{\tilde{d}})\Big)\subseteq B_{\tau^{3/2}\cdot {s}^{i}}(\tilde{\cC}_i\setminus\tilde{\cC}_{i-1})$.
	\item[i.7] $r_h(x)\le 2r_x$ for any $x\in \tilde{\cC}_i$.
	\end{enumerate}

We therefore wish to build a weak $(\delta,\tau)$-neck $\tilde\cN^{i+1}$ which satisfies $(i+1).1$-$(i+1).7$.  In order to build this weak neck region we must therefore break apart the $c$-balls in our covering.  
First we can consider the total approximate singular set on scale ${s}^i$ given by
	\begin{align}
		\cL^i\equiv \bigcup_c \cL^i_c \cap B_{2{s}^i}(x_c^i)\cap B_{12}(p')\, .
	\end{align}
	
	 Let us now consider a collection of balls $\{B_{{s}^{i+1}}(x^{i+1}_f)\}$ such that
	\begin{align}
		&\cL^{i}\setminus\Big(\cup_{x_c^{i}}B_{2\tau^{5/3}\cdot {s}^{i+1}}(x_c^{i}) \bigcup_{j\leq i} B_{2\tau^{5/3}\cdot{s}^j}(x^j_{\tilde{d}})\Big)\subseteq \bigcup_f B_{\tau^{3/2}\cdot{s}^{i+1}}(x^{i+1}_f)\, ,\notag\\
		&x^{i+1}_f\in \cL^{i}\setminus\Big(\cup_{x_c^{i}}B_{2\tau^{5/3}\cdot {s}^{i+1}}(x_c^{i}) \bigcup_{j\leq i} B_{2\tau^{5/3}\cdot{s}^j}(x^j_{\tilde{d}})\Big)\, ,\notag\\
		&\big\{B_{\tau^{5/3}\cdot{s}^{i+1}}(x^{i+1}_f)\big\} \text{ are disjoint }\, .
	\end{align}
	
	Now each ball $\big\{B_{{s}^{i+1}}(x^{i+1}_f), B_{{s}^{i+1}}(x_c^i)\big\}$  is either a $c$ ball which satisfies $\Vol\big(B_{\delta\cdot {s}^{i+1}}\cE_\eta(x^{i+1}_f,{s}^{i+1})\big)> \epsilon\,\delta^4\, ({s}^{i+1})^{n}$ with $x_f^{i+1}\in B_{\delta^2\cdot {s}^{i+1}}\cL_f^{i+1}$ by $\epsilon$-regularity, or is a ${\tilde{d}}$-ball which satisfies the converse inequality $\Vol\big(B_{\delta\cdot {s}^{i+1}}\cE_\eta(x^{i+1}_f,{s}^{i+1})\big)\leq \epsilon\,\delta^4\,({s}^{i+1})^{n}$.  Thus we can separate the balls into these two types by
	\begin{align}
		\big\{B_{{s}^{i+1}}(x^{i+1}_f),B_{{s}^{i+1}}(x_c^i)\big\} = \big\{B_{{s}^{i+1}}(x^{i+1}_c)\big\}_c \cup \big\{B_{{s}^{i+1}}(x^{i+1}_{\tilde{d}})\big\}_{\tilde{d}}\, .
	\end{align}
	This allows us to define the region $\tilde\cN^{i+1}$ by
	\begin{align}
		\tilde\cN^{i+1} \equiv B_{12}(p')\setminus \Big( \bigcup_c B_{{s}^{i+1}}(x^{i+1}_c)\cup \bigcup_{j\leq i+1} \bigcup_{\tilde{d}} B_{{s}^j}(x^j_{\tilde{d}})\Big)\, .
	\end{align}
	It is easy to check from the construction if $\eta\le \eta(\delta',n,\tau,\rv)$ that the region $\tilde{\cN}^{i+1}$ satisfies $(i+1).1$-$(i+1).7$, where $(i+1).7$ follows by the inductive fact that the ball $B_{{s}^i}(x_c^i)$ is close to a nontrivial cone $\mathbb{R}^{n-4}\times C(S^3/\Gamma)$.\\

Let us now check that if $\delta'\le \delta'(\delta,\tau,n,\rv)$ and $\eta\le \eta(\delta',\delta,n,\tau,\rv)$ the region $\tilde{\cN}^{i+1}$ is indeed a weak $(\delta,\tau)$-neck region as defined in Definition \ref{d:weak_neck} except the condition $p'\in \tilde\cC$. The condition ($\tilde{n}1$) follows directly from (i+1).3. To see the condition ($\tilde{n}2$) and ($\tilde{n}3$), for any $x\in \tilde{\cC}_{i+1}$ and $1\ge r\ge r_x={s}^j$, let us assume ${s}^{j_0}\le r<{s}^{j_0-1}$ with $1\le j_0\le j$. By the inductive scheme, there exists $\{x_c^{k}\in \tilde{\cC}_{k},k=j-1,\cdots,j_0-1\}$ such that 
\begin{align}\label{e:xcjchainrule}
x\in \cL_c^{j-1}\cap B_{{s}^{j-1}}(x_c^{j-1}),\, ~~x_c^{j-1}\in \cL_c^{j-2}\cap B_{{s}^{j-2}}(x_c^{j-2}),\, \cdots,\, x_c^{j_0}\in \cL_c^{j_0-1}\cap B_{{s}^{j_0-1}}(x_c^{j_0-1}).
\end{align}
By the cone-splitting theorem \ref{t:cone_splitting}, the $\epsilon$-regularity theorem \ref{t:eps_reg} and condition $k.7$ for $k=j-1,\cdots,j_0$, if $\eta\le \eta(n,\delta',\tau,\rv)$,  then $\cL_c^{k}\subset B_{\delta'{s}^{k+2}}\cL_c^{k-1}$. This implies that 
\begin{align}\label{e:xdistanceLcj0-1}
d(x,\cL_c^{j_0-1})\le \sum_{k=j_0}^{j-1}\delta'{s}^{k+2}\le \delta'{s}^{j_0+1}.
\end{align}
If $\delta'\le\delta'(n,\tau,\delta,\rv)$ and $\eta\le \eta(\delta',\delta,n,\tau,\rv)$,  the above gives ($\tilde{n}2$). To check condition ($\tilde{n}3$), by the inductive scheme $k.5$ for $k=i,\cdots,j_0$ we get 
\begin{align}\label{e:coveringtildecCj0}
\tilde{\cC}_{j_0}\subset \tilde{\cC}_{j_0+1}\subset \cdots \subset\tilde{\cC}_{i+1}.
\end{align}
By \eqref{e:xdistanceLcj0-1}, the cone-splitting theorem \ref{t:cone_splitting}, the $\epsilon$-regularity theorem \ref{t:eps_reg} and condition ($j_0-1$).7,  if $\eta\le \eta(\delta',\delta,n,\tau,\rv)$ we have 
\begin{align}\label{e:coveringtildecCj01}
\cL_{x,r}\subset B_{\delta' {s}^{j_0}}\cL_c^{j_0-1}, 
\end{align}
where $\cL_{x,r}$ is defined in condition ($\tilde{n}3$).  On the other hand, by condition $j_0$.6, we have 
\begin{align}
\cL_c^{j_0-1}\subset B_{\tau^{3/2}{s}^{j_0}}(\tilde{\cC}_{j_0})\cup_{x_d^j\in \tilde{\cC}_{j_0}} B_{\tau^{3/2}{s}^{j}}(x_d^j).
\end{align}
By the definition of $r_x$ for $x\in \tilde{\cC}_{i+1}$, this is equivalent to 
\begin{align}\label{e:coveringtildecCj02}
\cL_c^{j_0-1}\subset B_{\tau^{3/2}{s}^{j_0}}(\tilde{\cC}_{j_0}) \cup_{y\in \tilde{\cC}_{j_0} \cap \tilde{\cC}_{i+1}} B_{\tau^{3/2}r_y}(y).
\end{align}
Combining \eqref{e:coveringtildecCj0}, \eqref{e:coveringtildecCj01} and \eqref{e:coveringtildecCj02}, we get 
\begin{align}
\cL_{x,r}\subset B_{\tau^{4/3} {s}^{j_0}}\tilde{\cC}_{i+1}\cup_{y\in\tilde{\cC}_{i+1}} B_{\tau^{4/3} r_y}(y)\subset B_{\tau^{4/3} (r+r_y)}(\tilde{\cC}_{i+1}),
\end{align}
which proves one of the two covering conditions in ($\tilde{n}3$) of Definition \ref{d:weak_neck}.  For another covering condition, we will argue by contradiction.  Assume there exists $y\in \tilde{\cC}_{i+1}\cap B_r(x)\setminus B_{\delta r}(\cL_{x,r})$. If $r_y\le \tau^{-100}r$ then by condition ($\tilde{n}2$), which we have checked, if $\eta\le \eta(n,\rv,\tau,\delta,\delta')$ the ball $B_r(y)$ must be close to a metric cone. This contradicts to $(i+1).7$ by using cone-splitting theorem \ref{t:cone_splitting} and $\epsilon$-regularity Theorem \ref{t:eps_reg}. On the other hand, if $r_y\ge \tau^{-100}r$ then since $y\in B_r(x)$ we have $B_{r_x}(x)\subset B_r(x)\subset B_{\tau^3r_y}(y)$ which contradicts the fact that $B_{\tau^2 r_x}(x)\cap B_{\tau^2 r_y}(y)=\emptyset$. Thus we have shown $\tilde{\cN}^{i+1}$ is a weak $(\delta,\tau)$-neck region.

	Now to finish the proof of the Proposition let us consider the closed discrete sets $\cC^{i}_c\equiv \bigcup_c \{x^{i}_c\}$ and note by construction that $\cC^{i+1}_c\subseteq B_{{s}^i}(\cC^i_c)$.  Therefore we can define the Hausdorff limit
	\begin{align}
		\cC_0=\lim_{i\to \infty} \cC^{i+1}_c\, .
	\end{align}
Taking the weak neck regions $\tilde \cN^i$ and letting $i\to\infty$ we arrive at a  region
	\begin{align}
		\tilde\cN \equiv B_{12}(p')\setminus \Big( \cC_0\cup \bigcup_{j} \bigcup_{\tilde{d}} B_{{s}^j}(x^j_{\tilde{d}})\Big)\, .
	\end{align}
	One can check $\tilde{\cN}$ is a weak $(\delta,\tau)$-neck region except the condition $p'\in \tilde\cC$ by using the same argument as above. Now choose a point $q\in B_4(p)\cap \tilde\cC$ we have completed the proof.
\end{proof}

\vspace{.5cm}

\subsection{Strong $c$-Ball Covering and Construction of Maximal Neck Regions}

In this subsection we produce our more refined covering of a $c$-ball, which will turn our weak neck regions into actual neck regions.  The process goes roughly as follows.  We begin with a $(\delta',\tau)$-weak neck region for $\delta'<<\delta$, built in the last subsection.  Recall that this weak neck region satisfies all the properties of a neck region except the balls do not satisfy the strong lipschitz estimate of $(n4)$.  However, the Vitali covering property of the ball covering from the weak neck region automatically makes the radius function of the weak neck region lipschitz with some {\it large} lipschitz constant, so one can multiply the radius function of the weak neck region by a very small constant in order to make it satisfy $(n4)$.  This is at the expense that it does not actually cover the approximate singular set anymore.  A little bit of footwork will allow us to extend this new radius function into one which does cover, and thus give rise to an actual neck region.  Now the most challenging aspect of all of this is that during this process we have apriori lost the weak neck regions property that this covering is a maximal covering, that is we do not know some property of the neck region need fail on these balls so that we cannot extend it further.  This maximal property of the construction will be crucial for the end decomposition theorem.  What can be shown however, is that while not every ball in the new covering is maximal, indeed all but a small content will be maximal, and this will be good enough in the end decomposition theorem.\\

Our precise theorem for this subsection is the following:\\

\begin{proposition}[$c$-Ball Covering]\label{p:neck_decomp:c_ball_strong}
Let $(M^n_i,g_i,p_i)\to (X,d,p)$ satisfy $\Vol(B_1(p_i))>\rv>0$ with $\delta>0$ and $0<\epsilon,\tau<\tau(n)$, then for $\eta\leq \eta(n,\rv,\delta,\epsilon,\tau)$ let us assume the following holds:
\begin{enumerate}
\item $|\Ric|\leq \eta$,
\item If $\overline V\equiv \inf_{B_{4}(p)}\cV_{\eta^{-1}}(y)$ and $\cE_\eta \equiv \{x\in B_4(p): \cV_\eta(x)<\overline V+\eta\}$ then $\Vol\big(B_{\delta}\cE_\eta\big)> \epsilon\cdot \delta^{4}$.
\item $r_h(p)<2$.
\end{enumerate}
	then we can write $B_1(p) \subseteq \cC_0\cup\cN \cup \bigcup_b B_{r_b}(x_b)\cup \bigcup_c B_{r_c}(x_c)\cup \bigcup_d B_{r_d}(x_d) \cup \bigcup_e B_{r_e}(x_e)$, where
\begin{enumerate}
	\item[(a)] $\cN = B_{11}(q)\setminus \Big(\cC_0\cup\bigcup_b B_{r_b}(x_b)\cup \bigcup_c B_{r_c}(x_c)\cup \bigcup_d B_{r_d}(x_d) \cup \bigcup_e B_{r_e}(x_e)\Big)$ is a $(\delta,\tau)$-neck region and $\cC_0\subset S(X)$, where $q\in B_4(p)$.
	\item[(b)] For each $b$-ball $B_{r_b}(x_b)$ we have $r_h(x_b)>2r_b$,
	\item[(c)] For each $c$-ball $B_{r_d}(x_c)$, if $\cE_c\equiv \{x\in B_{4r_c}(x_c): \cV_{\eta r_c}(x)<\overline V+\eta\}$, then we have $\Vol\big(B_{\delta r_c}\cE_c\big)> \epsilon\cdot\delta^4\,r_c^{n}$.
	\item[(d)] For each $d$-ball $B_{r_d}(x_d)$, if $\cE_d\equiv \{x\in B_{4r_d}(x_d): \cV_{\eta r_d}(x)<\overline V+\eta\}$, then we have $0<\Vol\big(B_{\delta r_d}\cE_d\big)\leq \epsilon\cdot\delta^4\,r_d^{n}$.
	\item[(e)] For each $e$-ball $B_{r_e}(x_e)$, if $\cE_e\equiv \{x\in B_{4r_e}(x_e): \cV_{\eta r_e}(x)<\overline V+\eta\}$, then we have that
	 $\cE_e=\emptyset$.
\end{enumerate}
Further, we have the content bounds
\begin{align}
&\sum_{x_b\in B_{5}(q)} r_b^{n-4}+\sum_{x_d\in B_{5}(q)} r_d^{n-4}+\sum_{x_e\in B_{5}(q)} r_e^{n-4}<C(n,\tau)	\, ,\notag\\
&\sum_{x_c\in B_{5}(q)} r_c^{n-4}<C(n,\tau)\epsilon\, .
\end{align}
\end{proposition}

\begin{proof}
Let us begin the proof by applying Proposition \ref{p:neck_decomp:c_ball} with $(\delta',\tau)$ where $\delta'\leq \delta'(n,\rv,\delta,\tau)$ in order to get a covering
\begin{align}
	B_1(p) \subseteq \tilde\cC_0\cup\tilde\cN \cup \bigcup_{{\tilde{d}}\in \tilde\cC_+} B_{ \tilde{r}_{\tilde{d}}}(x_{\tilde{d}})\, ,
\end{align}
where $\tilde \cN\equiv B_{11}(q)\setminus \tilde\cC_0\cup\bigcup_{\tilde \cC_+} B_{\tilde{r}_{\tilde{d}}}( x_{\tilde{d}})$ is a weak $(\delta',{\tau})$-neck region with $q\in B_4(p)$, and $B_{ \tilde{r}_{\tilde{d}}}( x_{\tilde{d}})$ are all ${\tilde{d}}$-balls with respect to $\eta'$ as given from Proposition \ref{p:neck_decomp:c_ball}.  We will eventually pick our $\eta<<\eta'$. In the proof below we will use $\tilde{r}$ to represent the radius of a ball with respect to $\tilde{\cN}$.

Now since $\tilde \cN$ is a weak neck region, we have in particular that the collection of balls $\{B_{{\tau}^2  {\tilde r}_{\tilde{d}}}(x_{\tilde{d}})\}$ are all disjoint.  Now let us define the following radius function:
\begin{align}
	r_y\equiv \begin{cases}
 				\delta^2{\tau}^3  {\tilde r}_{\tilde{d}} & \text{ if }y\in \overline B_{{\tau}^3  {\tilde r}_{\tilde{d}}}( x_{\tilde{d}})\, ,\notag\\
 				\delta^2\,d(y,\tilde\cC) & \text{ if }y\not\in \bigcup B_{{\tau}^3 {\tilde r}_{\tilde{d}}}( x_{\tilde{d}})\, .
              \end{cases}
\end{align}
Let us note that we have the gradient bound $|\nabla r_y|\leq \delta^2$, and $r_y\ge \delta^2 d(y,\tilde{\cC})$. \\

In order to build our new covering let us make several observations.  To begin with for each $y\in B_2$ let $\tilde y\in \tilde\cC$ denote the center point which minimizes $d(y,\tilde\cC)$.  Recalling from Definition \ref{d:neck} that $\cL_{x,r}=\iota_{x,r}\Big(B_r(0^{n-4})\times\{y_c\}\Big)$ is the effective singular set on $B_r(x)$, we define the set for $\tilde\delta=\tilde\delta(\delta,\rv,n,\tau)<\delta$ to be fixed later that

\begin{align}\label{e:definitiontildecS}
\tilde \cS \equiv \big\{y\in B_{11}(q): y\in B_{\tilde\delta^2 r_y}\cL_{\tilde y,\max\{2d(y,\tilde\cC), 2{\tilde r}_{\tilde{y}}\} }\big\}\, ,
\end{align}
which is roughly to say that $\tilde \cS$ are the set of points which belong to the effective singular set on their own scale $r_y$.  Our first claim, which is immediate once you untangle the constants in the definition of $\tilde \cS$, is that points of $\tilde \cS$ look very close to $\dR^{n-4}\times C(S^3/\Gamma)$ at scales bigger than $r_y$.  Precisely,\\

{\bf Claim:} For any given $\tilde\delta>0$, let $y\in \tilde \cS$ with $\delta'\leq \delta'(n,\rv,\tau,\tilde\delta)$.  Then for all $r_y<r\leq 2$ we have that $B_{\tilde\delta^{-1}r}(y)$ is pointed $\tilde{\delta} r$-GH close to $B_{\tilde{\delta}^{-1}r}(0^{n-4},y_c)\subseteq \dR^{n-4}\times C(S^3/\Gamma)$. $\square$\\

Let us now consider the covering of $\tilde \cS$ given by
\begin{align}
\tilde \cS\subseteq { \tilde\cC_0}\cup\bigcup_{y\in\tilde \cS} B_{\tau^{3/2}r_y}(y)\, ,
\end{align}
and let $\{B_{\tau^{3/2}r_i}(y_i)\}$ be a maximal Vitali subcovering such that $\{B_{\tau^2 r_i}(y_i)\}$ are disjoint and $\tilde{\cC}\subset \cC:=\tilde{\cC}_0\cup \{y_i\}$.  If we define
\begin{align}
\cN \equiv B_{11}(q)\setminus \tilde\cC_0\cup\bigcup_{i} \overline B_{r_i}(y_i)\, ,
\end{align}
then it follows from almost the same argument as Proposition \ref{p:neck_decomp:c_ball}  that $\cN$ is a $(\delta,\tau)$-neck region if $\tilde\delta\le \tilde\delta(n,\rv,\tau,\delta)$ and  $\delta'\leq \delta'(n,\rv,\tau,\delta)$. Actually, noting the above Claim,  let us validate the only nontrivial condition (n3).  For any $x\in \cC$ and $1\ge r\ge r_x$ with $B_{2r}(x)\subset B_{11}(q)$, we will show $\cL_{x,r}\cap B_r(x)\subset B_{\tau r}(\cC)$ and $\cC\cap B_r(x)\subset B_{\delta r}(\cL_{x,r})$ where $\cL_{x,r}$ is the approximate singular set corresponding to the $(n-4,\tilde\delta )$-symmetric ball $B_r(x)$ from the above Claim, see also Definition \ref{d:quant_symmetry}.  

Let us begin the proof of $\cL_{x,r}\cap B_r(x)\subset B_{\tau r}(\cC)$. For any $y\in\cL_{x,r}\cap B_r(x)$, we have 
\begin{align}\label{e:ryrxtildecC}
r_y&\le r_x+\delta^2 d(x,y)\le (1+\delta^2)r.
\end{align}
Thus by the definition of $r_y$ we have $\tilde{r}_{\tilde y}\le \delta^{-2}\tau^{-3}r_y\le (1+\delta^{-2})\tau^{-3}r.$ Since $|{\rm Lip}r_x|\le \delta^2$ and $\tilde{\cS}\subset  \tilde{\cC}_0\cup_{x\in \cC}B_{\tau^{3/2}r_x}(x)$ we have 
	$\tilde{\cS}\cap B_{2r}(x)\subset  B_{\tau^{5/4}r}(\cC)$.
Hence, to show $y\in B_{\tau r}(\cC)$ it will suffice to show that $B_{\tau^3r}(y)\cap \tilde{\cS}\ne\emptyset$.

 If $d(y,\tilde{\cC})=d(y,\tilde{y})\le \tau r$, then $y\in B_{\tau r}(\cC)$ and we are done. 
Let us assume $d(y,\tilde{y})\ge \tau r$. We thus have $r_y\ge \delta^2 d(y,\tilde{\cC})\ge \delta^2\tau r$. On the other hand,  by \eqref{e:ryrxtildecC} we have $r_y\le (1+\delta^2)r$.
That means $r_y$ and $r$ have comparable sizes and $\tilde{r}_{\tilde{y}}\le (1+\delta^{-2})\tau^{-3}r$. For any $z\in B_{\tau^3r}(y)$, we have $r_z\le r_y+\delta^2 \tau^3r\le (1+2\delta^2)r$ and $r_z\ge r_y-\delta^2\tau^3r\ge \delta^2\tau r/2$. Assume $\tilde{z}\in \tilde{\cC}$ such that $d(z,\tilde{\cC})=d(z,\tilde{z})$ then by the definition of $r_z$ we have $\tilde{r}_{\tilde{z}}\le \delta^{-2}\tau^{-3}r_z\le (2+\delta^{-2})\tau^{-3}r$ and 
\begin{align}
d(\tilde{y},\tilde{z})\le d(\tilde{y},z)+d(z,\tilde{z})\le 2d(\tilde{y},z)\le 2d(\tilde y, y)+2d(y,z)\le 2(\delta^{-2}r_y+\tau^3r)\le 2(1+\tau^3+\delta^{-2})r.	
\end{align}
Therefore, 
by the cone-splitting theorem \ref{t:cone_splitting} and $\epsilon$-regularity theorem \ref{t:eps_reg} we conclude that if $\delta'\le \delta'(n,\rv,\tau,\tilde\delta,\delta)$ then 
\begin{align}\label{e:tildeztau3r}
	\cL_{\tilde z,10\delta^{-2}\tau^{-3}r}\subset B_{\tilde \delta^5 r}(\cL_{\tilde y,30\delta^{-2}\tau^{-3}r} ) \text{ and }\cL_{\tilde y,30\delta^{-2}\tau^{-3}r}\subset B_{\tilde \delta^5 r}(\cL_{\tilde y,50\delta^{-2}\tau^{-3}r} )
\end{align}
On the other hand, by the cone-splitting theorem \ref{t:cone_splitting} and $\epsilon$-regularity theorem \ref{t:eps_reg} if $\delta'\le \delta'(n,\rv,\tau,\tilde\delta,\delta)$  then we have  $d(y,\cL_{\tilde y,30\delta^{-2}\tau^{-3}r})\le \tau^4r$ which in particular implies $B_{\tau^{3}r}(y)\cap \cL_{\tilde y,30\delta^{-2}\tau^{-3}r}\ne \emptyset.$ Combining with \eqref{e:tildeztau3r} we have that the point $z\in B_{\tau^{3}r}(y)\cap \cL_{\tilde y,30\delta^{-2}\tau^{-3}r}$ must satisfy $z\in \tilde{\cS}$. Thus we have proved $B_{\tau^3r}(y)\cap \tilde{\cS}\ne\emptyset$. Hence we complete the proof of $\cL_{x,r}\cap B_r(x)\subset B_{\tau r}(\cC)$ if $\delta'\le \delta'(n,\rv,\tau,\tilde\delta,\delta)$. 

 For the covering $\cC\cap B_r(x)\subset B_{\delta r}(\cL_{x,r})$, this follows directly from cone-splitting theorem \ref{t:cone_splitting} and $\epsilon$-regularity Theorem \ref{t:eps_reg} if $\tilde{\delta}\le \tilde\delta(n,\rv,\tau,\delta)$ and $\delta'\le \delta'(n,\tau,\delta,\rv,\tilde\delta)$.\\

 By applying Theorem \ref{t:neck_region} we now get the estimate
\begin{align}
	H^{n-4}(\tilde\cC_0\cap B_5(q))+\sum_{x_i\in B_{5}(q)\cap \cC} r_i^{n-4}\leq C(n,\tau)\, .
\end{align}
Now certainly each ball $B_{r_i}(x_i)$ is either a $(b)\to (e)$ ball, and therefore we can write
\begin{align}
	\Big\{B_{r_i}(x_i)\Big\} = \Big\{B_{r_b}(x_b)\Big\}_b \cup \Big\{B_{r_c}(x_c)\Big\}_c \cup \Big\{B_{r_d}(x_d)\Big\}_d \cup \Big\{B_{r_e}(x_e)\Big\}_e\, .
\end{align}

The crucial aspect of the proof which is left to show that we have the estimate
\begin{align}
	\sum_{x_c\in B_{5}(q)\cap \cC} r_c^{n-4}\leq C(n,\tau)\epsilon\, ,
\end{align}
which is to say that the $c$-balls have small content.  In order to see this let us begin by observing that by the definition of $r_y$ and since $\tilde \cN$ is a weak $(\delta',\tau)$-neck we have the covering
\begin{align}\label{e:tildeScover}
\tilde \cS\subseteq \tilde\cC_0\bigcup_{{x_{\tilde{d}}}\in \tilde \cC } B_{10\tau  {\tilde r}_{\tilde{d}}}( x_{\tilde{d}})\, .	
\end{align}
To see \eqref{e:tildeScover}, for any $y\in \tilde\cS$  but not  in $\tilde\cC_0\bigcup_{{\tilde{d}}} B_{9\tau  {\tilde r}_{\tilde{d}}}( x_{\tilde{d}})$ assume $\tilde{y}\in \tilde{\cC}$ such that $d(y,\tilde{\cC})=d(y,\tilde{y})$. Then $r_y=\delta^2 d(y,\tilde{\cC})\ge 9\delta^2 \tau {\tilde r}_{\tilde{d}}$. By the property of weak $(\delta',\tau)$-neck region, we get 
\begin{align}
 \cL_{\tilde y,\max\{2d(y,\tilde\cC), 2{\tilde r}_{\tilde{y}}\} }\subset \cup_{x_{\tilde d}\in \tilde\cC}B_{4\tau (d(y,\tilde\cC)+\tilde{r}_{{\tilde{d}} })}({x}_{\tilde{d}})
 \end{align}
 Thus 
 \begin{align}
 y\in \cup_{x_{\tilde d}\in \tilde\cC}B_{4\tau (d(y,\tilde\cC+\tilde{r}_{{\tilde{d}} })+\delta^2 r_y}({x}_{\tilde{d}}),
 \end{align}
 where we use the fact that $\tilde\delta\le \delta$.
If $y\in B_{4\tau (d(y,\tilde\cC)+{{\tilde r}}_{\tilde{d} })+\delta^2 r_y}(\tilde{y})$, then $\delta^{-2}r_y=d(y,\tilde{y})\le 4\tau (d(y,\tilde\cC)+{\tilde {r}}_{\tilde{d} })+\delta^2 r_y\le 4\tau d(y,\tilde{y})+\delta^2r_y+4\delta^{-2}r_y/9$ which is a contradiction if $\delta$ is sufficiently small. Therefore, there exists some $x_{\tilde{d}}\in \tilde\cC$ such that $y\in B_{4\tau (d(y,\tilde\cC)+\tilde{r}_{{\tilde{d}} })+\delta^2 r_y}({x}_{\tilde{d}})$. Thus 
\begin{align}
\delta^{-2}r_y=d(y,\tilde{\cC})\le d(y,{x}_{\tilde{d}})\le 4\tau (d(y,\tilde\cC)+\tilde{r}_{{\tilde{d}} })+\delta^2 r_y.
\end{align}
If $\tau\le 1/8$ and $\delta\le \delta(n)$ sufficiently small, we conclude that $d(y,x_{\tilde d})<10\tau \tilde{r}_{{\tilde d}}$.  Hence $y\in \tilde\cC_0\bigcup_{x_{\tilde{d}}\in \tilde\cC} B_{10\tau {\tilde r}_{\tilde{d}}}( x_{\tilde{d}})$ which proves \eqref{e:tildeScover}.

 Using Theorem \ref{t:neck_region} we have that $\mu\equiv \sum r_i^{n-4}\delta_{x_i}$, which is the packing measure associated to the neck region $\cN$, is a doubling measure.  In particular, since $\{B_{\tau^2 {\tilde r}_{\tilde{d}}}( x_{\tilde{d}})\}$ are disjoint and $\tilde{\cC}\subset \cC$ and $r_{x_{\tilde d}}\le \delta^2 \tau^3 \tilde{r}_{\tilde d}$, this gives us
\begin{align}\label{e:strong_c_ball:1}
\sum_{ x_{\tilde{d}}\in B_{9}(q)}  {\tilde r}_{\tilde{d}}^{n-4}&\leq C(n,\tau)\sum_{ x_{\tilde{d}}\in B_{9}(q)} \mu(B_{ {\tilde r}_{\tilde{d}}}(x_{\tilde{d}})) \, ,\notag\\
&\leq C(n,\tau)\sum_{ x_{\tilde{d}}\in B_{9}(q)} \mu(B_{\tau^2 {\tilde r}_{\tilde{d}}}( x_{\tilde{d}}))\, ,\notag\\
&\leq C(n,\tau)\mu(B_{10}(q))\leq C(n,\tau)\, .
\end{align}
Now for each ball $B_{ {\tilde r}_{\tilde{d}}}( x_{\tilde{d}})$ let us consider the set
\begin{align}
C_{\tilde d}\equiv\{x_i\in \cC\cap B_{10\tau {\tilde r}_{\tilde{d}}}( x_{\tilde{d}}):  \;\text{ s.t. }B_{r_i}(x_i) \text{ is a $c$-ball}\}\, .	
\end{align}
The following is our main claim about this set:\\

{\bf Claim: } We have the estimate $\mu(C_{\tilde d})\leq C(n,\tau)\epsilon\,  {\tilde r}_{\tilde{d}}^{n-4}$.\\

In order to prove the claim let us first note that if $x_i\in C_{\tilde d}$ then $r_i< \delta^2  {\tilde r}_{\tilde{d}}$ and thus $B_{4r_i}(x_i)\subseteq B_{ {\tilde r}_{\tilde{d}}}(x_{\tilde{d}})$.  Let us also observe that if $x_i\in C_{\tilde d}$ then there exists $y_i\in B_{4r_i}(x_i)$ such that $\cV_{\eta r_i}(y_i)<\bar V+\eta<\bar V+\eta'$.  Then $y_i\in \cE_{\tilde d,\eta'}$, hence $B_{10^{-1}\delta  {\tilde r}_{\tilde{d}}}(x_i)\subset B_{5^{-1}\delta  {\tilde r}_{\tilde{d}}}(y_i)$, and therefore we have $B_{10^{-1}\delta {\tilde r}_{\tilde{d}}}(x_i)\subseteq B_{\delta  {\tilde r}_{\tilde{d}}}\cE_{\tilde d,\eta'}$.  Thus, let us choose a maximal collection of balls $\{B_{\delta  {\tilde r}_{\tilde{d}}}(x'_i)\}_1^{N'}$ with $x'_i\in C_{\tilde d}$ such that $\{B_{10^{-1}\delta {\tilde  r}_{\tilde{d}}}(x'_i)\}$ are disjoint.  Note then that because $B_{{\tilde r}_{\tilde{d}}}(x_{\tilde{d}})$ is a $\tilde d$-ball and since $B_{10^{-1}\delta  {\tilde r}_{\tilde{d}}}(x'_i)\subseteq B_{\delta  {\tilde r}_{\tilde{d}}}\cE_{\tilde d,\eta'}$ we have the estimate
\begin{align}
N'\delta^n {\tilde r}_{\tilde{d}}^n\leq \sum_{x'_i} \Vol(B_{\delta {\tilde r}_{\tilde{d}}}(x'_i))\leq C(n)\sum \Vol(B_{10^{-1}\delta  {\tilde r}_{\tilde{d}}}(x'_i)) \leq C(n)\Vol\big(B_{\delta  {\tilde r}_{\tilde{d}}}\cE_{\tilde d,\eta'}\big)	\leq C(n)\epsilon\cdot \delta^{4} {\tilde r}_{\tilde{d}}^n \, ,
\end{align}
so that we get the estimate $N'\leq C(n)\epsilon \delta^{4-n}$ as a bound for the number of balls in the covering $\{B_{\delta  {\tilde r}_{\tilde{d}}}(x'_i)\}$.  Using the Ahlfors regularity of Theorem \ref{t:neck_region} we can therefore estimate:
\begin{align}
	\mu(C_{\tilde d}) &\leq \sum_1^{N'} \mu(B_{\delta {\tilde r}_{\tilde{d}}}(x'_i))\leq C(n,\tau)\delta^{n-4} {\tilde r}_{\tilde{d}}^{n-4} N'\leq C(n,\tau)\epsilon {\tilde r}_{\tilde{d}}^{n-4}\, ,
\end{align}
which finishes the proof of the Claim. $\square$\\

Finally let us now consider the set
\begin{align}
C\equiv \{x_i\in B_{5}(q): x_i \text{ is a $c$-ball}\}\, ,	
\end{align}
then we can combine \eqref{e:strong_c_ball:1} with the previous claim and the Ahlfors regularity of Theorem \ref{t:neck_region} in order to estimate
\begin{align}
	\sum_{x_c\in B_{5}(q)} r_c^{n-4}&\leq C(n,\tau)\mu(C)\leq C(n)\sum_{\tilde x_d\in B_{9}(q)} \mu(C_{\tilde d})\notag\\
	&\leq C(n,\tau)\epsilon \sum_{ x_{\tilde{d}}\in B_{9}(q)}{\tilde r}_{\tilde{d}}^{n-4}\leq C(n,\tau)\epsilon\, ,
\end{align}
which finishes the proof of the Proposition.
\end{proof}
\vspace{.5cm}

\subsection{Refinement of Balls with Less than Maximal Symmetries}

In this subsection we will deal with balls for which the set of volume pinched points has small $n-4$ content.  That is, we will decompose the $d$-balls in this subsection.  In this case, we will simply recover our ball until we arrive at balls of other types.  However, what we will gain is that for the ball types of the new covering which are not regular balls, the content will not only be bounded, it will be small.  This is crucial, as otherwise there could be a pile of errors if one were forced to continually cover in this manner.  As we will see in subsequent sections, the smallness of the content bound will allow us to produce a converging geometric series when considering these errors.  The main result of this subsection is the following:\\

\begin{proposition}[$d$-ball Covering]\label{p:neck_decomp:d_ball}
Let $(M^n_i,g_i,p_i)\to (X,d,p)$ satisfy $\Vol(B_1(p_i))>\rv>0$ with $\delta,\tau>0$ and $0<\overline V\leq \inf_{y\in B_{4}(p)} \cV_{\eta^{-1}}(y)$, then for $\epsilon\leq \epsilon(n,\rv)$, $\eta\leq \eta(n,\bar V,\delta,\epsilon,\tau)$ let us assume the following holds:
\begin{enumerate}
\item $|\Ric|\leq \eta$,
\item If $\cE_\eta\equiv \{y\in B_{4}(p):\cV_{\eta}(y)<\bar V+\eta\}$, then $0<\Vol\big( B_{\delta}\cE_\eta\big)<\epsilon\cdot \delta^4$.
\end{enumerate}
then we can write $B_1 = \tilde S\cup\bigcup_b B_{r_b}(x_b)\cup \bigcup_c B_{r_c}(x_c) \cup \bigcup_e B_{r_e}(x_e)$, where
\begin{enumerate}
    \item[(s)] $\tilde S\subseteq \cS(X)$ and satisfies $H^{n-4}(\tilde S)=0$.
	\item[(b)] For each $b$-ball $B_{r_b}(x_b)$ we have $r_h(x_b)>2r_b$,
	\item[(c)] For each $c$-ball $B_{r_c}(x_c)$, if $\cE_c\equiv \{y\in B_{4r_c}(x_c):\cV_{\eta r_c}(y)<\bar V+\eta\}$ then $\Vol\big( B_{\delta r_c}\cE_c\big)>\epsilon\cdot \delta^4 r_c^n$.
	\item[(e)] For each $e$-ball $B_{r_e}(x_e)$, $\cE_e\equiv \{y\in B_{4r_e}(x_e):\cV_{\eta r_e}(y)<\bar V+\eta\}$ then $\cE_e = \emptyset$.
\end{enumerate}
Further, we have the content estimates $\sum_b r_b^{n-4} + \sum_e r_e^{n-4}<C(n,\delta)$ and $\sum_c r_c^{n-4}\leq C(n,\rv)\epsilon$.
\end{proposition}
\vspace{.5cm}

\begin{proof}

The procedure of the proof will be to iterate a certain covering construction and keep track of the estimates.  To illustrate this let us begin by considering a Vitali covering of $B_1(p)$ given by	
\begin{align}
	B_1(p)\subseteq \bigcup_f B_{\delta}(x^1_f)\, ,
\end{align}
with $\{B_{\delta/10}(x^1_f)\}$ disjoint.  Recall that for any ball $B_r(x)$ we can define the $\eta$-pinched subset $\cE_\eta$ given by
\begin{align}
		&\cE_\eta(x,r)\equiv \{y\in B_{4r}(x): \cV_{\eta r}(y)<\overline V + \eta\}\, .
\end{align}
In this way we wish to separate the balls $\{B_\delta(x_f)\}$ into three types, depending on whether they are $c$-balls, $d$-balls or $e$-balls based on the conditions:
 \begin{enumerate}
		\item[(c)] $\Vol\big(B_{\delta\cdot\delta}\cE_\eta(x^1_f,\delta)\big)> \epsilon\, \delta^4 \delta^{n}$,
		\item[(d)] $0<\Vol\big(B_{\delta\cdot\delta}\cE_\eta(x^1_f,\delta)\big)\leq \epsilon\, \delta^4 \delta^{n}$.
		\item[(e)] $\cE_\eta(x^1_f,\delta)=\emptyset$.
\end{enumerate}
Thus, by breaking up the balls $\{B_\delta(x_f)\}$ into these categories we can write our covering as
\begin{align}
	B_1(p)\subseteq \bigcup_{c=1}^{N^1_c} B_{\delta}(x^1_c) \cup \bigcup_{d=1}^{N^1_d} B_{\delta}(x^1_d) \cup \bigcup_{e=1}^{N^1_e} B_{\delta}(x^1_e)\, .
\end{align}

Since $B_1(p)$ is itself a $d$-ball, we can use the Vitali condition of the covering and assumption $(2)$ to conclude by a standard covering argument that
\begin{align}
	& \sum_{e=1}^{N^1_e} \delta^{n-4} \leq C(n,\delta)\, ,\notag\\
	&\sum_{c=1}^{N^1_c} \delta^{n-4} + \sum_{d=1}^{N^1_d} \delta^{n-4} \leq C(n,\rv)\epsilon\, .
\end{align}

This does not quite finish the Proposition because we still have a collection of $d$-balls in our covering, therefore we must recover them.  Let us remark that the only aspect about $B_1(p)$ used in the above covering was that $B_1(p)$ was a $d$-ball which satisfied condition $(2)$.  Thus, for each $d$-ball $B_{\delta}(x^1_d)$ let us repeat this covering process just introduced for $B_1(p)$.  If we do this for every $d$-ball then we arrive at the covering
\begin{align}
	\bigcup_{d=1}^{N^1_d} B_{\delta}(x^1_d)\subseteq  \bigcup_{c=1}^{N^2_c} B_{\delta^2}(x^2_c) \cup \bigcup_{d=1}^{N^2_d} B_{\delta^2}(x^2_d) \cup \bigcup_{e=1}^{N^2_e} B_{\delta^2}(x^2_e)\, ,
\end{align}
such that we have the estimates
\begin{align}
    & \sum_{d=1}^{N^2_d} (\delta^2)^{n-4}\leq C(n,\rv)\epsilon \sum_{d=1}^{N^1_d} \delta^{n-4}\leq \big(C(n,\rv)\epsilon\big)^2 \, ,\notag\\
    &\sum_{c=1}^{N^2_c} (\delta^2)^{n-4} \leq C(n,\rv)\epsilon\sum_{d=1}^{N^1_d} \delta^{n-4}  \leq \big(C(n,\rv)\epsilon\big)^2\, , \notag\\
	& \sum_{e=1}^{N^2_e} (\delta^2)^{n-4} \leq C(n,\delta)\sum_{d=1}^{N^1_d} \delta^{n-4}  \leq C(n,\delta)\cdot C(n,\rv)\epsilon\, .
\end{align}
Combining this with the original covering we obtain a covering of $B_1(p)$ given by
\begin{align}
		B_1(p)\subseteq &\bigcup_{c=1}^{N^1_c} B_{\delta}(x^1_c) \cup \bigcup_{e=1}^{N^1_e} B_{\delta}(x^1_e)\cup \bigcup_{c=1}^{N^2_c} B_{\delta^2}(x^2_c) \cup \bigcup_{d=1}^{N^2_d} B_{\delta^2}(x^2_d) \cup \bigcup_{e=1}^{N^2_e} B_{\delta^2}(x^2_e)\, ,\notag\\
		=&\bigcup_{d=1}^{N^2_d} B_{\delta^2}(x^2_d)\cup\bigcup_{j\leq 2} \bigcup_{c=1}^{N^j_c} B_{\delta^j}(x^j_d) \cup \bigcup_{j\leq 2} \bigcup_{e=1}^{N^j_e} B_{\delta^j}(x^j_e)\, ,
\end{align}
with
\begin{align}
	&\sum_{d=1}^{N^2_d} (\delta^2)^{n-4}\leq \big(C(n,\rv)\epsilon\big)^2\, ,\notag\\
	&\sum_{j\leq 2}\sum_{c=1}^{N^j_c} (\delta^j)^{n-4} \leq C(n,\rv)\epsilon+(C(n,\rv)\epsilon)^2\, , \notag\\
	&\sum_{j\leq 2}\sum_{e=1}^{N^j_e} (\delta^j)^{n-4}\leq C(n,\delta)\big(1+C(n,\rv)\epsilon\big)\, .
\end{align}
Of course, we may now proceed to cover the $d$-balls $\{B_{\delta^2}(x^2_d)\}$ in the same manner once again.  In fact, after repeating this inductive covering $i$ times we see we arrive at a covering
\begin{align}\label{e:d_ball:manifold}
		B_1(p)\subseteq \bigcup_{d=1}^{N^i_d} B_{\delta^i}(x^i_d)\cup\bigcup_{j\leq i} \bigcup_{c=1}^{N^j_c} B_{\delta^j}(x^j_c) \cup \bigcup_{j\leq i} \bigcup_{e=1}^{N^j_e} B_{\delta^j}(x^j_e)\, ,
\end{align}
with the estimates
\begin{align}\label{e:p:d_ball_cover:1}
	&\sum_{d=1}^{N^i_d} (\delta^i)^{n-4}\leq \big(C(n,\rv)\epsilon\big)^i\, ,\notag\\
	&\sum_{j\leq i}\sum_{c=1}^{N^j_c} (\delta^j)^{n-4} \leq \sum_{1\leq j\leq i} \big(C(n,\rv)\epsilon\big)^j\, , \notag\\
	&\sum_{j\leq i}\sum_{e=1}^{N^j_e} (\delta^j)^{n-4}\leq C(n,\delta)\sum_{0\leq j\leq i}\big(C(n,\rv)\epsilon\big)^j\, .
\end{align}
To finish the proof consider the closed discrete sets $\tilde \cS^i = \bigcup_{d=1} \{x^i_d\}$.  Note that by construction we have $\tilde \cS^{i+1}\subseteq B_{\delta^i}\big(\tilde\cS^i\big)$ and we can use \eqref{e:p:d_ball_cover:1} to conclude the estimate
\begin{align}\label{e:p:d_ball_cover:2}
	\Vol(B_{\delta^i}\big(\tilde\cS^i\big))\leq \big(C(n,\rv)\epsilon\big)^i(\delta^i)^4\, .
\end{align}
Taking the Hausdorff limit we can construct
\begin{align}
\tilde \cS = \lim_{i\to\infty} 	\tilde \cS^{i+1}\, ,
\end{align}
and we arrive at the covering
\begin{align}
		B_1(p)\subseteq \tilde\cS \cup\bigcup_{j} \bigcup_{c} B_{\delta^j}(x^j_d) \cup \bigcup_{j} \bigcup_{e} B_{\delta^j}(x^j_e)\, .
\end{align}
Using \eqref{e:p:d_ball_cover:1}, \eqref{e:p:d_ball_cover:2} and the inclusion $\tilde\cS\subseteq B_{\delta^i}\big(\tilde\cS^i\big)$ we see for $\epsilon\leq\epsilon(n,\rv)$ and $0<r\leq 1$ that we get the estimates
\begin{align}
	&r^{-4}\Vol(B_{r}\big(\tilde\cS\big))\to 0\text{ as }r\to 0\, ,\notag\\
	&\sum_{j}\sum_{c} (\delta^j)^{n-4} \leq C(n,\rv)\epsilon\, , \notag\\
	&\sum_{j\leq i}\sum_{e=1}^{N^j_e} (\delta^j)^{n-4}\leq C(n,\delta)\, ,
\end{align}
which finishes the proof of Proposition \ref{p:neck_decomp:d_ball}.
\end{proof}
\vspace{.5cm}

\subsection{Inductive Covering}

In this subsection we combine the coverings of Proposition \ref{p:neck_decomp:c_ball}, and Proposition \ref{p:neck_decomp:d_ball} into a geometric series of coverings in order to take a generic ball and produce from it a covering by balls which are either necks, regularity regions, or for which the volume ratio increases by some definite amount.\\

Let us start with the following lemma, which in terms of our ball notation will take a generic ball and produce from it a collection of $a$-balls, $b$-balls, and $e$-balls.  Combined with a modest additional recovering argument this will allow us to prove the inductive covering proposition of this subsection.  This in turn will be used in the next section to prove the neck decomposition theorem itself.  Let us begin with the following lemma:\\

\begin{lemma}\label{l:neck_decomp:inductive_cover}
Let $(M^n_i,g_i,p_i)\to (X,d,p)$ satisfy $Vol(B_1(p))>v>0$ and $|\Ric|\leq n-1$ with $0<\delta$ and $0<\tau\leq \tau(n)$ fixed constants and $\bar V\equiv \inf_{y\in B_{4}} \cV_{\eta^{-1}}(y)$.  Then there exists $\eta(n,\rv,\delta,\tau)>0$ such that we can write $$B_1 \subseteq \tilde\cS\cup\bigcup_a \big(\cC_{0,a}\cup \cN_a\cap B_{r_a}\big) \cup \bigcup_b B_{r_b}(x_b)\cup \bigcup_e B_{r_e}(x_e)\, ,$$ where
\begin{enumerate}
    \item[(a)] $\cN_a\subseteq B_{2r_a}(x_a)$ is a $(\delta,\tau)$-neck region with associated singular set $\cC_{0,a}$,
	\item[(b)] For each $b$-ball $B_{r_b}(x_b)$ we have $r_h(x_b)>2r_b$,
	\item[(d)] For each $e$-ball $B_{r_e}(x_e)$, if $\cE_e\equiv\{y\in B_{4r_e}(x_e): \cV_{\eta r_e}< \overline V +\eta\}$, then $\cE_e = \emptyset$.
	\item[(s)] $\tilde\cS\subseteq \cS(X)$ with $H^{n-4}(\tilde\cS)=0$.
\end{enumerate}
Further, we have the content bound $\sum_a r_a^{n-4}+\sum_b r_b^{n-4}+\sum_e r_e^{n-4}<C(n,\tau,\delta)$ and radius bound $r_a,r_b,r_e\leq \eta^2$.
\end{lemma}
\vspace{.5cm}
\begin{proof}

Let us pick $\eta$ as in Propositions \ref{p:neck_decomp:c_ball} and \ref{p:neck_decomp:d_ball}.  Note first that we can cover $B_1(p)$ by the Vitali covering
\begin{align}
B_1(p)\subseteq \bigcup_f B_{\eta^2}(x_f)\, ,	
\end{align}
where $\{B_{\eta^2/10}(x_f)\}$ are disjoint.  Then by a standard covering argument there are at most $C(n,\tau, \delta)$ balls in this set.  Therefore if we focus our covering on each of these balls individually, we may then take the union and not effect the content bound by more than another factor.\\

Thus, by rescaling one of these balls to scale one and focusing on it, we can assume without loss of generality that $|\Ric|\leq \eta$ on $B_1(p)$.  Now let us recall that for every ball $B_r(x)$ we can define the $\eta$-pinched points $\cE_\eta(x,r)$ by
\begin{align}
		&\cE_\eta(x,r)\equiv \{y\in B_{4r}(x): \cV_{\eta r}(y)<\overline V + \eta\}\, .
\end{align}
In our decomposition we want to first distinguish between whether $B_1(p)$ is a $b$-ball, $c$-ball, $d$-ball, or $e$-ball.  In the cases where it is either a $b$-ball or $e$-ball we are of course done, therefore we can assume $B_1(p)$ is either a $c$-ball or $d$-ball.  The handling of the two cases is almost verbatim, therefore we will assume $B_1(p)$ is a $c$-ball.  In this case, we can apply Proposition \ref{p:neck_decomp:c_ball} in order to build the covering
\begin{align}
	B_1 \subseteq \cC_0\cup \cN \cup \bigcup_b B_{r_b}(x_b)\cup \bigcup_c B_{r_c}(x_c)\cup \bigcup_d B_{r_d}(x_d) \cup \bigcup_e B_{r_e}(x_e)\, ,
\end{align}
where $\cN\subseteq B_{11}(q)$ is a $(\delta,\tau)$-neck for some $q\in B_4(p)$ and we have the content estimates
\begin{align}
&\sum_b r_b^{n-4} + \sum_d r_d^{n-4} + \sum_e r_e^{n-4} \leq C(n,\tau)\, ,\notag\\
&\sum_c r_c^{n-4} \leq C(n,\tau)\epsilon\, .
\end{align}
What we have gained from the above is that the remaining $c$-balls have small $n-4$-content.  We will need to remove both the $c$-balls and $d$-balls in this decomposition before the proof is complete.  Let us first deal with the $d$-balls.  Indeed, if we apply Proposition \ref{p:neck_decomp:d_ball} to each $d$-ball $B_{r_d}(x_d)$ then we obtain the covering
\begin{align}
	B_1(p) \subseteq \tilde\cS\cup\cC_0\cup \cN \cup \bigcup_b B_{r_b}(x_b)\cup \bigcup_c B_{r_c}(x_c) \cup \bigcup_e B_{r_e}(x_e)\, ,
\end{align}
where $\tilde\cS = \bigcup_d \tilde\cS_d$ is a countable union of $n-4$-measure zero sets, and thus itself has $n-4$ measure zero.  We also now have the estimates
\begin{align}
    &\sum_c r_c^{n-4}\leq C(n,\tau)\cdot \epsilon\, ,\notag\\
	&\sum_b r_b^{n-4} + \sum_e r_e^{n-4} \leq C(n,\tau)\, .
\end{align}
Let us remark that our only original assumption to construct this covering was that $B_1(p)$ was a $c$-ball.  Therefore, we can repeat this construction on each $c$-ball $B_{r_c}(x_c)$ in order to build a covering
\begin{align}
	B_1(p) \subseteq \tilde\cS\cup\bigcup_a \big(\cC_{0,a}\cup\cN_a\cap B_{r_a}\big) \cup \bigcup_b B_{r_b}(x_b)\cup \bigcup_c B_{r_c}(x_c) \cup \bigcup_e B_{r_e}(x_e)\, ,
\end{align}
with the estimates
\begin{align}
    &\sum_a r_a^{n-4}\leq 1+C(n,\tau)\epsilon\, ,\notag\\
    &\sum_c r_c^{n-4}\leq \Big(C(n,\tau)\cdot \epsilon\Big)^2\, ,\notag\\
	&\sum_b r_b^{n-4} + \sum_e r_e^{n-4} \leq C(n,\tau)\Big(1+C(n,\tau)\epsilon\Big)\, .
\end{align}
If we continue to recover the $c$-balls, then after $i$ iterations we have the covering
\begin{align}
	B_1(p) \subseteq \tilde\cS\cup\bigcup_a \big(\cC_{0,a}\cup\cN_a\cap B_{r_a}\big) \cup \bigcup_b B_{r_b}(x_b)\cup \bigcup_c B_{r_c}(x_c) \cup \bigcup_e B_{r_e}(x_e)\, ,
\end{align}
with the estimates
\begin{align}
    &\sum_a r_a^{n-4}\leq \sum_{j=0}^i \Big(C(n,\tau)\epsilon\Big)^j\, ,\notag\\
    &\sum_c r_c^{n-4}\leq \Big(C(n,\tau)\cdot \epsilon\Big)^i\, ,\notag\\
	&\sum_b r_b^{n-4} + \sum_e r_e^{n-4} \leq C(n,\tau)\sum_{j=0}^i \Big(C(n,\tau)\epsilon\Big)^j\, .
\end{align}
Now we may consider the discrete sets $\tilde\cS_c^i\bigcup_c \{x^i_c\}$ and their Hausdorff limit $\tilde\cS_c=\lim \tilde\cS^i_c$.  Arguing as in Proposition \ref{p:neck_decomp:d_ball} we see for $\epsilon\leq \epsilon(n,\tau)$ that $H^{n-4}(\tilde\cS_c)=0$.  Combining this with our previous $\tilde\cS$ set we then arrive at the covering
\begin{align}
	B_1(p) \subseteq \tilde\cS\cup\bigcup_a (\cC_{0,a}\cup\cN_a\cap B_{r_a}\big) \cup \bigcup_b B_{r_b}(x_b)\cup \bigcup_e B_{r_e}(x_e)\, ,
\end{align}
which for $\epsilon\leq \epsilon(n,\tau)$ satisfies
\begin{align}
    &H^{n-4}(\tilde\cS)=0\, ,\notag\\
    &\sum_a r_a^{n-4}\leq 2\, ,\notag\\
	&\sum_b r_b^{n-4} + \sum_e r_e^{n-4} \leq C(n,\tau)\, ,
\end{align}
which completes the proof of the Lemma.
\end{proof}
\vspace{.5 cm}

Let us now apply the above Lemma in order to prove the following corollary, which is the main result of this section:\\

\begin{proposition}[Inductive Covering]\label{p:neck_decomp:inductive_cover}
Let $(M^n_i,g_i,p_i)\to (X,d,p)$ satisfy $|\Ric|\leq n-1$ with $V\equiv \inf_{y\in B_4(p)}\cV_1 >\rv>0$.  Let us fix $0<\delta$ and $0<\tau\leq \tau(n)$, then there exists $v_0(n,\rv,\delta,\tau)>0$ such that we can write $$B_1 \subseteq \tilde\cS\cup\bigcup_a \big(\cC_{0,a}\cup\cN_a\cap B_{r_a}\big) \cup \bigcup_b B_{r_b}(x_b)\cup \bigcup_v B_{r_v}(x_v)\, ,$$ where
\begin{enumerate}
    \item[(a)] $\cN_a\subseteq B_{2r_a}(x_a)$ is a $(\delta,\tau)$-neck region with associated singular set $\cC_{0,a}$,
	\item[(b)] For each $b$-ball $B_{r_b}(x_b)$ we have $r_h(x_b)>2r_b$,
	\item[(d)] For each $v$-ball $B_{r_v}(x_v)$, if $V_v\equiv \inf_{y\in B_{4r_v}(x_v)} \cV_{r_v}$ then $V_v \geq V+v_0$.
	\item[(s)] $\tilde\cS\subseteq \cS(X)$ with $H^{n-4}(\tilde\cS)=0$.
\end{enumerate}
Further, we have the content bound $\sum_a r_a^{n-4}+\sum_b r_b^{n-4}+\sum_v r_v^{n-4}<C(n,\tau,\delta)$.
\end{proposition}
\begin{proof}
This Proposition is very similar to the previous Lemma, we essentially need only one additional covering in order to deal with the fact that $V$ here is coming from the volume pinching at scale one, instead of scale $\eta^{-1}$.

Now more precisely, let us begin by picking $\eta$ as in Propositions \ref{p:neck_decomp:c_ball} and \ref{p:neck_decomp:d_ball} and let us cover $B_1(p)$ by the Vitali covering
\begin{align}
B_1(p)\subseteq \bigcup_f B_{\eta^2}(x_f)\, ,	
\end{align}
where $\{B_{\eta^2/10}(x_f)\}$ are disjoint.  By a standard covering argument we have that there are at most $C(n,\tau,\delta)$ balls in this collection.  For each ball $B_{\eta^2}(x_f)$ let us apply Lemma \ref{l:neck_decomp:inductive_cover} to find the covering given by
\begin{align}
	B_1(p)\subseteq \bigcup B_{\eta^2}(x_f) \subseteq \tilde\cS\cup\bigcup_a \big(\cC_{0,a}\cup\cN_a\cap B_{r_a}\big) \cup \bigcup_b B_{r_b}(x_b)\cup \bigcup_e B_{r_e}(x_e)\, ,
\end{align}
	such that we have the content bounds $\sum_a r_a^{n-4}+\sum_b r_b^{n-4}+\sum_e r_e^{n-4}<C(n,\tau,\delta)$ and such that $r_e\leq \eta^4$.  We will see that these $e$-balls will be the $v$-balls of the Proposition with $v_0$ defined appropriately.  To see this, let us be careful and note that the $e$-balls in this covering are with respect to one of the balls $B_{\eta^2}(x_f)$.  Therefore for each $e$-ball we have the estimate
\begin{align}
\inf_{B_{4r_e}(x_e)}\cV_{\eta r_e}(y)	\geq \inf_{B_{4\eta^2}(x_f)}\cV_{\eta}(y)+\eta\geq \inf_{B_{(1+4\eta^2)}(p)}\cV_{1}(y)+\eta.
\end{align}

In order to finish the proof let us notice the simple estimate
\begin{align}
		\inf_{B_{(1+4\eta^2)}(p)}\cV_{1}(y)&\geq \frac{\Vol(B^{-1}_{1})}{\Vol(B^{-1}_{(1+4\eta^2)})}\inf_{B_{1}(p)}\cV_{1}(y)>(1-c(n)\eta^2)V>V-\frac{1}{2}\eta\, ,
\end{align}
	where $\Vol(B^{-1}_{r})$ is the volume of the ball of radius $r$ in hyperbolic space, and the last inequality holds for $\eta<\eta(n)$.  Combining this with the previous inequality we arrive at the estimate
\begin{align}
	\inf_{B_{4r_e}(x_e)}\cV_{\eta r_e}(y)	\geq \inf_{B_{(1+4\eta^2)}(p)}\cV_{1}(y)+\eta \geq V +\frac{1}{2}\eta \equiv V+v_0\, .
\end{align}
	Choosing a Vitali covering $\{B_{r_v}(x_v),x_v\in B_{r_e}(x_e)\}$ with $r_v=\eta r_e$ of $B_{r_e}(x_e)$, we thus finish the proof of the Proposition.
\end{proof}
\vspace{.5cm}

\subsection{Proof of Neck Decomposition of Theorem \ref{t:neck_decomposition}}

In this subsection we finish the proof of the Neck Decomposition Theorem.  The proof will conclude by recursively applying the inductive covering of Proposition \ref{p:neck_decomp:inductive_cover} a finite number of times until we have our desired covering.\\

More precisely, let $M$ satisfy the assumptions of the Theorem, and let $V\equiv \inf_{B_4(p)}\cV_1(y)>0$.  Let us then fix $v_0(n,V,\delta,\tau)$ as in Proposition \ref{p:neck_decomp:inductive_cover}.  Then if we apply Proposition \ref{p:neck_decomp:inductive_cover} we arrive at the covering
\begin{align}
	B_1 \subseteq \tilde\cS^1\cup\bigcup_a \big((\cC^1_{0,a}\cup\cN^1_a)\cap B_{r^1_a}\big) \cup \bigcup_b B_{r^1_b}(x^1_b)\cup \bigcup_v B_{r^1_v}(x^1_v)\, ,
\end{align}
such that we have the content estimates
\begin{align}
	\sum_a (r^1_a)^{n-4}+\sum_b (r^1_b)^{n-4}+\sum_v (r^1_v)^{n-4}<C(n,V,\tau,\delta)\, .
\end{align}
and such that for each $v$-ball $B_{r^1_v}(x^1_v)$ we have
\begin{align}
		\inf_{B_{4r^1_v}(x^1_v)}\cV_{r^1_v}(y) > V+v_0\, .
\end{align}

Now let us observe that we may apply Proposition \ref{p:neck_decomp:inductive_cover} to each $v$-ball $B_{r^1_v}(x^1_v)$ itself.  In this case, we obtain a new covering of $B_1(p)$ given by
\begin{align}
	B_1 &\subseteq \tilde\cS^1\cup\bigcup_a \big((\cC^1_{0,a}\cup\cN^1_a)\cap B_{r^1_a}\big)\cup \bigcup_b B_{r^1_b}(x^1_b)\cup \tilde\cS^2\cup\bigcup_a \big((\cC^2_{0,a}\cup\cN^2_a)\cap B_{r^2_a}\big) \cup \bigcup_b B_{r^2_b}(x^2_b)\cup\bigcup_v B_{r^2_v}(x^2_v)\, ,\notag\\
	&=\bigcup_{j\leq 2}\tilde\cS^j\cup \bigcup_{j\leq 2}\bigcup_a \big((\cC^j_{0,a}\cup\cN^j_a)\cap B_{r^j_a}\big)\cup \bigcup_{j\leq 2}\bigcup_b B_{r^j_b}(x^j_b) \cup \bigcup_v B_{r^2_v}(x^2_v)\, ,
\end{align}
such that we have the content estimates
\begin{align}
	\sum_a (r^2_a)^{n-4}+\sum_b (r^2_b)^{n-4}+\sum_v (r^2_v)^{n-4}<C(n,V,\tau,\delta)\sum_v (r^1_v)^{n-4}\leq C(n,V,\tau,\delta)^2\, ,
\end{align}
and such that for each $v$-ball $B_{r^2_v}(x^2_v)$ we have
\begin{align}
		\inf_{B_{4r^2_v}(x^2_v)}\cV_{r^2_v}(y) > V+2v_0\, .
\end{align}
By continuing this scheme and applying Proposition \ref{p:neck_decomp:inductive_cover} to the $v$-balls $i$ times we arrive at a covering
\begin{align}
	B_1 &\subseteq \bigcup_{j\leq i}\tilde\cS^j\cup \bigcup_{j\leq i}\bigcup_a \big((\cC^j_{0,a}\cup\cN^j_a)\cap B_{r^j_a}\big)\bigcup_{j\leq i}\bigcup_a \big(\cN^j_a\cap B_{r^j_a}\big) \cup \bigcup_{j\leq i}\bigcup_b B_{r^j_b}(x^j_b) \cup \bigcup_v B_{r^i_v}(x^i_v)\, ,
\end{align}
such that we have the content estimates
\begin{align}
	\sum_{j\leq i}\sum_a (r^j_a)^{n-4}+\sum_{j\leq i}\sum_b (r^j_b)^{n-4}+\sum_v (r^i_v)^{n-4}\leq C(n,V,\tau,\delta)^i\, ,
\end{align}
and such that for each $v$-ball $B_{r^i_v}(x^i_v)$ we have
\begin{align}
		\inf_{B_{4r^i_v}(x^i_v)}\cV_{r^i_v}(y) > V+iv_0\, .
\end{align}

In order to finish the proof let us observe that if $i>v_0^{-1}$ then there must be no $v$-balls, as in this case $\inf_{B_{4r^i_v}(x^i_v)}\cV_{r^i_v}(y) > V+iv_0>1$, which is not possible.  Thus for $i>v_0^{-1}$ we have built our desired covering and finished the proof of Theorem \ref{t:neck_decomposition}.

\vspace{1cm}

\section{Proof of $L^2$-Curvature Estimate}

In this short section we combine the neck region estimates of Theorem \ref{t:neck_region} with the neck decomposition of Theorem \ref{t:neck_decomposition} in order to finish the proof of Theorem \ref{t:main_L2_estimate}.  Indeed, for $\delta>0$ let us use Theorem \ref{t:neck_decomposition} and consider the covering
\begin{align}
	B_1(p)\subseteq \bigcup_a \big(\cN_a\cap B_{r_a}\big) \cup \bigcup_b B_{r_b}(x_b)\, ,
\end{align}
where $\cN_a\subseteq B_{2r_a}(x_a)$ is a $(\delta,\tau)$-neck, $r_h(x_b)\geq 2r_b$, and
\begin{align}\label{e:L2proof:content}
	\sum_a r_a^{n-4} + \sum_b r_b^{n-4}< C(n,\rv,\delta)\, .
\end{align}

Now since the Ricci curvature is uniformly bounded, as in the remark after Definition \ref{d:harmonic_radius} we have scale invariant $W^{2,2}$-estimates on the metric on $B_{r_b}(x_b)$.  In particular, we have the estimate
\begin{align}
r^{4-n}_b \int_{B_{r_b}(x_b)} |\Rm|^2(z)\,dz < C(n)\, .	
\end{align}

On the other hand, if we fix a $\delta=\delta(n,\rv)$, then by Theorem \ref{t:neck_region} we have the scale invariant $L^2$ estimate on each neck region $\cN_a$ given by:
\begin{align}
r^{4-n}_a \int_{\cN_a\cap B_{r_a}} |\Rm|^2(z)\,dz < C(n)\, .	
\end{align}

Combining these estimates with the content estimate of \eqref{e:L2proof:content} we get the estimate
\begin{align}
\int_{B_1(p)} |\Rm|^2(z)\,dz &\leq \sum_a \int_{\cN_a\cap B_{r_a}} |\Rm|^2(z)\,dz + \sum_b \int_{B_{r_b}(x_b)}	|\Rm|^2(z)\,dz\, \notag\\
&\leq C(n)\Big( \sum_a r_a^{n-4} + \sum_b r_b^{n-4}\Big) \leq C(n,\rv)\, ,
\end{align}
which completes the proof of Theorem \ref{t:main_L2_estimate}. $\square$ \\

\section{Proof of Theorem \ref{t:main_L4/3_estimate}}

The proof relies on an improved Kato inequality.\\

\textbf{Claim:} there exist $\eta(n)>0$ and $C(n)>0$ such that 
\begin{align}
|\nabla |\Rm||^2\le \big(1-\eta(n)\big)|\nabla\Rm|^2+C(n)|\nabla \Ric|^2\,.
\end{align}
Let first us assume the claim and prove the theorem. By direct computation and the Kato inequality,  we have
\begin{align}\label{e:Delta_Rm}
\Delta \sqrt{|\Rm|^2+1}&=\frac{\Delta |\Rm|^2}{2\sqrt{|\Rm|^2+1}}-\frac{|\nabla|\Rm|^2|^2}{4(|\Rm|^2+1)^{3/2}}\\
&\ge \frac{\eta(n)|\nabla \Rm|^2}{\sqrt{|\Rm|^2+1}}-\frac{C(n)|\nabla \Ric|^2}{\sqrt{|\Rm|^2+1}}-C(n)|\Rm|^2+\frac{\nabla^2\Ric\ast \Rm}{\sqrt{|\Rm|^2+1}}\\
&\ge  \frac{\eta(n)|\nabla \Rm|^2}{\sqrt{|\Rm|^2+1}}-{C(n)|\nabla \Ric|^2}-C(n)|\Rm|^2+\frac{\nabla^2\Ric\ast \Rm}{\sqrt{|\Rm|^2+1}}\, .
\end{align}
where one can find an explicit formula of the quadratic term $\nabla^2 \Ric\ast \Rm$ in Proposition 2.4.1 of \cite{Topping}.
Choose a cutoff function $\phi$ as in \cite{ChC1} such that $\phi\equiv 1$ on $B_1(p)$ and $\phi\equiv 0$ outside $B_2(p)$ with $|\nabla \phi|+|\Delta\phi|\le C(n)$. Multiplying $\phi$ to (\ref{e:Delta_Rm}) and integrating by parts, we have
\begin{align}
\int_{B_1(p)}\frac{|\nabla \Rm|^2}{\sqrt{|\Rm|^2+1}}(z)\,dz\le C(n)\int_{B_2(p)}\left(|\nabla \Ric|^2+\sqrt{|\Rm|^2+1}+|\Rm|^2\right)(z)\,dz\, ,
\end{align}
where we use the integrating by parts and the Cauchy inequality to handle the term $\frac{\nabla^2\Ric\ast \Rm}{\sqrt{|\Rm|^2+1}}$.
On the other hand, by Cauchy inequality, we have
\begin{align}
\int_{B_1(p)}|\nabla \Rm|^{4/3}(z)\,dz\le \left(\int_{B_1(p)}\frac{|\nabla \Rm|^2}{\sqrt{|\Rm|^2+1}}(z)\,dz\right)^{2/3}\left(\int_{B_1(p)}{(|\Rm|^2+1)}(z)\,dz\right)^{1/3}\,.
\end{align}
By the $L^2$ curvature estimate of Theorem \ref{t:main_L2_estimate}, we have proved Theorem \ref{t:main_L4/3_estimate} by assuming Claim.

Now we wish to prove the claim, which is essentially just a repeated application of the second Bianchi identity.  For any fixed $y\in M$, choose a normal coordinate such that $|\nabla |\Rm||(y)=\partial_1 |\Rm|(y)$. Then at $y$, we have
\begin{align}\label{e:nablaRm}
|\nabla |\Rm|^2|=|\partial_1|\Rm|^2|=2|\langle\nabla_1\Rm,\Rm\rangle|\le 2\left(\sum_{ijkl}R_{ijkl,1}^2\right)^{1/2}\left(\sum_{ijkl}R_{ijkl}^2\right)^{1/2}\,.
\end{align}
To prove the claim, it suffices to show $(1+\sigma(n))\sum_{ijkl}R_{ijkl,1}^2\le \sum_{ijklp}R_{ijkl,p}^2+C(n)\sum_{ijp}{R}_{ij,p}^2$ for some dimensional constant $\sigma(n)$ and $C(n)$. In fact, it suffices to show
$\sum_{ijkl}R_{ijkl,1}^2\le C(n)\left(\sum_{ijkl}\sum_{p\ge 2}R_{ijkl,p}^2+\sum_{ijp}{R}_{ij,p}^2 \right):=C(n)\Xi$. Hence we only need to prove $R_{\alpha\beta\gamma\delta,1}^2\le C(n)\Xi$ for any fixed $\alpha,\beta,\gamma,\delta=1,\cdots,n$.

\noindent Case 1: For $\alpha,\beta,\gamma,\delta\ge 2$, then by second Bianchi identity and Cauchy inequality, we have
\begin{align}
R_{\alpha\beta\gamma\delta,1}^2=\left(R_{\alpha\beta\delta 1,\gamma}+R_{\alpha\beta 1\gamma,\delta}\right)^2\le 2\left(R_{\alpha\beta\delta 1,\gamma}^2+R_{\alpha\beta 1\gamma,\delta}^2\right)\le 2\Xi\, .
\end{align}
Case 2: For $\alpha=1$ and $\beta,\gamma,\delta\ge 2$, then by second Bianchi identity and Cauchy inequality, we have
\begin{align}
R_{1\beta\gamma\delta,1}^2=\left((R_{1\beta\delta 1,\gamma}+R_{1\beta 1\gamma,\delta}\right)^2\le 2\left(R_{1\beta\delta 1,\gamma}^2+R_{1\beta 1\gamma,\delta}^2\right)\le 2\Xi\, .
\end{align}
Case 3: For $\alpha=\delta=1$ and $\beta,\gamma\ge 2$, then by Cauchy inequality, we have
\begin{align}
R_{1\beta\gamma 1,1}^2=\left({R}_{\beta\gamma,1}-\sum_{\alpha\ge 2}R_{\alpha\beta\gamma\alpha,1}\right)^2\le n\left({R}_{\beta\gamma,1}^2+\sum_{\alpha\ge 2}R_{\alpha\beta\gamma\alpha,1}^2\right)\le n\left(\Xi+2\sum_{\alpha\ge 2}\Xi\right)\le n(2n-1)\Xi\, ,
\end{align}
where we have used the estimates in Case 1 to the second inequality.

Thus, by the symmetric relations of curvature tensor, we have proved $R_{\alpha\beta\gamma\delta,1}^2\le C(n)\Xi$. Hence, we prove $\sum_{ijkl}R_{ijkl,1}^2\le C(n)\left(\sum_{ijkl}\sum_{p\ge 2}R_{ijkl,p}^2+\sum_{ijp}{R}_{ij,p}^2 \right)$.  This finishes the proof of the claim and then we complete the proof of Theorem \ref{t:main_L4/3_estimate}.
\vspace{1cm}

\section{Proof of Theorem \ref{t:main_limits}}

We now finish the proof of Theorem \ref{t:main_limits}.  To see that the singular set is $n-4$ rectifiable we use the singular neck decomposition of Theorem \ref{t:neck_decomposition} in order to write
\begin{align}
	&B_1(p)\cap \cS(X) \subseteq \bigcup_a \big(\cC_{0,a}\cap B_{r_a}\big)\cup \tilde S_\delta(X)\, ,
\end{align}
where $\cC_{0,a}\subseteq B_{2r_a}$ is the singular set of a $(\delta,\tau)$-neck region $\cN_a\subseteq B_{2r_a}$ and $\tilde \cS_\delta$ has $n-4$ measure zero.  Since by Theorem \ref{t:neck_region} we have that each $\cC_{0,a}$ is rectifiable, we therefore have that $\cS(X)$ is rectifiable as well, as claimed.

In order to prove $H^{n-4}(\cS(X))<C(n,\rv)$, let us observe by Theorem \ref{t:neck_region} that for $\tau<\tau(n)$ and $\delta<\delta(n,\rv)$ we have $H^{n-4}(\cC_{0,a}\cap B_{r_a})<C(n,\rv) r_a^{n-4}$, and by Theorem \ref{t:neck_decomposition} we have that $\sum r_a^{n-4}\leq C(n,\rv,\delta,\tau)\leq C(n,\rv)$.  Combining these we get the estimate
\begin{align}
	H^{n-4}(\cS(X)\cap B_1)\leq \sum_a H^{n-4}(\cC_{0,a}\cap B_{r_a})\leq C(n,\rv)\,\sum_a r_a^{n-4}\leq C(n,\rv)\, ,
\end{align}
which proves the desired finiteness estimate. \\

Finally, let us show that $n-4$ a.e. tangent cone is unique and isometric to $\dR^{n-4}\times C(S^3/\Gamma)$.  Indeed, by the neck decomposition we have for each $\delta<\delta(n,\rv)$ that $x\in \cS(X)\setminus \tilde \cS_\delta$ lives in some $\cC_{0,a}\cap B_{r_a}$, where $\cC_{0,a}$ is the singular set of a $(\delta,\tau)$-neck region $\cN_a\subseteq B_{2r_a}$.  From this we immediately get for any such $x\in \cS(X)\setminus \tilde \cS_\delta$ that {\it every} tangent cone is $\delta$-GH close to $\dR^{n-4}\times C(S^3/\Gamma)$.  Now $H^{n-4}(\tilde \cS_\delta)=0$, so let us define $\tilde \cS\equiv \bigcup_j \tilde\cS_{2^{-j}}$.  Note that $H^{n-4}(\tilde \cS)=0$ and for each $x\in \cS(X)\setminus \tilde \cS$ we must therefore have that every tangent cone is actually isometric to $\dR^{n-4}\times C(S^3/\Gamma)$.  This finishes the proof of Theorem \ref{t:main_limits}.
\vspace{1cm}

\section*{Acknowlegments}
The first author would like to thank his advisor Gang Tian for constant encouragement and for useful conversations during this work. Partial work was done while the first author was visiting the Mathematics Department at Northwestern University and he would like to thank the department for its hospitality and for providing a good academic environment. The first author was partially supported by  NSFC (No. 11701507, No. 12071425) and ARC DECRA190101471. The second author would like to thank NSF for its support under grant DMS-1406259.

\bibliographystyle{plain}

\end{document}